%% file: grand-Schnyder.tex
\documentclass[letter]{amsart}
\input{defs.tex}

\begin{document}

\author[Olivier Bernardi, \'Eric Fusy, and Shizhe Liang]{Olivier Bernardi$^{*}$ \and \'{E}ric Fusy$^{\dagger}$ \and Shizhe Liang$^{+}$}
\thanks{$^{*}$Department of Mathematics, Brandeis University, Waltham MA, USA,
bernardi@brandeis.edu.\\
$^{\dagger}$LIGM/CNRS, Universit\'e Gustave Eiffel, Champs-sur-Marne, France, eric.fusy@univ-eiffel.fr.\\
$^{+}$Department of Mathematics, Brandeis University, Waltham MA, USA, Shizhe1011@brandeis.edu.\\
}

\title[Grand Schnyder Woods]{Grand Schnyder Woods}
\date{\today}

\begin{abstract}
We define a far-reaching generalization of Schnyder woods which encompasses many classical combinatorial structures on planar graphs.

\emph{Schnyder woods} are defined for planar triangulations as certain triples of spanning trees covering the triangulation and crossing each other in an orderly fashion. They are of theoretical and practical importance, as they are central to the proof that the order dimension of any planar graph is at most 3, and they are also underlying an elegant drawing algorithm.
In this article we extend the concept of Schnyder wood well beyond its original setting: for any integer $d\geq 3$ we define a ``grand-Schnyder'' structure for (embedded) planar graphs which have faces of degree at most $d$ and non-facial cycles of length at least $d$. We prove the existence of grand-Schnyder structures, provide a linear construction algorithm, describe 4 different incarnations (in terms of tuples of trees, corner labelings, weighted orientations, and marked orientations), and define a lattice for the set of grand Schnyder structures of a given planar graph. We show that the grand-Schnyder framework unifies and extends several classical constructions: Schnyder woods and Schnyder decompositions, regular edge-labelings (a.k.a. transversal structures), and Felsner woods.
\end{abstract}

\maketitle


\tableofcontents


\section{Introduction}\label{sec:intro}
\input{intro}

\section{Notation and background}\label{sec:notation}
\input{notation}

\section{Incarnations of grand-Schnyder structures}\label{sec:incarnations}
\input{incarnation}




\section{Main results}\label{sec:statements}
\input{main-results}

\section{A further incarnation of grand-Schnyder structures for a subclass of $d$-adapted maps}\label{sec:arc_labeling}
\input{edge-tight}

\section{Connections to previously known structures}\label{sec:classical}
\input{other_structures}




\section{Bipartite case and connection to known structures}\label{sec:bipartite}
\input{bipartite}
\input{other_structures_bipartite}



\section{Lattice structure}\label{sec:lattice}
\input{lattice}

\section{Dual grand-Schnyder structures}\label{sec:dual}
\input{dual}


\section{Proofs of GS-woods properties, and of the bijection with GS-labelings}\label{sec:remaining-proofs}
\input{technical-proofs-woods}

\section{Proof of the existence of grand-Schnyder structures}\label{sec:proof-existence} 
\input{existence-proof}

\section{Quasi-Schnyder structures}\label{sec:level2-GS}
\input{quasi_structures}



\section{Concluding remarks}\label{sec:conclusion}
\input{conclusion}

\bigskip

\noindent{\bf Acknowledgments.} Olivier Bernardi was partially supported by NSF Grant DMS-2154242. \'Eric Fusy was partially supported by the project ANR19-CE48-011-01 (COMBIN\'E), and the project ANR-20-CE48-0018 (3DMaps). 

\noindent{\bf Confict of Interest Statement.} On behalf of all authors, the corresponding author states that there is no conflict of interest.

\bibliographystyle{plain}
\bibliography{biblio-Schnyder}

\end{document}

%% file: defs.tex
\usepackage{amssymb,amsmath}
\usepackage{stmaryrd,paralist}
\usepackage{bbm}
\usepackage{graphics,graphicx}

\usepackage[latin1]{inputenc}
\usepackage{subfigure}
\usepackage[margin=1.4 in]{geometry}
\usepackage{hyperref,color}
\usepackage[dvipsnames]{xcolor}
\graphicspath{{Figures/}} 
\graphicspath{{./Figures/}}

\usepackage{url}

\newtheorem{theo}{Theorem}
\numberwithin{theo}{section}
\newtheorem{thm}[theo]{Theorem}
\newtheorem{lem}[theo]{Lemma}
\newtheorem{lemma}[theo]{Lemma}
\newtheorem{prop}[theo]{Proposition}
\newtheorem{claim}[theo]{Claim}
\newtheorem{Def}[theo]{Definition}
\newtheorem{definition}[theo]{Definition}
\newtheorem{cor}[theo]{Corollary}
\theoremstyle{remark}
\newtheorem{remark}[theo]{Remark}

\newenvironment{Remark}{
\refstepcounter{theo}
\smallbreak\noindent%
\textbf{Remark~\thetheo.}}%
{\medbreak}

\newcommand{\ds}{\displaystyle}

\newcommand{\cacher}[1]{}

\def\ZZ{\mathbb{Z}}
\def\NN{\mathbb{N}}

\def\cA{\mathcal{A}}
\def\cB{\mathcal{B}}
\def\cM{\mathcal{M}}
\def\cF{\mathcal{F}}

\def\cR{\mathcal{R}}
\def\cT{\mathcal{T}}
\def\cAL{\mathcal{AL}}
 
\def\cW{\mathcal{W}}
\def\cL{\mathcal{L}}

\def\cX{\mathcal{X}}

\def\bA{\mathbf{A}}
\def\bE{\mathbf{E}}
\def\bL{\mathbf{L}}
\def\bM{\mathbf{M}}
\def\bW{\mathbf{W}}
\def\bX{\mathbf{X}}
\def\bAL{\mathbf{AL}}
\def\bT{\mathbf{T}}

\def\bBA{\mathbf{BA}}
\def\bBL{\mathbf{BL}}
\def\bBM{\mathbf{BM}}
\def\bBW{\mathbf{BW}}

\def\ab{\alpha/\beta}

\def\hC{\hat{C}}

\def\Gb{G_\bullet}
\def\hGb{\widehat{G}_\bullet}
\def\bGb{\overline{G}_\bullet}
\def\tGb{\widetilde{G}_\bullet}

\def\cAb{\cA_\bullet}
\def\cBb{\cB_\bullet}
\def\bAb{\overline{\cA}_\bullet}
\def\hAb{\widehat{\cA}_\bullet}
\def\hom{\widehat \om}

\newcommand{\Id}{\textrm{Id}}

\newcommand{\al}{\alpha}
\newcommand{\be}{\beta}

\newcommand{\de}{\delta}
\newcommand{\om}{\omega}
\newcommand{\si}{\sigma}
\newcommand{\eps}{\epsilon}

\newcommand{\De}{\Delta}
\newcommand{\Th}{\Theta}

\newcommand{\ov}[1]{\overline{#1}}

\newcommand{\bal}{\ov{\al}}

\newcommand{\bPhi}{\ov{\Phi}}
\newcommand{\bPsi}{\ov{\Psi}}
\newcommand{\bTh}{\ov{\Th}}
\newcommand{\bLa}{\ov{\Lambda}}
\newcommand{\beeta}{\ov{\eta}}

\def\Gcr{G^{\times}}
\def\jp{\mathrm{jump}}
\def\Sigmacw{J_{\mathrm{cw}}}
\def\Sigmaccw{J_{\mathrm{ccw}}}

\newcommand{\fig}[3]{\begin{figure}[h!]\begin{center}\includegraphics[#1]{#2.pdf}\end{center}\caption{#3}\label{fig:#2}\end{figure}}

\newcommand{\ee}{\textbf{e}}
\newcommand{\ff}{\textbf{f}}
\newcommand{\vv}{\textbf{v}}

\newcommand{\sups}[1]{\textsuperscript{#1}}

\newcommand{\cwjump}{\textrm{cw-jump}}
\newcommand{\ccwjump}{\textrm{ccw-jump}}

\newcommand{\toward}{\textrm{in}}
\newcommand{\inweight}{\textrm{inweight}}
\newcommand{\pseudol}{\textrm{enclength}}

%% file: intro.tex

In 1989, Walter Schnyder showed that planar triangulations can be endowed with remarkable combinatorial structures, which now go by the name of \emph{Schnyder woods}~\cite{Schnyder:wood1}. A Schnyder wood of a planar triangulation (drawn without crossing in the plane) is a partition of its inner edges into three trees, crossing each other in a specific manner. A Schnyder wood is represented in Figure~\ref{fig:triangulation2}, and more details about the definition are given in the caption.

\fig{width=\linewidth}{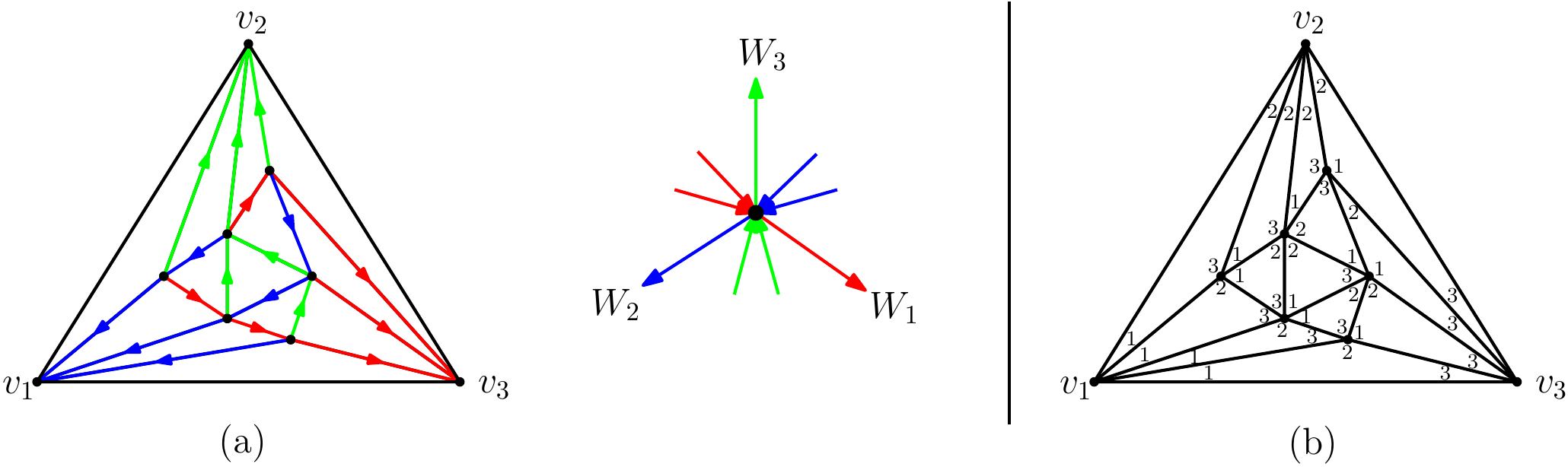}{(a) A Schnyder wood of a triangulation $G$, where the three trees are indicated by three colors.
A Schnyder wood of $G$ is a partition of the inner edges of $G$ into three trees $W_1$, $W_2$, $W_3$ satisfying two conditions: 
(1) 
 for all $i\in\{1,2,3\}$, the tree $W_i$ spans all the inner vertices and the outer vertex $v_{i-1}$ which is chosen as its root,  
(2) if the trees are oriented toward their roots, then in clockwise direction around each inner vertex one has: the outgoing edge of $W_1$, the ingoing edges of $W_3$, the outgoing edge of $W_2$, the ingoing edges of $W_1$, the outgoing edge of $W_3$, and finally the ingoing edges of $W_2$. 
(b) Encoding the Schnyder wood by a corner labeling.} 

Schnyder used the existence of Schnyder woods to show that the incidence poset of any planar triangulation has dimension at most 3, thereby completing the proof that graphs are planar if and only if their incidence poset has dimension at most 3~\cite{Schnyder:wood1}. Another application explored by Schnyder is the possibility of drawing planar graphs with straight edges~\cite{Schnyder:wood1} (reproving a fact established by Wagner~\cite{Wagner:straght-line-drawing}), and he showed further that this can be done with all vertices on the lattice points of a $(n+2)\times(n+2)$ grid~\cite{Schnyder:wood2}. Since then, numerous other applications have been found for Schnyder woods~\cite{louigi2017scaling,bonichon2006planar,aleardi2018array,dhandapani2010greedy,F01,Felsner:posets,gonccalves2012triangle,li2017schnyder,Poulalhon:triang-3connexe+boundary}.\\

In 1993, Xin He showed that triangulations of the square without separating 3-cycles can also be endowed with remarkable combinatorial structures~\cite{He93:reg-edge-labeling} called \emph{regular edge-labelings}. These structures were later rediscovered by the second author~\cite{Fu07b} who named them \emph{transversal structures}, and we will adopt this name throughout the article. There are several ways to encode transversal structures, and one of the encodings given in~\cite{Fu07b} is as a partition of the inner edges of the triangulation into two graphs, crossing each other in a specific manner. A transversal structure is represented in Figure~\ref{fig:transversal}. 

\fig{width=\linewidth}{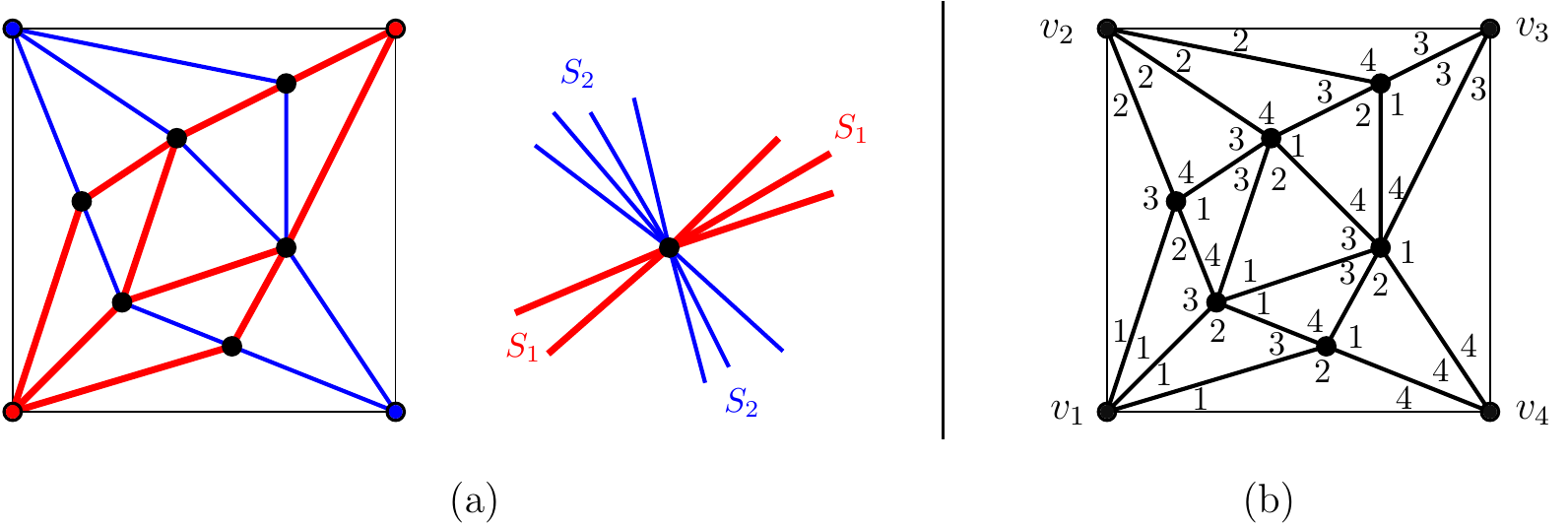}{(a) A transversal structure for a triangulation of the square $G$, where the two subgraphs are indicated using colors.
A transversal structure of $G$ is a partition of the inner edges into 2 subgraphs $S_1$ and $S_2$ satisfying 2 conditions (1) $S_1$ is incident to all inner vertices as well as the outer vertices $v_1$ and $v_3$, while $S_2$ is incident to all inner vertices as well as the outer vertices $v_2$ and $v_4$, (2) in clockwise order around each inner vertex of $G$ one has: a non-empty set of edges in $S_1$, a non-empty set of edges in $S_2$, a non-empty set of edges in $S_1$, and finally a non-empty set of edges in $S_2$. 
(b) Encoding the transversal structure by a corner labeling.}

Xin He~\cite{He93:reg-edge-labeling} used transversal structures to give an algorithm for realizing these triangulations as the contact graphs of rectangles (equivalently, drawing their dual in such a way that every face is a rectangle), and together with Goos Kant showed that the transversal structure and drawing can be computed in linear time~\cite{KantHe97:reg-edge-labeling-linear}. Since then, several other graph drawing algorithms based on transversal structures have been obtained~\cite{biedl2018embedding,eppstein2012area,felsner2013rectangle,felsnerFewLines,Fu07b}.\\

In 2012~\cite{OB-EF:Schnyder}, the first and second authors gave an analogue of Schnyder woods for \emph{$d$-angulations} (a planar graph drawn in the plane such that every face has degree $d$). This analogue, named \emph{$d$-Schnyder decomposition} is a $d$-tuple of spanning trees crossing each other in a specific manner and such that every inner edge belongs to exactly $d-2$ trees. A 5-Schnyder decomposition is shown in Figure~\ref{fig:pentagulation}. It is shown in~\cite{OB-EF:Schnyder} that a $d$-angulation admits a $d$-Schnyder decomposition if and only if its girth is $d$ (equivalently, it has no cycle of length less than $d$). The definition of 3-Schnyder decomposition (for triangulations) coincides with the classical definition of Schnyder woods. In~\cite{OB-EF:Schnyder} 4-Schnyder decompositions were also used to design a drawing algorithm for planar 4-valent graphs of min-cut 4. 

\fig{width=\linewidth}{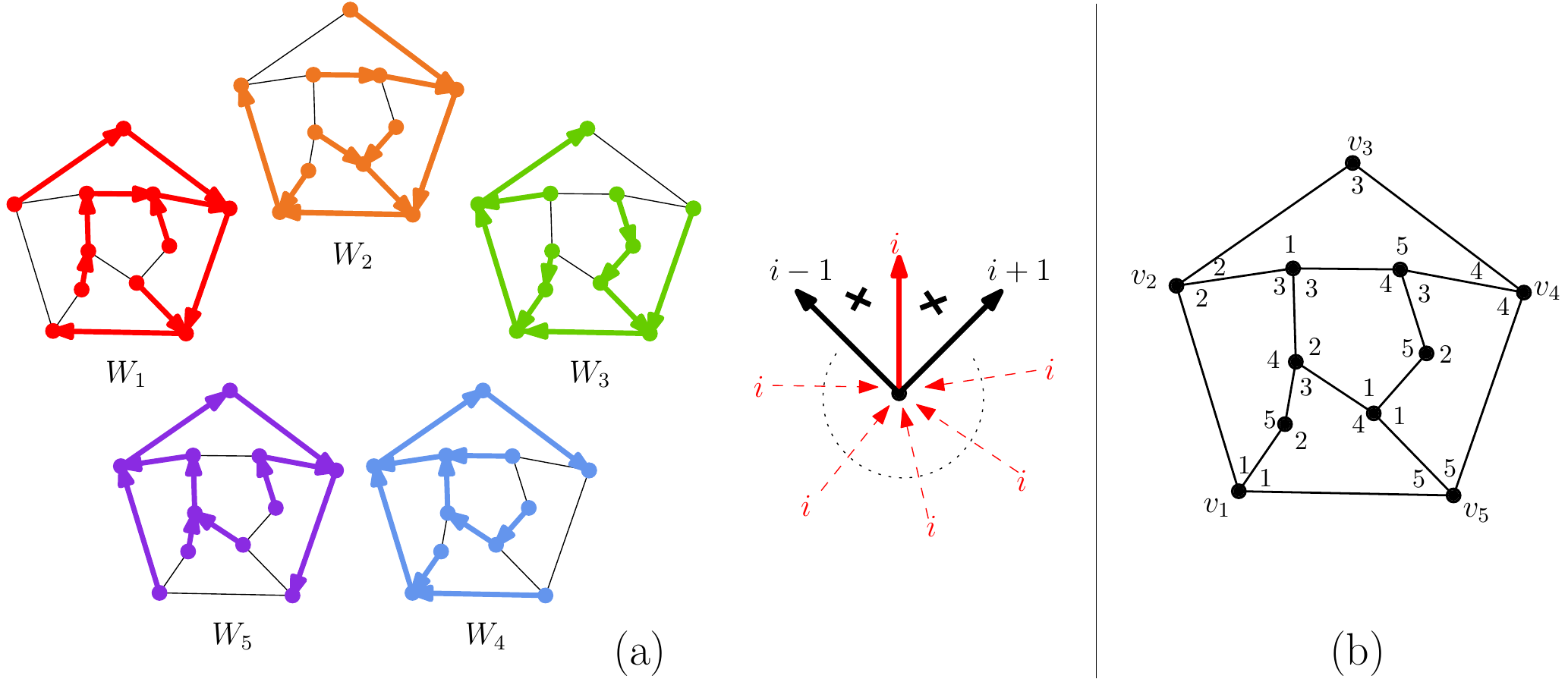}{(a) The five trees forming the 5-Schnyder decomposition. (b) Encoding the Schnyder decomposition by a corner labeling.}

At first sight, Schnyder woods and transversal structures may not appear to have much in common. However, we will show in this article that they can be given a common definition. One way to make the commonality more apparent, is to encode both structures by certain labelings of the corners. For Schnyder woods, this is a classical encoding, defined by Schnyder himself~\cite{Schnyder:wood1}, as a labeling of the corners of the triangulation with numbers in $\{1,2,3\}$ satisfying certain conditions. For transversal structures, we will define a similar encoding by a labeling of the corners of the triangulation with numbers in $\{1,2,3,4\}$. These corner labelings are indicated in Figures~\ref{fig:triangulation2}(b) and~\ref{fig:transversal}(b) respectively. With this ``labeling incarnation'' the conditions defining Schnyder woods and transversal structures look pretty similar.
\\

In this article we define a general combinatorial structure, the \emph{grand-Schnyder woods}, which put Schnyder woods and transversal structures under one roof.
 For $d\geq 3$, we call \emph{$d$-map} a connected planar graph drawn in the plane without edge crossing such that the outer face has degree $d$ (and is incident to $d$ distinct vertices) and the inner faces have degree at most $d$.
For a $d$-map $G$, a \emph{$d$-grand-Schnyder wood} is a $d$-tuple of spanning trees of $G$ crossing each other in a specific manner. An example is given in Figure~\ref{fig:4-grand-Schnyder}, and a precise definition is given in Section~\ref{sec:incarnations}.\\

When the $d$-map $G$ is a $d$-angulation, then the $d$-grand-Schnyder woods of $G$ coincide with the $d$-Schnyder decompositions of $G$ (up to minor differences in conventions). 
When the map $G$ is a triangulation of the square (a 4-map), then the 4-grand-Schnyder woods of $G$ are in bijection with the transversal structures of $G$. 
Hence grand-Schnyder woods are a far reaching generalization of both $d$-Schnyder decompositions and transversal structures. 
\\

One of our main results is that a $d$-map $G$ admits a $d$-grand-Schnyder wood if and only if all the non-facial cycles of $G$ have length at least $d$ (generalizing the existence results known for Schnyder decompositions and transversal structures~\cite{OB-EF:Schnyder,He93:reg-edge-labeling}). We call \emph{$d$-adapted} a $d$-map satisfying this condition.
In a forthcoming article~\cite{OB-EF-SL:4-GS-drawing}, we show that $4$-grand-Schnyder woods can be used for defining some graph-drawing algorithms. 
Schnyder decompositions and transversal structures have also been used to define bijections between classes of planar maps and classes of trees~\cite{albenque2013generic,Bernardi-Fusy:dangulations,Fu07b,FuPoScL,Poulalhon:triang-3connexe+boundary,Schaeffer:these}, and we plan to investigate whether these bijections can be extended thanks to the general framework of grand-Schnyder structures. 


As mentioned above, Schnyder woods and transversal structures have several incarnations. For instance, Schnyder woods can be encoded by either a triple of trees, or by a corner labeling. There are more incarnations, and this extends to the $d$-grand-Schnyder woods setting. More precisely, $d$-grand-Schnyder woods can be naturally encoded in four distinct ways.
\begin{compactitem}
\item As a $d$-tuple of trees crossing each other in a specific manner. We call such a structure \emph{$d$-grand-Schnyder wood}, or \emph{$d$-GS wood} for short.
\item As a labeling of the corners with values in $[d]:=\{1,2,\ldots,d\}$ satisfying certain local conditions. We call such a structure \emph{$d$-grand-Schnyder corner labeling}, or \emph{$d$-GS labeling} for short.
\item As a weighted orientation of $G$ together with marks at corners. We call such a structure \emph{$d$-grand-Schnyder marked orientation}, or \emph{$d$-GS marked orientation} for short.
\item As a weighted orientation of the \emph{angular map} of the $d$-map $G$ (the angular map of $G$ is obtained from $G$ by adding a vertex in each face of $G$ and joining that vertex to each vertex of $G$ incident to the face). We call such a structure a \emph{$d$-grand-Schnyder angular orientation}, or \emph{$d$-GS angular orientation} for short.
\end{compactitem}
These four incarnations are represented in Figure~\ref{fig:4-grand-Schnyder}.
We will define these three structures in Section~\ref{sec:incarnations}, and show that they are in bijection with each other in Section~\ref{sec:statements}. 

\fig{width=\linewidth}{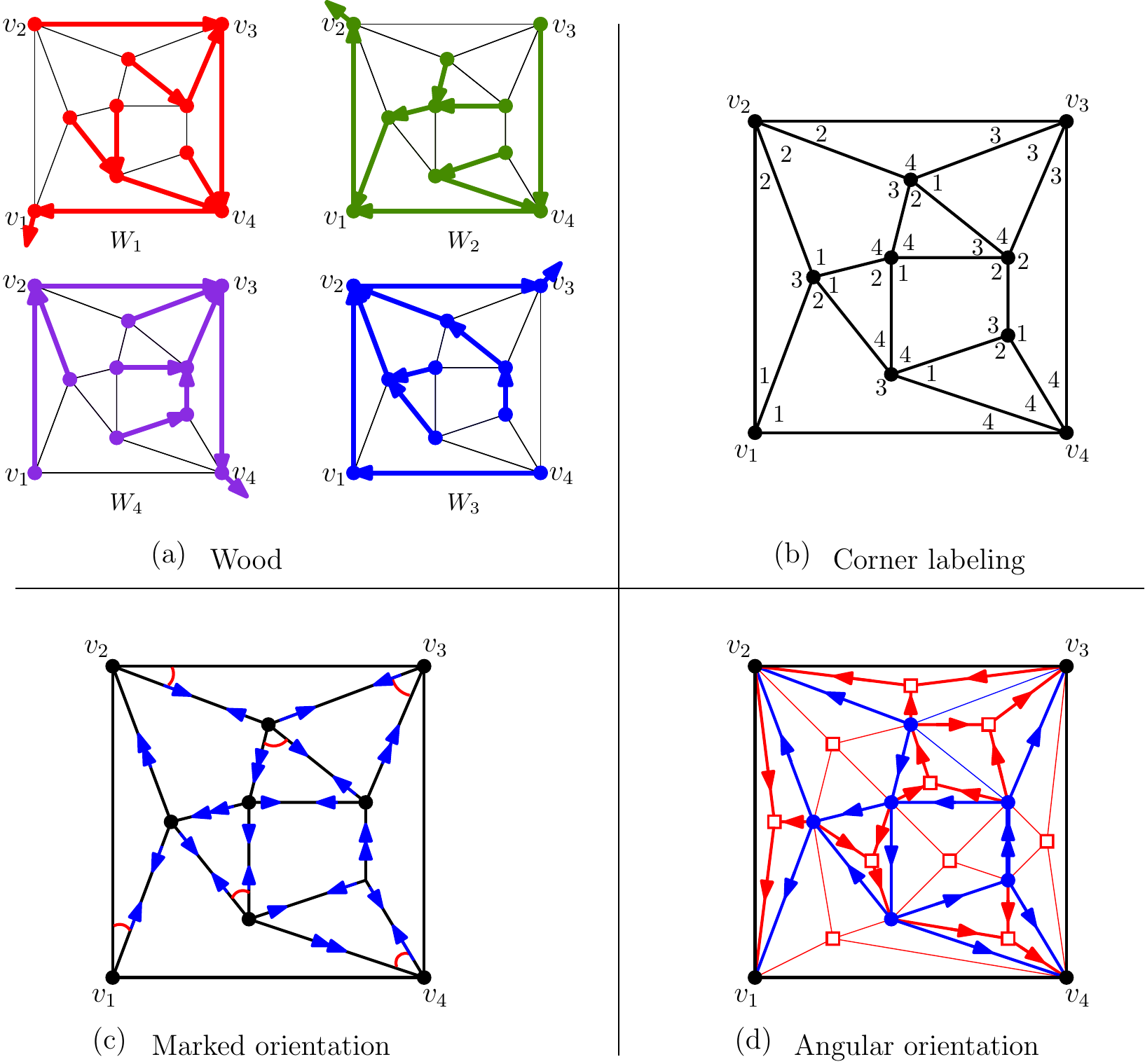}{Four incarnations of the same 4-grand-Schnyder structure.}

In Section~\ref{sec:dual}, we will also consider some incarnations of $d$-grand-Schnyder woods of a $d$-map $G$ as decorations of the dual graph $G^*$. This can be done either as a corner labeling of $G^*$ or as a $d$-tuple of spanning trees of $G^*$ crossing each other in a specific manner.\\

Schnyder woods and transversal structures are known to have two additional interesting properties. First, Schnyder woods and transversal structures are known to be computable in linear time~\cite{Schnyder:wood1,KantHe97:reg-edge-labeling-linear}. It was left as an open question in~\cite{OB-EF:Schnyder} to find a linear time algorithm for computing $d$-Schnyder decompositions for $d\geq 4$. We provide such an algorithm in Section~\ref{sec:proof-existence}. Precisely, for all $d\geq 3$, we give an algorithm for computing a $d$-GS wood for any $d$-map $G$ having all non-facial cycles of length at least $d$, in a number of operations which is linear in the number of vertices of $G$. 
Second, the set of Schnyder woods of any given triangulation is known to have a \emph{lattice structure} (in the sense of poset theory). This was first discovered in~\cite{Mendez:these} and reinterpreted in~\cite{Brehm:latticeSchnyder,Felsner:lattice}. Similarly, the set of transversal structures of a given triangulations of the square has a lattice structure~\cite{Fu07b}. In both cases, the covering relations in the lattice can be described in a simple local way. We generalize this lattice structure to the set of $d$-GS woods of a $d$-map and describe the covering relations in Section~\ref{sec:lattice}.\\

Schnyder woods and transversal structures are not the only combinatorial structures which can be captured by the grand-Schnyder framework. More precisely, there are additional structures which can be identified with \emph{bipartite grand-Schnyder woods}. For an even integer $d=2b$, the bipartite $d$-adapted maps admit a subclass of $d$-GS structures which we call \emph{even $d$-GS structures}. Even $d$-GS structures are a bit simpler than arbitrary $d$-GS structures, and after simplifications we arrive at the notion of \emph{$b$-bipartite grand-Schnyder structures} (or \emph{$b$-BGS} for short), which again have 4 different incarnations described in Section~\ref{sec:bipartite}. We will show that \emph{bipolar orientations} of 2-connected graphs, and \emph{Felsner woods} of 3-connected graphs~\cite{F01} can be identified with classes of bipartite grand-Schnyder structures (2-BGS and 3-BGS respectively).\\


Recall that a \emph{bipolar orientation} of a graph is an acyclic orientation with a unique source (vertex with no ingoing edge) and a unique sink (vertex with no outgoing edge). Given a planar graph $G$ drawn in the plane with 2 marked outer vertices $s,t$, one can associate a 4-angulation $Q_G$ by the process indicated in Figure~\ref{fig:2-orientations}(a), and any planar 4-angulation arise in this way. It is known that the bipolar orientations of $G$ with source $s$ and sink $t$ are in bijection with the \emph{2-orientations of $Q_G$}, that is, the orientations of the inner edges of $Q_G$ such that every inner vertex has outdegree 2. The correspondence is shown in Figure ~\ref{fig:2-orientations}(b). As we will see in Section~\ref{sec:bipartite}, $2$-orientations are one of the incarnations of $2$-BGS. This gives a bijection between the set of plane bipolar orientations and the set of $2$-BGS of quadrangulations.\\

 
\fig{width=.9\linewidth}{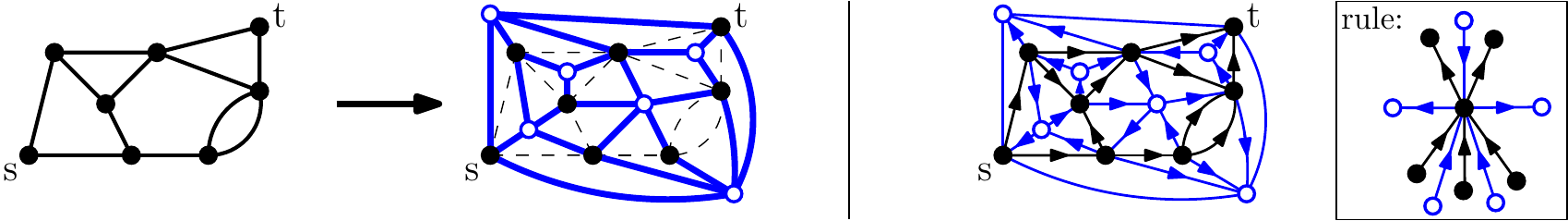}{Left: A connected planar map $M$ with two distinguished outer vertices, and the associated 4-angulation $G$. Right: a bipolar orientation of $M$ and the associated 2-orientation of $G$ (an incarnation of 2-BGS structures).}

Second, recall that there exists a generalization of Schnyder woods defined by Felsner~\cite{F01} (and independently in~\cite{di1999output}) for so-called \emph{suspended 3-connected plane graphs}.  
Let us call \emph{Felsner woods} this generalization of Schnyder woods (triangulations are a special case of suspended 3-connected plane graphs, and Felsner woods correspond to Schnyder woods in the case of triangulations). By the process indicated in Figure~\ref{fig:Felsner_woods_intro_bis}, one can associate to each suspended 3-connected plane graph $G$ a 4-angulation of the hexagon $M_G$ for which all non-facial cycles have length at least 6 (this process is almost a bijection). We will show in Section~\ref{sec:bip-classical} that the Felsner woods of $G$ 
are in bijection with the 3-BGS of $M_G$ (as we will explain, the 3-BGS of $M_G$ are closely related to some tricoloring of the edges of $M_G$ which is an incarnation of Felsner woods~\cite{F01} which was also discovered independently by Miller \cite{Miller:FelsnerWoods}). This gives a bijection between the set of Felsner woods and the set of $3$-BGS. \\

\fig{width=\linewidth}{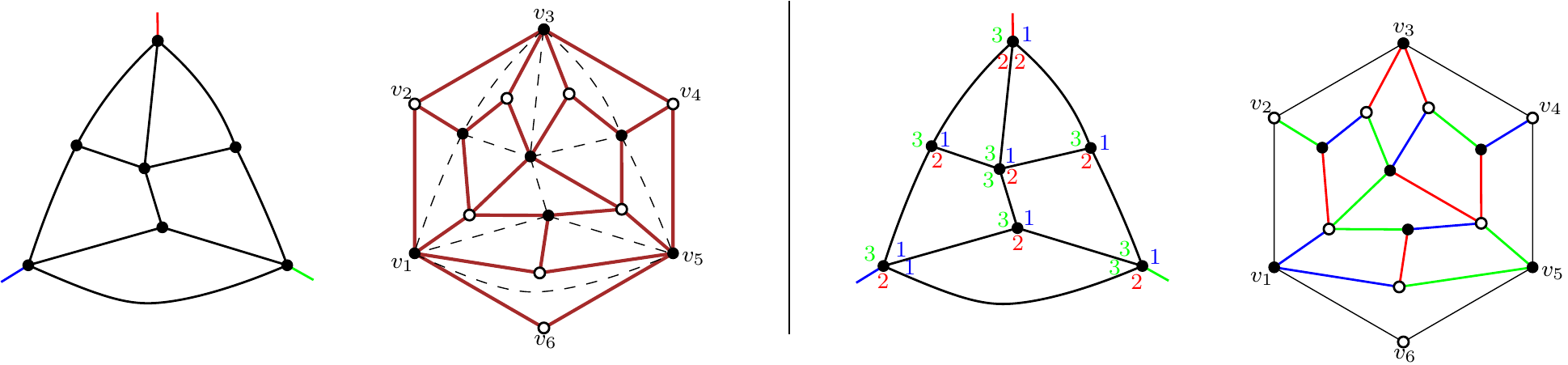}{Left: a suspended 3-connected plane graph $M$ (three outer vertices are marked), and the associated 4-angulation of the hexagon $G$. Right: a Felsner corner-labeling of $M$ (an incarnation of Felsner woods) 
and the corresponding tricoloration of inner edges of $G$ (an incarnation of 3-BGS structures).}

Lastly, in Section~\ref{sec:level2-GS} we will consider an extension of grand-Schnyder woods, called \emph{quasi-Schnyder structures}, which exist for $d$-maps which are not quite $d$-adapted. 
We say that a map is \emph{quasi $d$-adapted} if it is a $d$-map such that simple cycles of length less than $d$ are either facial cycles, or cycles of length $d-1$ containing a single edge and no vertex. \\ 
We show that a $d$-map with no face of degree less than $3$ admits a quasi-Schnyder wood if and only if it is quasi $d$-adapted. 
A quasi-Schnyder wood of a (quasi 5-adapted) triangulation of a 5-gon is represented in Figure~\ref{fig:5c_wood}. 
In this figure, the quasi-Schnyder structure is represented in terms of woods, whereas the other incarnations (in terms of orientations and labelings) 
will be discussed in Section~\ref{sec:level2-GS}. In \cite{OB-EF-SL:5QS-drawing} we focus on quasi-Schnyder woods of quasi $5$-adapted triangulations, discuss additional incarnations, and use these structures to define a graph-drawing algorithm (for triangulations of the 5-gon, the quasi 5-adapted condition is closely related to 5-connectedness).\\

Before we close this introduction, let us mention two links between Schnyder woods and transversal structures which have been previously established, but are not directly related to the present article. First, Kant and He showed in~\cite[Sec.4]{KantHe97:reg-edge-labeling-linear} that 4-connected triangulations admit a special kind of Schnyder woods, and that for $T$ any 4-adapted triangulation of the 4-gon, the transversal structures of $T$ can be mapped surjectively to the special Schnyder woods of the 4-connected triangulation obtained from $T$ by adding a diagonal in the outer 4-gon. 
Second Felsner, Schrezenmaier and Steiner~\cite{felsner2018equiangular} have considered families of orientations for planar triangulations of the $d$-gon (for $d\geq 3$) that are inspired by representations of triangulations by contact-systems of equi-angular $d$-gons. The definition of these orientations depends on the parity of $d$. The family of orientations obtained for odd $d$ correspond to Schnyder woods in the case $d=3$, while the family of orientations for even $d$ correspond to transversal structures for $d=4$. These orientations are not closely related to grand-Schnyder woods for higher values of $d$ however, and they are defined on triangulations of the $d$-gon without girth constraint for $d\geq 5$.


\fig{width=\linewidth}{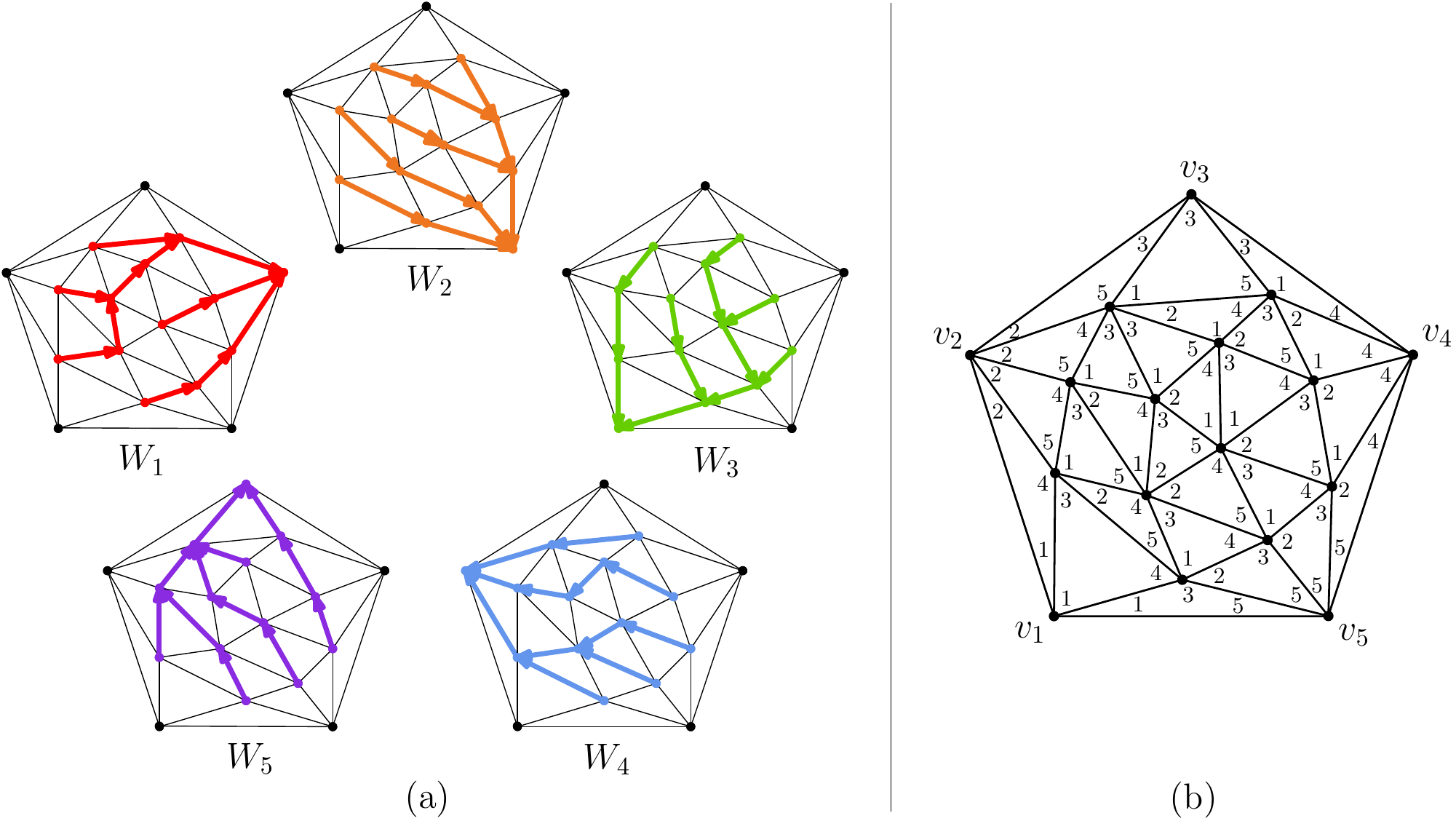}{(a) A 5-quasi-Schnyder wood of a (quasi 5-adapted) triangulation of a 5-gon. (b) Encoding the quasi-Schnyder structure by a corner labeling.}

This article is organized as follow. In Section~\ref{sec:notation} we set some notation and define the relevant classes of plane maps. 
In Section~\ref{sec:incarnations} we give four distinct incarnations of $d$-GS structures, as woods, corner labelings, marked orientations or angular orientations.
In Section~\ref{sec:statements} we state our main results: the existence result for $d$-GS structures, and the fact that the set of $d$-GS woods, labelings, marked orientations and angular orientations are in bijection.
We delay some of the proofs about these bijections to Section~\ref{sec:remaining-proofs}, and the proof of existence of $d$-GS structure to Section~\ref{sec:proof-existence} (which also contains the proof that these structures can be computed in linear time). 
In Section~\ref{sec:arc_labeling}, we provide a further incarnation of $d$-GS structure as arc-labelings for a restricted class of $d$-maps called \emph{edge-tight}. This further incarnation makes the connection to transversal structures and Felsner woods more straightforward and yield a decomposition of the map into plane bipolar orientations.
In Section~\ref{sec:classical}, we explain in detail how $d$-GS structures specialize to the classical Schnyder woods and to transversal structures.
In Section~\ref{sec:bipartite}, we define bipartite $d$-GS structures and how these structures are related to 2-orientations and Felsner woods.
In Section~\ref{sec:lattice}, we study the lattice structure for the set of $d$-GS structures of a given $d$-map.
In Section~\ref{sec:dual}, we consider 
the dual $d$-GS structures (either as dual-corner labelings or as dual-woods), on the dual $G^*$ of a $d$-adapted map.
In Section~\ref{sec:level2-GS}, we discuss quasi-Schnyder structures.
We conclude in Section~\ref{sec:conclusion} with some perspectives and open questions.


%% file: notation.tex
%
For a positive integer $d$, we define $[d]:=\{1,2,\ldots,d\}$. 
The set of non-negative integers is denoted by $\NN:=\{0,1,2,\ldots\}$.
The cardinality of a set $S$ is denoted by $|S|$.

Our \emph{graphs} are finite and undirected. We allow multiple edges and loops. A \emph{simple graph} is a graph without multiple edges or loops.

A \emph{plane map} is a connected planar graph drawn in the plane without edge crossing, considered up to continuous deformation. 
The \emph{faces} of a plane map are the connected components of the complement of the graph. The infinite face is called \emph{outer face}, and the finite faces are called \emph{inner faces}. The vertices and edges incident to the outer face are called \emph{outer} while the other are called \emph{inner}. 
The numbers $\vv$, $\ee$ and $\ff$ of vertices, edges and faces of a plane map are related by the \emph{Euler relation}  $\vv+\ff=\ee+2$.

We now define the class of plane maps which will be relevant for this article.
\begin{definition}\label{def:d-adapted}
A \emph{$d$-map} is a plane map such that the inner faces have degree at most $d$, and the outer face has degree $d$ and is incident to $d$ distinct vertices (in other words, the contour of the outer face is a simple cycle). 
We will assume that the outer vertices of a $d$-map are labeled $v_1,v_2,\ldots, v_d$ in clockwise order along the boundary of the outer face. 
A \emph{$d$-adapted map} is a $d$-map such that any simple cycle which is not the contour of a face has length at least $d$.\\
\end{definition}
We point out that $d$-adapted maps are necessarily 2-connected (because a cut point in a $d$-map $G$ implies the existence of a simple cycle of length strictly less than the degree of an inner face of $G$, which shows that $G$ is not $d$-adapted).

In a plane map, a \emph{corner} is the sector delimited by two consecutive (half-)edges around a vertex. It is called an \emph{inner corner} if it lies in an inner face, and an \emph{outer corner} otherwise.
The \emph{degree} of a vertex or face is its number of incident corners. A  \emph{$d$-angulation} is a plane map with all faces of degree $d$. A \emph{$d$-angulation of the $k$-gon} is a plane map with inner faces of degree $d$, and outer face of degree $k$. 
A graph is \emph{bipartite} if it admits a bicoloring of its vertices such that adjacent vertices have different colors. It is known that a plane map is bipartite if and only if all its faces have even degree. For $k\geq 2$, a graph is called \emph{$k$-connected} if it is connected and the deletion of any subset of $(k-1)$ vertices does not disconnect it (loops are forbidden for $k\geq 2$, multiple edges are forbidden for $k\geq 3$).

Let $G$ be an undirected graph. An \emph{arc} of $G$ is an edge $e$ of $G$ together with a chosen orientation of $e$ (so each edge of $G$ correspond to two arcs). The arc \emph{opposite} to an arc $a$, denoted by $-a$, is the arc corresponding to the same edge as $a$ but with the opposite direction. 
The endpoints of an arc $a$ are called the \emph{initial} and \emph{terminal} vertices of $a$ (with $a$ oriented from the initial vertex to the terminal vertex).  If $v$ is the initial (resp. terminal) vertex of the arc $a$, then we say that $a$ is an \emph{outgoing arc} (resp. \emph{ingoing arc}) at $v$. 
\\

A \emph{path} in an undirected graph $G$ is a sequence of arcs $a_1,a_2,\ldots,a_k$ such that the terminal vertex of $a_i$ is the initial vertex of $a_{i+1}$ for all $i\in[k-1]$. It is called a \emph{closed path} if the terminal vertex of $a_k$ is the initial vertex of $a_1$. A \emph{cycle} is a (cyclically ordered) closed path. A path or cycle is called \emph{simple} if it does not pass twice by the same vertex. The \emph{girth} of a graph is the minimum length of its simple cycles.   In a plane map, a closed path formed by the arcs around a face is called \emph{contour} of that face. It is known that face contours are simple cycles if the plane map is 2-connected. 
A simple cycle on a plane map is called \emph{counterclockwise} (resp. \emph{clockwise}) if the direction of arcs is counterclockwise (resp. clockwise) around the cycle.

Let $G$ be a graph.  Given an orientation of $G$, a \emph{directed path} (resp. \emph{directed cycle}) is a path (resp. cycle) $a_1,a_2,\ldots,a_k$ such that every arc $a_i$ is oriented according to the orientation of $G$.
A \emph{weighted orientation} of $G$ is an assignment of a non-negative integer to each arc of $G$. Given a weighted orientation $\cW$ of $G$, we call \emph{weight} of an edge the sum of the weights of the two corresponding arcs. 
Weighted orientations are a generalization of the classical (unweighted) orientations of $G$. Indeed the ``unweighted'' orientations of $G$ can be identified to the weighted orientations of $G$ such that the weight of every edge is 1 (for each edge, the arc of weight 1 is taken as the orientation of the edge). The \emph{outgoing weight} (shortly, the \emph{weight}) of a vertex $v$ is the sum of the weights of the arcs going out of $v$. Given a weighted orientation, we call \emph{positive path} (resp. \emph{positive cycle}) a path (resp. cycle) $a_1,a_2,\ldots,a_k$ such that the weight of every arc is positive (this generalizes the notion of \emph{directed path} and \emph{directed cycle}).

A \emph{tree} is a connected, acyclic graph. For a tree $T$ with a vertex $v$ distinguished as its \emph{root}, we apply the usual ``genealogy'' vocabulary about trees, where $v$ is an \emph{ancestor} of all the other vertices, and every non-root vertex incident to $T$ has a \emph{parent} in $T$, etc. 
We say that we \emph{orient the tree $T$ toward its root} by orienting every edge from child to parent. With this orientation, every non-root vertex of $T$ is incident to one outgoing edge in $T$ (the edge leading to its parent).
A \emph{subtree} of a graph $G$ is a subset of edges of $G$ such that this set of edges together with the incident vertices forms a tree. A \emph{spanning tree} of $G$ is a subtree of $G$ incident to every vertex of $G$.


%% file: incarnation.tex

In this section we define the $d$-grand-Schnyder structures under their four possible incarnations: \emph{GS woods}, \emph{GS labelings}, \emph{GS marked orientations}, and \emph{GS angular orientations}. The definition of those structures is summarized in Figure~\ref{fig:def-incarnations}.
As we will state in Section \ref{sec:statements}, these structures exist for a $d$-map if and only if this $d$-map is $d$-adapted. 

\subsection{Grand-Schnyder corner labelings}\label{subsec:GS-labeling}
In this subsection, we define the \emph{grand-Schnyder corner labelings} for $d$-maps.   
A \emph{$d$-labeling} of a $d$-map $G$ is an assignment of a \emph{label} in $[d]:=\{1,\ldots,d\}$ to each inner corner of $G$.
A $d$-grand-Schnyder corner labelings of $G$ is a $d$-labeling satisfying certain local conditions represented in Figure~\ref{fig:def-incarnations} (top row).
These conditions are best expressed in terms of \emph{jumps}.

Consider a $d$-labeling of a $d$-map $G$. Let $c$ and $c'$ be two inner corners of $G$, and let $i$ and $i'$ be their respective labels. The \emph{label jump} from the corner $c$ to the corner $c'$ is defined as the integer $\delta$ in $\{0,1,\ldots,d-1\}$ such that $i+\delta=i'\mod d$ (in other words, the label jump $\delta$ is $i'-i$ if $i'-i\geq 0$, and $i'-i+d$ otherwise).
For an inner vertex $v$ of $G$, the \emph{sum of clockwise jumps} around~$v$ is the sum of label jumps between consecutive corners in clockwise order around~$v$. Similarly, the \emph{sum of clockwise jumps} around a face $f$ is the sum of label jumps between consecutive corners in clockwise order around~$f$. Note that the sum of clockwise jumps around a vertex or face is necessarily a multiple of $d$: it is equal to $k\, d$, where $k$ is the number of strict decreases of labels in clockwise order around the vertex or face (a strict decrease is when the label of a corner is strictly larger than the label of the following corner).

\begin{definition}\label{def:GS-labeling}
Let $d\geq 3$, and let $G$ be a $d$-map. As usual, we assume that the outer vertices of $G$ are denoted by $v_1,\ldots,v_d$ in clockwise order around the outer face.
A \emph{$d$-grand-Schnyder corner labeling} of $G$ is an assignment to each inner corner of $G$ of a label in $[d]$ satisfying the following conditions.
\begin{itemize}
\item[(L0)] For all $i\in [d]$, all the corners incident to $v_i$ have label $i$.
\item[(L1)] For every inner vertex or face of $G$, the sum of clockwise jumps around this vertex or face is $d$.
\item[(L2)] The label jumps from a corner to the next corner around a face are always positive (equivalently, consecutive corners around a face have distinct labels).
\item[(L3)] Let $e$ be an inner edge and let $f$ be an incident face. Let $c$ and $c'$ be the consecutive corners incident to $e$ in clockwise order around~$f$, and let $v$ be the vertex incident to $c'$. The label jump $\delta$ from $c$ to $c'$ and the label jump $\eps$ from $c'$ to the next corner in clockwise order around~$v$ satisfy $\delta+\eps> d-\deg(f)$.
\end{itemize}
\end{definition}
The definition of  $d$-grand-Schnyder corner labelings, or \emph{$d$-GS labelings} for short, is represented in the top row of Figure~\ref{fig:def-incarnations}.

\fig{width=\linewidth}{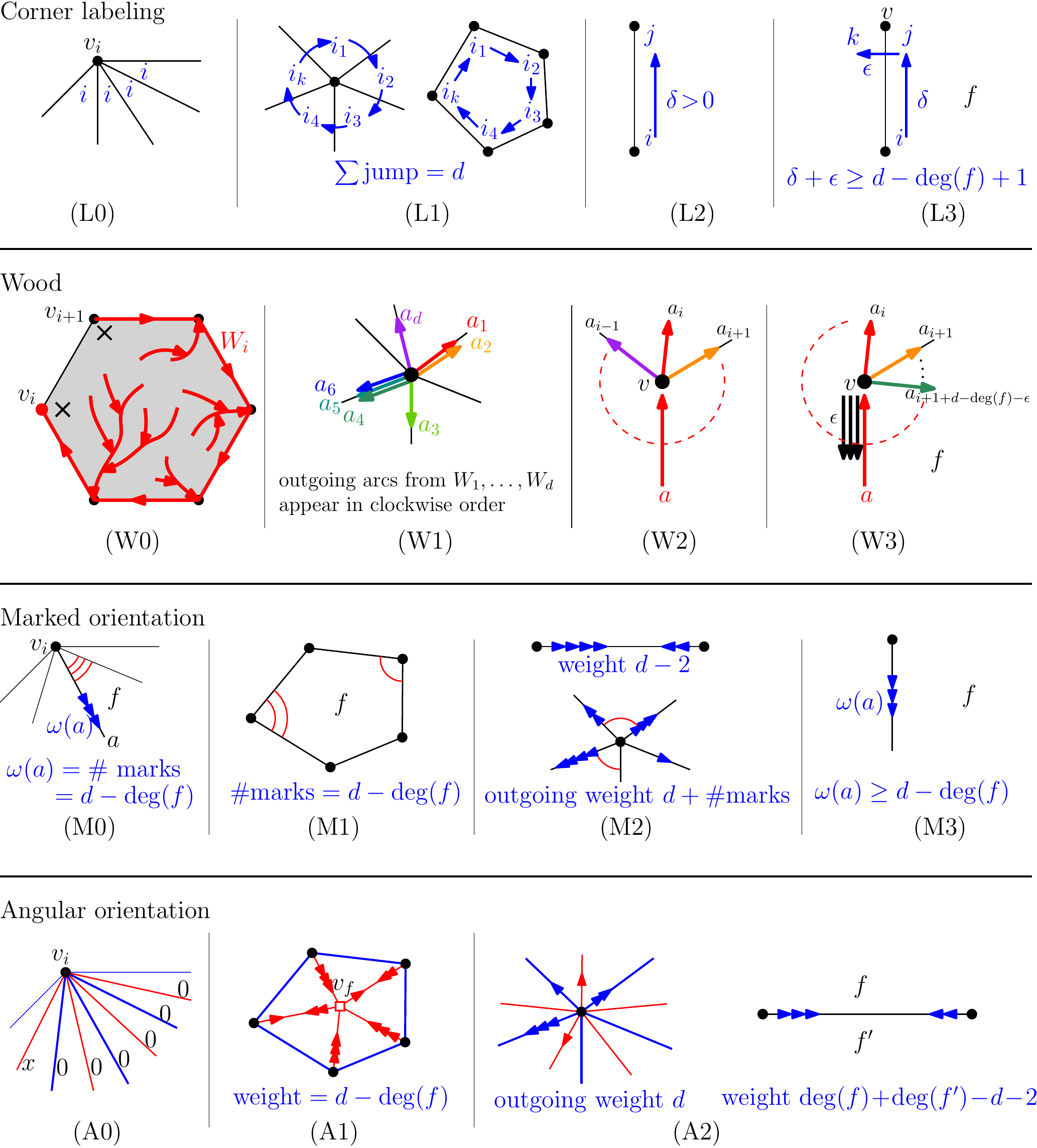}{Conditions defining the $d$-grand-Schnyder structures. The orientation weights are indicated by arrowheads (the number of arrowhead indicates the weight), while the markings are indicated by arcs in the corners (the number of arcs indicates the number of marks).}

Let us make a few easy observations about Definition~\ref{def:GS-labeling}.
\begin{remark}
\begin{compactitem}
\item Condition (L1) could alternatively be stated as saying that there is exactly one decrease of labels in clockwise order around any inner vertex or face. 
\item Conditions (L1) and (L2) imply that, clockwise around an inner face $f$, the label jumps from a corner to the next are all between $1$ and $1+d-\deg(f)$. These conditions also imply that the labels around any inner face are all distinct.
\end{compactitem}
\end{remark}

Let us prove one more easy property of $d$-GS labelings.
\begin{lemma}\label{lem:ccw-jumps-edges}
Let $\cL$ be a $d$-GS labeling of a $d$-map $G$. For an inner edge $e$, we consider the 4 corners incident to $e$ and the 4 label jumps in counterclockwise order around~$e$ as represented in Figure~\ref{fig:edge-jumps}. The sum of label jumps counterclockwise around $e$ is equal to~$d$.
\end{lemma}

\fig{width=.9\linewidth}{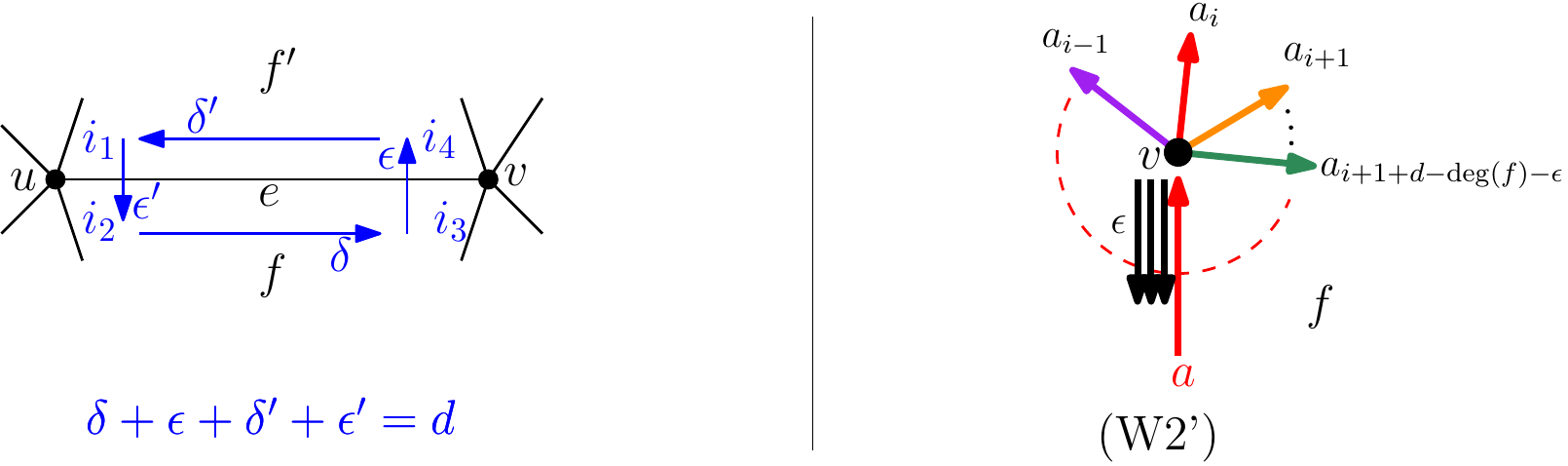}{(a) The label jumps in counterclockwise order around an inner edge~$e$. By Lemma~\ref{lem:ccw-jumps-edges} these label jumps always add up to $d$ in a $d$-GS labeling. (b) The Condition (W2') replacing Conditions (W2) and (W3) for $d$-GS woods of $d$-adapted maps.}

Lemma \ref{lem:ccw-jumps-edges} is illustrated in Figure \ref{fig:edge-jumps}(a). Note that by Lemma \ref{lem:ccw-jumps-edges}, Condition (L3) for $d$-GS labelings can be completed to say that the label jumps $\delta, \eps$ must satisfy $d-\deg(f)< \delta+\eps\leq d-1$.

\begin{proof} Let $\cwjump(x)$ (resp. $\ccwjump(e)$) be the sum of label jumps in clockwise order around an inner vertex or inner face $x$ (resp. in counterclockwise order around an inner edge $e$). We observe that a clockwise jump around a vertex or a face is a counterclockwise jump around an inner edge, or a jump along one of the outer edges (there are $d$ such jumps, and each has value $1$). This gives
 \begin{equation}\label{eq:sum-jumps-relation}
\sum_{f\in F}\cwjump(f)+ \sum_{v\in V}\cwjump(v)=d+\sum_{e\in E}\ccwjump(e),
\end{equation}
where $V,F,E$ are the sets of inner vertices, inner faces, and inner edges respectively. By Condition (L1), the left-hand side is $d(|V|+|F|)$, which is equal to $d+d|E|$ by the Euler relation. By Condition (L2), we have $\ccwjump(e)\geq d$ for all $e\in E$ (since $\ccwjump(e)$ is a positive multiple of $d$), which by the preceding implies $\ccwjump(e)= d$ for all $e\in E$.
\end{proof}

\subsection{Grand-Schnyder woods}\label{subsec:GS-woods}
In this subsection we define grand-Schnyder woods for $d$-maps. These are $d$-tuples of oriented spanning trees satisfying certain conditions represented in Figure~\ref{fig:def-incarnations} (second row). The definition below does not actually specify that the woods are trees; this property is actually a consequence of the other conditions, as stated in Proposition \ref{prop:Wi-are-trees}.


\begin{definition}\label{def:woods}
Let $d\geq 3$, and let $G$ be a $d$-map. As usual, we assume that the outer vertices of $G$ are denoted by $v_1,\ldots,v_d$ in clockwise order around the outer face.
A \emph{$d$-grand-Schnyder wood} of $G$ is a $d$-tuple $W=(W_1,\ldots,W_d)$ of subsets of arcs of $G$ satisfying the following conditions:
\begin{itemize}
\item[(W0)] 
For all $i\in [d]$, every vertex $v\neq v_i$ has exactly one outgoing arc in $W_i$, while $v_i$ has no outgoing arc in $W_i$.
For all $k\neq i$, the arc in $W_i$ going out of the outer vertex $v_k$ is the  outer arc oriented from $v_k$ to $v_{k+1}$.
Lastly, $W_i$ does not contain any inner arc oriented toward $v_{i}$ or $v_{i+1}$. 
\item[(W1)] Let $v$ be an inner vertex with incident outgoing arcs $a_1,a_2,\ldots,a_d$ in the sets $W_1,W_2,\ldots,W_d$ respectively. The arcs  $a_1,a_2,\ldots,a_d$ are not all equal, and they appear in clockwise order around~$v$ (with the situation $a_i=a_{i+1}$ allowed).
\item[(W2)] Let $v$ be an inner vertex with incident outgoing arcs $a_1,a_2,\ldots,a_d$ in the sets $W_1,W_2,\ldots,W_d$ respectively. Let $a$ be an arc oriented toward $v$. 
If the arc $a$ belongs to the set $W_i$, then $a$ appears strictly between $a_{i+1}$ and $a_{i-1}$ in clockwise order around~$v$ (by ``strictly'', we mean that $a$ is not on the same edges as $a_{i-1}$ or $a_{i+1}$).\footnote{There is a possible ambiguity in Definition~\ref{def:woods} that we ought to clarify. Let $(W_1,\ldots,W_d)$ be a $d$-GS wood of $G$. Let $v$ be an inner vertex with incident outgoing arcs $a_1,a_2,\ldots,a_d$ in the sets $W_1,W_2,\ldots,W_d$ respectively (recall that they are not all equal). If $a_i=a_{i+1}=a_{i+2}=\cdots=a_j$, then we consider that there is no edge incident to $v$ appearing strictly between $a_i$ and $a_j$, but there are edges appearing strictly between $a_j$ and $a_i$  (every edge incident to $v$ except the one containing $a_i=a_j$).}
\item[(W3)] Let $v$ be a vertex with outgoing arcs $a_1,a_2,\ldots,a_d$ in the sets $W_1,W_2,\ldots,W_d$ respectively (if $v=v_k$ is the root of $W_{k}$ we adopt the convention $a_{k}:=(v_{k},v_{k-1})$). Let $a$ be an inner arc oriented toward $v$, let $f$ be the face at the right of $a$, and let $\eps$ be the number of sets $W_1,W_2,\ldots,W_d$ containing the opposite arc $-a$. 
If $d-\deg(f)-\eps\geq 0$ and the arc $a$ belongs to the set $W_i$, then $a$ appears strictly between $a_{i+1+d-\deg(f)-\eps}$ and $a_{i}$ in clockwise order around~$v$. The same holds if $d-\deg(f)-\eps\geq 0$ and the arc $a$ belongs to none of the sets $W_1,W_2,\ldots,W_d$ but appears between the outgoing arcs in $W_i$ and $W_{i+1}$ in clockwise order around the initial vertex of~$a$.
\end{itemize}
\end{definition}

The definition of $d$-grand-Schnyder woods, or \emph{$d$-GS woods} for short, is illustrated in the second row of Figure~\ref{fig:def-incarnations}. 
 
We now prove a few additional facts that can be deduced from Conditions (W0-W3). 
Conditions (W0-W3) are all ``local'' (specifying what happens around vertices), but they actually imply the ``global property'' that the sets $W_1,\ldots,W_d$ are spanning trees as stated below (and as reflected in the figure representing Condition (W0)).
\begin{prop}\label{prop:Wi-are-trees}
Let $G$ be a $d$-map. If $\cW=(W_1,\ldots,W_d)$ is a $d$-GS wood (or even if these sets of arcs only satisfy Conditions (W0-W2)), then for all $i\in[d]$, $W_i$ is a spanning tree of $G$ oriented toward its root $v_i$.
\end{prop}

\begin{proof} Suppose that $\cW$ satisfies Conditions  (W0-W2). 
For an inner vertex $v$ and an index $i\in[d]$, let $P_i(v)$ be the directed path starting at $v$ which is obtained by following the outgoing arcs in $W_i$ until reaching an outer vertex or an inner vertex already visited. We want to show that $P_i(v)$ reaches an outer vertex.
Suppose for contradiction that there exists an inner vertex $v$ and a color $i\in [d]$, such that $P_i(v)$ ends at an inner vertex, so that $P_i(v)$ contains a cycle $C_i(v)$. We pick $v$ and $i$ such that the number of faces contained in the cycle $C_i(v)$ is minimal. Suppose, for concreteness, that $C_i(v)$ is directed clockwise. In this case, because of Condition~(W2), for every vertex $u$ on $C_i(v)$, the outgoing edge in $W_{i+1}$ at $u$ is either the same as the outgoing arc in $W_i$ or goes strictly inside $C_i(v)$. This implies that the path $P_{i+1}(u)$ cannot reach vertices laying outside of the region enclosed by $C_i(v)$. By the minimality condition on $C_i(v)$, we conclude that $C_{i+1}(u)$ is equal to $C_i(v)$ for any vertex $u$ on $C_i(v)$. Repeating the argument, we get $C_j(u)=C_i(v)$ for all $j\in [d]$. This means that the outgoing arcs in $W_1,\ldots, W_d$ at $u$ are all equal, which contradicts Condition (W1).
Similarly, if one assumes that $C_i(v)$ is directed counterclockwise, then one can prove that $C_{i-1}(u)$ is equal to $C_i(v)$ for any vertex $u$ on $C_i(v)$, and this leads to a contradiction as before. This concludes the proof that for all $v$, the path $P_i(v)$ reaches an outer vertex. Given Condition (W0), we see that there is a directed path in $W_i$ from any vertex to $v_i$, hence $W_i$ is a spanning tree of $G$ rooted at $v_i$.
\end{proof}

\begin{remark} Let $W=(W_1,\ldots,W_d)$ be a $d$-GS wood for a $d$-map $G$. By definition, the vertices $v_i$ and $v_{i+1}$ have degree 1 in the spanning tree $W_i$. Hence, removing from $W_i$ the outer edges, one gets a forest $W_i'$ made of $d-2$ subtrees rooted at the outer vertices $v_j,~j\neq i,i+1$, and spanning all the inner vertices. In the classical case $d=3$ of Schnyder, $W_i'$ is a subtree rooted at $v_{i-1}$ and spanning all the inner vertices as represented in Figure \ref{fig:triangulation2}. Note also that Condition~(W3) adds additional constraints about the incidence of $W_i$ with the outer vertices. For instance, if the inner faces of $G$ have degree at most $d-k$, then $v_j$ is not incident to any edge of $W_i'$ for all $j$ in $\{i,i+1,\ldots, i+k+1\}$. Hence in this case the forest $W_i'$ consists of $d-2-k$ subtrees rooted at the outer vertices $v_j,~j\notin \{i,i+1,\ldots, i+k+1\}$ (the other outer vertices are isolated).
\end{remark}


Lastly, let us mention that Condition (W3) can be simplified if we suppose that the $d$-map $G$ is $d$-adapted. More precisely, for a $d$-adapted map $G$, the last sentence in Condition (W3) (about arcs not belonging to any of the oriented trees $W_1,\ldots W_d$) can be removed (because it becomes redundant with the other conditions). This will be shown in Section~\ref{sec:statements}, where we will prove the following statement.
\begin{lemma}\label{lem:W2'}
For a $d$-adapted map $G$, Conditions (W2) and (W3) can be replaced by the following single condition:
\begin{itemize}
\item[(W2')] Let $v$ be a vertex with outgoing arcs $a_1,a_2,\ldots,a_d$ in the trees $W_1,W_2,\ldots,W_d$ respectively (if $v=v_k$ is the root of $W_{k}$ we adopt the convention $a_{k}:=(v_{k},v_{k-1})$). Let $a$ be an inner arc oriented toward $v$. 
If the arc $a$ belongs to the oriented tree $W_i$, then $a$ appears strictly between $a_{i+1+m}$ and $a_{i-1}$ in clockwise order around~$v$, where $m=\max(0,d-\deg(f)-\eps)$, $f$ is the face at the right of $a$, and $\eps$ is the number of oriented trees $W_1,W_2,\ldots,W_d$ containing the opposite arc $-a$. 
\end{itemize}
\end{lemma}

It is convenient to think of a $d$-GS wood of $G$ as a \emph{coloring} of the arcs of $G$ with some subsets of colors in $[d]$. We say that an arc $a$ of $G$ has a \emph{color} $i\in[d]$ if this arc belong to the set $W_i$. Note that a given arc can have several colors but that these colors are necessarily consecutive (modulo $d$), because of Condition (W1). Note also that, by condition (W2), if an arc has color~$i$, then the opposite arc $-a$ does not have color $i-1$, $i$ or $i+1$. Hence the intervals of colors of the arcs $a$ and $-a$ are disjoint and non-consecutive. In Section \ref{sec:remaining-proofs}, we will establish the following bounds for the number of colors of an inner edge.
\begin{lemma} \label{lem:nb-color-arcs}
Let $G$ be a $d$-map. For any $d$-GS wood of $G$, the total numbers of colors $n_e$ of an inner edge $e$ (that is, the sum of the numbers of colors of the two arcs corresponding to $e$) satisfies
$$\deg(f)+\deg(f')-d-2~\leq ~n_e~\leq~ d-2,$$ 
where $f,f'$ are the faces incident to $e$. In particular, if $G$ is a $d$-angulation, then every inner edge has $d-2$ colors. 
\end{lemma}





We now explain how a $d$-GS wood structure makes it possible to define paths and regions associated to each vertex.
Let $G$ be a $d$-map and let $\cW=(W_1,\ldots,W_d)$ be a $d$-GS wood. Let~$v$ be an inner vertex. For all $i$ in $[d]$, let $r_i(v)$ be the first outer vertex on the directed path from~$v$ to the root $v_{i}$ in the oriented spanning tree $W_i$. We define $P_i(v)$ as the directed path of $W_i$ from $v$ to $r_i(v)$, and call it the \emph{path of color $i$ starting at $v$}. 

The paths $P_1(v),P_2(v),\ldots,P_d(v)$ can have vertices and edges in common but they cannot ``cross'' each other, as we now explain. Let $P,Q$ be two simple directed paths on $G$. We say that \emph{$P$ crosses $Q$ from left to right} if $P$ has a sequence of successive arcs $a_0,a_1,\ldots a_k$, for some $k\geq 1$, such that
\begin{compactitem}
\item the arcs $a_1,\ldots,a_{k-1}$ are all on $Q$ (either in the direction of $Q$ or in the opposite direction),
\item the terminal vertex $u$ of $a_0$ is on $Q$ but is not an extremity of $Q$, and $a_0$ is strictly on the left of $Q$ around~$u$,
\item the initial vertex $u'$ of $a_k$ is on $Q$ but is not an extremity of $Q$, and $a_k$ is strictly on the right of $Q$ around~$u'$.
\end{compactitem}
A path $P$ crossing a path $Q$ from left to right is represented in Figure~\ref{fig:crossing-pathsa}. We define a crossing from right to left symmetrically. Note that if $P$ crosses $Q$ from left to right, then $Q$ crosses $P$ from right to left. 
We say that the paths $P$ and $Q$ are \emph{non-crossing} if $P$ does not cross $Q$ from left to right nor from right to left. 
A basic observation is that, by Condition (W2) of $\cW$, the path $P_{i}(v)$ cannot cross $P_{i+1}(v)$ from left to right, and $P_{i}(v)$ cannot cross $P_{i-1}(v)$ from right to left. We will prove that more is true.
 
\fig{width=.25\linewidth}{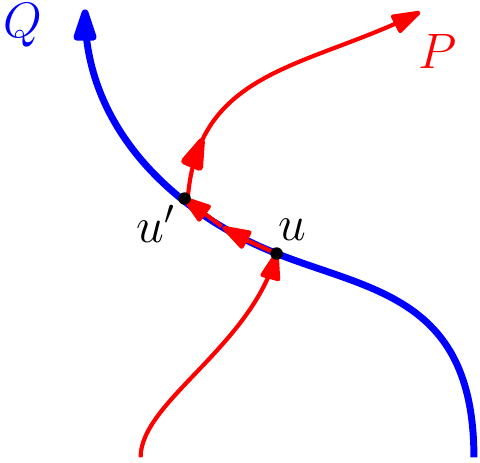}{A path $P$ crossing a path $Q$ from left to right.} 

\begin{lemma} \label{lem:beam-of-paths}
Let $G$ be a $d$-map and let $\cW=(W_1,\ldots,W_d)$ be a $d$-GS wood (or even if these sets of arcs only satisfy Conditions (W0-W2)). Let $v$ be an inner vertex and let $P_1(v),P_2(v),\ldots,P_d(v)$ be the paths of color $1,2,\ldots,d$ starting at $v$. The paths $P_1(v),P_2(v),\ldots,P_d(v)$ are pairwise non-crossing and their endpoints $r_1(v),r_2(v),\ldots,r_d(v)$ appear in clockwise order (weakly) around the outer face of $G$.
\end{lemma}

\fig{width=.8\linewidth}{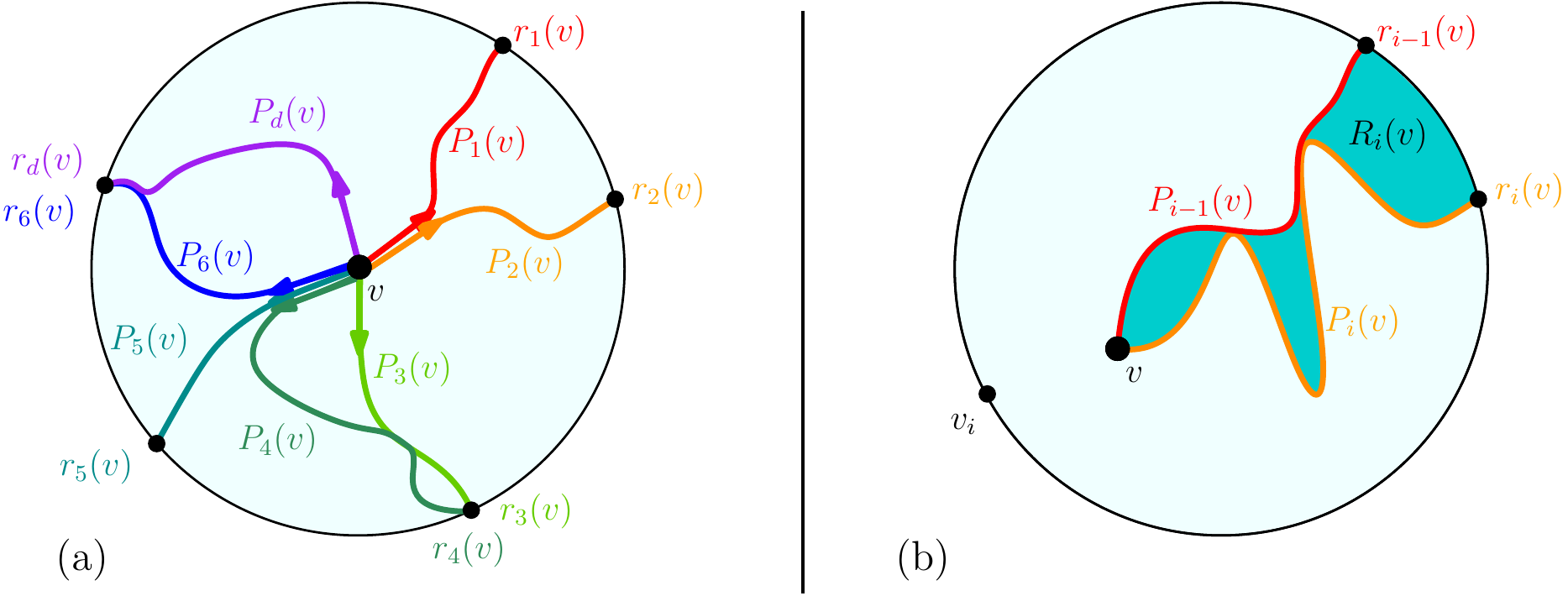}{(a) The paths $P_1(v),P_2(v),\ldots,P_d(v)$ are non-crossing. (b) The region $R_i(v)$.}
Lemma~\ref{lem:beam-of-paths} is represented in Figure~\ref{fig:beam-of-paths}(a). We postpone the proof to Section \ref{sec:remaining-proofs}.
Since the paths $P_1(v),\ldots,P_d(v)$ starting at a vertex $v$ are non-crossing, one can define the region $R_1(v),\ldots,R_d(v)$ that they delimit. Precisely, we define $R_i(v)$ as the submap of $G$ lying between $P_{i-1}(v)$ and $P_i(v)$ as represented in Figure~\ref{fig:beam-of-paths}(b): $R_i(v)$ is the set of vertices, faces and edges enclosed by the cycle made of $P_{i-1}(v)$, $P_i(v)$ and the set of outer edges between $r_{i-1}(v)$ and $r_i(v)$ in clockwise direction around the outer face (the vertices and edges on this cycle are part of $R_i(v)$). The regions $R_1(v),\ldots,R_d(v)$ play an important part in the applications of Schnyder woods in the classical case $d=3$ \cite{Schnyder:wood1,Schnyder:wood2}.

\subsection{Grand-Schnyder marked orientations}\label{subsec:GS-marked}
In this subsection, we define grand-Schnyder marked orientations for $d$-maps. 
Recall from Section~\ref{sec:notation}, that a \emph{weighted orientation} of a graph is an assignment of a non-negative integer to each arc of this graph called its \emph{weight}. A \emph{corner marking} of a plane map $G$ is the assignment of a non-negative number to each inner corner of $G$; this number is interpreted as the ``number of marks" of the corner. 
A $d$-GS marked orientation is a weighted orientation together with a corner marking satisfying certain conditions represented in Figure~\ref{fig:def-incarnations} (third row).

\begin{definition}\label{def:marked}
Let $d\geq 3$, and let $G$ be a $d$-map.
A \emph{$d$-grand-Schnyder marked orientation} of $G$ is a weighted orientation of $G$, together with a corner marking satisfying the following conditions.
\begin{itemize}
\item[(M0)] The weight of every outer arc is 0. For any inner arc $a$ whose initial vertex is an outer vertex $v_i$, the weight of $a$ and the number of marks in the corner of $v_i$ on the left of $a$ are both equal to $d-\deg(f)$, where $f$ is the face on the left of $a$. 
\item[(M1)] For any inner face $f$, the total number of marks in the corners of $f$ is $d-\deg(f)$.
\item[(M2)] The weight of every inner edge is $d-2$, and the outgoing weight of every inner vertex $v$ is $d+m$,  where $m$ is the number of marks in the corners incident to~$v$. 
\item[(M3)] The weight of every inner arc $a$ is at least $d-\deg(f)$, where $f$ is the face on the left of~$a$.
\end{itemize}
\end{definition}
The definition of $d$-grand-Schnyder marked orientations, or \emph{$d$-GS marked orientations} for short, is illustrated in the third row of Figure~\ref{fig:def-incarnations}.


\subsection{Grand-Schnyder angular orientations}\label{subsec:GS-angular}
In this subsection, we define grand-Schnyder angular orientations of $d$-maps. 
The \emph{angular map} of a plane map $G$ is the plane map $G^+$ obtained from $G$ by placing a new vertex $v_f$ in each inner face $f$ of $G$, and joining $v_f$ to all the vertices of $G$ incident to $f$ (more precisely, one edge from $v_f$ to each corner of $f$). The angular map $G^+$ has two types of vertices and edges: the \emph{original vertices and edges} of $G$, and the new vertices and edges that we call \emph{star vertices and edges}. The angular map is represented in Figure \ref{fig:angular-map}.

\fig{width=\linewidth}{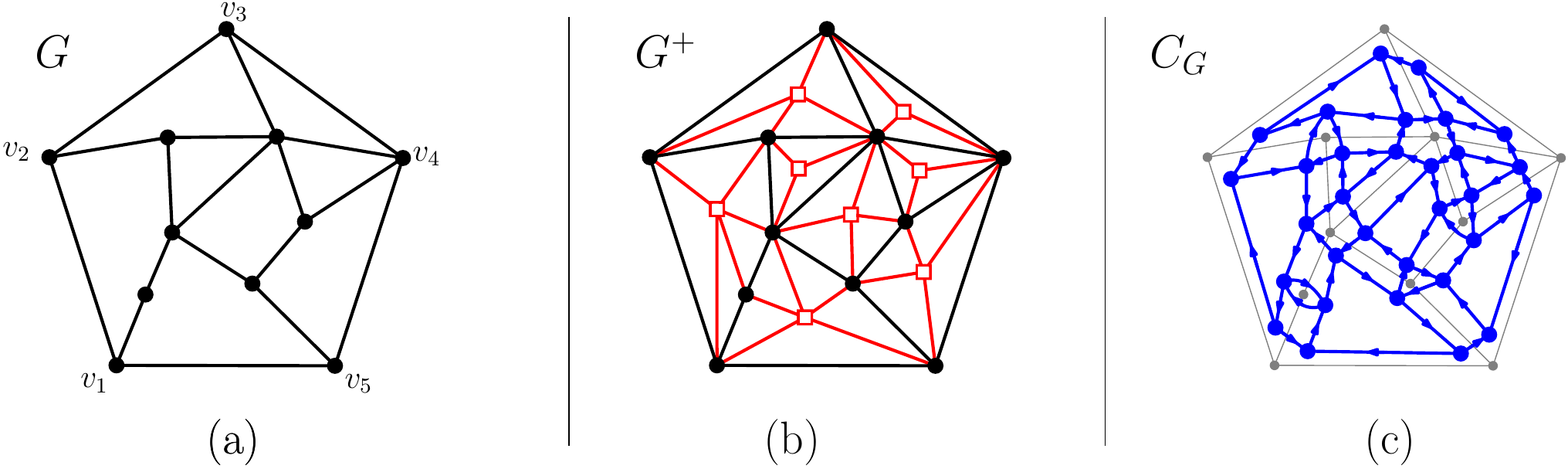}{(a) A 5-map $G$. (b) The angular map $G^+$. (c) The corner graph $C_G$.}

\begin{definition}\label{def:angular}
Let $d\geq 3$, and let $G$ be a $d$-map.
A \emph{$d$-grand-Schnyder angular orientation} of $G$ is a weighted orientation of the angular map $G^+$ satisfying the following conditions.
\begin{itemize}
\item[(A0)] The weight of every outer arc is 0. Any inner arc $a$ of $G^+$ whose initial vertex is an outer vertex $v_i$ has weight $0$, unless $a$ is the arc following the outer edge $(v_i,v_{i-1})$ around~$v_i$ (for this arc there is no condition).
\item[(A1)] The outgoing weight of any star vertex $v_f$ is $d-\deg(f)$, and the weight of every star edge incident to $v_f$ is $d-\deg(f)$.
\item[(A2)] The outgoing weight of every inner original vertex is $d$. The weight of any original inner edge $e$ is $\deg(f)+\deg(f')-d-2$, where $f,f'$ are the faces of $G$ incident to~$e$. 
\end{itemize}
\end{definition}
The definition of $d$-grand-Schnyder angular orientations, or \emph{$d$-GS angular orientations} for short, is illustrated in the bottom row of Figure~\ref{fig:def-incarnations}. 

\begin{remark}\label{rk:frozen}
In the representation of Condition (A0) the weight of one of the star edges out of $v_i$ is indicated as $x$. Although this weight is not explicitly specified by Condition (A0), this weight is actually $d-\deg(f)$, where $f$ is the face of $G$ containing this star edge. Indeed, this is implied by combining Condition (A1) at $v_f$ with Condition (A0) for the other outer vertices incident to $f$.
\end{remark}


%% file: main-results.tex
%

In this section we state our main result, which is the existence condition for $d$-GS structures. We then describe the bijections between the different incarnations of $d$-GS structures.

\begin{thm}\label{thm:main}
Let $d\geq 3$ and let $G$ be a $d$-map. There exists a $d$-GS wood (resp. labeling, marked orientation, angular orientation) for $G$ if and only if $G$ is $d$-adapted (that is, if its simple non-facial cycles have length at least $d$). 

Moreover for any fixed $d$, there is an algorithm which takes as input a $d$-adapted map, and computes a $d$-GS wood (resp. labeling, marked orientation, angular orientation) in linear time in the number of vertices.

Lastly, the set $\bW_G$ of $d$-GS woods of $G$, the set $\bL_G$ of $d$-GS labelings of $G$, the set $\bM_G$ of $d$-GS marked orientations of $G$, and the set $\bA_G$ of $d$-GS angular orientations of $G$ are all in bijection.
\end{thm}


We will prove the existence result stated in Theorem~\ref{thm:main} in Section~\ref{sec:proof-existence}, and we will also describe there the algorithm for computing $d$-GS structures. 
\\

In Section~\ref{sec:lattice} we will show that the set of $d$-GS structure of a given $d$-adapted map can be given a \emph{lattice structure} (in the sense of partially ordered sets), and characterize the covering operations. 

For the rest of this section, we focus on defining the bijections between the different incarnations of $d$-GS structures.
From now on, we fix an integer $d\geq 3$ and a $d$-map $G$. The outer vertices of $G$ are denoted by $v_1,\ldots,v_d$, and they appear in clockwise order around the outer face of $G$. The bijections between the sets $\bW_G$, $\bL_G$, $\bM_G$ and $\bA_G$ are represented in Figure~\ref{fig:bij-labeling-marked},~\ref{fig:bij-marked-angular} and~\ref{fig:bij-labeling-wood}. We start with the bijections between $\bL_G$, $\bM_G$ and $\bA_G$, which are easier to prove.


First, we define the bijection $\Phi$ between $d$-GS labelings and $d$-GS marked orientations. Roughly speaking, the bijection $\Phi$ is as follows: the marks encode the label jumps around a face, while the arc weights encode the label jumps around vertices. The bijection $\Phi$ is represented in Figure~\ref{fig:bij-labeling-marked}. We call \emph{marked orientation} of $G$ a weighted orientation of $G$ together with assigning a number of marks to each of its inner corners. 

\begin{definition}
Given a $d$-GS labeling $\cL$ of $G$, we define a marked orientation $\Phi(\cL)$ of $G$ as follows. 
For an inner corner $c$, we denote by $c^-$ the corner preceding $c$ in clockwise order around the face containing $c$, and by $c^+$ the corner following $c$ in clockwise order around the vertex incident to $c$.
\begin{compactitem}
\item The number of marks of an inner corner $c$ in $\Phi(\cL)$ is $\delta-1$, where $\delta$ is the label jump from $c^-$ to $c$.
\item The weight of outer arcs in $\Phi(\cL)$ is 0. 
Let $a$ be an inner arc, and let $c$ be the corner preceding $a$ in clockwise order around the initial vertex of $a$. Then the weight of the inner arc $a$ in $\Phi(\cL)$ is $\delta +\eps-1$, where $\delta$ is label jump from $c^-$ to $c$ and $\eps$ is label jump from $c$ to $c^+$.
\end{compactitem}
\end{definition}

\fig{width=\linewidth}{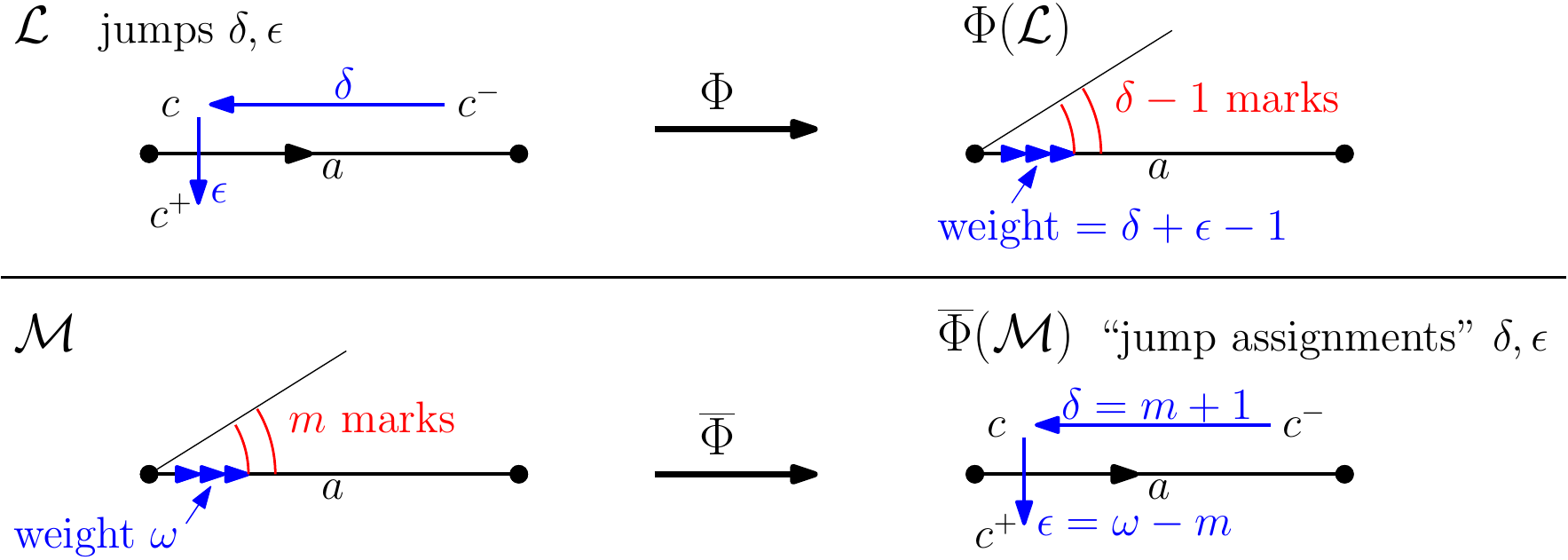}{Bijection $\Phi$ from $d$-GS labelings to $d$-GS marked orientations.}

\begin{prop}\label{prop:bij-beta}
The map $\Phi$ is a bijection between the set $\bL_G$ of $d$-GS labelings of $G$ and the set $\bM_G$ of $d$-GS marked orientations of $G$. 
\end{prop}

\begin{proof}
We first show that the image of a $d$-GS labeling by $\Phi$ is a $d$-GS marked orientation. 
Let $\cL\in \bL_G$ and let $\cM=\Phi(\cL)$. 
First observe that Condition~(L3) for $\cL$ translates into Condition~(M3) for $\cM$. 
Also, Condition (L1) for $\cL$ implies Condition (M1) for $\cM$ as well as the part of Condition~(M2) about weights around vertices. The part of Condition (M2) about edges (weight $d-2$ for every inner edge) is a consequence of the fact that the sum of counterclockwise label jumps around inner edges is $d$, which holds for any $d$-GS labeling by Lemma~\ref{lem:ccw-jumps-edges}. Lastly, Condition (L0) for $\cL$ implies that for any inner arc $a$ whose initial vertex is an outer vertex $v_i$, the weight $\om(a)$ of $a$ is equal to the number of marks $m(a)$ in the corner of $v_i$ on the left of $a$. Moreover, Condition (M1) gives $m(a)\leq d-\deg(f)$ and condition (M3) gives $\om(a)\geq d-\deg(f)$. This gives (M0). Hence, $\cM$ is a $d$-GS marked orientation.

By the above, $\Phi$ is a map from $\bL_G$ to $\bM_G$. It is also clear that $\Phi$ is injective because knowing $\Phi(\cL)$ allows one to determine all the label jumps between ``adjacent corners'' (corners that are consecutive around faces or vertices) which together with Condition (L0) determines the corner labeling $\cL$. The label jumps determined from a marked orientation is indicated in the bottom part of Figure~\ref{fig:bij-labeling-marked}, and provides a tentative inverse mapping for $\Phi$.
In order to prove that $\Phi$ is surjective, we need to take a closer look at this inverse mapping. Namely, we need to check that, starting from a marked orientation $\cM\in \bM_G$, the `jump assignments'' determined as indicated in Figure~\ref{fig:bij-labeling-marked} can be satisfied by a $d$-labeling of $G$ (recall that a \emph{$d$-labeling} is simply an assignment of a value in $[d]$ to each inner corner of $G$).

The \emph{corner-graph} of $G$ is the directed graph $C_G$ defined as follows: the vertices of $C_G$ are the inner corners of $G$, and there is an oriented edge from a corner $c$ to a corner $c'$ in $C_G$ if $c'$ is the corner following $c$ in clockwise order around a face or vertex. The corner graph is represented in Figure~\ref{fig:angular-map}(c).
Note that the corner graph comes with an embedding in the plane determined by $G$; it has three types of inner faces corresponding to the inner vertices, inner faces, and inner edges of $G$ respectively. 
A \emph{$d$-jump function} for $G$ is an assignment of a number in $\{0,1,\ldots,d-1\}$ to each oriented edge of the corner graph $C_G$. A $d$-jump function is called \emph{exact} if its values are equal to the label jumps of a $d$-labeling of $G$. It is easy to see that a $d$-jump function $f$ is exact if and only if every simple cycle $C$ of $C_G$ satisfies:
\begin{equation}\label{eq:jumps-exact}
\sum_{a\in C^+}f(a)-\sum_{a\in C^-}f(a)\in d\ZZ,
\end{equation}
where $C^+$ (resp. $C^-$) is the set of edges of $C_G$ appearing clockwise (resp. counterclockwise) on~$C$. Indeed~\eqref{eq:jumps-exact} is the condition that ensures that labels can be ``propagated'' according to the jump assignments without encountering any conflict. Furthermore, it is easy to check that~\eqref{eq:jumps-exact} holds if and only if it holds for every cycle which is the contour of an inner face of $C_G$. 

In conclusion, if one fixes an assignment of jumps from every inner corner to the next corner around the faces and vertices of $M$, this assignment of jumps corresponds to a $d$-labeling of $G$ if and only if the sum of assigned jumps around each vertex, face and edge of $G$ is a multiple of~$d$. 
Using this criteria, we see that for any marked orientation $\cM\in\bM_G$, Conditions (M1) and (M2) imply that the `jump assignments'' determined as indicated in Figure~\ref{fig:bij-labeling-marked} can be realized by a $d$-labeling of $G$. There is a unique such $d$-labeling such that the corners incident to $v_1$ have label 1, and we denote this $d$-labeling by $\bPhi(\cM)$. Using Condition (M0) and (M1) for $\cM$, we see that the jump assignments for $\bPhi(\cM)$ along the outer edges of $G$ are equal to 1. This implies that $\bPhi(\cM)$ satisfies Condition (L0). Lastly, it is easy to see that Conditions (M1) and (M2) for $\cM$ imply Condition (L1) for $\bPhi(\cM)$, and that Condition (M3) implies Condition (L3), while Condition (L2) holds for $\bPhi(\cM)$ by definition of the jump assignments. Thus, $\bPhi$ is a map from $\bM_G$ to $\cL_G$. 

Finally, it is clear that $\bPhi\circ \Phi=\Id_{\cL_G}$ and $\Phi\circ \bPhi=\Id_{\cM_G}$, hence these are inverse bijections between the $d$-GS labelings and the $d$-GS marked orientations of $G$. 
\end{proof}

Next, we define the bijection $\Psi$ between $d$-GS marked orientations and $d$-GS angular orientations of $G$. This bijection is represented in Figure~\ref{fig:bij-marked-angular}.
In the angular map $G^+$, the \emph{original arcs} are those on original edges (edges of $G$) while the \emph{star arcs} are those on star edges. 

\begin{definition}
Given a $d$-GS marked orientation $\cM$ of $G$ with arc weights denoted by $\om$, we define a weighted orientation $\Psi(\cM)$ of $G^+$ as follows.
\begin{compactitem}
\item The weight in $\Psi(\cM)$ of an original inner arc $a$ is $\om(a)+\deg(f)-d$, where $f$ is the face of $M$ at the left of $a$. The weight of outer arcs in $\Psi(\cM)$ is 0.
\item Let $c$ be an inner corner of $G$, let $m$ be the number of marks of $c$ and let $f$ be the face of $G$ containing $c$. The weight in $\Psi(\cM)$ of the star arc from the star vertex $v_f$ to $c$ is~$m$, and the weight of the star arc from $c$ to $v_f$ is $d-\deg(f)-m$.
\end{compactitem}
\end{definition}

\fig{width=\linewidth}{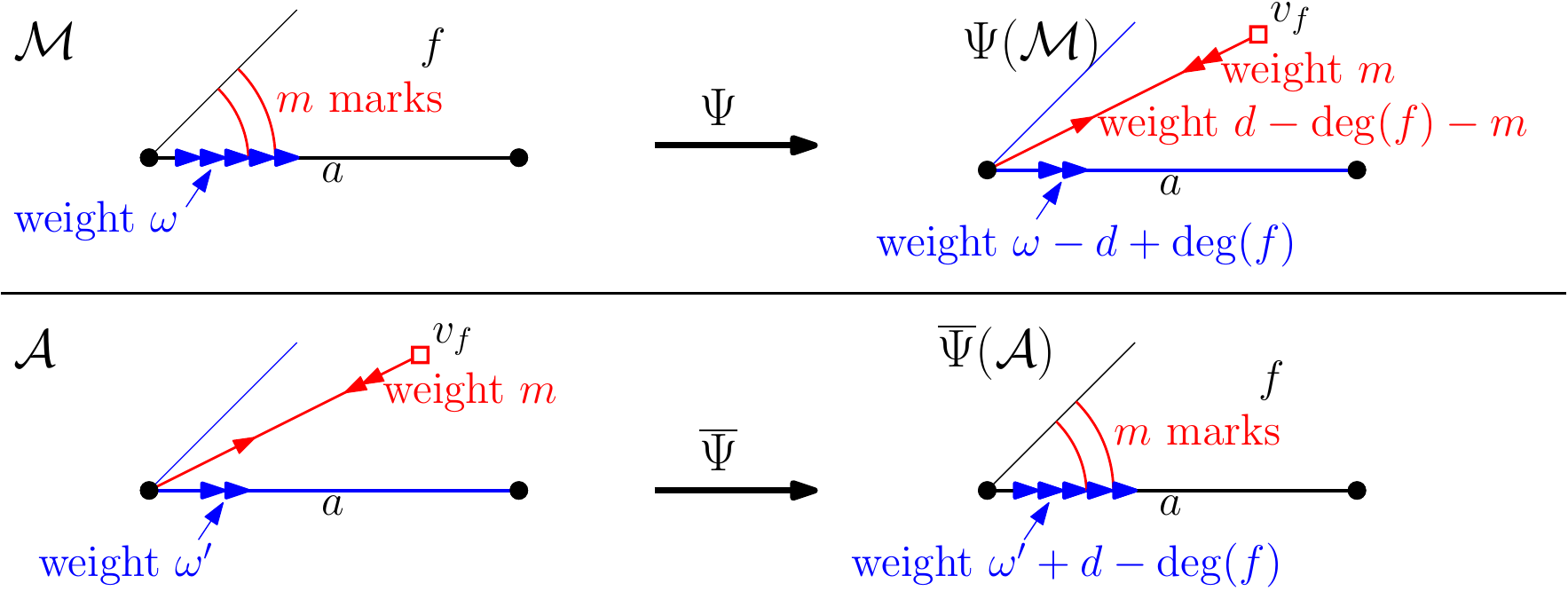}{Bijection $\Psi$ between $d$-GS marked orientations and $d$-GS angular orientations.}

\begin{prop}\label{prop:bij-gamma}
The map $\Psi$ is a bijection between the set $\bM_G$ of $d$-GS marked orientations of $G$ and the set $\bA_G$ of $d$-GS angular orientations of $G$. 
\end{prop}

\begin{proof}
We first show that the image of a $d$-GS marked orientation by $\Psi$ is a $d$-GS angular orientation. 
Let $\cM\in \bM_G$ and let $\cA=\Psi(\cM)$. It is clear that Conditions (M0) and (M1) for $\cM$ imply Conditions (A0) and (A1) respectively for $\cA$. Also Condition (M2) for an inner edge $e$ of $\cM$ (weight $d-2$ in $\cM$) implies condition (A2) for $e$ (weight $d-2-(d-\deg(f))-(d-\deg(f'))=\deg(f)+\deg(f')-d-2$ in $\cA$). Lastly, Condition (M2) for a vertex $v$ of $M$ (outgoing weight $d+\#marks$ in $\cM$) implies Condition (A2) for $v$. Indeed, for any original arc $a$ out of $v$, with $a'$ the preceding star arc in clockwise
order around $v$, the sum of the weights (in $\cA$) of $a$ and $a'$ is equal to $\om(a)-m$, where $m$ is the number of marks in the corner preceding $a$. 

It is easy to invert the mapping $\Psi$. Given a $d$-GS angular orientation $\cA$ of $G$ with arc weights denoted by $\om^+$, we define a marked orientation $\bPsi(\cM)$ of $G$ as follows.
\begin{compactitem}
\item The weight in $\bPsi(\bM)$ of an arc $a$ of $M$ is $\om^+(a)-\deg(f)+d$, where $f$ is the face of $M$ at the left of $a$. The weight of outer arcs in $\bPsi(\cA)$ is 0.
\item Let $c$ be an inner corner of $G$ in a face $f$. The number of marks of $c$ in $\bPsi(\cA)$ is the weight of the star arc from the star vertex $v_f$ to $c$.
\end{compactitem}

As before, it is easy to see that Conditions (A0), (A1) and (A2) for $\cA$ imply Conditions (M0), (M1) and (M2) respectively for $\bPsi(\cM)$. Moreover, Condition (M3) for $\bPsi(\cM)$ is immediate from the definition, hence $\bPsi(\cM)$ is a $d$-GS marked orientation.
Lastly, it is clear that $\bPsi\circ\Psi=\Id_{\bM_G}$ and $\Psi\circ\bPsi=\Id_{\bA_G}$, thus $\Psi,\bPsi$ are bijections.
\end{proof}

Lastly, we define a bijection $\Th$ between $d$-GS labelings and $d$-GS woods of $G$. This bijection is represented in Figure~\ref{fig:bij-labeling-wood}. Recall from Section~\ref{subsec:GS-woods} that we can interpret a $d$-GS wood $(W_1,\ldots,W_d)$ of $G$ in terms of ``arc colorings'': we say that an arc $a$ has color $i$ if it belongs to the spanning tree $W_i$, where $W_i$ is oriented toward its root $v_{i}$ as usual. By Condition (W1), the colors of a given arc are cyclically consecutive, and it will be convenient to use some notation for such intervals. For elements $i,j$ of $[d]$, we denote by $[i:j[$ the set of integers $\{i,i+1,i+2,\ldots,j-1\}$ modulo $d$. More precisely, if $i\leq j$ then $[i:j[:=\{k\in [d] \mid i\leq k<j\}$ and if $j<i$, then $[i:j[:=\{k\in [d]\mid k\geq i\textrm{ or }k<j\}$ (note the special case $[i:i[=\emptyset$). Roughly speaking, the bijection $\Th$ from $d$-GS labelings to $d$-GS woods is obtained by assigning to each arc $a$ the set of colors $[i:j[$, where $i,j$ are the labels preceding and following $a$. The precise definition is as follows.

\begin{definition}[Bijection between labelings and woods]
Given a $d$-GS labeling $\cL$ of $G$, we define a tuple $\Th(\cL)=(W_1,\ldots,W_d)$ of subsets of arcs of $G$ (interpreted as an arc coloring) as follows.
\begin{compactitem}
\item For all $i\in [d]$, the outer arc from $v_i$ to $v_{i+1}$ has all the colors except $i$, while the outer arc from $v_{i+1}$ to $v_{i}$ has no color.
\item An inner arc $a$ of $G$ has color set $[i:j[$, where $i$ and $j$ are the labels of the corners at the left and right of the arc $a$ respectively, at the initial vertex of~$a$. 
\end{compactitem}
Given a $d$-GS wood $\cW$ of $G$ (interpreted in terms of arc coloring), we define a corner labeling $\bTh(\cW)$ as follows.
\begin{compactitem}
\item The inner corners incident to the outer vertex $v_{i}$ receive label $i$.
\item A corner $c$ incident to an inner vertex $v$ has label $i$ if it is between the outgoing arcs of color $i-1$ and $i$ in clockwise order around $v$.
\end{compactitem}
\end{definition}

\fig{width=\linewidth}{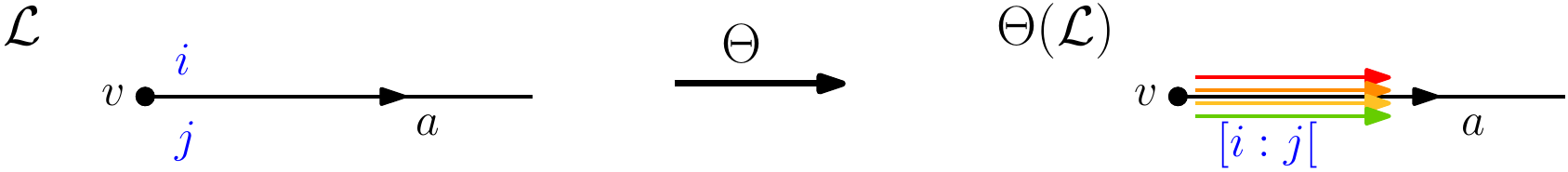}{Bijection $\Th$ between $d$-GS labelings and $d$-GS woods.}

\begin{remark}
In the definition of $\bTh$, the rule for setting the corner labels for inner vertices and outer vertices can be unified, up to using the appropriate convention. Namely, given a $d$-GS wood $\cW$ of $G$, one can add an outgoing edge of color $k$ from $v_k$ to $v_{k-1}$ for all $k\in[d]$. With this convention, every vertex $v$ of $G$ has one outgoing arc of each color, and each inner corner incident to $v$ has label $i$ if it is between the outgoing arcs of color $i-1$ and $i$ in clockwise order around $v$.
\end{remark}

\begin{prop}\label{prop:bij-theta}
The map $\Th$ is a bijection between the set $\bL_G$ of $d$-GS labelings of $G$ and the set $\bW_G$ of $d$-GS woods of $G$. The map $\bTh$ is the inverse bijection. 
\end{prop}

The proof of Proposition~\ref{prop:bij-theta} is a bit technical and we postpone it to Section~\ref{sec:remaining-proofs}.

%% file: edge-tight.tex

We give here a further incarnation of grand-Schnyder structures for a subclass of $d$-adapted maps, called \emph{edge-tight}. This incarnation
will be particularly convenient to establish connections between grand-Schnyder structures and classical objects such as transversal structures and Felsner woods.

A $d$-map $G$ is called \emph{edge-tight} if for every inner edge $e$ the incident faces $f,f'$ satisfy $\deg(f)+\deg(f')=d+2$. Equivalently, $G$ is 
edge-tight if for every inner edge $e$ the cycle enclosing the two faces incident to $e$ has length~$d$. 
It is easy to see that, for such a map, there exist some integers $g,h$ such that $g+h=d+2$ and all inner faces have degree $g$ or $h$, and every inner edge has a face of degree $g$ on one side and of degree $h$ on the other side. Note also that if $g\neq h$, then every inner vertex has even degree. The example in Figure~\ref{fig:example_edge_tight} has $g=3,h=4$. 

For a $d$-map $G$, an \emph{arc labeling} is the assignment of a label in $[d]$ to any inner arc, and the assignment of label $i$ to the outer arc $(v_i,v_{i+1})$ for all $i\in[d]$. 
A \emph{label jump} from an arc $a$ labeled $i$ to an arc $a'$ labeled $i'$ is defined as the integer $\delta\in \{0,1,\ldots,d-1\}$ such that $i+\delta=i'$ modulo $d$. For an inner vertex $v$ of $G$, we consider the arcs having initial vertex $v$, and we call \emph{$v$-clockwise jump} a label jump from such an arc to the next in clockwise order around $v$.
For an inner face $f$ of $G$, we consider the incident arcs having $f$ on their right, and we call \emph{$f$-clockwise jump} a label jump from such an arc to the next in clockwise order around $f$.

\begin{figure}
\begin{center}
\includegraphics[width=0.9\linewidth]{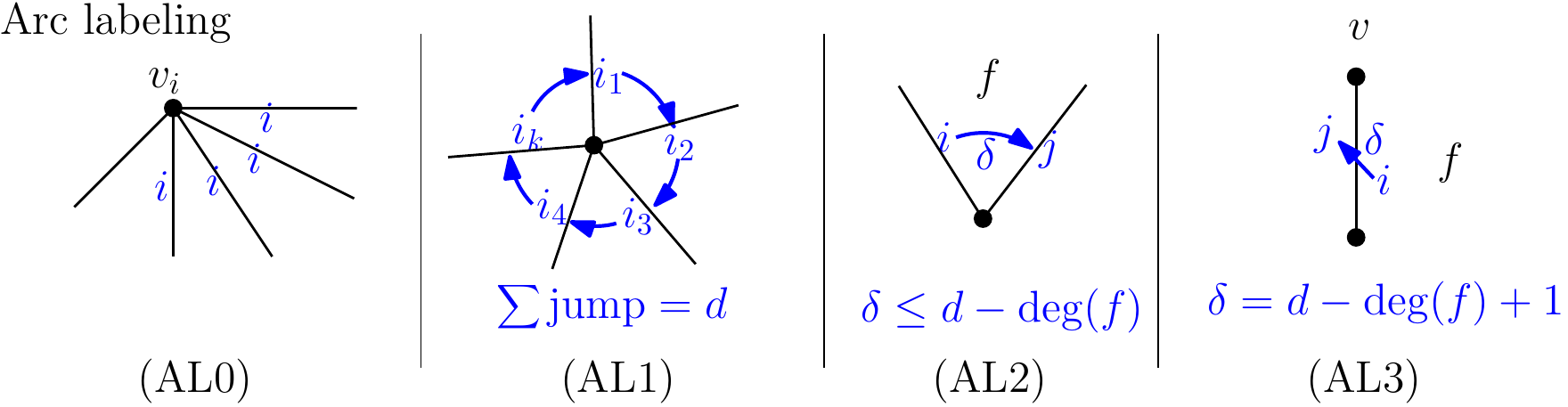}
\end{center}
\caption{Conditions defining grand-Schnyder arc labelings. The drawing convention is that arc labels are indicated on the right of each arc.}
\label{fig:local_arc}
\end{figure}

\begin{definition}\label{def:arcLabeling}
For $G$ a $d$-map, a \emph{grand-Schnyder arc labeling} of $G$ (or \emph{$d$-GS arc labeling} for short) is an arc labeling of $G$  satisfying the following conditions which are represented in Figure~\ref{fig:local_arc}:
\begin{itemize}
\item[(AL0)] For all $i\in [d]$, all the inner arcs with initial vertex $v_i$ have label $i$.
\item[(AL1)] For every inner vertex $v$ of $G$, the sum of $v$-clockwise jumps  is $d$.
\item[(AL2)] 
For every two consecutive arcs $a,a'$ in clockwise order around an inner vertex $v$ (both having initial vertex $v$), the label jump from $a$ to $a'$ is at most $d-\deg(f)$, where  $f$ is the face containing the corner between $a$ and $a'$. 
\item[(AL3)] For two opposite inner arcs $a$ and $-a$,  the label-jump from $a$ to $-a$ is $d+1-\deg(f)$, where $f$ is the face on the right of $a$. 
\end{itemize}
\end{definition}
An example is given in Figure~\ref{fig:example_edge_tight}. 

\begin{figure}
\begin{center}
\includegraphics[width=\linewidth]{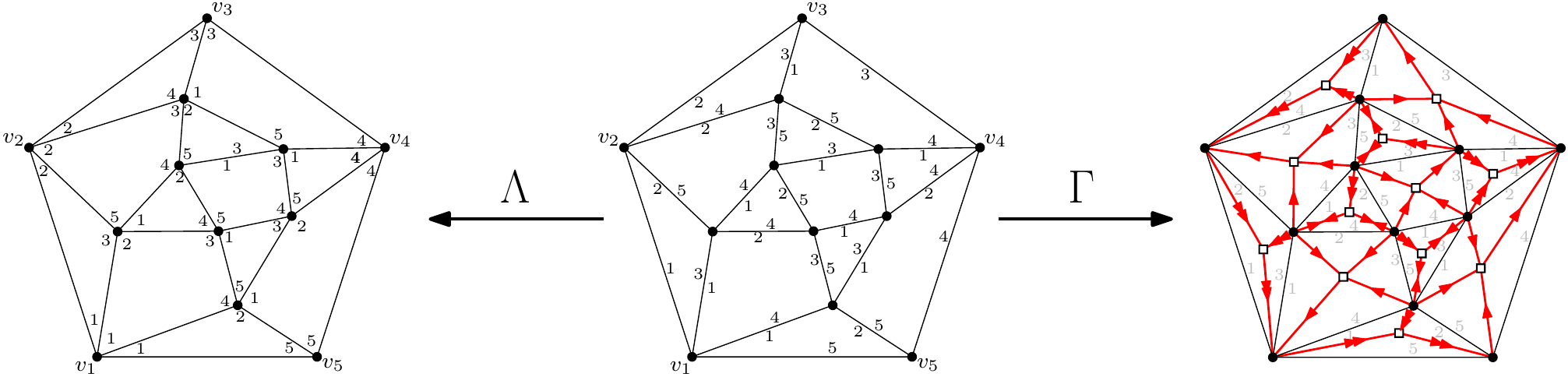}
\end{center}
\caption{Center: an edge-tight 5-adapted map (with inner face degrees in $\{3,4\}$), 
endowed with a $5$-GS arc labeling, where labels of arcs are indicated on the right of the arcs. 
Left: the corresponding corner labeling obtained by applying the bijection $\Lambda$ (that is, pulling 
the label of every arc to the corner at the initial vertex of the arc, on the right side of the arc).  
Right: the corresponding $5$-GS angular orientation obtained by applying the bijection $\Gamma$ (the weights of original arcs are 0 since the map is edge-tight).}
\label{fig:example_edge_tight}
\end{figure}

\begin{lemma}\label{lem:necessity_balanced}
If a $d$-map $G$ can be endowed with a $d$-GS arc labeling, then $G$ is edge-tight.
\end{lemma}
\begin{proof}
Let $a,a'$ be two opposite inner arcs, and let $f,f'$ be the incident faces on the right of $a$ and $a'$ respectively. Condition (AL3) implies that the label-jump from $a$ to $a'$ is $d+1-\deg(f)$, while the  label-jump from $a'$ to $a$  is $d+1-\deg(f')$. Since these jumps are non-zero (and less than $d$), they must add up to $d$, so that $\deg(f)+\deg(f')=d+2$. 
\end{proof}


We also have the following property (the analog of Property (L2) and Property (L1) for faces in the definition of $d$-GS corner labelings, the property is here a consequence of the definition):

\begin{lemma}\label{lem:arc-label}
Let $G$ be a $d$-map  endowed with a $d$-GS arc labeling. For any inner face $f$, the $f$-clockwise jumps are non-zero and add up to $d$.
\end{lemma}
\begin{proof}

For an inner arc $a$, we denote by $-a$ the opposite arc, and by $a'$ the arc following $a$ in clockwise order around the face $f$ at the right of $a$. This is represented in Figure \ref{fig:local_arc_facejumps}.
 We claim that the label jumps between $a,a'$ and $-a$ are related by 
 \begin{equation}\label{eq:arc-labeling-jumps}
 \jp(a,a')=\jp(a,-a)- \jp(a',-a).
 \end{equation}
Indeed, by (AL3) one has $\jp(a,-a)=d-\deg(f)+1$, and by  (AL2) one has $\jp(a',-a)<d-\deg(f)+1$ (this holds even if the initial vertex of $a'$ is an outer vertex since in this case $\jp(a',-a)=0$), thus $\jp(a,-a)-\jp(a',-a)>0$. This inequality implies \eqref{eq:arc-labeling-jumps}, and it also implies $\jp(a,a')>0$ for every inner arc $a$. Note that the label jump $\jp(a,a')$ from an outer arc $a=(v_{i},v_{i+1})$ to the next arc $a'$ around the incident inner face is equal to $1$ by (AL0). Thus, for every inner face $f$, all the $f$-clockwise jumps are non-zero.

\fig{width=.15\linewidth}{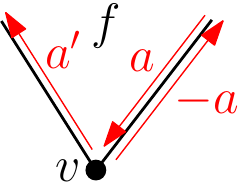}{Notation for the proof of Lemma \ref{lem:arc-label}.}
It remains to show that the $f$-clockwise jumps add up to $d$ for any inner face $f$.
 Let $V,E,F$ and $A$ be the sets of inner vertices, inner edges, inner faces, and inner arcs of $G$ respectively.
 The sum of clockwise-jumps around the inner faces of $G$ is
 $$J=d+\sum_{a\in A}\jp(a,a'),$$
where the term $d$ corresponds to the label jumps from the outer arcs $a_i=(v_i,v_{i+1})$ to the arc $a_i'$ following $a_i$ in clockwise order around the face $f_i$ at the right of $a_i$.
By~\eqref{eq:arc-labeling-jumps} we get 
$$J=d+\sum_{a\in A}\jp(a,-a)- \sum_{a\in A}\jp(a',-a).$$
Note that $\sum_{a\in A}\jp(a,-a)=d\,|E|$, since for any inner arc $a$, one has $\jp(a,-a)+\jp(-a,a)=d$. Moreover, $\sum_{a\in A}\jp(a',-a)=d\,|V|$, since this is the sum of clockwise jumps around the inner vertices of $G$ (the jumps around outer vertices have contribution 0 by (AL0)). This gives  $J=d+d\,|E|-d\,|V|=d\,|F|$ by the Euler formula.
Since the sum of clockwise jumps around every inner face of $G$ is a non-zero multiple of $d$, the identity $J=d\,|F|$ implies that the sum of clockwise jumps around each inner face is $d$.
\end{proof}

As we now explain, for an edge-tight $d$-map, the $d$-GS arc labelings are in easy bijection with the $d$-GS corner labeling. The bijection $\Lambda$ is illustrated in Figure \ref{fig:example_edge_tight}.

\begin{prop}
For $G$ an edge-tight $d$-map endowed with a $d$-GS arc labeling $\cAL$, let $\cL=\Lambda(\cAL)$ be the 
labeling of inner corners of $G$ such that any inner corner receives the label of the arc that has the corner on its right at its origin.  Then $\Lambda$ is a bijection between the set $\bAL_G$ of $d$-GS arc labelings of $G$, and the set $\bL_G$ of $d$-GS corner labelings of $G$.

Hence, by Theorem~\ref{thm:main}, a $d$-map admits a $d$-GS arc labeling if and only if it is edge-tight and $d$-adapted.
\end{prop}

\begin{proof}
Conditions (AL0), (AL1), (AL3), and Lemma~\ref{lem:arc-label} ensure that $\Lambda(\cAL)$ is a $d$-GS corner labeling. The inverse mapping $\bLa$ is as follows. For $\cL$ a $d$-GS corner labeling, $\cAL=\bLa(\cL)$ is the arc labeling of $G$ such that each arc $a$ receives the label of the corner at its origin and in the face on its right. Since $\cL$ is a $d$-GS corner labeling, Properties (AL0), (AL1) are satisfied by $A$. Next, we prove Property (AL3).
Let  $a,a'$ be opposite inner arcs, and let $f,f'$ be the faces on the right of $a,a'$ respectevily. Condition (L3) implies that the label-jump from $a$ to $a'$ satisfies $\jp(a,a')\geq 1+d-\deg(f)$, 
and that the label-jump from $a'$ to $a$ satisfies $\jp(a',a)\geq 1+d-\deg(f')$. Since 
$\deg(f)+\deg(f')=d+2$, we must have  $\jp(a,a')= 1+d-\deg(f)$ and
 $\jp(a',a)= 1+d-\deg(f')$, so that (AL3) holds. It remains to prove (AL2).
Let $a,a'$ be two consecutive arcs in clockwise order around an inner vertex $v$, let $f$ be the face on the right of $a$, and let $a''$ be the arc opposite to $a'$. 
Conditions (L1) and (L2) imply that $1\leq \jp(a'',a)\leq d+1-\deg(f)$. Since $\jp(a'',a')=d+1-\deg(f)$, we conclude that $\jp(a,a')\leq d-\deg(f)$, hence Condition (AL2) holds.  Thus, $\cAL=\bLa(\cL)$ is a $d$-GS arc labeling.

Lastly, it is clear that $\bLa\circ\Lambda=\Id_{\bAL_G}$ and $\Lambda\circ\bLa=\Id_{\bL_G}$, thus $\Lambda,\bLa$ are bijections.
\end{proof}

In terms of $d$-GS angular orientations, the condition of being edge-tight is reflected by the fact that all original edges have weight $0$. Let $G$ be a $d$-adapted edge-tight map. The bijection $\Gamma$ between the $d$-GS arc labelings of $G$ and the $d$-GS angular orientations of $G$ is represented in Figure \ref{fig:bij-arc-labeling-angular}. For an arc labeling $\cAL$ of $G$, the angular orientations $\Gamma(\cAL)$ is obtained as follows:  for each pair $a,a'$ of consecutive inner arcs in clockwise order around a vertex $v$, with $f$ the face incident to $v$ between $a$ and $a'$, the weight of the star-arc from $v$ to $v_f$ is equal to the clockwise jump from $a$ to $a'$ (moreover, letting $s_i$ be the star vertex in the face incident to $(v_{i-1},v_i)$, as usual the weight of the star-arc from $v_i$ to $s_i$ is set to $d-\mathrm{deg}(s_i)$ and the weight of the star-arc from 
$v_{i+1}$ to $s_i$ is set to $0$). An example is given in Figure~\ref{fig:example_edge_tight}.  
\fig{width=\linewidth}{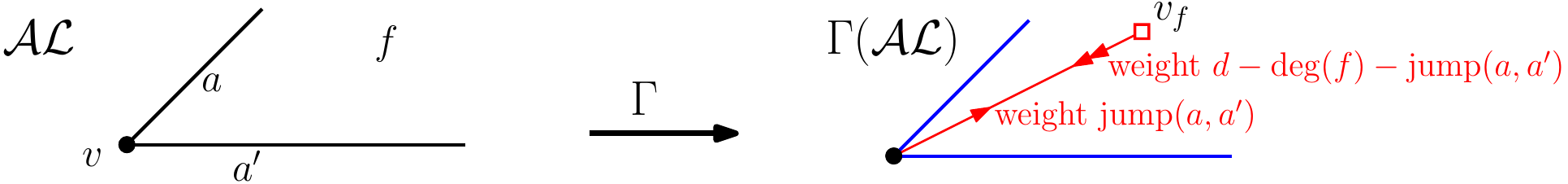}{Bijection $\Gamma=\Psi\circ\Phi\circ\Lambda$ between the sets $\bAL$ of $d$-GS arc labelings and the set $\bA$ of angular orientations for an edge-tight map $G$}

The conditions of being $d$-adapted and edge-tight are quite restrictive, and impose some constraints on the degrees of the faces. Recall that for an edge-tight $d$-adapted map there exist some integers $g, h$, with $g+h=d+2$, such that every inner edge has a face of degree $g$ on one side and a face of degree $h$ on the other side. 

\begin{lem}\label{lem:degree-edge-tight}
Let $g\leq h$. The set $\bE_{g,h}$ of edge-tight $(g+h-2)$-adapted maps with inner face degrees in $\{g,h\}$ and at least one inner vertex is non-empty if and only if $(g,h)$ belongs to the set
\begin{equation*}
S=\{(2,d),~d\geq 3\}\cup \{(3,3), (3,4), (3,5),(4,4),(5,5)\}.
\end{equation*}
Moreover, in that case, the family $\bE_{g,h}$ is infinite. 
\end{lem}


\begin{proof} 
It is easy to see that the inner vertices in edge-tight adapted maps have degree at least 3. Moreover, if $g\neq h$ the degree of vertices is even, hence at least 4.
Combining these observations with the Euler relation (and the incidence relation between faces and edges), one can check that $E_{g,h}$ is empty unless it is in the set $S$. 
About the second statement,  
the maps in $\bE_{2,d}$ are those obtained from $d$-angulations of girth $d$, opening 
every inner edge (and an arbitrary subset of the outer edges) into a face of degree $2$. 
The maps in $\bE_{3,3}$ (resp. $\bE_{4,4}$) are the so-called \emph{irreducible} triangulations of the 4-gon (resp. \emph{irreducible} quadrangulations of the hexagon) which are known to form an infinite family \cite{tutte1962census,mullin1968enumeration,FuPoScL,Fu07b,bouttier2014irreducible}.  
Finally, for the three other cases $(g,h)\in\{(5,5),(3,4),(3,5)\}$, as shown in Figure~\ref{fig:fractal}, one can construct infinitely many maps in $\bE_{g,h}$ by a ``tunnel-construction": given a sequence $L_0,\ldots,L_k$
such that $L_i$ has a marked corner in the outer face for $0\leq i<k$, and a marked corner in an inner face for $0< i\leq k$, the assembled map is obtained by identifying the 
marked outer corner of $L_i$ with the marked inner corner of $L_{i+1}$ for $0\leq i<k$ (thereby identifying the contours of their incident faces). 
\end{proof}

\begin{figure}
\begin{center}
\includegraphics[width=\linewidth]{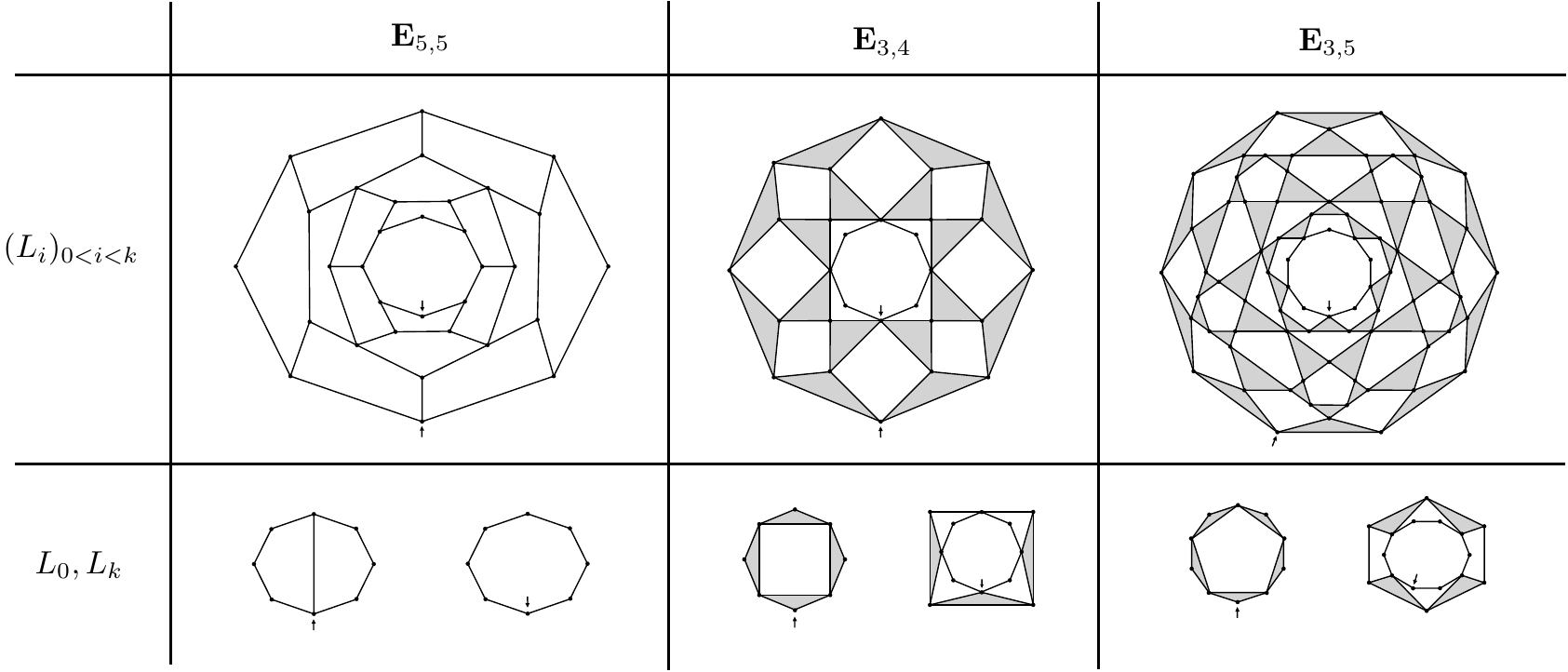}
\end{center}
\caption{The maps $(L_i)_{0\leq i\leq k}$ whose tunnel-assembling yields a map in $\bE_{5,5}$ (left column), $\bE_{3,4}$ (middle column), and $\bE_{3,5}$ (right column).} 
\label{fig:fractal}
\end{figure}


\begin{figure}
\begin{center}
\includegraphics[width=\linewidth]{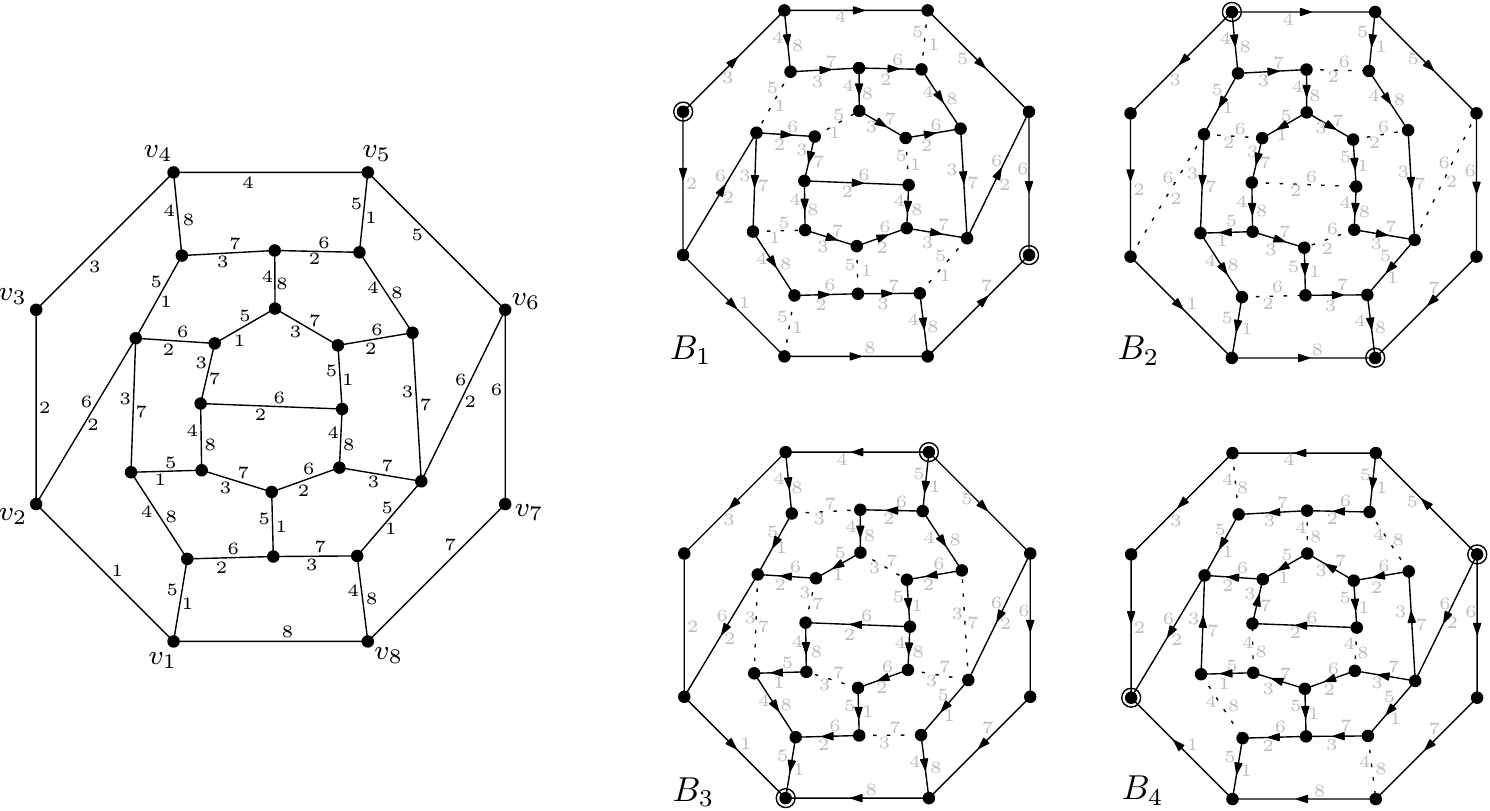}
\end{center}
\caption{Left: an edge-tight pentagulated 8-adapted map endowed with an $8$-GS arc labeling. 
Right: the 4 associated plane bipolar orientations $B_1,B_2,B_3,B_4$.}
\label{fig:octogon}
\end{figure}

For the rest of this section we focus on the special case of $2b$-adapted maps with all inner faces of the same degree $b+1$. These are edge-tight maps, and by Lemma \ref{lem:degree-edge-tight} only the values $b\in \{2,3,4\}$ are relevant. Note that for any arc labeling $\cAL$ of such a map, the  labels of any two opposite inner arcs are opposite modulo $2b$ (that is, the labels of an edge are of the form $\{j,j+b\}$). We now associate to $\cAL$ a tuple of oriented submaps, as illustrated in  Figure~\ref{fig:octogon}.
\begin{definition}\label{def:Delta}
Let $G$ be such a $(b+1)$-angulation of the $2b$-gon. To an arc labeling $\cAL$ of $G$, we associate a tuple $\beta(\cAL)=(B_1,\ldots,B_{2b})$ of oriented maps. For each  $i\in[2b]$,  the oriented map $B_i$ is obtained from $\cAL$ as follows:
\begin{itemize}
\item
The inner edges of $G$ with arc labels $\{i,i+b\}$ are removed, and the other inner edges are oriented as their arc whose label $j$ satisfy $i<j<i+b$.
\item
The outer edges of $G$ are directed so as to form two paths from $s_i:=v_{i+\lfloor b/2\rfloor}$ to $t_i:=v_{i-\lceil b/2\rceil}$.
\end{itemize} 
\end{definition}
Note that if $\beta(\cAL)=(B_1,\ldots,B_{2b})$, then the oriented map $B_{i+b}$ is the reverse of $B_i$. 

\begin{Def}
A \emph{plane bipolar orientation} is a plane map $B$ endowed with an acyclic orientation having a unique source $s$ and a unique sink $t$ (also called its target), both incident to the outer face. 
\end{Def}

Let $B$ be a plane bipolar orientation with source $s$ and sink $t$. It is well known that for any vertex $v\notin\{s,t\}$, the edges incident to $v$ form an interval of ingoing edges and an interval of outgoing edges~\cite{de1995bipolar}. 
The \emph{leftmost outgoing tree} of $B$ is the spanning tree rooted at the sink, where the parent-edge of any non-sink vertex is its leftmost outgoing edge. The \emph{rightmost ingoing tree} of $B$ is the spanning tree rooted at the source, where the parent-edge of any non-source vertex is its rightmost ingoing edge (the last one in counterclockwise order).  

\begin{prop}
Let $b\in\{2,3,4\}$, and let $G$ be an edge-tight $2b$-adapted map with inner faces of degree $b+1$, endowed with a $(2b)$-GS arc labeling $\cAL$. Then for $i\in[2b]$, the oriented map $B_i$ is a plane bipolar orientation with source $s_i$ and target $t_i$.  
\end{prop}

\begin{proof}
Let $i\in[2b]$. Property (AL2) ensures that, in $G$, the clockwise-jumps around an inner vertex $v$ are at most $b-1$, hence in $B_i$ $v$ has at least one outgoing and at least one ingoing edge. Also, by construction, the only source (resp. sink) among the outer vertices is $s_i$ (resp. $t_i$). It remains to show that $B_i$ is acyclic. 

First, one easily checks that no directed cycle can visit an outer vertex (indeed, for each of the two outer paths from $s_i$ to $t_i$, the vertices on the first half of the path have all their incident inner edges outgoing, and the vertices on the second half of the path have all their incident inner edges ingoing). 

Next, one checks that no contour of an inner face $f$ of $G$ can be a directed cycle in $B_i$. 
From what we have seen above, we can assume that $f$ is not incident to any outer vertex. In that case,  
the clockwise arcs around $f$ that are (seeing an arc as a directed edge) in $B_i$ are those of label between $i$ and $i+b$ (excluded).  By Lemma~\ref{lem:arc-label},   
all clockwise-jumps around $f$ are at most~$b$.
Hence, not all edges around $f$ can be clockwise in $B_i$, so there is no inner face of $G$ whose contour is a clockwise cycle of $B_i$. Similarly,  
no contour of inner face of $G$ is a   counterclockwise cycle of $B_i$.    

Finally, for $i\in[2b]$, assume that $B_i$ has a directed cycle, and let $C$ be a minimal one, i.e., the interior of $C$ does not contain the interior of another directed cycle $C'$. Assume there is a vertex $v$ in the interior of $C$. Since the vertices in the interior of $C$ are different from  $\{s_i,t_i\}$, from $v$ starts a directed path in $B_i$. By minimality of $C$, this path can not loop into a cycle, hence 
it has to reach a vertex on $C$. Similarly, from $v$ starts a path of edges of $B_i$ taken in reverse direction, which also has to reach a vertex on $C$. The concatenation of these two paths forms a directed path in $B_i$, in the interior of $C$, starting and ending on $C$. This contradicts the minimality of $C$. Hence there is no vertex in the interior of $C$. Similarly there can be no edge of $B_i$ inside $C$ (it would form a chord, contradicting the minimality of $C$). Since $C$ is not the contour of a face of $G$, there must a chord $e$  inside $C$ with labels $(i,i+b)$. However, if $C$ is clockwise, Condition (AL1) ensures that no arc of label $i$ can leave a vertex of $C$ toward the interior of $C$; similarly, if $C$ is counterclockwise, no arc of label $i+b$ can leave a vertex of $C$ toward the interior of $C$. 
We reach a contradiction. Hence, $B_i$ is a plane bipolar orientation.        
\end{proof}

Figure~\ref{fig:octogon} shows an example for $b=4$. We will discuss the case $b=2$ (connected to transversal structures) in Section~\ref{sec:transversal}, and the case $b=3$ (connected to Felsner woods) in Section~\ref{sec:Felsner}. 

\begin{remark}\label{rk:bipolarBi}
Let $b\in\{2,3,4\}$, and let $G$ be an edge-tight $2b$-adapted map with inner faces of degree $b+1$, endowed with a $(2b)$-GS arc labeling $\cAL$. 
There is a close connection between the plane bipolar orientations $B_i$, and the $(2b)$-GS wood $(W_1,\ldots,W_{2b})$ in bijection with $\cAL$.
Namely, for $i\in[2b]$, $W_i$ is the leftmost outgoing tree of $B_i$ (and also the rightmost ingoing tree of $B_{i+b}$), up to changing the root-vertex from $v_i$ to $v_{i-\lceil b/2\rceil}$, and changing the missing outer edge from $(v_i,v_{i+1})$ to $(v_{i+\lfloor b/2\rfloor -1},v_{i+\lfloor b/2\rfloor})$ (that is, the missing outer edge is the same for $b\in\{2,3\}$, and changes to $(v_{i+1},v_{i+2})$ for $b=4$).
\end{remark}


%% file: other_structures.tex

In this section, we explain how the $d$-GS framework emcompasses previously known structures such as Schnyder woods (and more generally Schnyder decompositions) and transversal structures (a.k.a. regular edge labelings).

\subsection{Schnyder decompositions as GS-structures on $d$-angulations}\label{sec:Schnyder-decompositions}
We start by showing that Schnyder woods, as introduced in \cite{Schnyder:wood1,Schnyder:wood2}, can be identified with the grand-Schnyder structures on triangulations. Recall that a plane triangulation admits a Schnyder wood if and only if it it has no loop or multiple edges, that is, if its girth is $3$. A generalization of Schnyder woods, are the \emph{Schnyder decompositions}  defined in \cite{OB-EF:Schnyder} for all $d\geq 3$. Several incarnations of Schnyder decompositions were given in \cite{OB-EF:Schnyder}, and it was shown that a $d$-angulation admits a Schnyder decomposition if and only if its girth is $d$ (that is, every cycle has length at least $d$).

Let us now compare Schnyder decompositions to grand-Schnyder structures. The first observation is that a $d$-angulation is $d$-adapted if and only if its girth is $d$. Let $G$ be a  $d$-angulation of girth $d$. By definition the  $d$-GS marked orientations of $G$ have no mark at all, hence the definition of $d$-GS marked orientations simplifies: these are the weighted orientations where the outer vertices and edges have weight $0$, the inner edges have weight $d-2$, and the inner vertices have weight $d$. Such weighted orientations are called \emph{$d/(d-2)$-orientations} in~\cite{Bernardi-Fusy:dangulations,OB-EF:Schnyder}, and they are one of the incarntions of Schnyder decompositions. Hence, the notions of  Schnyder decompositions and grand-Schnyder structures coincide for $d$-angulations. Note that for the original case $d=3$, the $d/(d-2)$-orientations are classical orientations (weight 1 per inner edge) such that every inner vertex has outdegree 3.

Let us now discuss the other incarnations of $d$-GS structures for $d$-angulations, and how they compare to the definitions in \cite{OB-EF:Schnyder}
First, observe that for a $d$-angulation $G$, the  $d$-GS angular orientation have no weight on star edge, hence the notion of  $d$-GS angular orientation again simplifies and identifies with $d/(d-2)$-orientations.

Next, consider $d$-GS labelings of $d$-angulations.
In each inner face, Conditions~(L1)  and (L2) imply that 
the corner labels have to be $1,2,\ldots,d$ in clockwise order around the face (no label is missing), while
the total clockwise jump around any inner vertex is $d$. With Conditions~(L0),(L1),(L2) we thus 
recover the definition of corner labelings as defined in~\cite{OB-EF:Schnyder} for $d$-angulations of girth $d$. Condition~(L3) is actually redundant in that case since it is a consequence of (L2). 
Moreover Lemma~\ref{lem:ccw-jumps-edges} ensures that for each inner edge $e=(u,v)$, the clockwise
jump accross $e$ at $u$ plus the clockwise jump across $e$ at $v$ add up to $d-2$, so that clockwise-jumps across edges at inner vertices are at most $d-2$. In particular, for $d=3$, 
at inner vertices all corner labels appear (the incident corners form three non-empty groups, of 1's, 2's and 3's in clockwise order), as in the definition of Schnyder labelings for triangulations~\cite{Schnyder:wood1}.  

Finally, let us consider the $d$-GS woods of $d$-angulations.
Conditions (W0), (W1), and (W2) in Definition~\ref{def:woods} give exactly the definition
of Schnyder woods for $d$-angulations of girth $d$ as defined in~\cite{OB-EF:Schnyder}. This includes the classical case $d=3$ of Schnyder woods of triangulations~\cite{Schnyder:wood1} (with a slight change on the rooting convention: in classical Schnyder woods, for $i\in\{1,2,3\}$ the tree $W_i$ would be naturally rooted at $v_{i-1}$, and the missing outer edge would be $(v_{i},v_{i+1})$). For $d$-angulations, the additional Condition~(W3) in Definition~\ref{def:woods} is redundant.  
Indeed, it only gives the constraint that no inner arc in $W_i$ is ending at $v_i$ or $v_{i+1}$ (which is already required by (W0)), 
and that, for every inner arc $a$ ending at an inner vertex $v$, if $a$ is in $W_i$ and the opposite arc is in none of $W_1,\ldots,W_d$, then $a$ appears strictly between  the outgoing arcs $a_{i+1}$ and $a_i$ in clockwise order around $v$ (which is already required by (W2)).  

We mention that, in the definition of Schnyder woods for $d$-angulations given in~\cite{OB-EF:Schnyder}, it is also required that every edge belongs to $d-2$ trees, but this requirement actually follows from (W0), (W1), (W2). Indeed, these conditions easily imply that every inner edge belongs to at most $d-2$ trees, while the total number of edges in the $d$ trees is $d$ times the number of inner vertices, which is also $(d-2)$ times the number of inner edges.


\subsection{Transversal structures as GS-structures on triangulations of the square}\label{sec:transversal} 
Recall that regular edge-labelings were introduced by He in~\cite{He93:reg-edge-labeling}. They were later rediscovered by the second author~\cite{Fu07b} who coined the term \emph{transversal structures}, which we will adopt here. Transversal structures are defined on \emph{triangulations of the square} (that is, 4-maps such that every inner face has degree 3). We now show that the notions of transversal structures and $4$-GS structures coincide for triangulations of the square.

\begin{figure}[h!]
\begin{center}
\includegraphics[width=12cm]{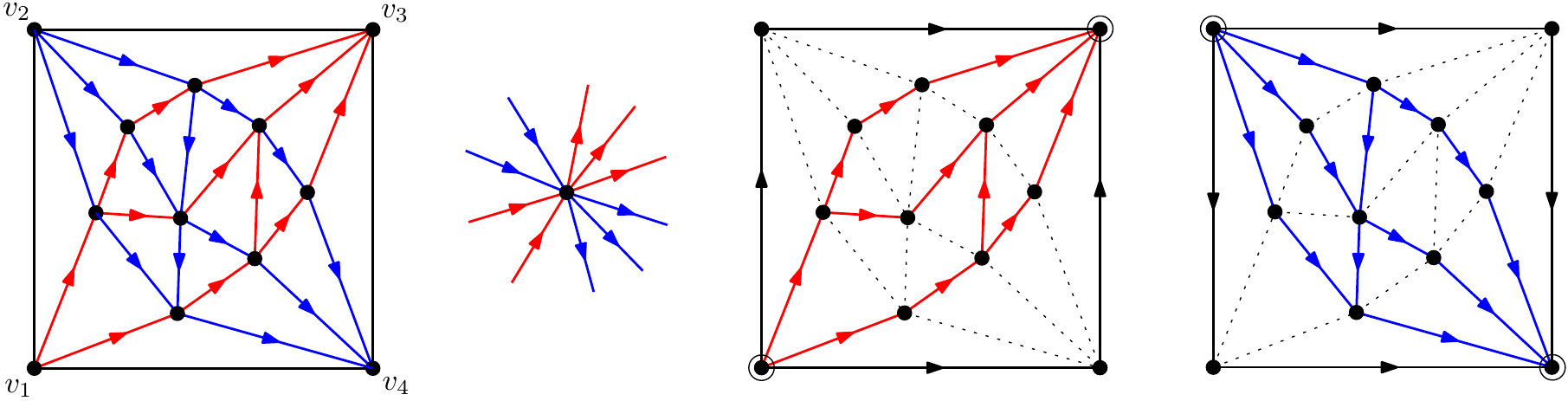}
\end{center}
\caption{From left to right: a transversal structure on a triangulation of the square, the local condition at inner vertices, the red bipolar orientation, and the blue
bipolar orientation.}
\label{fig:transvers}
\end{figure}

First, let us recall that a triangulation of the square admits a transversal structure if and only if its non-facial cycles have length at least 4.\footnote{These are sometimes called  \emph{irreducible} triangulations of the square in the literature~\cite{Fu07b,bouttier2014irreducible}.} In other words, a triangulation of the square admits a transversal structure if and only if it is $4$-adapted. Note that triangulations of the 4-gon are also \emph{edge-tight} in the sense of Section \ref{sec:arc_labeling}.

For a triangulation of the square $G$, a \emph{transversal structure} is
an orientation and partition  of the inner edges into red and blue oriented edges such that the following conditions hold (see Figure~\ref{fig:transvers}):
\begin{itemize}
\item[(T0):]
The inner edges at $v_1,v_2,v_3,v_4$ are respectively outgoing red, outgoing blue, ingoing red, and ingoing blue.
\item[(T1):]
Around each inner vertex, the incident edges in clockwise order form $4$ non-empty groups: ingoing red, ingoing blue, outgoing red, outgoing blue. 
\end{itemize}

There is a simple bijection $\al$ between the set $\bT_G$ of transversal structures of $G$ and the set $\bAL_G$ of $4$-GS arc labelings of $G$; see Figure \ref{fig:corresp_GS_transversal}. From a transversal structure $\cT\in \bT_G$, one defines an arc labeling $\al(\cT)$ by assigning the label 1 (resp. 2) to the red (resp. blue) arcs of $\cT$, and the label 3 (resp. 4) to the opposite of red (resp. blue) arcs of $\cT$. It is clear that  $\al(\cT)$ satisfy conditions (AL0-AL4) of arc labelings, hence $\al(\cT)\in \bAL_G$. Conversely, for $\cAL\in\bAL_G$, Condition (AL3) implies that opposite arcs have opposite label modulo 4, and Condition (AL2) implies that around every inner vertex there are 4 non-empty groups of outgoing arcs of label 1,2,3, and 4 respectively (in this order clockwise around $v$). Hence,  one can construct a transversal structure $\bal(\cAL)$ whose red (resp. blue) arcs are the arcs labeled 1 (resp. 2) in  $\cAL$. Obviously $\al$ and $\bal$ are inverse mappings, hence bijections, between $\bT_G$ and $\bAL_G$.

It is well known~\cite{He93:reg-edge-labeling} that 
 a transversal structure yields two plane bipolar orientations: the \emph{red bipolar orientation} is obtained by erasing the blue edges, and orienting the outer edges in the direction from $v_1$ to $v_3$; the \emph{blue bipolar orientation} is obtained by erasing the red edges, and orienting the outer edges in the direction from $v_2$ to $v_4$.  
Note that the red (resp. blue) bipolar orientation of the transversal structure $\cT$  is the plane bipolar orientation $B_4$ (resp. $B_1$) associated to $\cAL=\al(\cT)$ by the mapping $\beta$ of Definition \ref{def:Delta}.

\begin{figure}
\begin{center}
\includegraphics[width=12cm]{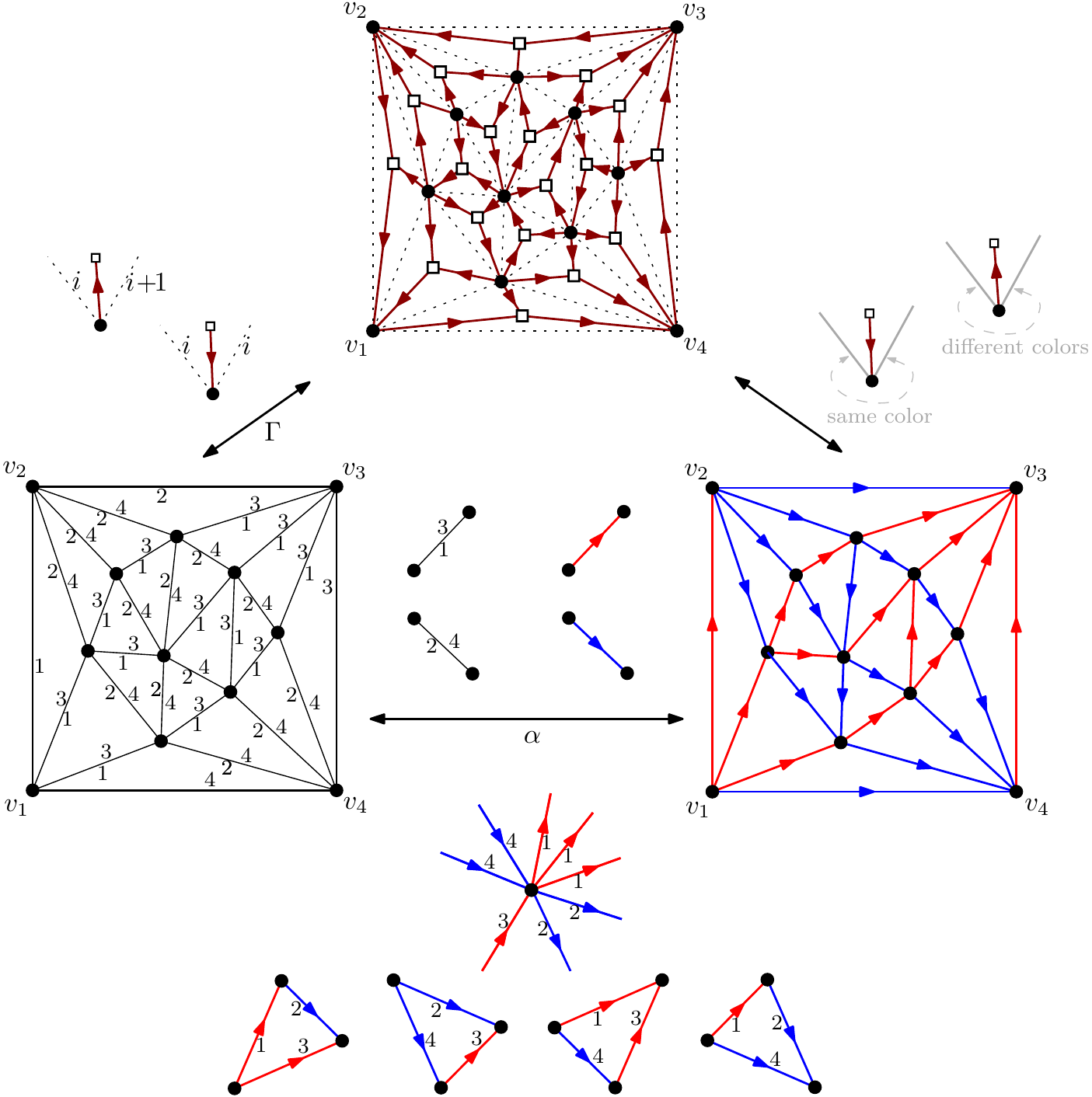}
\end{center}
\caption{On the left, a $4$-GS arc labeling on a triangulation of the 4-gon. On the right, the corresponding transversal structure (upon orienting the outer edges from $v_1$ toward $v_3$,  and coloring $(v_i,v_{i+1})$ red/blue for $i$ odd/even). The two bottom rows show the  local conditions at inner vertices and inner faces (of 4 possible types) when superimposing both structures. The top row shows the associated $4$-GS angular orientation. 
}
\label{fig:corresp_GS_transversal}
\end{figure}




We now discuss other incarnations of $d$-GS structures for triangulations of the square, and  compare them to the known incarnations of  transversal structures  \cite{Fu07b}. We start with the incarnation in terms of $4$-GS woods. Let $\cT$ be a transversal structure, let  $\cAL=\al(\cT)$ be the corresponding arc labeling, and les  $\cW=(W_1,W_2,W_3,W_4)$ be the 4-GS wood associated to $\cAL$. By  Remark~\ref{rk:bipolarBi}, the tree $W_i$ (after changing its root from  $v_i$ to $v_{i-1}$) is obtained from $\cT$ as follows: the tree $W_1$ (resp. $W_3$) is the leftmost outgoing tree (resp. rightmost ingoing tree) of the blue bipolar orientation, while the tree $W_4$ (resp. $W_2$) is the leftmost outgoing tree (resp. rightmost ingoing tree) of the red bipolar orientation. 

Next, we consider the $4$-GS angular orientations for triangulations of the square. It was shown in~\cite{Fu07b} that transversal structures correspond to orientations of star edges in $G^+$, where $v_i$ has outdegree $2$ (resp. $0$) for $i\in\{1,3\}$ (resp. $i\in\{2,4\}$), original inner vertices have outdegree $4$, and star vertices have outdegree $1$. Letting $s_i$ be the star vertex in the inner face containing the outer edge $(v_{i-1},v_i)$, these coincide with the $4$-GS angular orientations for that case, upon returning the edges $(v_2,s_2), (v_1,s_2), (v_4,s_4), (v_3,s_4)$. Moreover, as shown in the top-part of Figure~\ref{fig:corresp_GS_transversal}, the correspondence in~\cite{Fu07b} commutes with our correspondence from $4$-GS arc labelings to $4$-GS angular orientations.

%

%% file: bipartite.tex

In this section we study a class of structures that can be defined on bipartite $d$-maps, when $d=2b$ is an even integer. We will call such structures \emph{$b$-bipartite-grand-Schnyder structures}, or \emph{$b$-BGS structures} for short. These structures 
can be identified with a subclass of $2b$-GS structures called \emph{even} (even $2b$-GS structures form a proper, non-empty, subclass of $2b$-GS structures for bipartite $2b$-maps). As we will see, specific classes of $b$-BGS structures can be identified with previously known structures.

We start this section by giving the four incarnations of $b$-BGS structures and then state their equivalence. 
In the second part we will state the existence condition for $b$-BGS structures. 
In the third part we explain the connection between $b$-BGS structures and previously known structures.

\subsection{Incarnations of bipartite grand-Schnyder structures}\label{sec:incar_bipartite}\hfill\\
In this section, we fix an integer $b\geq 2$ and a bipartite $2b$-map $G$. 
We also fix the bicoloring of the vertices of $G$ in black and white by requiring the outer vertex $v_1$ to be black (and hence the outer vertex $v_i$ is black if and only if $i$ is odd). 

We now define 4 incarnations of $b$-BGS structures. They are represented in Figure~\ref{fig:bipartite_compiled2}. Our first incarnation is in terms of corner labelings.
\fig{width=\linewidth}{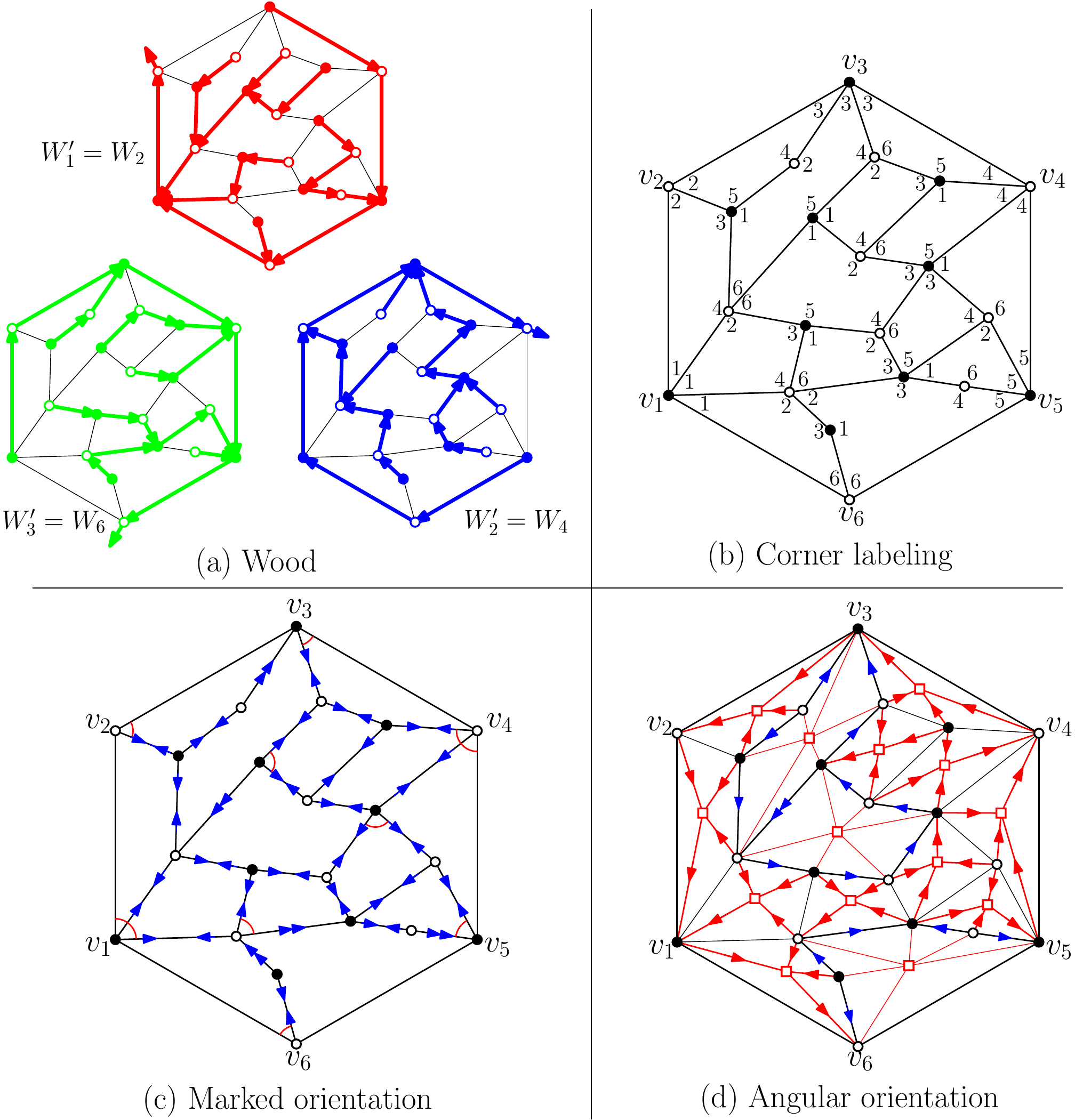}{Four incarnations of a 3-bipartite-grand-Schnyder structure.}

\begin{definition}\label{def:BGS-labeling}
A \emph{$b$-bipartite-grand-Schnyder corner labeling}, or \emph{$b$-BGS labeling}, of $G$ is a $2b$-GS labeling of $G$ such that the corners incident to black vertices have odd labels, while the corners incident to white vertices have even labels.
\end{definition}


The parity condition is equivalent to requiring the \emph{label jumps} (as defined in Section~\ref{sec:incarnations}) between consecutive corners in clockwise order around a vertex to be even, and the label jumps between consecutive corners in clockwise order around an inner face to be odd. This observation underlies the other incarnations of BGS structures.

First, recall the mapping $\Phi$ defined in Section~\ref{sec:statements} between corner labelings and marked orientations. Observe that by specializing $\Phi$, the set of $b$-BGS labelings is in bijection with the subclass of $2b$-GS marked orientations of $G$ such that the weight of each inner arc is even, and the number of marks in each inner corner is even. Let us call this subclass the \emph{even $2b$-GS marked orientations} of $G$. For an even $2b$-GS marked orientation, we can divide the weight of each arc and the number of marks in each corner by 2 without loss of information. This leads to the following definition.

\begin{definition}\label{def:BGS-marked}
A \emph{$b$-bipartite-grand-Schnyder marked orientation}, or \emph{$b$-BGS marked orientation}, of $G$ is a weighted orientation of $G$ together with a corner marking satisfying the following conditions. 
\begin{itemize}
  \item[(BM0)] The weight of every outer arc is 0. For any inner arc $a$ whose initial vertex is an outer vertex $v_i$, the weight of $a$ and the number of marks in the corner of $v_i$ on the left of $a$ are both equal to $b-\mathrm{deg}(f)/2$, where $f$ is the face on the left of $a$.
  \item[(BM1)] For any inner face $f$, the total number of marks in the corners of $f$ is $b - \mathrm{deg}(f)/2$.
  \item[(BM2)] The weight of every inner edge is $b-1$, and the outgoing weight of every inner vertex $v$ is $b+m$, where $m$ is the number of marks in the corners incident to $v$.
  \item[(BM3)] The weight of every inner arc $a$ is at least $b-\mathrm{deg}(f)/2$, where $f$ is the face on the left of $a$.
\end{itemize}
\end{definition}


It is clear that $b$-BGS marked orientations are in bijection with even $2b$-GS marked orientations. Hence, from the above discussion we get:

\begin{lem}\label{lem:bij-BGS-beta}
The set $\bBL_G$ of $b$-BGS labelings of $G$ and the set $\bBM_G$ of $b$-BGS marked orientations of $G$ are in bijection. 
\end{lem}

The next incarnation is in terms of angular orientations. Recall the bijection $\Psi$ between marked orientations and angular orientations. Under $\Psi$, the even $2b$-GS marked orientations of $G$ correspond to the $2b$-GS angular orientations of $G$ such that the weight of every arc of the angular map $G^+$ is even. We shall call this subclass of $2b$-GS angular orientations \emph{even}. As before, we can divide every weight by 2, which leads to the following definition.

\begin{definition}\label{def:BGS-angular}
A \emph{$b$-bipartite-grand-Schnyder angular orientation}, or \emph{$b$-BGS angular orientation}, of $G$ is a weighted orientation of the angular map $G^+$ satisfying the following conditions: 
\begin{itemize}
  \item[(BA0)] The weights of outer arcs are 0. Any inner arc $a$ of $G^+$ whose initial vertex is an outer vertex $v_i$ has weight 0 unless $a$ is the arc preceding the outer edge ($v_i,v_{i-1}$) in clockwise order around $v_i$. 
  \item[(BA1)] The outgoing weight of any star vertex $v_f$ is $b-\mathrm{deg}(f)/2$, and the weight of every star edge incident to $v_f$ is $b-\mathrm{deg}(f)/2$.
  \item[(BA2)] The outgoing weight of any inner original vertex is $b$. The weight of any original inner edge is $\mathrm{deg}(f)/2+\mathrm{deg}(f')/2-b-1$, where $f,f'$ are the faces of $G$ incident to $e$.
\end{itemize}
\end{definition}

It is clear that $b$-BGS angular orientations are in bijection with even $2b$-GS angular orientations. Hence, from the above discussion we get:

\begin{lem}\label{lem:bij-BGS-gamma}
The set $\bBM_G$ of $b$-BGS marked orientations of $G$ and the set $\bBA_G$ of $b$-BGS angular orientations of $G$ are in bijection.
\end{lem}

\begin{remark} \label{rk:weight-frozen-bip}
Note that Conditions (BA0) and (BA1) together imply that the inner arc $a$ whose initial vertex is an outer vertex $v_i$ and is pointing toward the star vertex $v_{f_i}$, where $f_i$ is the inner face incident to the outer edge $\{v_i,v_{i-1}\}$, has weight $b-\mathrm{deg}(f_i)/2$.
\end{remark}

Now we give the final incarnation of bipartite $2b$-GS structures as a $b$-tuple of trees. Let us recall the bijection $\Theta$ defined in Section~\ref{sec:statements} between GS corner labelings and GS woods. Observe that under $\Theta$ the set of $b$-BGS labelings of $G$ is in bijection with the subclass of $2b$-GS woods satisfying the following condition: 
\begin{itemize}
\item[($\dagger$)] \emph{For every $i \in \{1,\ldots,b\}$, and each black (resp. white) inner vertex $v$, the incident outgoing arcs in $W_{2i}$ and in $W_{2i-1}$ (resp. in $W_{2i}$ and $W_{2i+1}$) are the same.}
\end{itemize}

%
 
 Let us call this subclass the \emph{even $2b$-GS woods}. Note that Condition ($\dagger$) implies that there are redundancies in considering both the odd and the even colors. Focusing on the even colors leads to the following definition.

\begin{definition}\label{def:BGS-woods}
A \emph{$b$-bipartite-grand-Schnyder wood}, or \emph{$b$-BGS wood}, of $G$ is a $b$-tuple $(W_1',\ldots,W_b')$ of subsets of arcs satisfying:\begin{itemize}
  \item[(BW0)] For all $i \in [b]$, every vertex $v \neq v_{2i}$ has exactly one outgoing arc in $W_i'$, while $v_{2i}$ has no outgoing arc in $W_i'$. For $k \neq 2i$, the arc in $W_i'$ going out of the outer vertex $v_k$ is the outer arc oriented from $v_k$ to $v_{k+1}$. Lastly, $W_i'$ does not contain any inner arc oriented toward $v_{2i}$ or $v_{2i+1}$.
  \item[(BW1)] For every inner vertex $v$, the incident outgoing arcs $a_1',\ldots,a_b'$ in $W_1',\ldots,W_b'$ are not all the same, and they appear in clockwise order around $v$.
  \item[(BW2)] Let $v$ be a black (resp. white) vertex with outgoing arcs $a_1',\ldots,a_b'$ in $W_1',\ldots,W_b'$ respectively (for $v = v_{2i}$ we adopt the convention $a_i' = (v_{2i}, v_{2i-1})$). Let $a$ be an inner arc oriented toward $v$, let $f$ be the face on its right, let $\epsilon$ be the number of sets in $W_1',...,W_b'$ containing the opposite arc $-a$ and let $m = \max(0,b-\deg(f)/2-\eps)$.

 If $a$ belongs to $W_i'$, then $a$ appears strictly between $a_{i+1+m}'$ and $a_i'$ (resp. $a_{i+m}'$ and $a_{i-1}'$) in clockwise order around $v$. 
  
 If $a$ belongs to none of the sets $W_1',\ldots,W_b'$ but is between the outgoing arcs in $W_i'$ and $W_{i+1}'$ in clockwise order around the initial vertex of $a$, then $a$ appears strictly between $a_{i+1+m}'$ and $a_{i+1}'$ (resp. $a_{i+m}'$ and $a_i'$) in clockwise order around $v$. If $m=0$ this condition means that $-a\neq a_{i+1}'$ (resp. $-a\neq a_{i}'$).
\end{itemize}
\end{definition}

\fig{width=\linewidth}{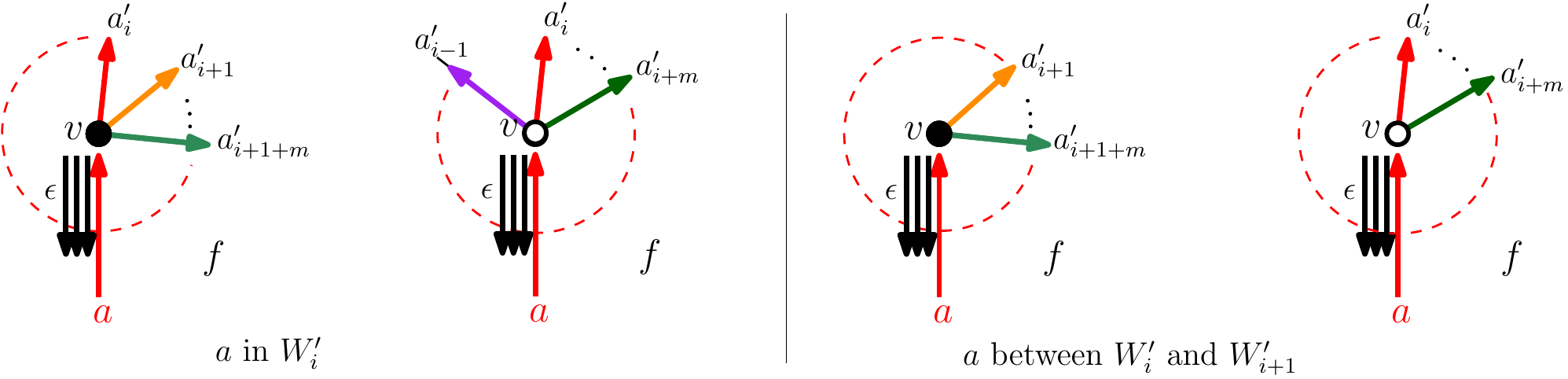}{Condition (BW2)  of $b$-BGS woods, where $m := \max(0,b-\deg(f)/2-\eps)$.}

Conditions (BW2)  are represented in Figure~\ref{fig:def-bip-woods}. We now mention a few facts about $b$-BGS woods, but omit their proofs because they are either easy or similar to the one provided for ordinary GS woods in Section~\ref{sec:remaining-proofs}.

\begin{remark}\label{rem:reduced-wood}\hfill
\begin{compactitem}
  \item For each $i \in [b]$, $W_i'$ is a spanning tree of $G$ oriented toward its root $v_{2i}$.
  \item  Similar to the case of ordinary GS woods, if $G$ is $2b$-adapted, then the last statement in Conditions (BW2) (about arcs which are in none of the trees $W_1',\ldots,W_b'$) can be dropped because it is redundant with the other conditions.
\end{compactitem}
\end{remark}

Next, we claim that $b$-BGS woods are in bijection with even $2b$-GS woods. This is not obvious and will be elaborated in the following lemma.

\begin{lemma}\label{lem:reduced-wood}
Let $G$ be a bipartite $2b$-map. Let $\cW=(W_1,...,W_{2b})$ be an even $2b$-GS wood and define the mapping $\lambda$ as $\lambda(W_1,...,W_{2b}) = (W_2,...,W_{2b}) = (W_1',...,W_b')$. Then $\lambda$ is a bijection between even $2b$-GS woods and $b$-BGS woods.
As a result, the set $\bBL_G$ of $b$-BGS labelings of $G$ and the set $\bBW_G$ of $b$-BGS woods of $G$ are in bijection.
\end{lemma}

\begin{proof}
 First we show that $(W_1',...,W_b') = \lambda(W_1,...,W_{2b})$ is indeed a $b$-BGS wood. Conditions (BW0) and (BW1) are immediate from Conditions (W0) and (W1) of ordinary $2b$-GS woods.
Condition (BW2) can also be easily deduced from Conditions (W2), (W3) and the evenness of $(W_1,...,W_{2b})$ as follows.

Let $a = (u,v)$ be an inner arc of $G$. We consider the case where $v$ is black.
Let $f$ be the face on the right of $a$, and let $2\eps$ be number of sets in $(W_1,\ldots,W_{2b})$ containing the opposite arc $-a$. Let us first assume  $a$ is in $W_i'$. By definition, $a$ is in $W_{2i}$, and since $u$ is white, $a$ is also in $W_{2i+1}$. By applying Conditions (W2) and (W3) to the color $2i$ and $2i+1$, we conclude that $a$ appears strictly between the outgoing arcs in $W_{2i+2k+2}$ and $W_{2i}$ in clockwise order around $v$, where $k = \max(0, 2b-\deg(f)-2\eps)$. This translates to $a$ being strictly between $a_{i+1+m}'$ and $a_i'$, where $m = k/2$. 

Now, suppose $a$ is not in any of the sets $W_1',...,W_b'$, but between the outgoing arcs  $W_i'$ and $W_{i+1}'$ in clockwise order around $u$. Since $u$ is white, $a$ is between the outgoing arcs in $W_{2i+1}$ and $W_{2i+2}$ in clockwise order around $u$.
The last statement of (W3) implies that $a$ appears strictly between the outgoing arcs in $W_{2i+2k+2}$ and $W_{2i+1}$ in clockwise order around $v$, where $k = \max(0, 2b-\deg(f)-2\eps)$. This translates into $a$ being strictly between $a_{i+1+m}'$ and $a_{i+1}'$. The case where $v$ is white is similar.

Next we prove that $\lambda$ is a bijection. Injectivity follows directly from definition~\ref{def:BGS-woods} as the odd colors $(W_1,W_3,...,W_{2b-1})$ can be recovered from the even ones: the outgoing edge of $W_{i}' = W_{2i}$ at a black (resp. white) inner vertex now also belongs to $W_{2i-1}$ (resp. $W_{2i+1}$). The assignment of odd colors to outer edges is different: for each $i \in [b]$, we simply force that the outer arc $(v_{2i},v_{2i+1})$ to have all the odd colors, and the outer arc $(v_{2i-1},v_{2i})$ to have all the odd colors except for color $2i-1$.

To prove surjectivity, we need to show for any $b$-tuple $\cW'=(W_1',...,W_b')$ satisfying Conditions (BW0-BW2), the tuple $\cW=(W_1,W_2,...,W_{2b})$ obtained by the recovery rule outlined above is an even $2b$-GS wood. As before, it is easy to check that Conditions (BW0-BW2) for $\cW'$ imply Conditions (W0-W3) for $\cW$.
Hence $\cW$ is a $2b$-GS wood. Moreover it is clear from the definition that $\cW$ is even, which complete the proof of the surjectivity of $\lambda$.

The argument above immediately implies that the set $\bBL_G$ of $b$-BGS labelings of $G$ and the set $\bBW_G$ of $b$-BGS woods of $G$ are in bijection.  
\end{proof}




We conclude this subsection with following Theorem, which summarizes Lemma \ref{lem:bij-BGS-beta},  Lemma \ref{lem:bij-BGS-gamma} and Lemma \ref{lem:reduced-wood}:

\begin{thm}\label{thm:bij-BGS}
Given a bipartite $2b$-map $G$, the sets $\bBL_G$ of $b$-BGS labelings, $\bBW_G$ of $b$-BGS woods, $\bBM_G$ of $b$-BGS marked orientations and $\bBA_G$ of $b$-BGS angular orientations are in bijections.
\end{thm}

\subsection{Existence Result}
In this subsection, we state the existence theorem for $b$-BGS structures, and show that it is a consequence of Theorem~\ref{thm:main} for non-bipartite GS structures.

\begin{thm}\label{thm:BGS-main}
Let $b \geq 2$, and let $G$ be a bipartite $2b$-map. There exists a $b$-BGS wood (resp. labeling, marked orientation, angular orientation) for $G$ if and only if $G$ is $2b$-adapted. 

Moreover for any fixed $b$, there is an algorithm which takes as input a $2b$-adapted map, and outputs a $b$-BGS wood (resp. labeling, marked orientation, angular orientation) in linear time in the number of vertices. 
\end{thm}

\begin{proof}
We assume Theorem~\ref{thm:main}, which will be proved in Section~\ref{sec:proof-existence}. Note that $b$-BGS woods can be identified as a subclass of $2b$-GS wood (which is especially clear from the $b$-BGS labeling incarnation), so the necessity of $2b$-adaptedness is clear from Theorem~\ref{thm:main}. 

It remains to prove the sufficiency of $2b$-adaptedness. Let $G$ be a $2b$-adapted bipartite map. By Theorem~\ref{thm:bij-BGS} the different incarnations of $b$-BGS structures are in bijection. Hence it suffices to prove the existence of a $b$-BGS angular orientation. In turn this reduces to proving the existence of an even $2b$-GS angular orientation.


By Theorem~\ref{thm:main}, $G$ admits a $2b$-GS angular orientation $\cA$. Let us call \emph{odd} the arcs of $G^+$ having an odd weight in $\cA$. If there is no odd arc, then $\cA$ is even and we are done. Otherwise, let us explain how to produce an angular orientation $\cA'$ with fewer odd arcs. 
Since $G$ is bipartite, the degree of every face is even. Hence the outgoing weight of every vertex of $G^+$ is even and the total weight of every edge of $G^+$ is even. Now, suppose that the arc $a_1 = (u_0, u_1)$ is odd, then the opposite arc $-a = (u_1, u_0)$ is also odd. Since the total outgoing weight of $u_1$ is even, $u_1$ must have another outgoing odd arc $a_2 = (u_1,u_2) $. By repeating this process we get a path of odd arcs $a_1 = (u_0, u_1), a_2 = (u_1,u_2),\ldots$. Note that this path cannot reach the outer vertices, because every arc incident to an outer vertex has even weight. Therefore, the sequence of odd arcs must contain a directed cycle. Note that if we subtract 1 from the weight of every arc in this cycle and add 1 to their opposite arcs, the resulting weighted orientation $\cA'$ is still a $2b$-GS angular orientation of $G$. Moreover, $\cA'$ has fewer odd arcs as promised, and repeating this process leads to an even $2b$-GS angular orientation. 

For the runtime, if $G$ has $n$ vertices, once we obtain some $2b$-GS angular orientation $\cA$, which takes linear time by Theorem~\ref{thm:main}, testing whether $\cA$ is even, or finding all the odd arcs if it is not, takes only linear time since the total number of edges is linear in $n$. The runtime of eliminating all odd arcs is linear in the total number of odd arcs. Hence the total runtime to construct an even $2b$-GS angular orientation is linear in $n$. The conversion from an even $2b$-GS structure to a $b$-BGS structure takes linear time, and the bijections between the different incarnations are also linear, which concludes the proof.
\end{proof}

\begin{Remark}
In the above proof we showed that Theorem~\ref{thm:main} implies Theorem~\ref{thm:BGS-main}. We mention that, conversely, Theorem~\ref{thm:BGS-main} implies Theorem~\ref{thm:main}. Indeed, let us assume Theorem~\ref{thm:BGS-main}, and prove the existence of $d$-GS marked orientations for $d$-adapted maps. Given a $d$-adapted map $G$ we can draw an ``edge vertex'' at the center of every edge to obtain a bipartite $2d$-adapted $2d$-map $\overline{G}$. By Theorem~\ref{thm:BGS-main}, $\overline{G}$ admits a $2d$-GS marked orientation. We can then delete the edge vertices and obtain a $d$-GS marked orientation of $G$ in the way illustrated in Figure~\ref{fig:bgs_to_gs}.
\end{Remark}

\fig{width=\linewidth}{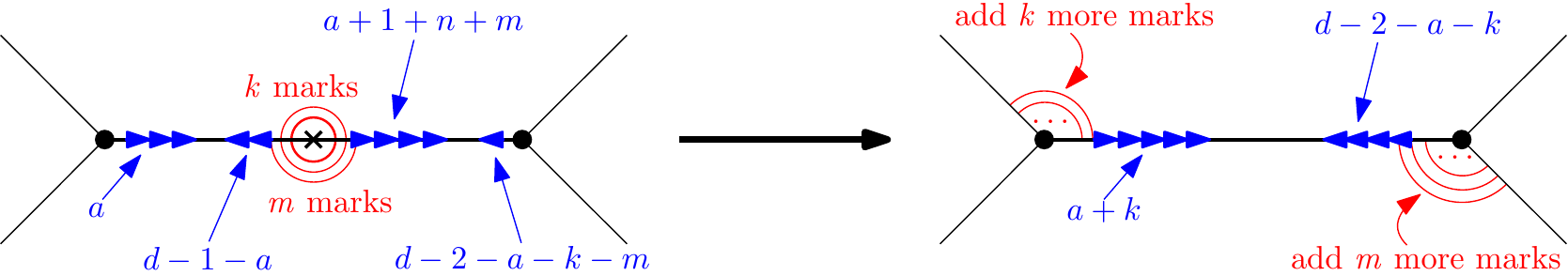}{Merge an edge of $\overline{G}$.}

\subsection{Connections to Known Structures}\label{sec:bip-classical}



%% file: other_structures_bipartite.tex

\subsubsection{Bipartite grand-Schnyder structures on $2b$-angulations of girth $2b$, and their relation to 2-orientations and bipolar orientations}\label{sec:bipartite_angulations}

In this subsection we relate $b$-BGS structures to the \emph{even Schnyder decompositions} defined in~\cite{OB-EF:Schnyder} for $2b$-angulations. The special case $b=2$ is the most classical, as $2$-BGS structures are in bijection with plane bipolar orientations.

Recall from Section~\ref{sec:classical} that a $d$-angulation $G$ is $d$-adapted if and only if it has girth $d$, and that all four incarnations of $d$-GS structures can be simplified in this case. These structures where studied in~\cite{OB-EF:Schnyder} under the name of \emph{$d$-Schnyder structures}. 
In particular, the $d$-GS marked orientation incarnation and the $d$-GS angular orientation incarnation both simplify into the same type of weighted orientations of $G$ (with no marks) called \emph{$d/(d-2)$-orientations}. In~\cite{OB-EF:Schnyder} the $d$-GS corner labelings and the $d$-GS woods of $d$-angulations were called \emph{$d$-Schnyder labelings} and \emph{$d$-Schnyder decompositions}, respectively.

When $d=2b$ is an even integer, a (nonempty) subclass of Schnyder structures on $2b$-angulations of girth $2b$ called \emph{even} was studied in~\cite{OB-EF:Schnyder}. 
The class of even $d$-Schnyder structures can easily be identified with the class of $b$-BGS structures of $2b$-angulations:
\begin{compactitem}
\item A $2b$-Schnyder labeling is \emph{even} if all the corners incident to black (resp. white) vertices have odd (resp. even) labels. This characterization exactly coincides with the definition of $b$-BGS labelings on $2b$-angulations.
\item A $2b/(2b-2)$-orientation is \emph{even} if the weight of every inner arc is even. Dividing every weight by 2 gives a structure called \emph{$b/(b-1)$-orientation} in~\cite{OB-EF:Schnyder} (these are weighted orientations of the inner edges such that edges have weight $b-1$ and vertices have weight $b$). For a $2b$-angulation the $b/(b-1)$-orientations exactly coincide with the $b$-BGS marked orientations (no mark) and $b$-BGS angular orientations (weight 0 on star edges).
\item A $2b$-Schnyder decomposition is \emph{even} if for every $i \in \{1,...,d\}$, and each black (resp. white) inner vertex $v$, the arcs leading $v$ to its parent in $W_{2i}$ and in $W_{2i-1}$ (resp. in $W_{2i}$ and $W_{2i+1}$) are the same. It was shown in~\cite{OB-EF:Schnyder} that keeping only the trees of even color does not result in any loss of information. This simplified structures, called \emph{reduced Schnyder decompositions} in~\cite{OB-EF:Schnyder}, coincide with the $b$-BGS woods of $2b$-angulations.
\end{compactitem}

The case $b=2$ (of $b$-BGS structures on $2b$-angulations) is classical and precedes~\cite{OB-EF:Schnyder}. Let $G$ be a quadrangulation. By definition, a $2/1$-orientation of $G$ is simply an (unweighted) orientation of the inner edges of $G$ such that every inner vertex has outdegree 2. These are simply called \emph{2-orientations} of $G$, and $G$ admit such an orientation if and only if it is simple (that is, has no double edge, which is equivalent to having girth 4 in this case). 
Next consider the corner labeling incarnation: because of the parity condition in BGS corner labeling, there is no loss of information in replacing each label $i$ by $\lfloor (i-1)/2 \rfloor$. 
This incarnation of 2-orientations was studied by Felsner et al. in~\cite{FeHuKa}.

As explained in the introduction, $2$-oriented quadrangulations are in bijection with plane bipolar orientations. The bijection is given by Figure~\ref{fig:2-orientations}. 


Let us finally mention that 2-orientations were used by Barri\'ere and Huemer~\cite{Barriere-Huemer:4-Labelings-quadrangulation} to design a straight-line drawing algorithm for quadrangulations. These structures (in the form of, dual, even $4$-Schnyder structures) were also used in~\cite{OB-EF:Schnyder} to design a drawing algorithm for 4-regular plane maps. In the forthcoming article~\cite{OB-EF-SL:4-GS-drawing}, we present extensions of these two algorithms (and of the drawing algorithm of He~\cite{He93:reg-edge-labeling}, which is based on transversal structures).

\subsubsection{BGS structures for quadrangulations of the hexagon, and their relation to Felsner woods}\label{sec:Felsner} 

In this subsection, we consider the $3$-BGS structures for \emph{quadrangulations of the hexagon} (6-maps where inner faces have degree 4). This case bears a strong analogy to the case of transversal structures discussed in Section~\ref{sec:transversal}. For these maps (which are edge-tight) the BGS structures can be identified with certain edge colorings, and they are related to the Felsner woods of 3-connected plane maps.

\begin{figure}
\begin{center}
\includegraphics[width=\linewidth]{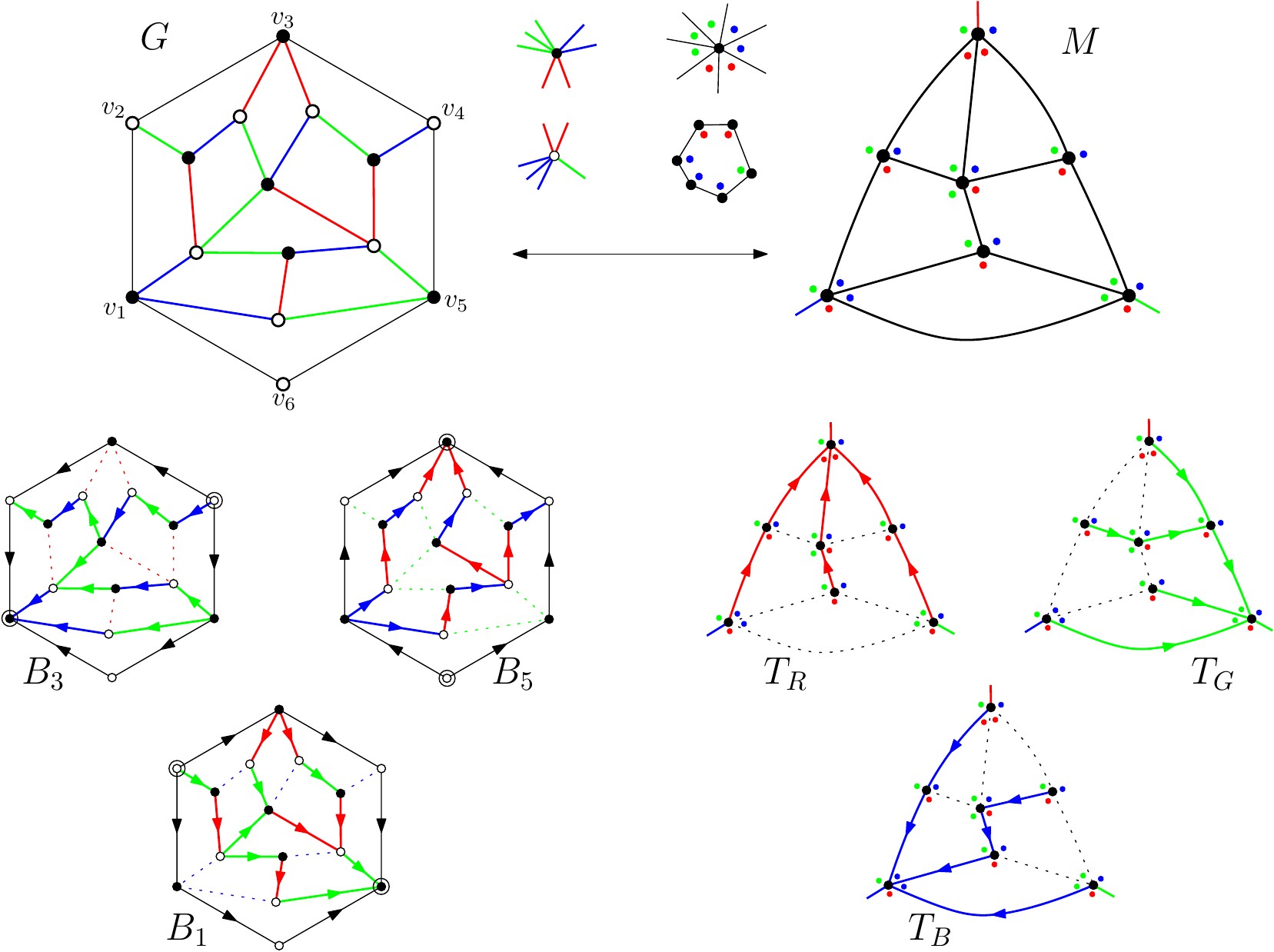}
\end{center}
\caption{On the left, a Felsner edge-coloring of an quadrangulation of the hexagon, and the 3 associated bipolar orientations (blue-green, red-blue, and green-red). 
On the right, the corresponding coloring of corners of the diagonal-map, and the 3 associated spanning trees.}
\label{fig:Felsner_woods}
\end{figure}

Let $G$ be a quadrangulation of the hexagon. Clearly, such a map is 6-adapted if it is simple and every 4-cycle bounds a face.\footnote{The 6-adapted quadrangulations of the hexagon are sometimes called \emph{irreducible quadrangulations of the hexagon} in the literature~\cite{FuPoScL,bouttier2014irreducible}.}
A \emph{Felsner edge-coloring} of $G$ is a coloring of the 
inner edges of $G$ in red, blue, green with the following properties (see the top-left part of 
Figure~\ref{fig:Felsner_woods}):
\begin{itemize}
\item[(C0)]
All inner edges incident to $v_1$ and $v_4$ are blue, all inner edges incident to $v_2$ and $v_5$ are green, and all inner edges incident to $v_3$ and $v_6$ are red.
\item[(C1)]
Around every inner vertex, the incident edges form 3 non-empty groups in clockwise order: red edges, green edges, and blue edges. 
\end{itemize} 
Felsner edge-colorings are closely related to extensions of Schnyder structures developed by Felsner~\cite{F01,Felsner:geodesic-embedbings,Felsner:lattice} for $3$-connected maps. 
Precisely, with the bi-partition of the vertices of $G$ into black and white vertices (where $v_1$ is black), one can classically associate a plane map $M$ to $G$, called the \emph{diagonal-map} of $G$, where the vertices of $M$ are 
the black vertices of $G$, and there is one edge of $M$ for each inner face $f$ of $G$, which connects the two diagonally opposed black vertices around $f$. 
The obtained map is actually a \emph{suspended map}, that is, a map with $3$ distinguished vertices ($v_1,v_3,v_5$) incident to the outer face, whose marking is indicated by a dangling half-edge incident to the outer face; the dangling half-edges at $v_1,v_3,v_5$ are colored blue, red, and green respectively. 
Let $M^{\infty}$ be the map obtained from $M$ by joining the dangling half-edges to an additional vertex $v_{\infty}$ in the outer face. 
The map $M$ is called \emph{quasi-3-connected} (case considered by Felsner) if $M^{\infty}$ is $3$-connected, which is equivalent to the fact that $G$ is 6-adapted and has at least one inner edge incident to each of $v_1,v_3,v_5$. Since each edge of $G$ corresponds to a corner of $M$, a Felsner edge-coloring is equivalent (see the top-right part of Figure~\ref{fig:Felsner_woods}) to a coloring of the 
corners of $M$ in red, blue or green such that:
\begin{itemize}
\item[(C0')]
For each color $c\in\{$red, blue, green$\}$, the corners of label $c$ in the outer face are those in the interval delimited by the dangling half-edges of the two other colors; 
and all inner corners incident to the distinguished outer vertex carrying the dangling half-edge of color $c$ have color~$c$.
\item[(C1')]
Around every non-distinguished vertex and every inner face, the incident corners in clockwise order form 3 non-empty groups: red corners, green corners, and blue corners. 
\end{itemize} 
These are exactly the \emph{Schnyder colorings} of corners defined by Felsner in~\cite{Felsner:geodesic-embedbings}, or equivalently the \emph{axial orthogonal colorings} defined by Miller in \cite{Miller:FelsnerWoods}. 
Felsner also shows that such a coloring yields $3$ spanning trees of $M$ (thus giving an extension of Schnyder woods to 3-connected plane maps). The red (resp. blue, green) tree is rooted at $v_3$ (resp. $v_1$, $v_5$), with the parent edge of each non-root vertex of $M$ being the unique incident edge marking the separation between the groups of green/blue edges (resp. of red/green edges, of blue/red edges). These trees are represented in the bottom-right part of Figure~\ref{fig:Felsner_woods}.

As we now explain, one can associate some plane bipolar orientations to a Felsner edge-coloring, in a way that is similar to the case of transversal structures treated in Section~\ref{sec:transversal}.
For a quadrangulation of the hexagon $G$ endowed with a Felsner edge-coloring, we define the \emph{red-blue bipolar orientation} as the oriented map obtained by deleting the green edges, orienting the red edges from white to black, orienting the blue edges from black to white, and orienting the $6$ outer edges in the flow-direction from $v_6$ to $v_3$. Similarly, the green-red (resp. blue-green) bipolar orientation is obtained by erasing the blue (resp. red) edges, orienting the red (resp. green) edges from black to white, orienting the green (resp. blue) edges from white to black, and orienting the outer edges in the flow-direction from $v_2$ to $v_5$ (resp. $v_4$ to $v_1$), see the bottom-left part of Figure~\ref{fig:Felsner_woods}. 

\begin{remark}
The three bipolar orientations defined above are also natural in the context of orthogonal surface representations (which are specific 3D representations) associated with the Felsner structures~\cite{Miller:FelsnerWoods,Felsner:geodesic-embedbings,felsner2008schnyder}. Then the faces of the red-blue (resp. green-red, blue-green) bipolar orientation correspond to the \emph{flats} of the orthogonal surface in the direction orthogonal to the $y$-axis (resp. $x$-axis, $z$-axis), and the dual bipolar orientation indicates order constraints on the $y$-coordinates (resp. $x$-coordinates, $z$-coordinates) of those flats so as to have a valid \emph{rigid} orthogonal surface representation of the Felsner structure. The red-blue bipolar orientation has also been recently used to obtain enumerative results on Felsner structures~\cite{enumerationFelsnerColorings}. 
\end{remark}

\begin{figure}
\begin{center}
\includegraphics[width=12cm]{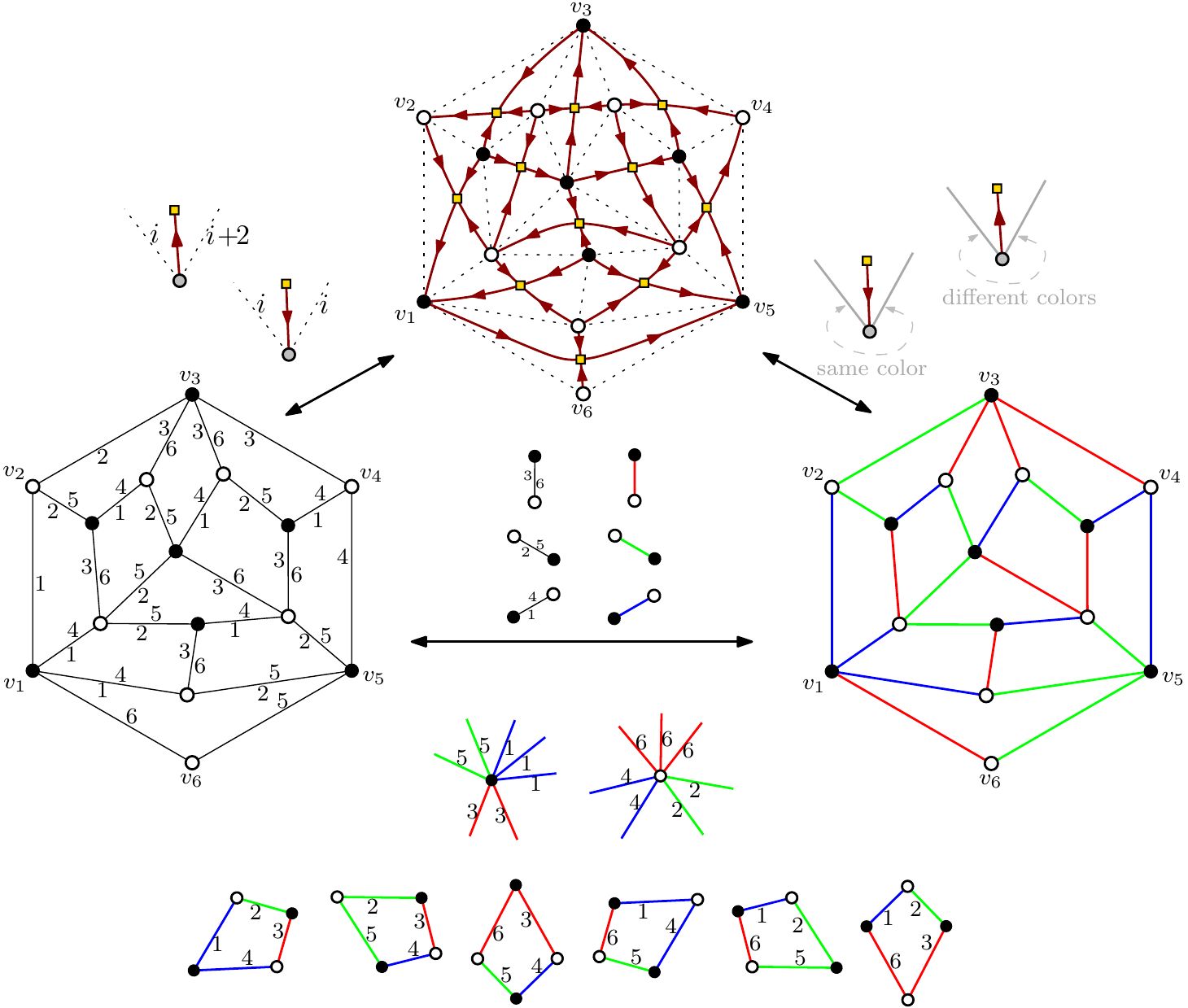}
\end{center}
\caption{On the left, a $3$-BGS arc labeling on a quadrangulation of the hexagon.
 On the right, the corresponding Felsner edge-coloring (upon coloring blue/green/red the outer edges $(v_i,v_{i+1})$ for $i=1/2/3$ modulo $3$). 
The two bottom rows show the local conditions at inner vertices and inner faces (of 6 possible types) when superimposing both structures. The top row shows the associated $3$-BGS angular orientation. 
}
\label{fig:corresp_BGS_Felsner_coloring}
\end{figure}

We now discuss the link with bipartite grand-Schnyder structures. First note that quadrangulations of the hexagon are edge-tight in the sense of Section~\ref{sec:arc_labeling}. 
As illustrated in Figure~\ref{fig:corresp_BGS_Felsner_coloring}, for $G$ a quadrangulation of the hexagon, there is a direct bijection between the Felsner edge-colorings of $G$ and the $3$-BGS structures of $G$. 
Let $\cAL$ be a $6$-GS arc labeling of $G$. It corresponds to a 3-BGS corner labeling if and only if the inner arcs with black (resp. white) initial vertex have odd (resp. even) labels. We call such an arc labeling a \emph{$3$-BGS arc labeling} of $G$. 
Condition (AL2) and the parity property imply that in a $3$-BGS arc labeling $\cAL$, around any black (resp. white) inner vertex there are 3 non-empty groups of outgoing arcs of label 1,3,5 (resp. 2,4,6). Hence, to a $3$-BGS arc labeling $\cAL$, one can associate a Felsner edge-coloring $\eta(\cAL)$ by coloring blue (resp. green, red) the inner edges with arc labels $\{1,4\}$ (resp. $\{2,5\}$, $\{3,6\}$). Conversely, to a Felsner edge-coloring $\cF$ of $G$, one associates a $3$-BGS arc labeling $\beeta(\cF)$ by giving the label 1 (resp. 3,5) to the arcs of color blue (resp. red, green) with black initial vertex, and the label 2 (resp. 4,6) to the arcs of color green (resp. blue, red) with white initial vertex. It is clear that $\eta$ and $\beeta$ are inverse mappings, hence bijections, between the set of $3$-BGS arc labelings and the set of Felsner edge-colorings of $G$.


Note that the red-blue (resp. green-red, blue-green) bipolar orientation of the Felsner edge-coloring $\eta(\cAL)$ is exactly the plane bipolar orientation $B_5$ (resp. $B_1$, $B_3$) associated to the $3$-BGS arc labeling $\cAL$ by the mapping $\beta$ (see Section~\ref{sec:arc_labeling}). Hence, according to Remark~\ref{rk:bipolarBi}, the even grand-Schnyder wood $\cW=(W_1,W_2,W_3,W_4,W_5,W_6)$ associated to $\cAL$ can be easily obtained from these bipolar orientations. Precisely, up to changing the tree-root of $W_i$ from $v_i$ to $v_{i-2}$ for all $i$, 
the tree $W_5$ (resp. $W_2$) is the leftmost outgoing tree (resp. rightmost ingoing tree) of the red-blue bipolar orientation, 
the tree $W_1$ (resp. $W_4$) is the leftmost outgoing tree (resp. rightmost ingoing tree) of the green-red bipolar orientation, and the tree $W_3$ (resp. $W_6$) is the leftmost outgoing tree (resp. rightmost ingoing tree) of the blue-green bipolar orientation. This correspondence is represented in Figure~\ref{fig:corresp_BGS_Felsner_coloring}. (Note that the trees $W_1,\ldots,W_6$ of the grand-Schnyder wood are closely related to the three bipolar orientations rather than to the three spanning trees of the Felsner wood.)

Let us finally consider the angular orientations incarnation.
In~\cite{Felsner:lattice} Felsner shows that (when $G$ has at least one inner edge incident to each of $v_1,v_3,v_5$) his corner labelings of $M$ correspond to orientations of the star edges of $G^+$ (these edges are those of the superimposition of $M$ with its dual, upon considering that there are 3 outer faces separated by the dangling half-edges) such that $v_1,v_3,v_5$ have outdegree~$2$, $v_2,v_4,v_6$ have outdegree $0$, all inner vertices of $G$ have outdegree $3$, and the star vertices have outdegree $1$. Letting $s_i$ be the star vertex in the inner face containing the outer edge $(v_{i-1},v_i)$, these orientations of $G^+$ defined in~\cite{Felsner:lattice} coincide with the $3$-BGS angular orientations of $G$, upon returning the edges $(v_2,s_2), (v_1,s_2), (v_4,s_4), (v_3,s_4), (v_6,s_6), (v_5,s_6)$. 
Moreover, as shown in the top-part of Figure~\ref{fig:corresp_BGS_Felsner_coloring}, the correspondence in~\cite{Felsner:lattice} commutes with our correspondence $\Gamma$ between $3$-BGS arc labelings and $3$-BGS angular orientations.



%% file: lattice.tex

In this section we show that the set of grand-Schnyder structures on a fixed $d$-adapted map has the structure of a distributive lattice, and that the covering relations (flip/flop operations) have a simple characterization. Possible applications are efficient algorithms for the exhaustive generation and uniform random generation  of $d$-GS structures on a fixed $d$-adapted map (as explained in Remark~\ref{rk:random_gen}). For instance, the exhaustive generation of transversal structures (equivalently, 4-GS structures for triangulations of the square), was used in~\cite{eppstein2012area} for the construction of certain rectangular tilings with prescribed tile areas. 

For bipartite $d$-adapted maps, the class of bipartite-grand-Schnyder structures also has a distributive lattice structure, and the results are similar.

\subsection{Lattice structure for orientations with prescribed weights}\label{sec:general-uld-framework}
In this subsection we recall some definitions and results from~\cite{felsner2009uld} which underlie the lattice results for grand-Schnyder structures.
Let $M$ be a plane map, let $V$ be its vertex-set, and let be $E$ its edge-set. Let $\alpha:V\to \mathbb{N}$ and $\beta:E\to\mathbb{N}$ be functions. 
An \emph{$\alpha/\beta$-orientation} of $M$ is a weighted orientation of $M$ 
such that every vertex $v\in V$ has outgoing weight $\alpha(v)$, and every edge $e\in E$ has weight $\beta(e)$. The pair $(\alpha,\beta)$ is called \emph{feasible} if 
$M$ admits at least one $\alpha/\beta$-orientation. In that case, letting $\bX$ be the set of $\ab$-orientations of $M$, it follows from~\cite[Theo.6]{felsner2009uld} (via duality, as explained in~\cite[Lem.14]{OB-EF:Schnyder}) that $\bX$ carries the structure of a distributive lattice, which we now describe. 


Recall that a \emph{positive cycle} in a weighted orientation is a cycle such that each arc of the cycle has positive weight. 
The \emph{push} of a positive cycle is the operation of decreasing by $1$ the weights of the arcs of the cycle, and increasing by $1$ the weights of their opposite arcs. Any $\alpha/\beta$-orientation can be obtained from another by a sequence of push. 
Moreover, the set $\bX$ of $\alpha/\beta$-orientations can be endowed with a lattice order $\preceq$ defined as follows: for two orientations $\cX_1,\cX_2$ in $\bX$, we have $\cX_1\preceq \cX_2$ if and only if $\cX_2$ can be obtained from $\cX_1$ by a sequence of push-operations at counterclockwise cycles. 
The minimal (resp. maximal) element $\cX_{\mathrm{min}}$ (resp. $\cX_{\mathrm{max}}$) in the lattice is the unique $\ab$-orientation with no clockwise positive cycle
(resp. no counterclockwise positive cycle).

A \emph{flip} (resp. \emph{flop}) is a push corresponding to a covering downward (resp. upward) relation in $\bX$. 
In other words, a flip (resp. flop) is a push on a clockwise (resp. counterclockwise) positive cycle, which cannot be realized by a sequence of several pushes on clockwise (resp. counterclockwise) positive cycles. Let $C$ be a simple cycle of $M$. A \emph{chordal path} for $C$ in a weighted orientation is a positive path staying strictly inside $C$ and such that its initial and terminal vertex are distinct and both on $C$.
We call a simple cycle of $M$ \emph{essential} if it has no chordal path in an $\alpha/\beta$-orientation (the existence of a chordal path does not depend on the $\alpha/\beta$-orientation considered since it is preserved by a cycle push). 
It is not hard to see that the flip/flop operations are exactly the pushes on essential cycles. Indeed a push on a clockwise (resp. counterclockwise) non-essential cycle can be performed by a sequence of 2 pushes on simple clockwise (resp. counterclockwise) cycles, while a push on a clockwise (resp. counterclockwise) essential cycle cannot be performed by a sequence of several pushes on clockwise (resp.  counterclockwise) simple cycles.\footnote{We mention that the definition of essential cycles in \cite{Felsner:lattice} is slightly incorrect: it forbids chordal paths starting and ending at the same vertex of $C$, which is too restrictive.}



The distributive lattice property is established by associating to each element $\cX\in\bX$ a ``potential-vector'' (see~\cite[Sec.2.3]{Felsner:lattice} and~\cite[Sec.3]{felsner2009uld}), which is an integer vector indicating for each essential cycle $C$ how many times $C$ is pushed in any sequence of flops from $\cX_{\mathrm{min}}$ to $\cX$; it can be shown that the constraints to be satisfied by such vectors are stable under componentwise min and componentwise max, and this defines the join and meet operation in the lattice $\bX$. 


Let us finally mention an easy extension to the context of $\ab$-orientations of plane maps with \emph{frozen edges}, that is, with a subset of edges whose arc-weights are fixed. In that context, the set of $\ab$-orientations (respecting the frozen edges conditions) is also a distributive lattice (one can argue by turning to $0$ the weights of arcs on frozen edges, and updating the functions $\alpha$ and $\beta$ accordingly), with the push operations restricted to positive cycles having no frozen edge. The minimal (resp. maximal) element in the lattice is the unique $\ab$-orientation such that all the positive clockwise (resp. counterclockwise) cycles have at least one frozen edge. 
Essential cycles correspond to cycles $C$ such that any chordal path of $C$ has at least one frozen edge.


\subsection{Lattice for grand-Schnyder structures}\label{sec:lattice_GS}
Let $G$ be a $d$-adapted map, and let $G^+$ be its angular map. 
Let $V_{G^+}$ and $E_{G^+}$ be the vertex-set and edge-set of $G^+$. For $i \in [d]$, let $s_i$ be the star vertex for the inner face of $G$ which is incident to the outer edge $(v_{i-1},v_{i})$, and let $\delta_i$ be its degree. 
Let $\alpha:V_{G^+}\to\mathbb{N}$ be the function such that:
\begin{itemize}
\item
for all $i\in[d]$, $\alpha(v_i)=d-\delta_i$,
\item
for $v$ an inner vertex of $G$, $\alpha(v)=d$,
\item
for $v$ a star vertex of degree $k$, $\alpha(v)=d-k$.
\end{itemize}
Let $\beta:E_{G^+}\to\mathbb{N}$ be the function such that:
\begin{itemize}
\item
for $e$ an outer edge, $\beta(e)=0$,
\item
for $e$ a star edge incident to a star vertex of degree $k$, $\beta(e)=d-k$,
\item
for $e$ an original inner edge of $G$ whose incident faces in $G$ have degrees $k,k'$, $\beta(e)=k+k'-d-2$.
\end{itemize} 
Declare the \emph{frozen edges} of $G^+$ as all the edges incident to outer vertices, where for all $i\in[d]$ the arc from $v_i$ to $s_i$ has weight $d-\delta_i$ (and the opposite arc has weight 0) and all the other arcs out of $v_i$ have weight 0 (and the opposite arcs have the weight required by $\beta$).
Then it follows from Definition~\ref{def:angular} and Remark~\ref{rk:frozen} that the $d$-GS angular orientations of $G^+$ are exactly the $\ab$-orientations with these frozen edges.


Thus, by the framework recalled in Section~\ref{sec:general-uld-framework}, the set $\bA_G$ of $d$-GS angular orientations of a $d$-adapted map $G$ carries the structure of a distributive lattice, where the order $\preceq$ on $\bA_G$ is defined by declaring $\cA\preceq \cA'$ if $\cA'$ can be obtained from $\cA$ by a sequence of push-operations at counterclockwise cycles not incident to outer vertices. Our aim is now to characterize the covering relation for this lattice, that is, to characterize the flip and flop operations. 

 Let $G$ be a $d$-map. A simple cycle $C$ of $G^+$ is called \emph{compatible} if it is simple, does not visit star-vertices of degree~$d$, and for every original edge $e$ belonging to $C$ the star vertex incident to the inner face of $G^+$ incident to $e$ and lying inside $C$ has degree $d$.
 For a cycle $C$ of $G^+$, the \emph{enclosing cycle} $\hC$ of $C$ is the cycle of original edges 
 (possibly not forming a simple cycle, but with simply connected interior) ``just outside" of $C$, that is, $\hC$ visits the original edges visited by $C$, and for each star vertex $u$ on $C$, $\hC$ passes by the original edges that are opposite to $u$ on the contours of the triangular faces incident to $u$ in the exterior of $C$, as shown in Figure~\ref{fig:comp_cycle}. Note that the cycle $\hC$ of $G$ cannot be the boundary of a face of $G$ of degree less than $d$, hence it has length at least $d$. The length of $\hC$ is called the \emph{enclosing length} of $C$. 
 
 \begin{figure}[h!]
\begin{center}
\includegraphics[width=0.6\linewidth]{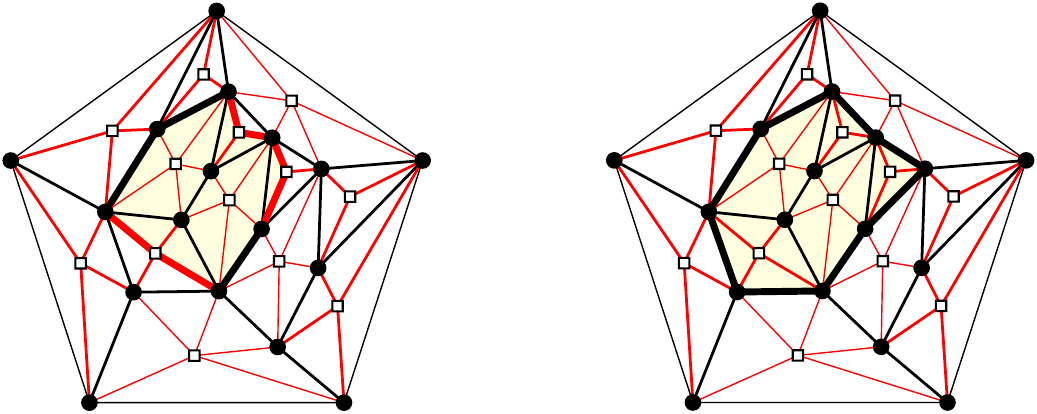}
\end{center} 
\caption{Left: a compatible cycle $C$ (shown in bold lines) on the angular map of a 5-map. Right: the same angular map, where the enclosing cycle $\hC$ of $C$ is shown in bold.}
\label{fig:comp_cycle} 
\end{figure}

Our main result is the following.

\begin{thm}\label{thm:flip}
For every $d$-adapted map $G$, the set of $d$-GS angular orientations on $G$ has the structure of a distributive lattice by the framework recalled in Section~\ref{sec:general-uld-framework}. A flip (resp. flop) for this lattice consists in pushing a clockwise (resp. counterclockwise) cycle not incident to outer vertices which is either
 \begin{itemize}
\item[(i)] the contour of an inner face of $G^+$, or 
 \item[(ii)] a compatible cycle of enclosing length $d$.
 \end{itemize}
\end{thm}

We will prove Theorem~\ref{thm:flip} in Section~\ref{sec:proof-lattice}. 

\begin{remark}\label{rk:flip_corner}
Via the bijective correspondences of Section~\ref{sec:statements}, the other incarnations of $d$-GS structures of $G$ inherit an isomorphic lattice structure. 

We first describe the covering relations in the corner labeling incarnation (see Figure~\ref{fig:flip_example}).  
For a $d$-cycle $C$ of $G$, a \emph{special corner} for $C$ is a corner $c$ of $G$ inside $C$ incident to a vertex $v\in C$, in a face of $G$ of degree smaller than $d$, and such that the edge of $G$ preceding $c$ in clockwise order around $v$ is on $C$. 
It is easy to check that, via the bijections, a flip (resp. flop) operation consists in either (i) decreasing (resp. increasing) by $1$ modulo $d$ the label of a single corner, or (ii) decreasing (resp. increasing) by $1$ modulo $d$ the labels inside a $d$-cycle $C$ of $G$ except for the special corners of $C$ whose labels are left unchanged. These operations correspond to a push on an essential cycle of type (i) or of type (ii) respectively in Theorem~\ref{thm:flip}, and they are only allowed if the resulting corner labeling is a valid $d$-GS labeling. 

We now describe the covering relations in terms of arc labelings, in the case where $G$ is edge-tight. Note that when $G$ is edge-tight all the essential cycles are of type (ii) since original edges have weight $0$ in $d$-GS angular orientations. In the incarnation as $d$-GS arc labelings, a flip (resp. flop) consists in decreasing (resp. increasing) by $1$ modulo $d$ all the arc labels inside a $d$-cycle $C$ of $G$ (of course, such an operation is only allowed if the the resulting arc labeling is a valid $d$-GS arc labeling). 
\end{remark}

 \begin{figure}[h!]
\begin{center}
\includegraphics[width=1\linewidth]{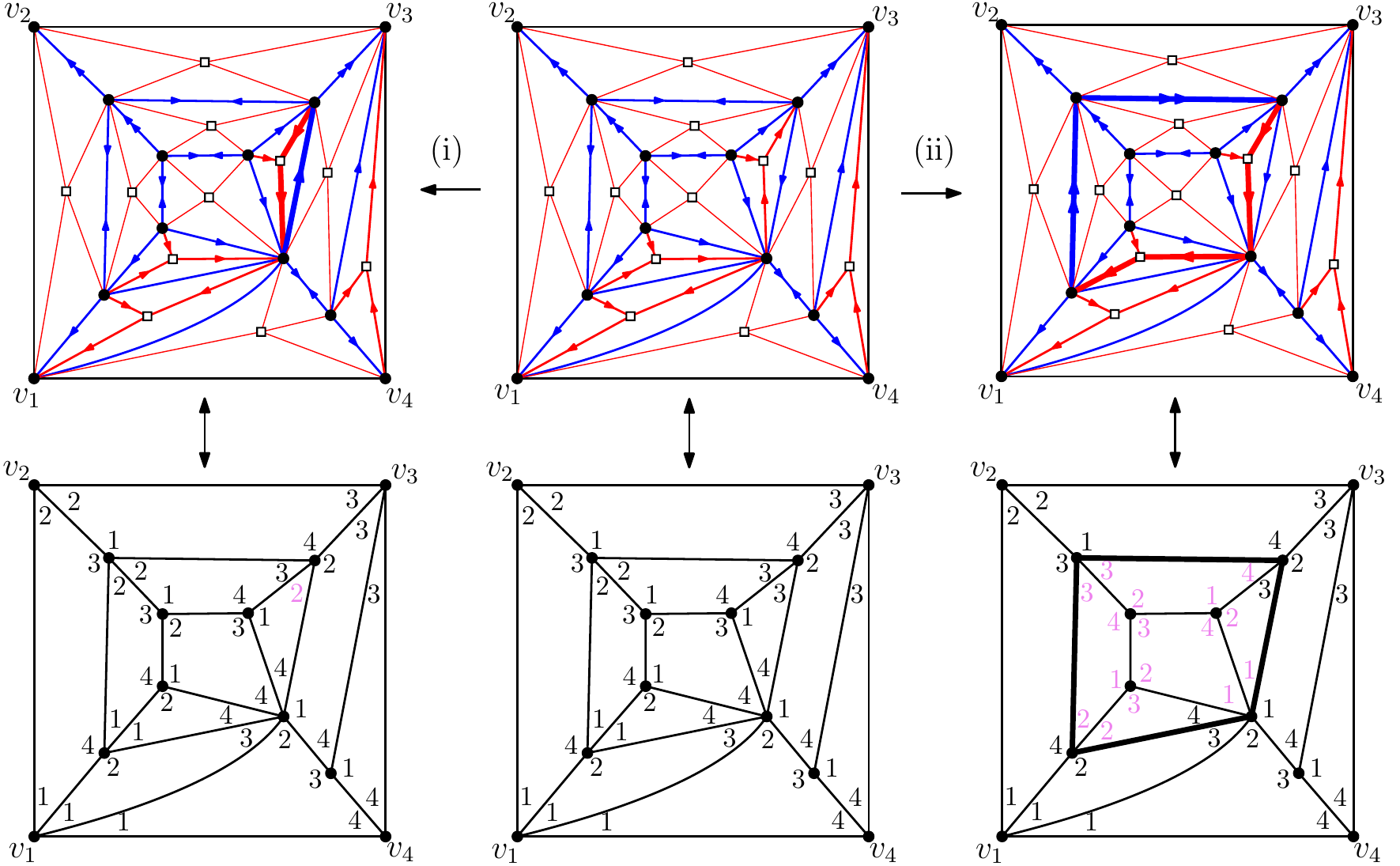}
\end{center} 
\caption{The top row shows three $4$-GS angular orientations. The left (resp. right) one is obtained from the middle one by pushing an essential clockwise cycle of type (i) (resp. 
pushing an essential counterclockwise cycle of type (ii)), where the edges of the pushed cycle are shown bolder. 
The bottom row shows the effect of the push operations on the associated $4$-GS labelings, where the modified corner labels are colored.}
\label{fig:flip_example} 
\end{figure}

\begin{remark}
It can be shown that, for a $d$-adapted map $G$, the compatible cycles of $G^+$ of enclosing length $d$ 
(corresponding to the flip/flop operations of type (ii) in Theorem~\ref{thm:flip}) are in 1-to-1 correspondence with the simple cycles of $G$ of length $d$. This is based on the following observations.
\begin{compactitem}
\item Any compatible cycle $C$ of $G^+$ has the following property $(\star)$: for any face $f$ of $G$ of degree smaller than~$d$ inside $\hat{C}$ and sharing at least one edge with $\hat{C}$, the set of edges shared by $f$ and $\hat{C}$ forms a path $P_f$ (which is not a cycle). See Figure~\ref{fig:comp_cycle}.
\item A compatible cycle $C$ is uniquely determined by its enclosing cycle $\hat{C}$. Indeed, $C$ can be recovered from $\hat{C}$ by  the following procedure:  for each face $f$ of degree smaller than~$d$ inside $C$ and sharing at least one edge with $\hat{C}$, replace the path $P_f$ by the path of length $2$ connecting the two extremities of $P_f$ via the star vertex $v_f$.
\item Any simple cycle $\hat{C}$ of $G$ satisfying Property $(\star)$ is the enclosing cycle of a (unique) compatible cycle of $G^+$. Indeed, the procedure described above produces a compatible cycle $C$ of $G^+$, whose enclosing cycle is $\hat{C}$. 
\end{compactitem}
It is easy to see, from the $d$-adaptiveness of $G$, that any simple cycle $C$ of $G$ of length $d$ satisfies Property $(\star)$. Hence, simple cycles of $G$ of length $d$ are in 1-to-1 correspondence with compatible cycles of $G^+$ of enclosing length $d$.
\end{remark}

\begin{remark}\label{rk:random_gen}
Theorem~\ref{thm:flip} can be used to define an algorithm for the uniform random generation in the set $\bA_G$ of $d$-GS angular orientations on a fixed $d$-adapted map $G$.
The algorithm is based on a Markov chain with a stopping time defined using the ``coupling from the past'' method~\cite{propp1996exact}. Let us first describe the Markov chain dynamics (expressed in~\cite[Sec.2.1]{heldt2017mixing} for outdegree-constrained orientations), that is, how the next element $A'$ is randomly chosen from the current element $A\in\bA_G$. 
Given the pre-computed list $L(G)$ of essential cycles of $G^+$ (the cycles characterized in Theorem~\ref{thm:flip}), we draw a random element $C\in L(G)$ and a random sign $\sigma\in\{-,+\}$. If $\sigma\!=\!-$ (resp. $\sigma\!=\!+$), and if $C$ forms a clockwise (resp. counterclockwise) cycle, then $A'$ is obtained from $A$ by a flip (resp. flop) at $C$; otherwise $A'=A$. This yields a symmetric irreducible Markov chain, whose unique stationary distribution is the uniform one on $\bA_G$. 
An important remark is that, once $C,\sigma$ chosen, the action $A\to A'$ can be formulated as the action of a function $\Phi_{C,\sigma}:\bA_G\to\bA_G$ (one function for each element of $L(G)\times\{-,+\}$) that is monotone with respect to the lattice order (the monotonicity follows from the above mentioned encoding by potential-vectors). 
Therefore, the coupling from the past method from~\cite[Sec.2.2]{propp1996exact} can be applied to generate an element of $\bA_G$ uniformly at random.
The precomputation of $L(G)$ can be done in polynomial time by the above results, and the computation of the minimal and maximal elements of $\bA_G$ (needed for the coupling from the past) can also be done in polynomial time (using Theorem~\ref{thm:main} and results from~\cite[Section 3]{khuller1993lattice}).
However, as shown in~\cite{miracle2016sampling} in the case of Schnyder woods, it may happen that the coupling time is exponential for certain maps~$G$.
\end{remark}


\subsection{Lattice for bipartite grand-Schnyder structures}
The set of $b$-BGS structures of a fixed bipartite $2b$-adapted map $G$ can also be given a lattice structure, and the results are very similar. 
First note that $b$-BGS angular orientations correspond to some $\ab$-orientations of $G^+$ with frozen edges defined as follows. Let $s_i$ (for $i\in [2b]$) be the star-vertex for the inner face of $G$ incident to $(v_{i-1},v_i)$, and let $\delta_i$ be the half-degree of $s_i$ (that is, $2\delta_i$ is the degree of $s_i$). Let $\alpha:V^+\to \NN$ and $\beta:V^+\to \NN$ be defined as follows:
\begin{itemize}
\item
for all $i\in [2b]$, $\alpha(v_i)=b-\delta_i$,
\item
for $v$ an inner original vertex, $\alpha(v)=b$,
\item
for $v$ a star vertex of degree $2k$, $\alpha(v)=b-k$, 
\item
for $e$ an outer edge, $\beta(e)=0$,
\item
for $e$ a star edge incident to a star vertex of degree $2k$, $\beta(e)=b-k$, 
\item
for $e$ an original inner edge of $G$, with $k,k'$ the half-degrees of the inner faces of $G$ incident to $e$, $\beta(e)=k+k'-b-1$. 
\end{itemize}
Moreover, we declare the \emph{frozen edges} to be all the edges incident to outer vertices, where for all $i\in[d]$ the arc from $v_i$ to $s_i$ has weight $b-\delta_i$ (and the opposite arc has weight 0) and all the other arcs out of $v_i$ have weight 0 (and the opposite arcs have the weight required by $\beta$).
It follows from Definition~\ref{def:BGS-angular} and Remark~\ref{rk:weight-frozen-bip} that the $b$-BGS angular orientations of $G^+$ are exactly these $\ab$-orientations with frozen edges. 
 
 \begin{thm}\label{thm:flip_bip}
For every bipartite $2b$-adapted map $G$, the set of $b$-BGS angular orientations on $G$ has the structure of a distributive lattice. The upward (resp. downward) covering relations consist in pushing a counterclockwise (resp. clockwise) simple cycle not incident to the outer vertices which is either
 \begin{itemize}
\item[(i)] 
the contour of an inner face of $G^+$, or 
 \item[(ii)]
a compatible cycle of enclosing length $2b$.
 \end{itemize}
\end{thm}

The proof of Theorem~\ref{thm:flip_bip} will be given in Section~\ref{sec:proof-lattice}.

\begin{remark}\label{rk:flip_corner_bip}
As in Remark~\ref{rk:flip_corner}, the other incarnations of $b$-BGS structures of $G$ inherit the lattice structure via the bijections. 
In the corner labeling incarnation, a flip consists in either (i) decreasing by $2$ modulo $d=2b$ the label of a single corner, or (ii) decreasing by $2$ modulo~$d$ all the labels inside a simple $d$-cycle $C$ of $G$ except for the labels of the special corners which are left unchanged (and such an operation is allowed only if the resulting corner labeling is a $b$-BGS labeling). 
When $G$ is edge-tight (which by Lemma~\ref{lem:degree-edge-tight} can only occur if $G$ is a quadrangulation of the hexagon or a map obtained from $2b$-angulation by opening some edges into 2-gons), only flips of type (ii) are ever valid. We call \emph{$b$-BGS arc labeling} a $2b$-GS arc labeling of $G$ whose arcs starting from black (resp. white) vertices have odd (resp. even) label. 
In the incarnation as $b$-BGS arc labelings, a flip consists in decreasing by $2$ modulo $2b$ the arc labels inside a $2b$-cycle $C$ of $G$ (and such an operation is allowed only if the resulting arc labeling is a valid arc labeling). 
\end{remark}


\subsection{Recovering known lattices of orientations}
As we now explain, Theorem~\ref{thm:flip} provides a common generalization of the lattice structures already already known for Schnyder decompositions of $d$-angulations~\cite{OB-EF:Schnyder} and for transversal structures~\cite{Fu07b}. 
Similarly, in the bipartite setting, Theorem~\ref{thm:flip_bip} provides a common generalization of the lattice structure for even Schnyder decompositions of bipartite $2b$-angulations \cite{OB-EF:Schnyder} and the lattice structure for Felsner woods~\cite{Felsner:lattice}. 

\medskip

\textbf{Lattice structure for Schnyder decompositions of $d$-angulations.}
Recall from Section~\ref{sec:Schnyder-decompositions} that $d$-GS angular orientations of a $d$-angulation $G$ coincide with the Schnyder decompositions $G$ as defined in \cite{OB-EF:Schnyder}. Via this correspondence, the lattice structure defined in Theorem~\ref{thm:flip} coincides with the lattice structure defined in~\cite{OB-EF:Schnyder} on Schnyder decompositions of $G$. In the particular case $d=3$, we recover the classical lattice structure defined on the set of Schnyder woods of a triangulation. Let us now fix a $d$-angulation $G$ and consider the characterization of flips given by Theorem~\ref{thm:flip}. The compatible cycles of $G$ are the simple cycles of the original map $G$. Hence, by Theorem~\ref{thm:flip}, the essential cycles are simple cycles of $G$ of length $d$ (essential cycle of type (i) cannot occur for essential cycles, since star edges have weight $0$). 
In terms of $d$-GS labelings, a flip (resp. flop) operation consists in decreasing (resp. increasing) by $1$ modulo $d$ the labels of corners inside a $d$-cycle (note that there is no special corner, of unchanged label, inside the $d$-cycle, since all faces have degree $d$). We thus recover the lattice structure properties of~\cite[Sec.3.3]{OB-EF:Schnyder}.

\medskip

\textbf{Lattice structure for even Schnyder decompositions of $2b$-angulations.}
As before, the lattice structure defined in \cite{OB-EF:Schnyder} even Schnyder decompositions of a bipartite $2b$-angulation $G$ coincides with the lattice structure on $b$-BGS structures given by Theorem~\ref{thm:flip_bip} (via the identification given in Section~\ref{sec:bipartite_angulations}). By Theorem~\ref{thm:flip_bip} the essential cycles are $2b$-cycles of the original map $G$, and in the corner labeling incarnation, a downward (resp. upward) covering relation consists in decreasing (resp. increasing) by $2$ modulo $2b$ the labels of corners inside a given $2b$-cycle. This recovers the flip descriptions from \cite{OB-EF:Schnyder}.

 \begin{figure}
\begin{center}
\includegraphics[width=\linewidth]{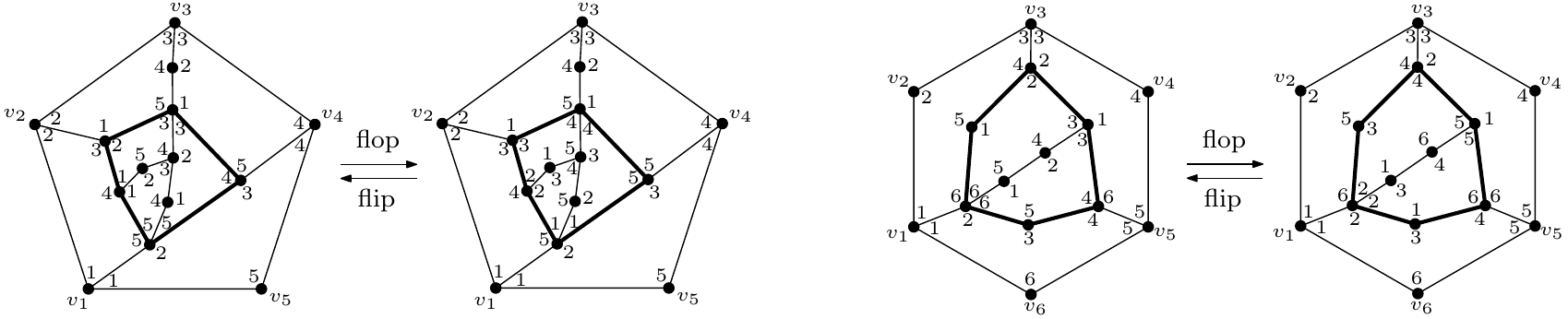}
\end{center} 
\caption{Left: a flip/flop operation on a $5$-GS corner labeling of a pentagulation.
Right: a flip/flop operation on a $3$-BGS) corner labeling of a hexangulation. 
}
\label{fig:flip_flop_angulation} 
\end{figure}

\medskip

\textbf{Lattice structure for transversal structures.}
Let us now examine the lattice obtained when $G$ is a triangulation of the square.
Recall from Section~\ref{sec:transversal} that the 4-GS structures of $G$ are in bijection with the transversal structures of $G$.
Via this correspondence, the lattice structure defined in Theorem~\ref{thm:flip} coincides with the lattice structure defined in~\cite{Fu07b} on the transversal structure of $G$.
Let us now consider the characterization of flips given by Theorem~\ref{thm:flip}
In the 4-GS angular orientation of $G$, the original edges have weight~$0$, hence only the essential cycles of type (ii) can be positive.
Hence the flip/flop operations correspond to pushing a directed cycle (made of star edges) of enclosing length $4$.
In the incarnation as transversal edge-partitions, the flip/flop operations also have a nice formulation, as illustrated in the left part of Figure~\ref{fig:flip_flop_transversal}. Given a transversal edge-partition of $G$, a 4-cycle $C$ is called \emph{alternating} if the edges around $C$ are red/blue/red/blue (such a cycle is the enclosing cycle of an essential cycle in the associated GS angular orientation). 
 For such a cycle $C$ and $v\in C$, the left (resp. right) edge of $v$ is the edge on $C$ just after (resp. just before) $v$ in clockwise order around $C$. 
It is shown in~\cite{Fu07b} that only two situations can occur for the color of the edges inside $C$ and incident to $C$:
 \begin{itemize}
 \item
either every edge $e$ inside $C$ and incident to a vertex on $C$ has the color of the right edge of $v$, in which case $C$ is called a right alternating 4-cycle (it is the enclosing cycle of an essential clockwise cycle in the 4-GS angular orientation),
 \item
 or every edge $e$ inside $C$ and incident to a vertex on $C$ has the color of the left edge of $v$, in which case $C$ is called a left alternating 4-cycle (it is the enclosing cycle of an essential counterclockwise cycle in the 4-GS angular orientation).
 \end{itemize} 
Then, a flip (resp. flop) operation consists in switching the colors of all the edges inside a right (resp. left) 
alternating 4-cycle, turning it into a left (resp. right) one~\cite[Theo.2]{Fu07b}. 
This is consistent with the formulation of the flip/flop operation on $4$-GS arc labelings given in Remark~\ref{rk:flip_corner}, and the correspondence between arc labelings and transversal structures illustrated in Figure~\ref{fig:corresp_GS_transversal}.

 \begin{figure}
\begin{center}
\includegraphics[width=\linewidth]{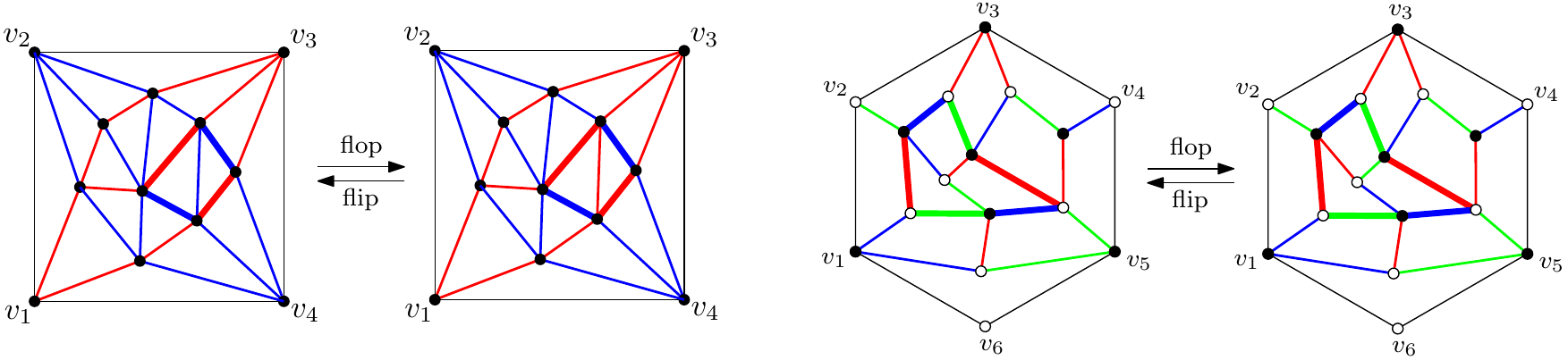}
\end{center} 
\caption{Left: a flip/flop operation on a transversal structure, corresponding to an alternating 4-cycle indicated in bold lines. Right: a flip/flop operation on a Felsner edge-coloring, corresponding to an alternating 6-cycle indicated in bold lines.}
\label{fig:flip_flop_transversal} 
\end{figure}

\medskip

\textbf{Lattice structure for Felsner woods.}
Let us lastly discuss the case where $G$ is a quadrangulation of the hexagon. 
In the 3-BGS angular orientations of $G$, the original edges have weight $0$, so that the positive cycles are made of star edges, and the essential ones are those of enclosing length $6$. 
We recover~\cite[Lem.17]{Felsner:lattice} (precisely, the identity $2k=a+6$ obtained in its proof). In the incarnation of Felsner edge-colorings, we have a formulation of the covering relations as for transversal edge-partitions, which is similar to the one given for transversal structures. This covering relation is indicated in the right side of Figure~\ref{fig:flip_flop_transversal}. In a Felsner edge-coloring, we call a 6-cycle $C$ \emph{alternating} if the colors of edges in a clockwise traversal of $C$ are red/blue/green/red/blue/green 
(these are the enclosing cycles of essential cycles in the associated $3$-BGS angular orientation). 
 For such a cycle $C$ and a vertex $v$ on $C$, the left (resp. right) edge of $v$ is the edge on $C$ just after (resp. just before) $v$ in a clockwise traversal of $C$. Then 
 two situations can occur: 
 \begin{itemize}
 \item
 every edge $e$ inside $C$ and incident to a vertex on $C$ has the color of the right edge of~$v$, in which case $C$ is called a right alternating 6-cycle (it is the enclosing cycle of an essential clockwise cycle in the 3-BGS angular orientation),
 \item
 or every edge $e$ inside $C$ and incident to a vertex on $C$ has the color of the left edge of~$v$, in which case $C$ is called a left alternating 6-cycle (it is the enclosing cycle of an essential counterclockwise cycle in the 3-BGS angular orientation).
 \end{itemize} 
A proof that only these two cases occur can be given along the same lines as in~\cite{Fu07b}: any 6-cycle $C$ is the enclosing cycle of an essential compatible cycle, and these have inward degree $0$ (see Lemma~\ref{lem:compa}). Hence (through the bijection $\Gamma$ between 3-BGS arc labelings and angular orientations of $G^+$),  at each vertex $v\in C$ the arcs inside $C$ with initial vertex $v$ all have the same color. Moreover, the fact that $C$ is left alternating (resp. right alternating) is equivalent (under $\Gamma$) to the fact that the corresponding compatible cycle is directed, in counterclockwise (resp. clockwise) direction.

Then, a flip operation turns a right alternating 6-cycle into a left one, by applying $\{$red$\to$blue, blue$\to$green, green$\to$red$\}$ to the colors of the edges inside $C$; and conversely a flop operation turns a left alternating 6-cycle into a right one, by applying $\{$red$\to$green, blue$\to$red, green$\to$blue$\}$ to the colors of the edges inside $C$. Again, this is consistent with the formulation of the flip/flop operation on arc labelings given in Remark~\ref{rk:flip_corner_bip}, and the correspondence from these labelings to Felsner edge-colorings illustrated in Figure~\ref{fig:corresp_GS_transversal}.


\subsection{Proof of Theorems~\ref{thm:flip} and~\ref{thm:flip_bip}}\label{sec:proof-lattice}
This subsection is devoted to the proof of Theorems~\ref{thm:flip} and~\ref{thm:flip_bip}. 
Roughly speaking, we will show that essential cycles for $d$-GS orientations need to be compatible and have ``inward weight'' 0. 
 
Let $C$ be a not necessarily simple cycle of $G$, whose interior is simply connected. 
The \emph{inward weight} of the cycle $C$, denoted by $\inweight(C)$, is the sum of the weights of the arcs strictly inside $C$ with initial vertex on $C$.

\begin{lemma} \label{lem:ingoing-weight-angular} 
Let $G$ be a $d$-map. Let $\cA$ be a weighted orientation of the angular map $G^+$ satisfying Conditions (A1) and (A2) of $d$-GS angular orientations. 
Let $C$ be a not necessarily simple cycle of $G$, which is the contour of a simply connected finite region. Then the inward weight of $C$ is
$$\inweight(C)=\ell(C)-d+\sum_{a\in C}(d-\deg(f_a)),$$
where $\ell(C)$ is the length of $C$, the sum is over the arcs of $C$, and for an arc $a$ of $C$ the face incident to $a$ inside $C$ is denoted by $f_a$.
\end{lemma}

\begin{proof} 
We start by computing the total weight of the edges of $G^+$ inside $G$.
Let $V,E,F$ be the set of vertices, edges and faces of $G$ strictly inside $C$. Each star edge $\epsilon=(v_f,v)$ can be associated to the original edge $e$ that follows $\epsilon$ in counterclockwise order around~$v$. Conditions (A1) and (A2) of $d$-GS angular orientations ensure that for each original edge $e$ inside~$C$, the total weight of $e$ together with the two associated star edges is equal to $d-2$. Hence,  the total weight of the edges of $G^+$ inside $C$ is $\ds (d-2)|E|+\sum_{a\in C}(d-\deg(f_a))$, 
where the second term gathers the contributions of star edges associated to edges on $C$.
The total weight of the vertices of $G^+$ inside $C$ is $d|V|+\sum_{f\in F}(d-\deg(f))$.
Hence, 
$$\inweight(C)=(d-2)|E|+\sum_{a\in C}(d-\deg(f_a))-d|V|-\sum_{f\in F}(d-\deg(f))$$
The Euler relation gives $|V|+|F|=|E|+1$, while the incidence relation between faces and edges gives $\sum_{f\in F}\deg(f)=2|E|+\ell(C)$. Combining these relations proves the lemma.
\end{proof}



Next we prove an accessibility property for $d$-GS angular orientations.
\begin{lemma}\label{lem:alpha-ori-acc}
Let $G$ be a $d$-adapted map endowed with a $d$-GS angular orientation. Then, for any vertex $v$ of $G^+$ which is not a star vertex of degree $d$, there exists a positive path in $G^+$ starting at $v$ and ending at an outer vertex. 
\end{lemma} 
\begin{proof}
It suffices to prove the property for the original vertices (because it then easily follow for the star vertices).
Let $U$ be the set of vertices $v$ of $G$ for which there exists a positive path in $G^+$ from $v$ to an outer vertex.
Let $G_U$ be the submap of $G$ made of $U$ and the edges of $G$ with both endpoints in $U$ (this is a connected submap containing the outer vertices and the outer edges). 
Suppose for contradiction that there is a vertex $v$ of $G$ which is not in $U$. Then consider the contour $C$ of the face of $G_U$ containing $v$. By definition of $G_U$, there cannot be any edge of $G$ strictly inside $C$ with both endpoints on $C$. Consider the total weight $\om(C)$ (resp. $\overline \om(C)$) of the arcs of $G^+$ strictly inside $C$ and having their initial (resp. terminal) vertex on $C$. We will prove $\overline \om(C)>0$ and reach a contradiction.
By Lemma~\ref{lem:ingoing-weight-angular},
$$\om(C)=\textrm{length}(C)-d+\sum_{a\in C}(d-\deg(f_a)),$$
where the sum is over the arcs of $C$ and $f_a$ is the face of $G$ incident to $a$ inside $C$. Let $J$ be the set of arcs of $G^+$ strictly inside $C$ whose terminal vertex is on $C$, and let $I\subseteq J$ be the subset of those on original edges. For $a\in J$, denote by $e_a$ the edge of $G^+$ that contains $a$, and for $a\in I$ denote by $f_a'$ the face of $G$ incident to $a$ and on the left of $a$. We claim that
$$\om(C)+\overline\om(C)=\sum_{a\in J}\beta(e_a)=\sum_{a\in C}(d-\deg(f_a))+\sum_{a\in I}(\deg(f_a')-2).$$
To see that this identity holds, we pair each arc $a\in I$ with the star arc $a'\in J$ preceding $a$ in clockwise order around the terminal vertex of $a$, and we observe $\beta(e_a)+\beta(e_{a'})=(d-2)-(d-\deg(f_a'))=\deg(f_a')-2$. The unpaired arcs in $J$ are the star arcs $a'$ followed (in clockwise order around the terminal vertex of $a'$) by an arc $a\in C$, and satisfy $\beta(e_{a'})=d-\deg(f_a)$; so the sum of $\beta(e_{a'})$ over these unpaired arcs is $\sum_{a\in C}(d-\deg(f_a))$. This proves the claimed identity, and from it we obtain
$$\overline\om(C)= \sum_{a\in I}(\mathrm{deg}(f_a')-2)-\textrm{length}(C)+d.$$
Moreover, $\textrm{length}(C)\leq \sum_{a\in I}\ell(f_a')$, where $\ell(f_a')$ is the number of arcs on $C$ incident to $f_a'$ (this is an equality unless $f_a'=f_b'$ for distinct arcs $a,b\in I$). Since there is no edge inside $C$ with both endpoints on $C$, we have 
$\ell(f_a')\leq \mathrm{deg}(f_a')-2$ for all $a\in I$, hence $\overline\om(C)\geq d>0$. 
So there exists an arc $a=(s,u)$ of $G^+$ with positive weight, such that $s$ is strictly inside $C$ (hence not in $U$) and $u$ is on $C$ (hence in $U$). The vertex $s$ cannot be a vertex of $G$ (otherwise it would be in $U$), hence it is a star vertex. However this implies that for every vertex $w\neq u$ adjacent to $s$, the arc $(w,s)$ has positive weight, so that $w\in U$. This implies that $C$ is the contour of the face of $G$ corresponding to the star vertex $s$, which contradicts our assumption that the interior of $C$ contains a vertex $v$ of $G$ which is not in $U$.
\end{proof}
\begin{remark}
We say that a weighted orientation of a plane map is \emph{accessible} if for every vertex~$v$, there exists a positive path from $v$ to an outer vertex. 
Lemma~\ref{lem:alpha-ori-acc} shows that any $d$-GS angular orientation is accessible, upon deleting the star vertices of degree $d$ and their incident edges. This generalizes the accessibility properties known for $d/(d-2)$-orientations~\cite[Theo.13]{Bernardi-Fusy:dangulations} and for the $\alpha$-orientations associated to transversal structures~\cite[Lem.3]{Fu07b}.
\end{remark}


We define a \emph{pseudo-compatible cycle} as a not necessarily simple cycle $C$ of $G$ whose interior is simply connected, that does not visit star-vertices of degree~$d$, and such that for every original edge $e$ belonging to $C$, the inner face of $G^+$ incident to $e$ and lying inside $C$ has its incident star vertex of degree $d$. Clearly, pseudo-compatible cycles are a generalization of compatible cycles (these are the \emph{simple} pseudo-compatible cycles), and we will now extend the definition of enclosing length to pseudo-compatible cycles; see Figure~\ref{fig:pseudo-compatible}(a).
Let $C$ be a pseudo-compatible cycle, let $\textrm{Orig}(C)$ be the set of arcs of $C$ belonging to original edges of $G$, and let $\textrm{Star}(C)$ be the set of arcs in $C$ whose terminal vertex is a star vertex. For $a\in \textrm{Star}(C)$, we consider the terminal vertex $s$ of $a$ (which is a star vertex) and the next arc $a'$ along $C$, and we define $\ell(a):=\deg(s)-k$, where $k$ is the number of corners incident to $s$ between $a$ and $a'$ on the side of the region included by $C$ (if $C$ is a clockwise contour, then $\ell(a)$ is the number of corners from $a'$ to $a$ in clockwise order around $s$).
We define the enclosing length of $C$ as 
$$\pseudol(C)=|\textrm{Orig}(C)|+\sum_{a \in \textrm{Star}(C)}\ell(a).$$
Note that for a compatible cycle $C$, $\pseudol(C)$ coincides with the length of the enclosing cycle of $C$, and Figure~\ref{fig:pseudo-compatible}(a) provides a generalization of this interpretation for pseudo-compatible cycles.

\fig{width=\linewidth}{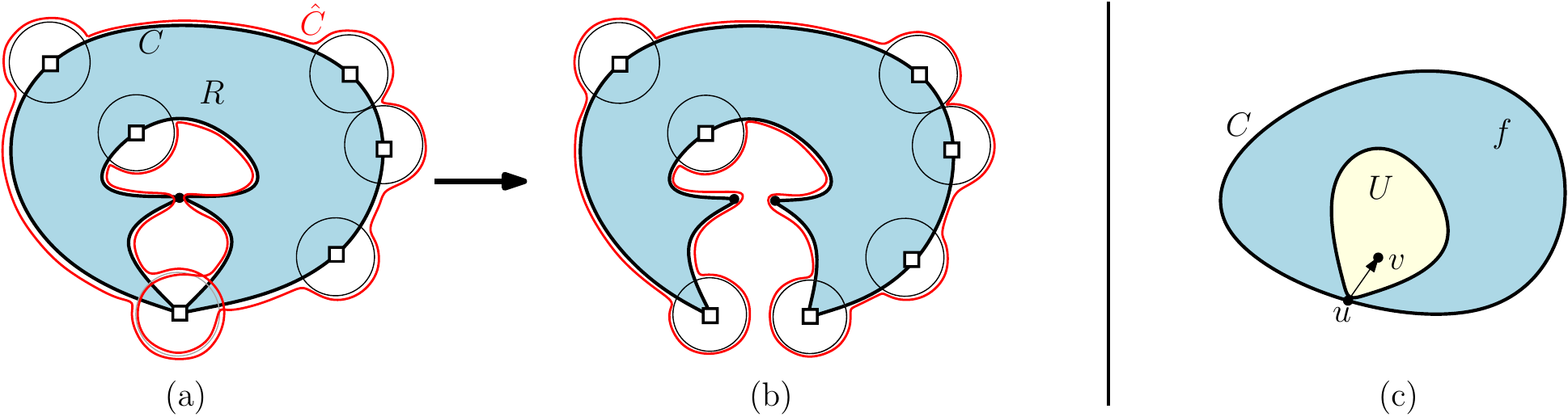}{(a) A pseudo-compatible $C$ in a map $G^+$ enclosing a simply connected region $R$. The star vertices on $C$ are indicated (by squares) together with the contour of the face of $G$ corresponding to these star vertices (represented as circles). For an arc $a$ in $\textrm{Star}(C)$, the quantity $\ell(a)$ represents the length of a portion of the contour of the face of $G$ corresponding to the star vertex. Putting together these partial face contours and the arcs in $\textrm{Orig}(C)$ gives a cycle $\hC$ of $G$ of length $\pseudol(C)$. (b) Duplicating part of the oriented map $G^+$, one can get a compatible cycle $C'$ and its enclosing cycle $\hC'$ (in an angular map which is not $G^+$) such that the interior of $C'$ is the same as the interior of $C$, and the interior of $\hC'$ satisfies the conditions of $d$-GS orientations (in terms of weights of vertices and edges). (c) Notation for the proof of Theorem~\ref{thm:flip}.}
 
 \begin{lemma}\label{lem:compa}
 Let $G$ be a $d$-map endowed with a $d$-GS angular orientation. Let $C$ be a pseudo-compatible cycle of $G^+$, and let $\ell$ be its enclosing length. Then the inward weight of $C$ is equal to $\ell-d$. 
\end{lemma}
\begin{proof}
Let us first prove the formula when $C$ is simple, that is, a compatible cycle.
Let $C$ be a compatible cycle and let $\hC$ be the enclosing cycle.
Let $V,E$ (resp. $\hat V,\hat E$) be the sets of vertices and edges of $G^+$ that are strictly inside the cycle $C$ (resp. $\hat C$).
The inward weights of $C$ and $\hat C$ are given by 
$$\inweight(C)=\sum_{e\in E}\be(e)-\sum_{v\in V}\al(v)~\textrm{ and }~\inweight(\hC)=\sum_{e\in \hat E}\be(e)-\sum_{v\in \hat V}\al(v).$$
Let $F$ be the set of faces of $G$, of degree smaller than $d$, which are inside $\hC$ and have at least one edge on $\hC$ (i.e., are inside $\hC$ but not inside $C$). For $f\in F$, let $\ell(f)$ be the number of arcs on $\hC$ that are on the contour of $f$. It is easy to see that $\sum_{e\in \hat E\backslash E}\be(e)=\sum_{f\in F}(\ell(f)+1)(d-\deg(f))$, and $\sum_{v\in \hat V\backslash V}\al(v)=\sum_{f\in F}(d-\deg(f))$.
Hence, $\inweight(\hC)-\inweight(C)=\sum_{f\in F}\ell(f)(d-\deg(f))$, and 
$$\inweight(C)=\inweight(\hC)-\sum_{a\in \hC}(d-\deg(f_a)),$$
where the sum is over the arcs of $\hC$ and $f_a$ is the face of $G$ incident to $a$ inside $\hC$. Using the expression for $\inweight(\hC)$ provided by Lemma~\ref{lem:ingoing-weight-angular} gives $\inweight(C)=\ell-d$ as claimed.

Let us now prove the formula for a pseudo-compatible cycle $C$ which is not simple. The idea of the proof is represented in Figure~\ref{fig:pseudo-compatible}(b). Upon duplicating part of the angular map $G^+$, it is possible to obtain a simple cycle $\hC'$, the interior of which is the angular map of a map $M'$ with outer face $\hC'$ endowed with a weighted orientation with the same edges and vertex weight as a $d$-GS orientation, and such that $\hC'$ is the enclosing cycle of a compatible cycle $C'$, such that 
$\pseudol(C)=\pseudol(C')$ (which is the length of $\hC'$) and $\inweight(C)=\inweight(C')$ (because the weighted orientations in the inner regions of $C$ and $C'$ are identical). Hence applying the above reasoning on the compatible cycle $C'$ gives the result for the pseudo-compatible cycle $C$.
\end{proof}


We can now complete the proof of Theorem~\ref{thm:flip}.
\begin{proof}[Proof of Theorem~\ref{thm:flip}]
Let $G$ be a $d$-map endowed with a $d$-GS angular orientation. A flip/flop operation in the lattice of $d$-GS angular orientations of $G$ correspond to pushing a positive simple cycle which is not incident to the outer vertices, and is essential. Let $C$ be a positive simple cycle not incident to the outer vertices, which is essential, and is not the contour of an inner face of $G^+$. We want to show that $C$ is a compatible cycle of enclosing length $d$.

Suppose for contradiction that $C$ is not compatible. Since $C$ is not compatible, it has at least one original edge $e=\{u,v\}$ whose incident face $f$ inside $C$ has degree less than $d$. Consider the star edges $\eps_1=\{v_f,u\}$ and $\eps_2=\{v_f,v\}$ joining $u,v$ to the star vertex $v_f$. The weights of $\eps_1,\eps_2$ in angular orientations are $\om(\eps_1)=\om(\eps_2)=\om(v_f)=d-\deg(f)>0$. Hence either (a) both arcs $a_1=(u,v_f)$ and $a_2=(v,v_f)$ have positive weights, or (b) $a_1$,$-a_2$ form a positive path from $u$ to $v$, or (c) $a_2,-a_1$ form a positive path from $v$ to $u$. Since $C$ is essential (and not reduced to $e,\eps_1,\eps_2$), cases (b) and (c) are excluded. Hence, we are in case (a). By Lemma~\ref{lem:alpha-ori-acc} there exists a positive path $P$ from $v_f$ to one of the vertices on $C$. Concatenating either $a_1$ or $a_2$ with $P$ gives a chordal path for $C$, which is a contradiction. 

Thus, $C$ is a compatible cycle, and it remains to prove that it has enclosing length $d$. Suppose for contradiction that the enclosing length is greater than $d$. Then the inward weight of $C$ is positive by Lemma~\ref{lem:compa}. Let $a=(u,v)$ be an arc inside $C$ with positive weight, having initial vertex $u$ on $C$. Let $U$ be the set of vertices on $C$ or inside $C$ that can be reached from $u$ by a positive path of arcs strictly inside $C$. The situation is represented in Figure~\ref{fig:pseudo-compatible}(c).
By assumption $C$ is essential, so $u$ is the only vertex from $U$ on $C$. Let $G_U$ be the submap of $G^+$ made of the vertices and edges on $C$, together with the vertices in $U$ and the edges with both ends in $U$. The submap $G_U$ is connected. Let $f$ be the inner face $f$ of $G_U$ incident to the edges of $C$, and let $C'$ be the clockwise contour of $f$.
We claim that $C'$ is a pseudo-compatible cycle of $G^+$. Indeed, otherwise there would
a face $f'\in G$ of degree smaller than $d$ within $f$ and sharing at least one edge $\{p,q\}$ with $C'\backslash C$. But then at least one of the arcs $(p,v_f)$ or $(q,v_f)$ would have positive weight, so $v_f$ would be in $U$ (since $p$ and $q$ are in $U$), a contradiction. Hence, $C'$ is a pseudo-compatible cycle of $G^+$.  
Moreover, $\inweight(C')<\inweight(C)$ (since any arc inside $f$ having initial vertex in $U$ has weight 0) and $\pseudol(C)<\pseudol(C')$. However these inequalities are incompatible with Lemma~\ref{lem:compa}, which gives a contradiction.

We have thus proved that any non-facial essential cycle of $G^+$ has to be a compatible cycle of enclosing length $d$. Conversely, by Lemma~\ref{lem:compa}, any compatible cycle of enclosing length $d$ 
has inward weight $0$, hence is essential.
\end{proof}

\begin{proof}[Proof of Theorem~\ref{thm:flip_bip}]
The proof of Theorem~\ref{thm:flip_bip} is almost identical to that of Theorem~\ref{thm:flip}. The inward weight results of Lemma~\ref{lem:compa} can be used to give an analogue for $b$-BGS angular orientations since such orientations are obtained by halving the weights of even $2b$-GS orientations. Then, the characterization of the essential cycles involved in flips/flops use the same arguments as for $d$-GS angular orientations.
\end{proof}

%% file: dual.tex

In this section, we study the structures obtained by taking the ``dual" of grand-Schnyder structures, and of their bipartite analogs. In the forthcoming article \cite{OB-EF-SL:4-GS-drawing}, we will use the dual of 4-grand-Schnyder woods to design some graph-drawing algorithms.

Recall that the \emph{dual} $G^*$ of a plane map $G$ is obtained as follows: 
\begin{compactitem}
\item we draw a vertex $v_f$ of $G^*$ inside every face $f$ of $G$ and call $v_f$ the \emph{dual vertex} of $f$, 
\item and we draw an edge $e^*$ of $G^*$ connecting $v_f$ and $v_g$ across each edge $e$ separating the faces $f$ and $g$ of $G$, and call $e^*$ the \emph{dual edge} of the \emph{primal edge} $e$.
\end{compactitem}
The dual vertex of the outer face of $G$ is called the \emph{root-vertex} of $G^*$ and is denoted by $v^*$. Under this construction, the degree of the vertex $v_f$ is equal to the degree of the face $f$. The \emph{dual of a corner} is also well-defined: it is the corner of $G^*$ that faces the original corner of $G$.

If $G$ is a $d$-map, then we call its dual $G^*$ a \emph{dual $d$-map}. More concretely, a \emph{dual $d$-map} is a vertex-rooted plane map whose vertices have degree at most~$d$ and whose root-vertex has degree~$d$ and is incident to $d$ distinct faces. In this case, the edges $e_1^*,\ldots,e_d^*$ in counterclockwise order around $v^*$ are called \emph{root-edges}, where $e_i^*$ is the dual edge of $\{v_i,v_{i+1}\}$. The faces $f_1^*,\ldots,f_d^*$ in counterclockwise order around $v^*$ are called \emph{root-faces}, where $f_i^*$ is dual to $v_i$.


We call a dual $d$-map $G^*$ \emph{dual $d$-adapted} if every \emph{cut} that does not separate a single vertex has size at least $d$. Note that this is exactly the dual notion of $d$-adaptedness, hence a dual $d$-map $G^*$ is dual $d$-adapted if and on if its primal map $G$ is $d$-adapted.

\subsection{Dual of grand-Schnyder structures}\label{sec:dual-GS} \hfill\\
We give two incarnations for the ``dual'' of $d$-GS structures, one in terms of corner labeling and the other as a tuple of trees. These two incarnations are represented in Figure~\ref{fig:4_dual}. They are closely related to the corner labeling and wood incarnations of GS structures of the primal map. 

\fig{width=\linewidth}{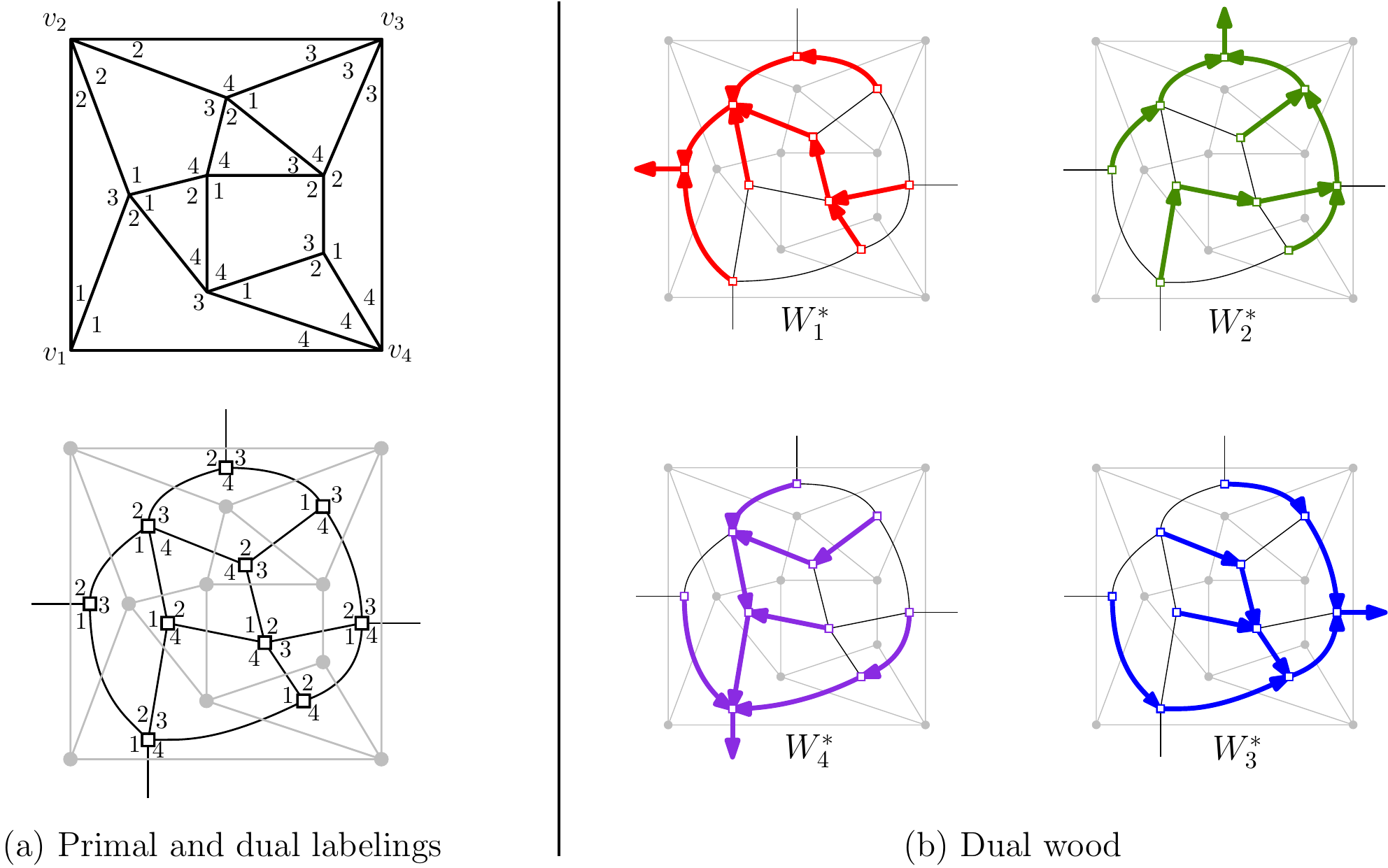}{Wood and labeling of a dual 4-GS structure.}

\begin{definition}\label{def:dual-GS-labeling}
Let $G$ be a $d$-map. A \emph{dual $d$-grand-Schnyder corner labeling}, or \emph{dual $d$-GS labeling}, of $G^*$ is an assignment to each corner of $G^*$ of a label in $[d]$ satisfying: \begin{itemize}
  \item[(L0\sups{*})] The corners incident to the root-face $f_i^*$ have label $i$.
  \item[(L1\sups{*})] Around every non-root face or non-root vertex, the sum of clockwise jumps is $d$.
  \item[(L2\sups{*})] Consecutive corners around a vertex have distinct labels.
  \item[(L3\sups{*})] Let $a$ be an arc not incident to the root-vertex, and let $v$ be its initial vertex. Then the sum of the clockwise jump across $a$ and the clockwise jump along $a$ is at least $1+d-\text{deg}(v)$.
\end{itemize} 
\end{definition}

\fig{width=\linewidth}{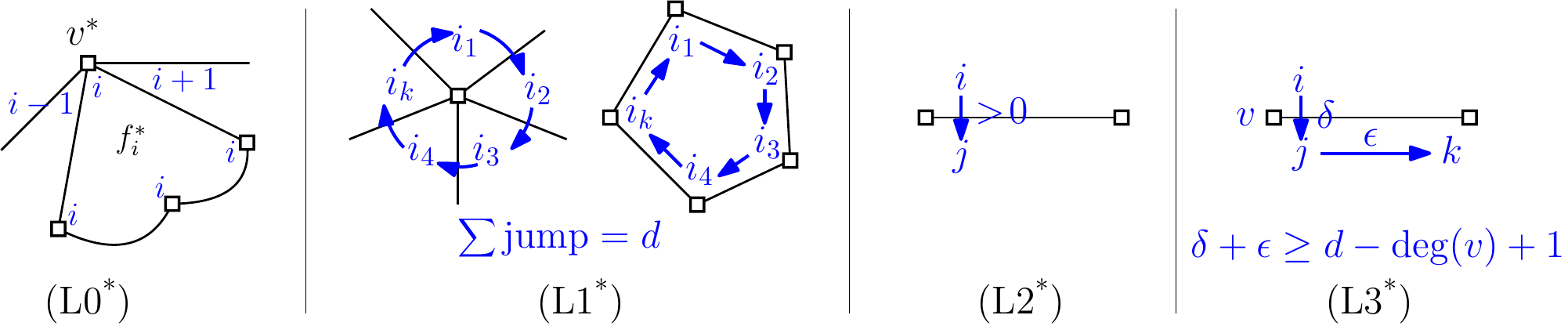}{Conditions defining dual $d$-GS labelings.}

The conditions defining dual $d$-GS labelings are represented in Figure~\ref{fig:dual-labeling}. As indicated in Figure~\ref{fig:4_dual}(a), there is a simple bijection between the $d$-GS labelings of $G$ and the dual $d$-GS labelings of $G^*$.
 Consider the following operation: given a corner labeling $\cL$ of $G$,  we first complement $\cL$ by labeling the outer corner at $v_i$ by $i$, and then give every corner of $G^*$ the label of its dual corner. We denote this operation by $\De$. It is easy to see that the Conditions (L0\sups{*}) $\sim$ (L3\sups{*}) are just translations of the Conditions (L0) $\sim$ (L3) of $d$-GS labelings via $\Delta$. Hence Theorem~\ref{thm:main} gives:

\begin{prop}\label{prop:existence-dual-labeling}
Let $G$ be a $d$-map. The mapping $\De$ is a bijection between the set of $d$-GS labelings of $G$ and the set of dual $d$-GS labelings of $G^*$. Hence $G^*$ admits a dual $d$-GS labeling if and only if $G^*$ is dual $d$-adapted, and in this case a dual $d$-GS labeling can be computed in linear time in the number of vertices of $G^*$.
\end{prop}

Now we present the incarnation of dual GS structures as a tuple of subsets of arcs $(W^*_1,...,W^*_d)$. We will see in Proposition~\ref{prop:primal_dual_woods} that a dual GS wood is also a tuple of spanning trees of the dual map, though this is not explicitly stated in the definition. As before, we interpret the tuple $(W^*_1,...,W^*_d)$ in terms of colorings, and say that an arc $a$ of $G^*$ \emph{has a color $i$} if this arc is in $W_i^*$. 

\begin{definition}\label{def:dual-GS-woods}
Let $G$ be a $d$-map. A \emph{dual $d$-grand-Schnyder wood}, or \emph{dual $d$-GS wood} of $G^*$ is a tuple $\mathcal{W^*} = (W^*_1,...,W^*_d)$ of subsets of arcs of $G^*$ satisfying: \begin{itemize}
  \item[(W0\sups{*})] For each $i \in [d]$, every non-root vertex $v$ has exactly one outgoing arc with color $i$. The arc of $e_i^*$ with initial vertex $v^*$ has no color, and its opposite arc has only color $i$.
  \item[(W1\sups{*})] Every arc whose initial vertex is not $v^*$ has at least one color. Let $v$ be a non-root vertex with outgoing arcs $a_1,...,a_d$ with colors $1,...,d$, respectively. The arcs $a_1,...,a_d$ appear in clockwise order around $v$ (with the situation $a_i = a_{i+1}$ allowed).
  \item[(W2\sups{*})] Let $a$ be an arc with initial vertex $v \neq v^*$ and let $\de$ be the number of colors of $a$. If $a$ has color $i$, then the colors of the opposite arc $-a$ form a subset of $[i+1+m:i[$, where $m = \max(0, 1+d-\text{deg}(v)-\de)$.
  
\end{itemize}
\end{definition}

A dual $d$-GS wood is represented in Figure~\ref{fig:4_dual}(b). 
\begin{remark}
Condition (W2\sups{*}) of dual $d$-GS woods can be interpreted as a ``crossing condition'' for the trees $(W^*_1,...,W^*_d)$. Precisely, let $a$ be an arc in $W_i^*$ with terminal vertex $v\neq v^*$, and let $a_1,...,a_d$ be the arcs in $W_1^*,...,W_d^*$ with initial vertex $v$. Then (W2\sups{*}) states that $a$ appears strictly between $a_{i+m}$ and $a_i$, in clockwise order around $v$, where $m = \max(0, 1+d-\text{deg}(v)-\de)$ and $\de$ is the number of colors of $a$ (when $m=0$ this simply means $a\neq -a_i$).
\end{remark}

In the definition of dual $d$-GS woods the arcs with initial vertex $v^*$ are quite special. We will call \emph{outside arcs} the arcs with initial vertex $v^*$, and \emph{inside arcs} all the other arcs.

We now define a mapping $\Theta^*$ between the dual $d$-GS labelings and dual $d$-GS woods that is very similar to the mapping $\Theta$ defined in Section \ref{sec:statements}: given a dual $d$-GS labeling $\cL^*$, let $\Theta^*(\cL^*) = (W^*_1,\ldots,W^*_d)$ be the tuple such that each arc whose initial vertex is not $v^*$ has the color set $[i:j[$, where $i$ and $j$ are the labels of the corners that are respectively on its left and on its right, at its initial vertex. We do not assign colors to outside arcs.

\begin{prop}\label{prop:bij-dual-labeling-wood}
The mapping $\Theta^*$ is a bijection between the dual $d$-GS labelings of $G^*$ and the dual $d$-GS woods of $G^*$.
\end{prop}

\begin{proof}
Let $\cL^*$ be a dual $d$-GS labeling on $G^*$. First we show that $\Theta^*(\cL^*)$ is indeed a dual $d$-GS wood. The properties (W0\sups{*}) and (W1\sups{*}) are clear from (L0\sups{*}) $\sim$ (L2\sups{*}) plus the convention that $\Theta^*$ uses about outside arcs. 

It remains to verify (W2\sups{*}). Observe that if $a$ incident to the root-verted, then (W2\sups{*}) holds by construction, so we now assume that $a$ is not incident to the root-vertex.  Observe that the sum of counterclockwise jumps around $e$ is equal to $d$. 
This property follows from Lemma~\ref{lem:ccw-jumps-edges} via duality. 

Let $i_1,i_2,i_3,i_4$ be the labels in counterclockwise order around $a$ as indicated in Figure \ref{fig:labels-around-dual-edge}. The map $\Theta^*$ assigns to $a$ the colors $[i_2:i_3[$ and to $-a$ the colors $[i_4:i_1[$. The above observation implies that the intervals $[i_2:i_3[$ and $[i_4:i_1[$ are disjoint. Hence if $a$ has color $i$, then $[i_4:i_1[\subseteq[i+1:i[$. Moreover, if $a$ has $\delta$ colors, then (L3\sups{*}) implies that the label jump from $i_3$ to $i_4$ is at least $m = \max(0, 1+d-\text{deg}(v)-\de)$, so $[i_4:i_1[\subseteq[i+1+m:i[$. 
This completes the proof that $\Theta^*(\cL^*)$ has property (W2\sups{*}), hence is a dual GS wood.

\fig{width=0.3\linewidth}{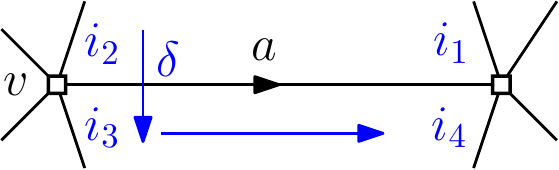}{The labels around a non-root edge.}

Now we show that $\Theta^*$ is a bijection. Injectivity is clear as it is easy to recover corner labels after the assignment of colors: the corners incident to the root-face $f_i^*$ all get label $i$ and the corners on the left and right side of an arc $a$ with color set $[p:q[$ will have labels $p$ and $q$, respectively. To prove surjectivity, we consider a tuple $\mathcal{W^*} = (W^*_1,\ldots,W^*_d)$ satisfying Conditions (W0\sups{*}-W2\sups{*}). We label the corners of $G^*$ according to the rule just mentioned, and want to show that the result $\cL^*$ is a dual $d$-GS labeling.
Note that Property (L0\sups{*}) holds by construction, and that (L2\sups{*}) follows directly from the first statement of (W1\sups{*}). Given the discussion in the previous paragraph, it is also easy to see that (L3\sups{*}) is a consequence of (W2\sups{*}). 

It remains to show that $\cL^*$ satisfies (L1\sups{*}). Let $\mathbf{v}, \mathbf{e}, \mathbf{f}$ be the number of vertices, edges and faces of $G^*$, respectively. Around a non-root vertex, the sum of clockwise jumps is $d$ by the second part of (W1\sups{*}), while that of the root-vertex $v^*$ is $d(d-1)$. The sum of counterclockwise jumps around each root-edge is $d$ by construction. Also, Condition (W0\sups{*}) implies that for a non-root edge the color sets of the two arcs are disjoint consecutive subsets of $[d]$, hence the sum of counterclockwise jumps around each non-root edge is also $d$. Therefore the sum of counterclockwise jumps around all edges is $d\mathbf{e}$, which is equal to the sum of clockwise jumps around all faces and all vertices. Consequently, the sum of clockwise jumps around all faces is equal to $d\mathbf{e} - d(\mathbf{v}-1)-d(d-1) = d(\mathbf{e}-\mathbf{v}+2-d) = d(\mathbf{f}-d)$, where the last equality is from Euler's relation. Note that $\mathbf{f}-d$ is exactly the total number of non-root faces. But since the sum of clockwise jumps around each root-face is 0 and that of each non-root face has to be a multiple of $d$, it is exactly $d$ for each non-root face. Hence $\cL^*$ satisfy (L1\sups{*}). This completes the proof that $\Theta^*$ is surjective, hence a bijection.
\end{proof}

The following theorem summarizes our results for dual $d$-GS structures. 

\begin{thm}\label{thm:dual-main}
Let $d \geq 3$ and let $G$ be a $d$-map. There exists a dual $d$-GS wood (resp. labeling) for $G^*$ if and only if $G^*$ is dual $d$-adapted. 

Moreover, for any fixed $d$, there is an algorithm which takes as input a dual $d$-adapted map and computes a dual $d$-GS wood (resp. labeling) in time linear in the number of vertices. 

Lastly, the set $\bL^*_G$ of dual $d$-GS labelings of $G^*$, the set $\bL_G$ of $d$-GS labelings of $G$, the set $\bW^*_G$ of dual $d$-GS woods of $G^*$, and the set $\bW_G$ of $d$-GS woods of $G$ are all in bijection.
\end{thm}

In fact, we can give a more direct description of the bijection between the dual $d$-GS woods of $G^*$ and the $d$-GS woods of $G$. This description shows that the dual $d$-GS woods is a tuple of spanning trees of $G^*$.

Recall the mapping $\Delta$ in Proposition~\ref{prop:existence-dual-labeling} between $d$-GS labelings of the primal map $G$ and dual $d$-GS labelings of the dual map $G^*$. By the above results, the map $\chi = \Theta^* \circ \Delta \circ \Theta^{-1}:\bW_G \to \bW^*_G$ is a bijection between the $d$-GS woods of $G$ and the dual $d$-GS woods of $G^*$. The local definition of these mappings around an inner edge $e$ is indicated in Figure~\ref{fig:primal_dual_woods}. 
Observe that, through $\chi$, the dual edge $e^*$ gets exactly the colors that the primal edge $e$ does not have. For outer edges/root-edges, the conventions were chosen so that the same property holds.

\fig{width=.7\linewidth}{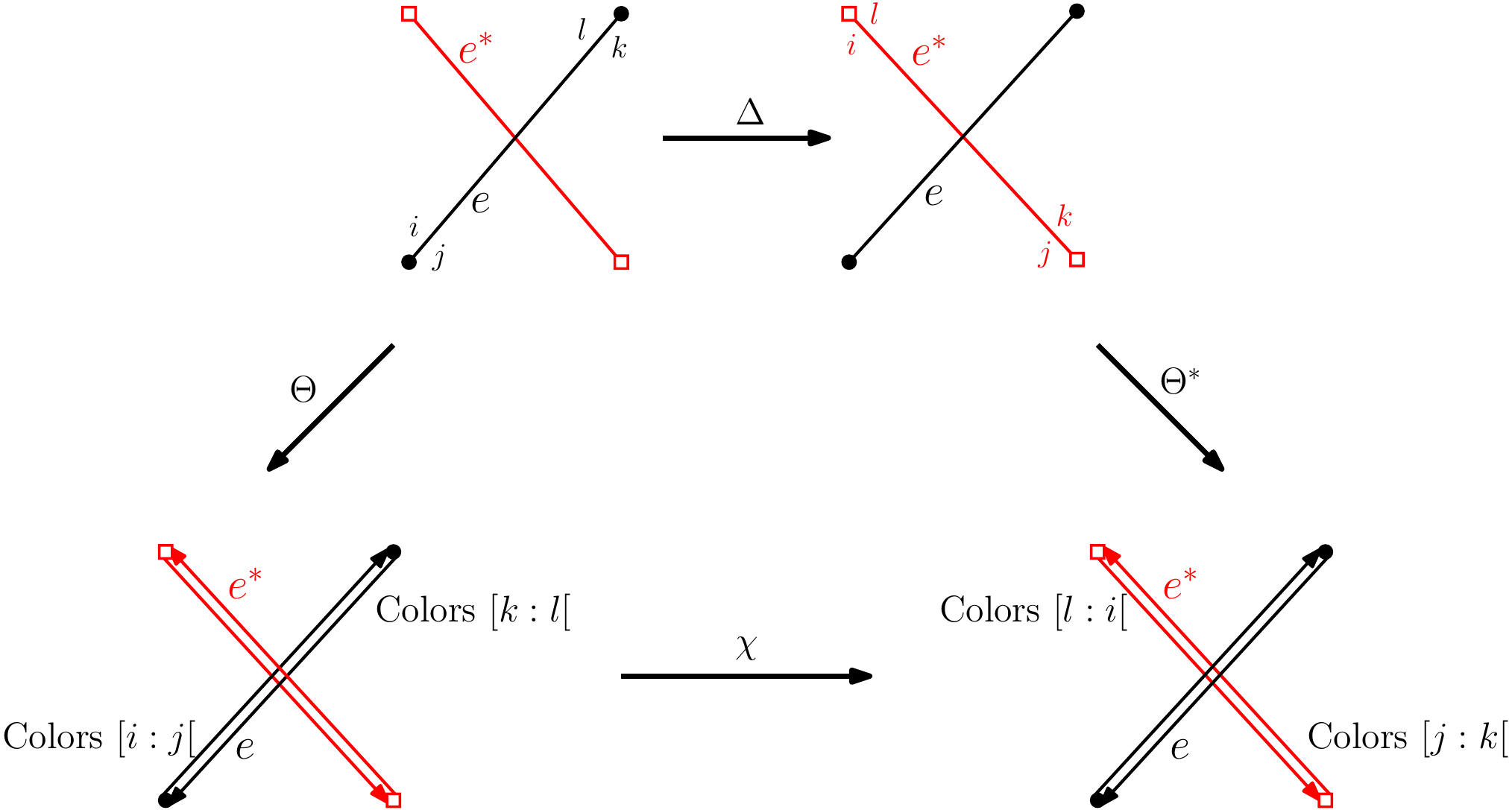}{The mappings $\Theta,\Theta^*$ and $\De$.}

For $i \in [d]$, let us call \emph{support} of $W^*_i$ the set of edges $\bar{W}^*_i$ of $G$ bearing the arcs in $W^*_i$. The \emph{support} $\bar{W}_i$ of the primal GS tree $W_i$ is defined similarly. By the above observation $\bar{W}^*_i$ is the \emph{dual-complement} of $\bar{W}_i$, that is, $\bar{W}^*_i$ is the set of edges which are dual to edges not in $\bar{W}_i$. 
It is well-known that the dual-complement of a spanning tree is also a spanning tree. Hence $\bar{W}^*_i$ is a spanning tree of $G^*$, for all $i\in [d]$. Moreover, by (W0\sups{*}), $v^*$ is the only vertex without an outgoing arc in $W^*_i$, so $W^*_i$ is a tree oriented toward the root-vertex $v^*$. To summarize: 

\begin{prop}\label{prop:primal_dual_woods}
Let $\cW = (W_1,...,W_d)$ be a GS wood of $G$. Then $\chi(\cW) = (W^*_1,...,W^*_d)$ is formed in the following way: for each $i \in [d]$, the support $\bar{W}^*_i$ of $W^*_i$ is the spanning tree of $G^*$ that is the dual-complement of the support of $W_i$: it contains edges of $G^*$ which are dual to the edges of $G$ not in $W_i$. $W^*_i$. Moreover, $W^*_i$ is obtained by orienting the tree $\bar{W}^*_i$ toward the root-vertex $v^*$ (every edge oriented from child to parent).
\end{prop}

Before closing this subsection let us observe that Lemma~\ref{lem:nb-color-arcs} for $d$-GS woods translates into the following property:

\begin{cor}\label{cor:nb-color-dual-arcs}
Let $G$ be a $d$-map, and $\cW^* = (W^*_1,...,W^*_d)$ be a dual $d$-GS wood of $G^*$. The total number of colors $n_e$ of a non-root edge $e = \{u,v\}$ satisfies $$2 ~\leq ~n_e ~\leq 2+(d-\deg(u))+(d-\deg(v)).$$ 
In particular, if $G^*$ is $d$-regular, then every non-root edge has exactly $2$ colors.
\end{cor}



\subsection{Dual of bipartite grand-Schnyder structures}\label{sec:dual-BGS} \hfill\\
In this subsection we investigate the bipartite case of dual GS structures. Let $d = 2b$ be an even integer, and let $G$ be a bipartite $2b$-map.  
Note that every vertex of the dual map $G^*$ has even degree.
As before we fix the bicoloring of the vertices of $G$ in black and white in which the outer vertex $v_1$ is black. The faces of the dual map $G^*$ are colored according to the colors of their corresponding primal vertices.


\begin{definition}\label{def:dual-BGS-labeling}  Let $G$ be a bipartite $2b$-map. A \emph{dual $b$-BGS labeling} of $G^*$ is a dual $2b$-GS labeling of $G^*$ such that the corners of black faces have odd labels, while the corners of white faces have even labels.
\end{definition}

Note that dual $b$-BGS labelings are exactly the dual, via $\Delta$, of $b$-BGS labelings defined in Section~\ref{sec:bipartite}. The parity condition is equivalent to requiring the label jump between consecutive corners in clockwise order around a vertex to be odd, and the label jump between consecutive corners in clockwise order around a face to be even.



Next we give the definition of dual bipartite GS woods. Under the bijection $\Theta^*$  between dual GS labelings and dual GS woods, the dual $b$-BGS labelings of $G$ are in bijection with the subclass of dual $2b$-GS woods satisfying the following condition: 
\begin{itemize}
  \item[($\dagger^*$)] \emph{For an inside arc, if the face on its right is black (resp. white), then it has exactly one more even (resp. odd) colors than odd (resp. even) colors.}
\end{itemize}

Let us call this subclass the \emph{even} dual $2b$-GS woods. Although this is not obvious, there is no loss of information in keeping only the information about even colors. In order to simplify the statements and analysis, we now make the further assumption that $G$ has no face of degree 2, or equivalently $G^*$ has no vertex of degree 2. Note that, under this asumption, no arc of $G^*$ can have $2b-1$ colors, hence no arc can have all the even colors.


\begin{definition}\label{def:dual-BGS-wood}
 Let $G$ be a bipartite $2b$-map having no face of degree 2. A \emph{dual $b$-BGS wood} of $G^*$ is a tuple $\mathcal{W'^*} = (W_1'^*,...,W_b'^*)$ of subsets of arcs of $G^*$ satisfying: \begin{itemize}
  \item[(BW0\sups{*})] For each $i \in [b]$, every non-root vertex $v$ has exactly one outgoing arc with color $i$. For all $i\in[b]$, the arc of the root-edge $e_{2i}^*$ going toward $v^*$ has only color $i$. The other arcs of root-edges have no color.
  
  \item[(BW1\sups{*})] Every inside arc with a black face on its right is in at least one color. 
  Let $v$ be a non-root vertex with outgoing arcs $a_1',...,a_b'$ with colors $1,...,b$, respectively. 
  The arcs $a_1',...,a_b'$ appear in clockwise order around $v$.
  
  \item[(BW2\sups{*})] Let $a$ be an inside arc with initial vertex $v$ and let $\de$ be the number of colors of $a$.

  If $a$ has a black (resp. white) face on its right and has color $i$, then the colors of the opposite arc $-a$ form a subset of $[i+1+m: i[$, where $m = \max(0, 1+b-\deg(v)/2-\de)$ (resp. $m = \max(0, b-\deg(v)/2-\de)$).

  If $a$ has a white face on its right and has no color, but is between the outgoing arcs of colors $i$ and $i+1$ in clockwise order around $v$, then the colors of the opposite arc $-a$ form a subset of $[i+1+m:i]$, where $m = \max(0, b-\deg(v)/2)$.
  
\end{itemize}
\end{definition}

\begin{lemma}\label{lem:reduced-dual-wood}
Let $G$ be a bipartite $2b$-map having no face of degree 2. For an even dual $2b$-GS wood $(W_1^*,...,W_{2b}^*)$, we define $\Lambda(W_1^*,...,W_{2b}^*) = (W_2^*,...,W_{2b}^*)$. Then, $\Lambda^*$ is a bijection between even dual $2b$-GS woods and dual $b$-BGS woods.
\end{lemma}

\begin{proof}
First we show that for any even dual $2b$-GS wood $\cW^*=(W_1^*,...,W_{2b}^*)$, the image  $\cW'^*=(W_1'^*,...,W_{b}'^*)=\Lambda(W_1^*,...,W_{2b}^*)$ is a $b$-BGS wood.
It is easy to check that $\cW'^*$ satisfies (BW0\sups{*}) and (BW1\sups{*}).  
For (BW2\sups{*}), there are three cases to check. Let us first consider an inside arc $a$ which has at least one even color $2i$ in $\cW^*$. If $a$ has a black face on its right, let $2\de-1$ be the number of colors of $a$, where $\de \geq 1$ is the number of even colors. By (W2\sups{*}), the colors of $-a$ in $\cW^*$ form a sub-interval of $[2i+1+k:2i[$, where $k = \max(0, 2+2b-\deg(v)-2\de)$ and $v$ is the initial vertex of $a$. Hence the colors of $-a$ in ${\cW'}^*$ form a sub-interval of $[i+1+m:i]$, where $m = k/2$. Hence $a$ satisfies (BW2\sups{*}). 
The argument is similar for an arc with a white face on its right and at least one even color. 
Lastly, consider an arc $a$ with no even color which is between outgoing arcs of even colors $2i$ and $2i+2$ in $\cW^*$. The arc $a$ has a single color $2i+1$ and has a white face on its right. By (W2\sups{*}), the colors of $-a$ form a sub-interval of $[2i+2+k:2i+1[$, where $k = \max(0, 2b-\deg(v))$. Hence the colors of $-a$ in $\cW'^*$ form a sub-interval of $[i+1+k/2:i]$, which is (BW2\sups{*}). Hence $\cW'^*$ satisfies (BW2\sups{*}) and is a $b$-BGS wood.

The map $\Lambda^*$ is injective because the colors assigned by $\Lambda^*$ correspond to the original even colors, and we can recover the odd colors from the even colors. 
Indeed, if an arc $a$ has a black face on its right, then it has at least one even color, but not all even colors (as noted above), so the odd colors can be recovered. Once the colors for the arcs with a black face on their right are known, this determines the colors for the arcs with a white face on their right by Condition (W1\sups{*}).


  Finally we prove surjectivity. Let $\cW'^*=(W_1'^*,...,W_b'^*)$ be a $b$-BGS wood of $G^*$. It is easy to see that the recovery rule layed out in the previous paragraph is still well-defined. Moreover, the result $\cW^*$ of this operation clearly satisfies (W0\sups{*}) and (W1\sups{*}) as well as the evenness condition~($\dagger^*$).  It remains to check Condition (W2\sups{*}) for every inside arc $a$.
There are again three cases (depending on whether $a$ has a color, and has a black or white face on its right), and one can check that Condition (BW2\sups{*}) for $\cW'^*$ translates into Condition (W2\sups{*}) for $\cW^*$ for each case.
This completes the proof that $\Lambda^*$ is surjective, hence a bijection.
\end{proof}

We now summarize our main result for dual bipartite GS structures.

\begin{thm}\label{thm:dual-BGS-main}
  Let $b \geq 2$ and let $G$ be a  bipartite $2b$-map with no face of degree 2. There exists a dual $b$-BGS wood (resp. labeling) for its dual map $G^*$ if and only if $G$ is $2b$-adapted.

  Moreover for any fixed $b$, there is an algorithm which takes as input a bipartite $2b$-adapted map and computes a dual $b$-BGS wood (resp. labeling) for $G^*$ in linear time.

  Lastly, the set $\mathbf{BL}_G^*$ of dual $b$-BGS labelings of $G^*$, the set $\mathbf{BL}_G$ of $b$-BGS labelings of $G$, the set $\mathbf{BW}_G^*$ of dual $b$-BGS woods of $G^*$, and the set $\mathbf{BW}_G$ of $b$-BGS woods of $G$ are all in bijection.
\end{thm}


%

%% file: technical-proofs-woods.tex

In this Section we prove Lemma~\ref{lem:beam-of-paths} about $d$-GS woods and establish some further consequences. Then, we prove Proposition~\ref{prop:bij-theta} establishing the bijection between $d$-GS woods and $d$-GS labelings, and some easy corollaries.

\subsection{Proof of Lemma~\ref{lem:beam-of-paths} and consequences}
\begin{proof}[Proof of Lemma~\ref{lem:beam-of-paths}]
Let $i,j$ be distinct colors in $[d]$. Let us prove that $P_i(v)$ and $P_j(v)$ are non-crossing for all vertices $v$ of $G$, with the convention that $P_i(v)$ is reduced to a vertex when $v$ is an outer vertex. We make an induction on $|P_i(v)|+|P_j(v)|$, where $|P|$ denotes the length of the path~$P$. The base case $|P_i(v)|+|P_j(v)|=0$ (outer vertices) is trivial. Let $v$ be an inner vertex. If the paths $P_i(v)$ and $P_j(v)$ have no common inner vertices beside $v$, then these paths are non-crossing. Assume now that there are other common inner vertices, and let $u$ be the first inner vertex on $P_i(v)$ which belongs to $P_j(v)$ and is different from $v$. By induction, we can assume that $P_i(u)$ and $P_j(u)$ are non-crossing, and we need to examine what happens at $u$.

If the first arc of $P_i(v)$ and $P_j(v)$ is equal (with endpoint $u$), then it is clear that $P_i(v)$ and $P_j(v)$ are non-crossing. Suppose now that the first arcs of $P_i(v)$ and $P_j(v)$ are different. 
The situation is represented in Figure~\ref{fig:crossing-pathsb}.
Let $P_i$ (resp. $P_j$) be the part of $P_i(v)$ ($P_j(v)$) from $v$ to~$u$. The union of $P_i$ and $P_j$ forms a simple cycle $C$. For concreteness, let us assume that 
 the finite region enclosed by $C$ is on the right of $P_i$ and on the left of $P_j$ (the other case being treated in the exact same manner). By Condition (W1), for all $k\in [i+1:j[$, the path $P_k(v)$ starts \emph{weakly inside} of $C$ (that is, on $C$ or strictly inside of $C$), and we consider the initial portion $P_k$ of $P_k(v)$ before the first arc strictly out of $C$.
Observe that $P_{i+1}$ ends on $P_j$ (since, by Condition (W2) $P_{i+1}(v)$ cannot cross $P_i(v)$ from right to left), hence $P_{i+2}$ ends on $P_j$ (since $P_{i+2}(v)$ cannot cross $P_{i+1}(v)$ from right to left), etc. Hence the path $P_k$ ends on $P_j$ for all $k\in [i+1:j[$. By a symmetric argument (starting with $P_{j-1}$), the path $P_k$ ends on $P_i$ for all $k\in [i+1:j[$. In conclusion, all the paths $P_k$ end at $u$. For $k\in[i:j+1[$, let $a_0^k$ be the last arc of $P_{k}(v)$ before $u$ and let $a_1^k$ be the following arc. Since $a_0^k$ is weakly inside $C$ for all $k\in [i:j+1[$, and $a_1^k$ is strictly outside $C$ for all $k\in [i+1:j[$, 
Condition (W2) implies that the arcs $a_0^i,a_1^i,a_1^{i+1},\ldots,a_1^{j},a_0^j$ appear in clockwise order (weakly) around~$u$. From this and the fact that $P_i(u)$ and $P_j(u)$ are non-crossing it is not hard to see that $P_i(v)$ and $P_j(v)$ are non-crossing (indeed, if $P_i(u)$ was to enter $C$ by crossing $P_j(v)$, then it would have to cross $P_j(v)$ again to exit $C$ and this would violate the induction hypothesis).
\fig{width=.3\linewidth}{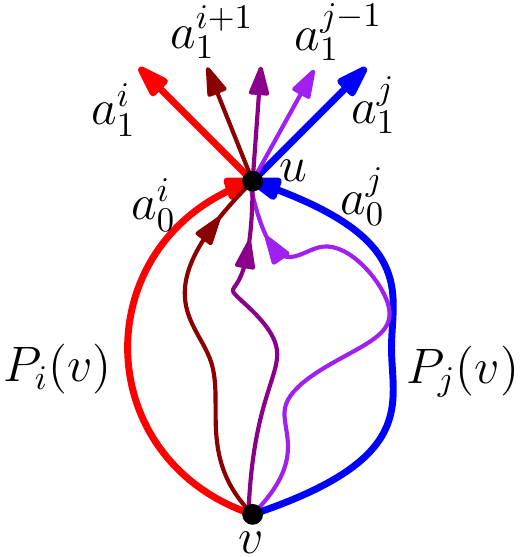}{The situation in the proof of Lemma~\ref{lem:beam-of-paths}.}

It remains to prove the statement about the clockwise order of $r_1(v),r_2(v),\ldots,r_d(v)$. This amounts to proving that for $i<j<k$ in $[d]$, the vertices $r_i(v),r_j(v)$ and $r_k(v)$ are (weakly) in clockwise order around the outer face of $G$. This easily follows from the non-crossing property together with Condition (W1), and we omit the details.
\end{proof}


\begin{lemma}\label{lem:vi-not-in-Ri}
Let $G$ be a $d$-map and let $\cW$ be a $d$-GS wood (or a $d$-tuple of sets of arcs only satisfying Conditions (W0-W2)). Let $v$ be an inner vertex. For $i$ in $[d]$, let $r_i(v)$ be the endpoint of the path $P_i(v)$ of color $i$ starting at $v$. For all $i\in [d]$, the outer vertex $v_i$ appears strictly between $r_{i}(v)$ and $r_{i-1}(v)$ in clockwise order around the outer face of $G$. Equivalently, $v_i$ does not belong to the region $R_i(v)$.
\end{lemma}
Lemma~\ref{lem:vi-not-in-Ri} is represented in Figure~\ref{fig:beam-of-paths}(b).

\begin{proof}
For some outer vertices $u_1,\ldots, u_k$, we say that $u_1,\ldots, u_k$ \emph{appear clockwise} to mean that they are appear in clockwise order weakly around the outer face of $G$. Suppose for contradiction that Lemma~\ref{lem:vi-not-in-Ri} does not hold. For concreteness, let us suppose that $r_{d}(v), v_1, r_{1}(v)$ appear clockwise.
Since $r_1(v)\neq v_1,v_2$ and $r_d(v)\neq v_d,v_1$ by Condition (W0), this implies that $r_d(v), v_d, v_1,v_2, r_1(v)$ appear clockwise. Since $r_2(v)\neq v_2,v_3$, Lemma~\ref{lem:beam-of-paths} implies that $r_d(v), v_d, v_1,v_2, v_3, r_2(v)$ appear clockwise. Continuing in this manner, we get that $r_d(v), v_d, v_1,v_2, v_d, r_{d-1}(v)$ appear clockwise. This gives $r_d(v)=v_d$, which contradicts (W0).
\end{proof}

\begin{cor} \label{cor:containment-regions}
Let $G$ be a $d$-map and let $\cW$ be a $d$-GS wood (or even if these sets of arcs only satisfy Conditions (W0-W2)). Let $v,v'$ be distinct inner vertices of~$G$. If $v'$ belongs to the region $R_i(v)$ for some $i\in[d]$, then $R_{i}(v')$ is contained in $R_{i}(v)$, and $R_i(v')\neq R_i(v)$.
\end{cor}

\begin{proof}
Suppose that $v'$ belongs to $R_i(v)$. Observe that if the paths $P_i(v')$ and $P_i(v)$ have a vertex in common, then these paths will ``merge'' and $r_{i}(v')=r_{i}(v)$. Moreover, Condition (W2) implies that $P_{i}(v')$ cannot cross $P_{i-1}(v)$ from right to left. Hence, the path $P_i(v')$ stays inside $R_{i}(v)$. Similarly, $P_{i-1}(v')$ stays inside $R_{i}(v)$. 
In particular, the endpoints $r_{i-1}(v')$ and $r_{i}(v')$ are both between $r_{i-1}(v)$ and $r_i(v)$ in clockwise order around the outer face of $G$. Furthermore, Lemma~\ref{lem:vi-not-in-Ri} implies that the vertices $r_{i-1}(v),\,r_{i-1}(v'),\,r_{i}(v'),\,r_i(v)$ appear in this order clockwise around the outer face of $G$. Thus $R_{i}(v')$ is contained in $R_{i}(v)$. Lastly, $R_i(v')\neq R_i(v)$ because otherwise one of the paths $P_{i-1}(v')$ or $P_i(v')$ would have to go through $v$, which is impossible by Condition (W2) at $v$.
\end{proof}

\subsection{Proof of Proposition~\ref{prop:bij-theta} and consequences.}
We will now prove Proposition~\ref{prop:bij-theta} and Lemma~\ref{lem:W2'}.
\begin{lemma}\label{lem:image-theta}
Let $G$ be a $d$-map. For any $d$-GS labeling $\cL$ of $G$, the image $\Th(\cL)$ is a $d$-GS wood.
\end{lemma}

\begin{proof}
Let $\cL\in \bL_G$, and let $\cW=(W_1,\ldots,W_d)=\Th(\cL)$ be the corresponding arc coloring (where $W_i$ is the set of arcs of color $i$).

In $\cL$ the sum of label jumps in clockwise order around any inner vertex is $d$, hence the inner vertices of $G$ are incident to one outgoing arc of color $i$ for all $i\in [d]$, and the colors are not all on the same arc. 
Moreover, by definition of $\Th$, the outer vertex $v_k$ of $G$ is incident to one outgoing arc of color $i$ for all $i\neq k$, and no outgoing arc of color $k$. Lastly, it is easy to see that Property (L0) of $\cL$ implies that $v_i$ and $v_{i+1}$ are not incident to ingoing arcs of color $i$. Hence $\cW$ satisfies (W0).

For any inner vertex $v$, the sum of label jumps in clockwise order around $v$ is $d$. This implies that the outgoing arcs of color $1,2,\ldots,d$ at $v$ appear in clockwise order around $v$. Hence $\cW$ satisfies (W1).

Next, we show that $\cW$ satisfies Condition (W2). 
Let $a$ be an inner arc of $G$ oriented from $u$ to $v$. We assume that $a$ has color $i$ in $\cW$, and that $v$ is an inner vertex. 
We want to show that $a$ appears strictly between the outgoing arc of color $i+1$ and the outgoing arc of color $i-1$ around $v$. 
Let $i_1,i_2,i_3,i_4$ be the labels in counterclockwise order around $a$ as indicated in Figure~\ref{fig:labels-around-edge}.
 Since the arc $a$ has color $i$, we have $i\in [i_1:i_2[$. Moreover, since the label jumps are all positive around faces by Condition (L2), we have $i-1,i, i+1\in [i_4:i_1[\cup [i_1:i_2[\cup [i_2:i_3[$. By Lemma~\ref{lem:ccw-jumps-edges}, the sum of label jumps counterclockwise around the arc $a$ is equal to $d$, hence the sets $[i_4:i_1[$, $[i_1:i_2[$, $[i_2:i_3[$, and $[i_3:i_4[$ are disjoint (and give a partition of $[d]$). This implies that the arcs of color $i-1$, $i$ and $i+1$ appear in this order in clockwise order around $v$ starting at the corner labeled $i_4$ and ending at the corner labeled $i_3$, which proves that the arc $a$ satisfies (W2).

\fig{width=.3\linewidth}{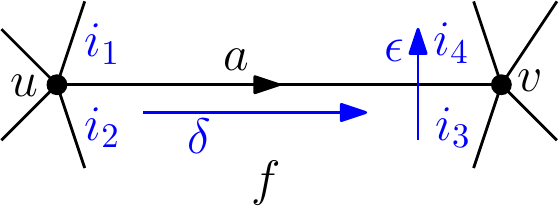}{Corner labels around an edge.}

It remains to show that $\cW$ satisfies (W3). Consider an inner arc $a$ oriented from $u$ to $v$. Let $f$ be the face at the right of $a$. Let $i_1,i_2,i_3,i_4$ be the labels in counterclockwise order around $a$ as indicated in Figure~\ref{fig:labels-around-edge}. 
Suppose that $a$ has color $i$, or that $a$ is strictly between the outgoing arcs of color $i$ and $i+1$ in clockwise order around $u$. Suppose also that the number $\eps:=|[i_3:i_4[|$ of colors of the arc $-a$ satisfies $d-\deg(f)-\eps>0$. 
In order to prove that the arc $a$ satisfies Condition~(W3), it suffices to prove that $[i:i+2+d-\deg(f)-\eps[\subseteq [i_4:i_3[$. Since, by Lemma~\ref{lem:ccw-jumps-edges}, the sum of label jumps counterclockwise around $a$ is $d$, we have $[i_4:i_3[=[i_4:i_2[\cup [i_1:i_3[$. Moreover, under our hypotheses, $i\in[i_4:i_2[$, and $[i+1:i+1+\delta[\subseteq [i_1:i_3[$, where $\delta=|[i_2:i_3[|$.
Lastly, by Condition (L3) we have $\delta\geq d+1-\deg(f)-\eps$, hence $[i :i+2+d-\deg(f)-\eps[\subseteq [i_4:i_3[$. This concludes the proof that Condition (W3) holds. Thus, $\cW$ is a $d$-GS wood.
\end{proof}

\begin{lemma}\label{lem:image-theta-inverse}
Let $G$ be a $d$-map. For any $d$-GS wood $\cW$ of $G$, the image $\bTh(\cW)$ is a $d$-GS labeling.
\end{lemma}

\begin{proof}
Let $\cW$ be a $d$-GS wood of $G$, and let $\cL=\bTh(\cW)$. Condition (L0) holds for $\cL$ by definition of $\bTh$. Moreover it is clear that Condition (W1) for $\cW$ implies that the sum of label jumps clockwise around inner vertices is always $d$. In order to establish that $\cL$ satisfies (L1) and (L2) we need to establish two technical results about the label situation around inner edges. 

\noindent \textbf{Claim 1:} For every inner edge $e$ of $G$, the 4 corners incident to $e$ cannot all have the same label.

Suppose, for contradiction, that all 4 corners incident to $e$ have the same label $j$. Let $a$ be an arc of $e$, and let $u$ and $v$ be the initial and terminal vertices respectively. Let $i=j-1$. By definition, the arc $a$ (resp. $-a$) belongs to none of the sets $W_1,\ldots, W_d$ but appears between the outgoing arc in $W_{i}$ and $W_{i+1}$ in clockwise order around $u$ (resp. $v$). This contradicts Condition (W3) for $a$.

\noindent \textbf{Claim 2:} Let $a$ be an inner arc of $G$, and let $i_1,i_2,i_3,i_4$ be the labels of the incident corners as indicated in Figure~\ref{fig:labels-around-edge}. The sum of label jumps in counterclockwise order around $a$ is $d$, and moreover $i_2\neq i_3$ and $i_1\neq i_4$.

Let us first prove Claim 2 in the case $i_1\neq i_2$. If $i_1\neq i_2$ then by property (W2) of $\cW$ we have $[i_1:i_2[\cap [i_3:i_4[=\emptyset$, hence the sum of label jumps in counterclockwise order around $a$ is~$d$. Moreover, still by property (W2), $i_2\neq i_3$ and $i_1\neq i_4$ so Claim 2 holds. 
By symmetry, if $i_3\neq i_4$, then Claim 2 holds.
Lastly, if $i_1=i_2$ and $i_3=i_4$, then $i_1=i_2\neq i_3=i_4$ by Claim 1, which implies again that the sum of label jumps in counterclockwise order around $e$ is $d$, and Claim 2 holds again.

Claim 2 trivially implies that $\cL$ satisfies (L2). Moreover, Claim 2 implies that the sum of label jumps clockwise around every inner face is at least $d$ (since it is a multiple of $d$ and cannot be 0). Next, we use Equation~\eqref{eq:sum-jumps-relation} between the sum of label jumps around vertices, edges and faces.
Using Claim 2, we get
\begin{equation}\label{eq:sum-jumps-relation2}
d(|V|+|F|)\leq \sum_{v\in V}\cwjump(v)+\sum_{f\in F}\cwjump(f) =d+\sum_{e\in E}\ccwjump(e)=d(1+|E|),
\end{equation}
where $V,F,E$ are the set of inner vertices, faces, and edges of $G$ respectively. By the Euler relation we have $|V|+|F|=1+|E|$, hence the inequality in~\eqref{eq:sum-jumps-relation2} is an equality. Thus the sum of label jumps clockwise around every inner face is $d$. This complete the proof that $\cL$ satisfies (L1).

It remains to prove that $\cL$ satisfies (L3). Consider an inner arc $a$ oriented from $u$ to $v$ with incident corners labeled $i_1,i_2,i_3,i_4$ as indicated in Figure~\ref{fig:labels-around-edge}. Let $f$ be the face at the right of $a$, let $\delta=|[i_2:i_3[|$ and let $\eps=|[i_3:i_4[|$. We want to show $\delta+\eps\geq d-\deg(f)+1$.
If $d-\deg(f)-\eps<0$, then this inequality clearly holds (since $\delta>0$). Suppose now that $d-\deg(f)-\eps\geq 0$, and consider Consider (W3) of $\cW$. If $i_1\neq i_2$, then Condition (W3) applied to color $i=i_2-1$ of $a$ gives $\delta\geq d-\deg(f)-\eps+1$ as wanted. If $i_1=i_2$, then $a$ is between the outgoing arcs of $W_i$ and $W_{i+1}$ around $u$ for $i=i_2-1$ and we also get $\delta\geq d-\deg(f)-\eps+1$ as wanted. This shows that $\cL$ satisfies (L3), which completes the proof that $\cL$ is a $d$-GS labeling.
\end{proof}

\begin{proof}[Proof of Proposition~\ref{prop:bij-theta}] Let $G$ be $d$-map. 
By Lemmas~\ref{lem:image-theta} and~\ref{lem:image-theta-inverse}, $\Th$ is a map from $\bL_G$ to $\bW_G$, and $\bTh$ is a map from $\bW_G$ to $\bL_G$. It is easy to see that $\bTh\circ \Th=\Id_{\bL_G}$ and $\Th\circ \bTh=\Id_{\bW_G}$. Thus these are inverse bijections. 
\end{proof}

Before closing this section, we prove Lemmas~\ref{lem:nb-color-arcs} and~\ref{lem:W2'} as corollaries of Proposition~\ref{prop:bij-theta}.

\begin{proof}[Proof of Lemma~\ref{lem:nb-color-arcs}]
Let $G$ be a $d$-map, let $e$ be an inner edge, and let $f$ and $f'$ be the faces incident to $e$. Let $n_e$ be the number of colors of $e$. Consider the $d$-GS labeling $\cL=\bTh(\cW)$. By Lemma~\ref{lem:ccw-jumps-edges}, the sum of label jumps in counterclockwise order around $e$ is $d$. This is represented in Figure~\ref{fig:edge-jumps}, and we refer to this figure to define the label jumps $\delta, \eps,\delta',\eps'$ around $e$. The number of colors of $e$ in $\cW$ is 
$$n_e=\eps+\eps'=d-\delta-\delta'.$$ 
Moreover, $1\leq \delta\leq d-\deg(f)+1$ and $1\leq \delta'\leq d-\deg(f')+1$ because the sum of label jumps around faces is $d$ and each jump is at least 1. This gives $\deg(f)+\deg(f')-d-2\leq n_e\leq d-2$ as claimed.
\end{proof}
 
\begin{proof}[Proof of Lemma~\ref{lem:W2'}]
Let $G$ be a $d$-adapted map, and let $\bW_G'$ be the set of maps satisfying Conditions (W0), (W1) and (W2'). We clearly have the inclusion $\bW_G\subseteq \bW_G'$ and want to show $\bW_G'= \bW_G$. 

Let $\bTh'$ be the extension of the map $\bTh$ to $\bW_G'$ (with the same definition as $\bTh$). Let $\cW\in \bW_G'$ and let $\cL=\bTh'(\cW)$. We want to show that $\cL$ is in $\bL_G$. We reason as in the proof of Lemma~\ref{lem:image-theta-inverse}, except we cannot assume that $\cW$ satisfies (W3). In the proof of Lemma~\ref{lem:image-theta-inverse}, (W3) was used to justify Claim 1, and to justify that $\cL$ satisfies (L3). Let us now give an alternative proof of Claim 1, which does not use (W3). 

Suppose for the sake of contradiction, that an inner edge $e$ has its 4 incident corners labeled $i$. Since $G$ is $d$-adapted, it has no loop, hence $e$ has 2 distinct endpoints $u,v$. Note that $u$ and $v$ cannot both be outer vertices by Condition (L0). Hence, we can suppose that $v$ is an inner vertex without loss of generality. Consider the paths $P_1(v),\ldots,P_d(v)$ of colors $1,2,\ldots,d$ starting at $v$ as defined in Section~\ref{subsec:GS-woods}. Let $R_i(v)$ be the region delimited by the paths $P_{i-1}(v)$ and $P_i(v)$. Since the corners incident to $e$ have label $i$, the edge $e$ is in $R_i(v)$, hence $u$ is in $R_i(v)$. By Lemma~\ref{lem:vi-not-in-Ri} this implies that $u\neq v_i$. Since $u$ is incident to some corners labeled $i$, we conclude that $u$ is not an outer vertex. Hence, both $u$ and $v$ are inner vertices. Moreover, $u$ is in $R_i(v)$, and symmetrically $v$ is in $R_i(u)$.
By Corollary~\ref{cor:containment-regions}, the region $R_i(u)$ is strictly contained in $R_i(v)$, and $R_i(v)$ is strictly contained in $R_i(u)$. This is a contradiction, hence Claim 1 holds.

It remains to prove that $\cL$ satisfies (L3). Let $a$ be an inner arc of $G$. 
We need to show $\delta+\eps\geq d-\deg(f)+1$. If $i_1\neq i_2$ then Condition (W2') applied to color $i=i_2-1$ of $a$ gives $\delta\geq d-\deg(f)-\eps+1$ as wanted. We now consider the case $i_1=i_2$. Let $f'$ be the face at the left of $a$. Since $G$ is $d$-adapted, we must have $\deg(f)+\deg(f')-2\geq d$ (because $\deg(f)+\deg(f')-2$ is the length of a non-facial cycle of $G$: the contour of the face one would obtain by deleting $e$ and merging $f$ and $f'$). Moreover, $|[i_4:i_1[|\leq d-\deg(f')+1$ because the sum of label jumps around $f'$ is $d$ and each jump is at least 1.
Since the sum of label jumps in counterclockwise order around $a$ is $d$, one obtains
$$\delta+\eps=d-|[i_4:i_1[|\geq \deg(f')-1\geq d-\deg(f)+1.$$
Thus Condition (L3) holds and $\cL$ is in $\bL_G$.

Since the image of $\bTh'$ is in $\bL_G$, we can compose this map with $\Th$. For all $\cW\in \bW_G'$ one clearly has $\Th \circ \bTh'(\cW)=\cW$, hence $\cW$ is in the image of $\Th$, which is in $\bW_G$ by Lemma~\ref{lem:image-theta}. This concludes the proof that $\bW_G'= \bW_G$. 
\end{proof}


%% file: existence-proof.tex

In this section we complete the proof of Theorem~\ref{thm:main}. Since we have already established the bijection between the different incarnations of grand-Schnyder structures (woods, labelings, marked orientations, and angular orientations), it suffices to show that a $d$-map $G$ admits a $d$-GS angular orientation if and only if $G$ is $d$-adapted.

The proof that $d$-adaptedness is a necessary condition for the existence of $d$-GS angular orientation is based on a simple counting argument.
Suppose that a $d$-map $G$ admits a $d$-GS angular orientation. Let $C$ be a non-facial simple cycle of $G$. We want to show that $C$ has length at least $d$.
Recall from Lemma~\ref{lem:ingoing-weight-angular} that the total weight $\om(C)$ of the arcs strictly inside $C$ with initial vertex on $C$ is
$$\om(C)=\ell(C)-d+\sum_{a\in C}(d-\deg(f_a)),$$
where $\ell(C)$ is the length of $C$, the sum is over the arcs of $C$, and for an arc $a$ of $C$ the face incident to $a$ inside $C$ is denoted by $f_a$.
Observe now for each inner face $f$ of $G$ inside $C$ and incident to $k\geq 1$ edges on $C$, there are at least $k+1$ corners of $f$ incident to a vertex on $C$. 
The weight condition (A1) then ensures that the total contribution to $\om(C)$ of the edges incident to $v_f$ is at least $k(d-\deg(f))$. Hence, $\om(C)\geq \sum_{a\in C}(d-\deg(f_a))$, so that $\ell(C)\geq d$.

This prove that $d$-adaptedness is a necessary condition for the existence of $d$-grand-Schnyder structures. In order to complete the proof of Theorem~\ref{thm:main}, it remains to prove that any $d$-adapted map admits a $d$-GS orientation and that such an orientation can be computed \emph{in linear time}. Here and in the rest of this section we say that a structure for $d$-adapted maps can be computed \emph{in linear time} if it can be computed in a number of operations which is linear in the number of vertices of the $d$-adapted map. It should be noted that for $d$-adapted maps, the number of edges $\ee$ is linear in the number of vertices $\vv$ (it is easy to show that $\vv\leq \ee\leq 6\vv$ using the fact that there cannot be 3 edges with the same endpoints), and that the bijections between the different incarnations of $d$-GS structures described in Section~\ref{sec:incarnations} can be performed in linear time.\footnote{In~\cite{OB-EF:Schnyder} an existence proof was given for the special case of Schnyder decompositions (which correspond to $d$-GS structures on $d$-angulations) by using a ``min-cut max-flow'' type of argument. It may be possible to adapt this argument in the general case, however the proof in the present article is different: it is constructive and can be used to define a linear time algorithm for computing $d$-GS structures, thereby answering a question which was left open for Schnyder decompositions in~\cite{OB-EF:Schnyder}.}

The existence proof is by induction on~$d$. 
Since the induction step is rather technical, we will first present the main ideas on a simple case: we will explain how to deduce the existence of transversal structures (for 4-adapted triangulations of the square), from the existence of Schnyder woods (for 3-adapted triangulations).

\subsection{A preview of the existence proof: existence of 4-GS angular orientations}\label{sec:special-case-proof-exist}

Let us fix a 4-adapted map $G$ such that the inner faces have degree 3 or 4 (no face of degree 2). We want to construct a 4-GS angular orientation of $G$. The process is represented in Figure~\ref{fig:4-GS-existence}.

\fig{width=.85\linewidth}{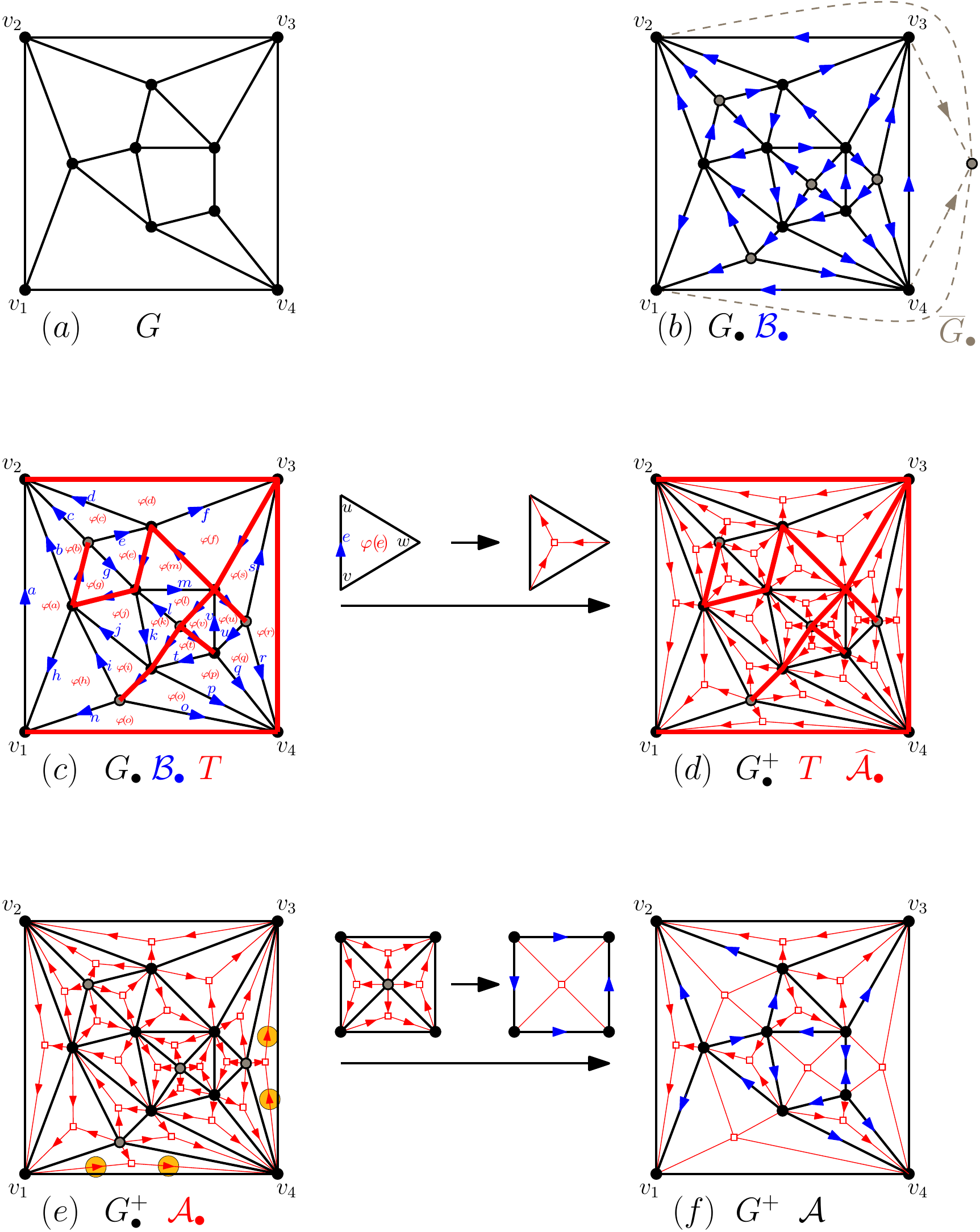}{Construction of a 4-GS angular orientation $\cA$ for a 4-adapted map $G$. (a) The map $G$. (b) 
The triangulation $\bGb$ with a 3-GS orientation. Restricting the 3-GS orientation to $\Gb$ gives an orientation $\cBb$ such every inner-vertex has outdegree 3. (c) A spanning tree $T$ of $\Gb$ oriented from the outer vertices to the inner vertices, and the bijection between the edges in $\Gb\setminus T$ and the faces of $\Gb$. (d) Construction of the angular orientation $\hAb$ from $\cBb$. (e) The 4-GS angular orientation $\cAb$ closely related to $\hAb$ (the differences with $\hAb$ are highlighted). (f) Construction of the 4-GS angular orientation $\cA$ from~$\cAb$.}

Let $\Gb$ be the triangulation obtained from $G$ by adding a vertex $u_f$ in each inner face $f$ of degree~4 and joining $u_f$ to each of the vertices of $f$, as represented in Figure~\ref{fig:4-GS-existence}(b). It is easy to see that $\Gb$ is 4-adapted. By adding a vertex in the outer face of $\Gb$ and joining it to the outer vertices of $\Gb$, one gets a 3-adapted triangulation $\bGb$. Since this triangulation admits a Schnyder wood, we can obtain the following structure for $\Gb$: an orientation $\cBb$ of the inner edges of $\Gb$ such that every inner vertex has outdegree 3. See Figure~\ref{fig:4-GS-existence}(b).

An important observation is that $\cBb$ is \emph{co-accessible}: for every inner vertex $v$ of $\Gb$ there exists a directed path in $\cBb$ from an outer vertex of $\Gb$ to $v$. We will justify this claim in Section~\ref{sec:induction-step}, and simply mention that this is a consequence of the fact that $\Gb$ is 4-adapted. As a consequence, there exists a spanning tree $T$ of $\Gb$ containing all the outer edges except $\{v_1,v_2\}$, such that the inner edges in $T$ form directed paths from the outer vertices to the inner vertices. See Figure~\ref{fig:4-GS-existence}(c). 

Next, we use $T$ to construct a 4-GS angular orientation of $\Gb$ as represented in Figure~\ref{fig:4-GS-existence}(c-d). We will use a bijection between the edges of $\Gb\setminus T$ and the inner faces of $\Gb$. This is a standard correspondence, that we now recall.

\begin{claim}\label{claim:inside-face}
Let $M$ be a plane map, and let $T$ be a spanning tree of $M$. To each edge $e$ in $M\setminus T$ we associate the face $\varphi(e)$ incident to $e$ which is enclosed by the unique cycle in $T\cup \{e\}$. The mapping $\varphi$ is a bijection between the set of edges in $M\setminus T$ and the inner faces of $M$.
\end{claim}
\begin{proof} The claim is classical, so we only sketch the proof idea here. The dual of the edges in $M\setminus T$ forms a spanning tree $T^*$ of the dual map $M^*$. We take the vertex of $M^*$ in the outer face to be the root of $T^*$. Then, the mapping $\varphi$ can be interpreted as the correspondence between the non-root vertices of $T^*$ and the edges of $T^*$ (each vertex is associated to its parent edge).
\end{proof}

Let $\Gb^+$ be the angular map of $\Gb$. We now describe an orientation of the star edges of $\Gb^+$.
Let $f$ be a face of $\Gb$, and let $e$ be the edge of $\Gb\setminus T$ such that $f=\varphi(e)$, where $\varphi$ is the bijection described in Claim~\ref{claim:inside-face}. Let $u$ be the terminal vertex of $e$ in the orientation $\cBb$ (by convention we orient the outer edge $\{v_1,v_2\}$ toward $v_1$) and let $v,w$ be the other vertices of $\Gb$ incident to $f$. We orient the star edges incident to the star vertex $v_{f}$ as follows: the star edge between $v_{f}$ and $u$ is oriented toward $u$ while the 2 others are oriented toward $v_{f}$. Let $\hAb$ be the resulting orientation of the star-edges of $\Gb^+$ (we forget the orientations of the original edges of $\Gb^+$). This is represented in Figure~\ref{fig:4-GS-existence}(c-d).

It is not hard to check that the original inner vertices of $\Gb$ have outdegree~4 in $\hAb$ (and we delay the proof to Section~\ref{sec:induction-step}). 
Hence, the orientation $\hAb$ can be identified with a $4$-GS angular orientation of $\Gb$: the star edges have weight 1, the original edges have weight 0, the star vertices have weight 1 and the original vertices have weight~4. 
More precisely, we see that $\hAb$ satisfy condition (A1) and (A2) of angular orientation. Up to a simple adjustment of the orientation of the star-edges incident to the outer vertices of $\Gb$, one can obtain a 4-GS angular orientation $\cAb$ of $\Gb$. 
See Figure~\ref{fig:4-GS-existence}(e), and the proof in Section~\ref{sec:induction-step}. 

Lastly, we use $\cAb$ to construct a 4-GS angular orientation $\cA$ of $G$ as represented in Figure~\ref{fig:4-GS-existence}(e-f).
Let $G^+$ be the angular map of $G$. Recall that an angular orientation of $G$ must satisfy the following conditions: 
\begin{compactitem}
\item the weight of original inner vertices is~4, 
\item the weight of any star vertex inside a face of $G$ of degree~$3$ (resp.~4) is~1 (resp.~0),
\item the weight of any original inner edge is equal to the number of incident faces of $G$ of degree~4, 
\end{compactitem}
We define $\cA$ as follows. For star arcs of $G^+$ inside faces of $G$ of degree $3$, we take their weight in $\cA$ to be the same as in $\cAb$. For an original inner arc $a$ of $G^+$, the weight $w_a$ of $a$ in $\cA$ is determined by looking at the triangles of $\Gb$ incident to $a$ which are inside a face of $G$ of degree~4: $w_a$ is the number of star edges in these triangles which are oriented toward the terminal vertex of $a$ in $\cAb$. 
This definition of $\cA$ is represented in Figure~\ref{fig:4-GS-existence}(e-f): the orientation $\cAb$ inside a face $f$ of $G$ of degree~4 is used to determine the weight of the arcs of $G$ incident to $f$ (one thing to observe is that the star-edges of $\Gb^+$ incident to the vertex $u_f$ of $\Gb$ are always oriented away from $u_f$). It is not hard to check that $\cA$ is indeed a 4-GS angular orientation of~$G$.

The above sketch of proof of the existence of a 4-GS angular orientation is constructive. Since there are algorithms for constructing 3-GS angular orientations in linear time~\cite{Schnyder:wood1}, the above process leads to an algorithm for constructing 4-GS orientations in linear time. This concludes this preview of our proof, and we will now tackle the general case.

\subsection{Existence of $d$-GS angular orientations by induction on $d$}\label{sec:induction-step}
In this subsection we present the general induction step for the existence of $d$-GS orientations. To be precise, we will prove that every $d$-adapted map \emph{without multiple edges} admits a $d$-GS angular orientation that can be computed in linear time. The case of maps with multiple edges will be treated in the next subsection. 

As noted above, the case $d=3$ was established by Schnyder~\cite{Schnyder:wood1}. Indeed the 3-adapted maps without multiple edges are the 3-connected triangulations, and the Schnyder woods of triangulations are in easy bijection with 3-GS angular orientations. Moreover, as established in~\cite{Schnyder:wood1}, Schnyder woods can be computed in linear time.

Suppose now that for some $d\geq 3$ the existence of $d$-GS angular orientations has been established for $d$-adapted maps without multiple edges. We consider a $(d+1)$-adapted map $G$ without multiple edges, for which we want to construct a $(d+1)$-GS orientation. As a guide, Figure~\ref{fig:diagram_proof} shows the sequence of steps to be performed for this inductive process.

\fig{width=.9\linewidth}{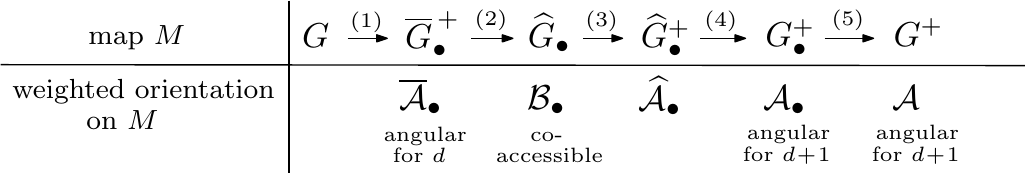}{The sequence of steps to obtain a $(d+1)$-GS angular orientation on a $(d+1)$-adapted map $G$. Regarding the sequence of operations on maps, in step (1), to obtain $\overline{G}_\bullet$ from $G$, a copy of $X_d$ is inserted in each face of degree $d+1$ (including the outer one), and the map is re-rooted at a $d$-face in the outer copy of $X_d$. In step (2), to obtain $\widehat{G}_{\bullet}$ from $\overline{G}_{\bullet}^+$, the outer copy of $X_d$ is deleted (giving $G_{\bullet}^+$), and then the star-vertices of degree $d$ and their incident edges are deleted. In (4), to recover $G_{\bullet}^+$ from $\widehat{G}_{\bullet}^+$, the edges of $\widehat{G}_{\bullet}^+\backslash \widehat{G}_{\bullet}$ (``small'' edges) that are incident to star vertices of $G_{\bullet}^+$ are contracted. In (5), to recover $G$ from $G_{\bullet}$, the inner copies of $X_d$ are deleted. Regarding weighted orientations, all steps involve a trivial transfer (under edge contractions, edge deletions, or merges of edges), except step (3) where a directed spanning tree of $\widehat{G}_\bullet$ is used to orient the small edges and to decrement some arc weights in $\widehat{G}_\bullet$.
}
 
We first construct a new map $\Gb$ from $G$ by dissecting the faces of $G$ of degree $d+1$.
Let $X_d$ be the $(d+1)$-map represented in Figure~\ref{fig:face-plug} (the definition of $X_d$ depends on the parity of~$d$). Let $\Gb$ be the map obtained by gluing $X_d$ in each inner face of $G$ of degree $d+1$. 
 \fig{width=.65\linewidth}{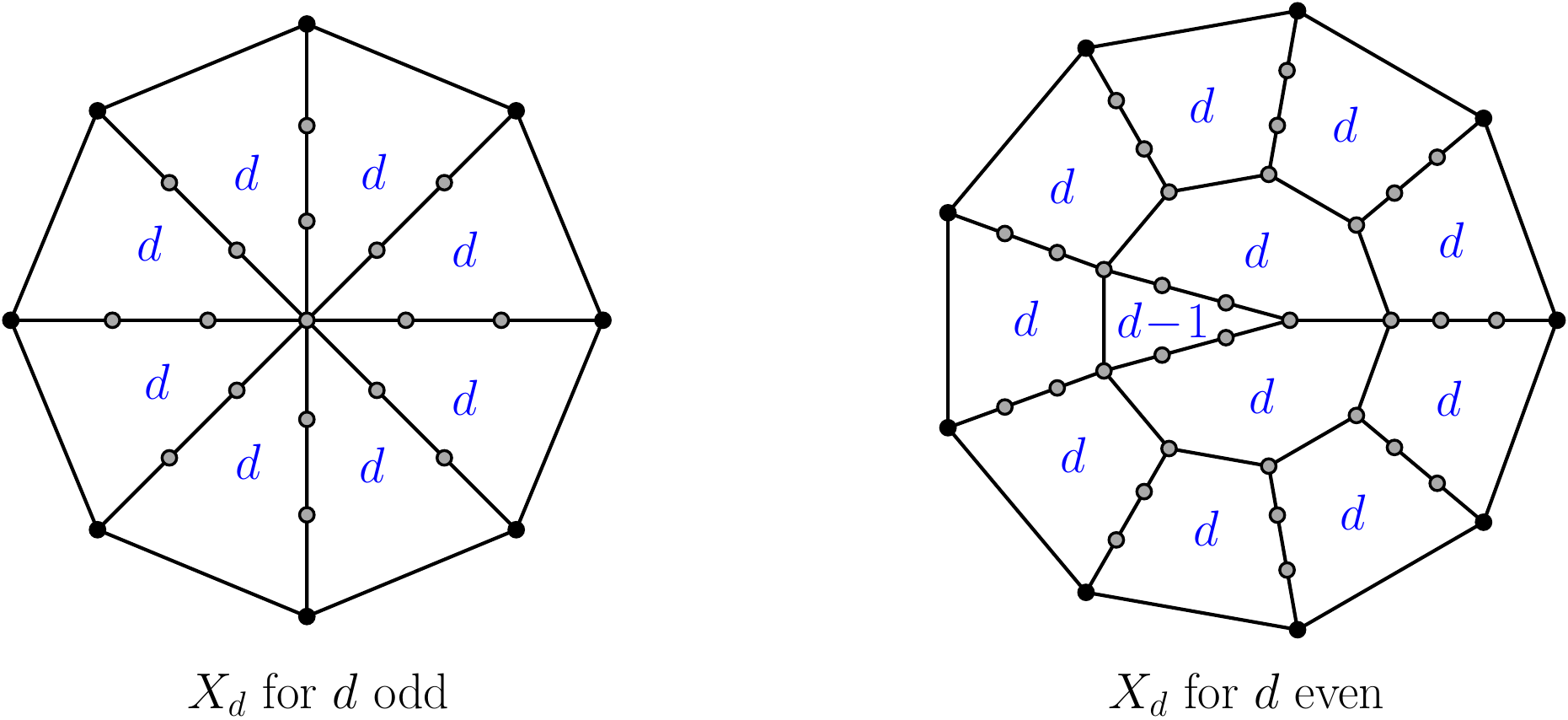}{The $(d+1)$-map $X_d$ in the case of odd $d$ (left, represented for $d=7$) and in the case of even $d$ (right, represented for $d=8$). The degree of the inner faces are indicated (the outer face has degree $d+1$).}

Note that the inner faces of $\Gb$ have degree at most~$d$. Moreover it is not hard to check the following claim.

\begin{claim}\label{claim:still-adapted} 
The map $\Gb$ has no non-facial simple cycle of length less than $d+1$.
\end{claim}

\begin{proof} 
If a non-facial simple cycle uses none of the inner edges from the copies of $X_d$, then it has length at least $d+1$ since $G$ is $(d+1)$-adapted. 
If a non-facial simple cycle stays within a copy of $X_d$ then it has length at least $d+1$. Finally, consider a non-facial simple cycle $C$ using some edges in a copy of $X_d$, but which is not confined to this copy. Then $C$ uses at least $d-1$ edges in the copy of $X_d$, and at least 2 additional edges (since $G$ does not have double edges, and cannot have an edge joining non-adjacent vertices of a face because it is $(d+1)$-adapted). 
\end{proof}

Let $X_d'$ be the map obtained from $X_d$ by ``taking one of its inner faces of degree $d$ to become the new outer face'' (this is better understood when picturing $X_d$ on a sphere with a face arbitrarily chosen to be the ``outer face''). The map $X_d'$ has an inner face of degree $d+1$, and we define $\bGb$ as the $d$-map obtained by gluing $\Gb$ in the face of degree $d+1$ of $X_d'$.
Using Claim~\ref{claim:still-adapted}, it it easy to see that $\bGb$ is $d$-adapted. Hence, by the induction hypothesis, $\bGb$ admits a $d$-GS angular orientation $\bAb$. By definition, the weight in $\bAb$ of a star edge $e$ is $d-\deg(f)$, where $f$ is the face of $\bGb$ containing $e$, and in particular this weight is 0 for star edges inside faces of $\bGb$ of degree $d$. Let $\hGb$ be the map obtained from $\bGb^+$ by 
\begin{compactitem}
\item deleting the star vertices and star edges (of weight 0) in the faces of $\Gb$ of degree $d$, and
\item deleting the vertices and edges $X_d'$ (except the outer vertices and edges of $G$) and the star vertices and star edges in the faces of $X_d'$.
\end{compactitem}
The map $\hGb$ is represented in Figure~\ref{fig:existence-proof-maps}. We mention that for the special case $d=3$ treated in Section~\ref{sec:special-case-proof-exist}, we had $\hGb=\Gb$ because $\Gb$ only had faces of degree $d$.

\fig{width=\linewidth}{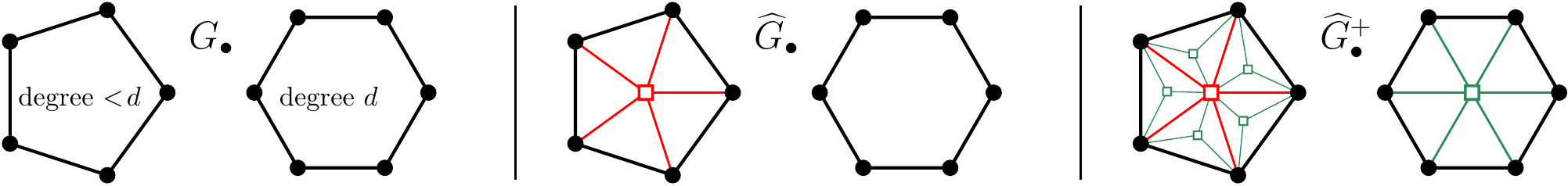}{The maps $\Gb$, $\hGb$ and $\hGb^+$. Left: a face of $\Gb$ of degree less than $d$, and a face of degree $d$. Middle: the corresponding faces in $\hGb$. Right: the corresponding faces in $\hGb^+$.}

Let $\cBb$ be the restriction of $\bAb$ to the map $\hGb$. It is not hard to check the following claim.

\begin{claim}\label{claim:positive-edges} Every inner edge of $\hGb$ has a positive weight in $\cBb$.
\end{claim}

\begin{proof} The claim is straightforward for star edges. For an original edge $e$ the weight in $\cBb$ is $\deg(f)+\deg(f')-d-2$, where $f$ and $f'$ are the faces of $\Gb$ incident to $e$. The cycle of $\Gb$ surrounding $f\cup f'$ has length $\deg(f)+\deg(f')-2$, hence this quantity is at least $d+1$ by Claim~\ref{claim:still-adapted}.
\end{proof}

Recall that a \emph{positive path} of a weighted orientation is a directed path of the underlying graph such that every arc on that path is positively weighted.
We say that $\cBb$ is \emph{co-accessible} if for every inner vertex~$v$ of $\hGb$ there exists a positive path of $\cBb$ starting at an outer vertex and ending at~$v$.

\begin{claim}\label{claim:co-accessible}
The weighted orientation $\cBb$ of $\hGb$ is \emph{co-accessible}.
\end{claim}

\begin{proof}
Let $U$ be the set of vertices of $\hGb$ made of the outer vertices and of the inner vertices that can be reached by a positive path of $\cBb$ starting from an outer vertex. Let $M_U$ be the submap of $\hGb$ induced by $U$. We think of $M_U$ as embedded simultaneously with $\hGb$. A \emph{hole} is an inner face of $M_U$ which is not empty, that is, which contains at least one vertex of $\hGb$ in its interior. Proving co-accessibility of  $\cBb$  thus amounts to showing that there is no hole. 

Assume for contradiction that there exists a hole. Then its contour $C$ is a non-necessarily simple cycle enclosing a simply connected region. Moreover, since $M_U$ is an induced submap, the cycle $C$ has no inside chord (that is, edge of $\hGb$ in the interior of $C$ with both endpoints on $C$); and since no vertex inside $C$ can be reached by a positive path, the inward degree of $C$ in $\cBb$ has to be $0$. 
Another simple observation is that there is at least one original vertex inside $C$; otherwise any star vertex $s$ inside $C$ would be only adjacent to original vertices on $C$, and due to the weight-conditions, at most one of these vertices has its arc to $s$ of weight $0$, which contradicts the fact that the inward weight of $C$ is~$0$. 

We now claim that $C$ is a pseudo-compatible cycle of $\Gb^+$ (as defined in Section~\ref{sec:proof-lattice}). First, note that $C$ cannot visit star vertices of degree $d$ since $M_U$ is a submap of $\hGb$. Second, we need to check that, in $\hGb$, there is no triangular face inside $C$ with two original vertices $u,v$ and one star-vertex $s$, with the (original) edge $\{u,v\}$ on $C$. There are two cases to check. (i) If $s$ is in the interior of $C$, then the arcs $(u,s)$ and $(v,s)$ are both taken into account in the inward degree of $C$; but they can not both have weight $0$ (otherwise the weight of $s$ would be too large), which gives a contradiction with the fact that $C$ has inward degree $0$. (ii) If $s$ is on $C$,  since $C$ has no inside chord we conclude that $C$ must be the contour of the triangular face, contradicting the fact that it contains at least one vertex in its interior. Thus, $C$ is pseudo-compatible. 
By Lemma~\ref{lem:compa}, its enclosing length has to be $d$. Hence, one can extract a cycle $C'$ of original edges of length at most $d$ and whose interior contains the interior of $C$ (see Figure~\ref{fig:pseudo-compatible}(a)). Recalling that there is at least one original vertex inside $C$, we conclude that $C'$ forms a cycle of $\Gb$ of length at most $d$ with at least one vertex inside, which gives a contradiction with Claim~\ref{claim:still-adapted}.     
\end{proof}

Since $\cBb$ is co-accessible, there exists a spanning tree $T$ of $\hGb$ satisfying the following conditions:
\begin{compactitem}
\item $T$ contains all the outer edges of $\hGb$ except $\{v_1,v_2\}$, 
\item taking $v_1$ as the root of $T$, for every inner edge $e$ of $\hGb$ belonging to the tree $T$, the arc of $e$ oriented from parent to child in $T$ is positively weighted in $\cBb$.
\end{compactitem}

We will now use $\cBb$ and $T$ to define a weighted orientation ${\hAb}$ of the angular map $\hGb^+$.
Recall from Claim~\ref{claim:inside-face} that there is a bijection $\varphi$ between the edges of $\hGb\setminus T$ and the inner faces of $\hGb$. For each inner face $f$ of $\hGb$, we consider the corresponding edge $e=\varphi^{-1}(f)$. By Claim~\ref{claim:positive-edges}, the weight of $e$ in $\cBb$ is positive and we pick an arc $a_f$ of $e$ with positive weight. Let $u_f$ be the terminal vertex of $a_f$. 
With this notation, we define the weighted orientation ${\hAb}$ of the inner edges of $\hGb^+$ as follows.
\begin{compactitem}
\item For an inner face $f$ of $\hGb$, the star edges of $\Gb^+$ incident to the star vertex $v_f$ have weight~1 and are oriented as follows: the edge $\{v_f,u_f\}$ is oriented toward $u_f$ (i.e. has weight 1 in this direction), while the other edges are oriented toward $v_f$.
\item For an original inner arc $a$ of $\hGb^+$, the weight $\hom_a$ of $a$ in ${\hAb}$ is equal to $\om_a-\eps_a$, where $\om_a$ is the weight of $a$ in $\cBb$ and $\eps_a=1$ if either $a=a_f$ for some inner face $f$ of $\Gb$, or if $a$ is in $T$ and oriented from parent to child (for $T$ rooted at the outer vertex $v_1$), and $\eps_a=0$ otherwise.
\end{compactitem}

Note that all the arc weights in ${\hAb}$ are non-negative. For an edge or vertex $x$ of $\hGb$, let $\om(x)$ and $\hom(x)$ be their weight in $\cBb$ and ${\hAb}$ respectively. It follows directly from the definition that for every inner edge $e$ of $\hGb$, one has $\hom(e)=\om(e)-1.$ 
We also claim that for every inner vertex $v$ of $\hGb$, one has $\hom(v)=\om(v)+1$. To show this, let us consider the set $\toward(v)$ of arcs of $\hGb$ having terminal vertex $v$, and the number $n_v$ of arcs in $\toward(v)$ of the form $a_f$ (equivalently, $n_v$ is the number of inner faces $f$ of $\hGb$ such that $u_f=v$). With this notation, the contribution of star edges to $\hom(v)$ is $\deg(v)-n_v$, while the contribution of original edges to $\hom(v)$ is $\om(v)-\eps_v$, where $\eps_v=\sum_{a\in \toward(v)}\eps_a=\deg(v)-1-n_v$.

Next, we use the orientation ${\hAb}$ of $\hGb^+$ to define an orientation ${\cAb}$ of $\Gb^+$. 
Let us call \emph{small triangles} of $\Gb^+$ the inner faces (of degree 3) of the angular map $\Gb^+$ incident to star vertices of degree less than~$d$ (see Figure~\ref{fig:existence-proof-maps}). Observe that $\hGb^+$ is obtained from $\Gb^+$ by placing a vertex in each small triangle of $\Gb^+$ and connecting it to the three incident vertices of $\Gb^+$. We call \emph{small} the vertices and edges of $\hGb^+$ which are in the small triangles of $\Gb^+$. 

By contracting all the small edges of $\hGb^+$ incident to star vertices of $\Gb^+$ one obtains a plane map $\tGb^+$ closely related to $\Gb^+$: precisely, $\tGb^+$ is the same as $\Gb^+$ except that for every star edge $e$ incident to a star vertex of degree less than $d$ one has 2 additional small edges $e_\ell,e_r$ with the same endpoint as $e$ ($e_\ell$ and $e_r$ are on the left and right of $e$ respectively). This is represented in Figure~\ref{fig:existence-proof-construct}(a).
We define $\cAb'$ as the weighted orientation of $\Gb^+$ obtained from ${\hAb}$ by contracting the small edges of $\hGb^+$ incident to star vertices of $\Gb^+$: precisely, the weights in $\cAb'$ are the same as in $\hAb$ except that for every star edge $e$ incident to a star vertex of degree less than $d$, the weights of the arcs in $e$ are taken as the sum of the weights in $\hAb$ of the corresponding arcs in $e$, $e_\ell$ and $e_r$.

\fig{width=.9\linewidth}{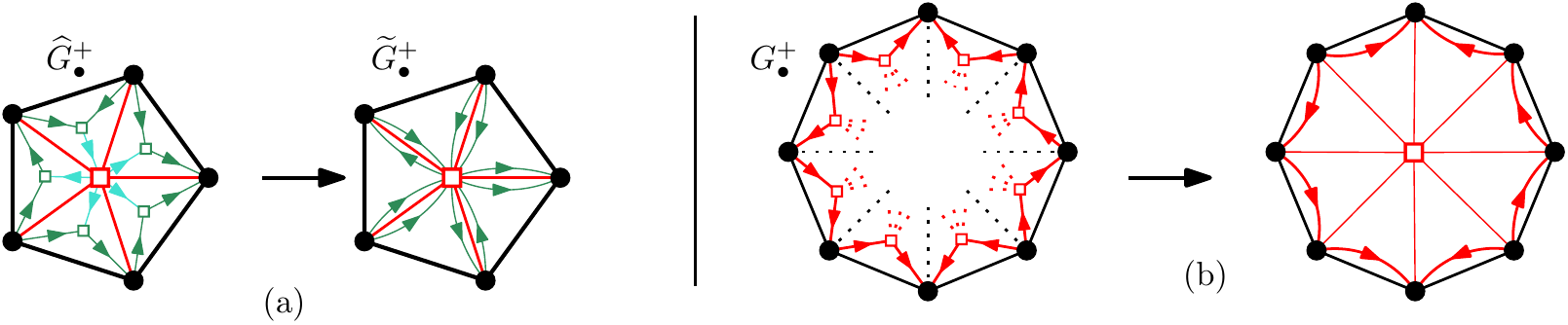}{(a) Contracting the small edges incident to a star vertex of $\Gb^+$. (b) Getting from the orientation $\cAb$ of $\Gb^+$ to the orientation $\cA$ of $G^+$ (situation inside a copy of $X_d$).}

\begin{claim}\label{claim:final-weights}
The weighted orientation $\cAb'$ of $\Gb^+$ satisfies the following properties:
\begin{compactitem}
\item[(i)] every original inner vertex has outgoing weight $d+1$,
\item[(ii)] every original inner edge $e$ has weight $\deg(f)+\deg(f')-d-3$, where $f$ and $f'$ are the faces of $\Gb$ incident to $e$,
\item[(iii)] every star vertex $v_f$ has outgoing weight $d+1-\deg(f)$,
\item[(iv)] every star edge $e$ incident to the star vertex $v_f$ has weight $d+1-\deg(f)$.
\end{compactitem}
\end{claim}

\begin{proof} Recall that the weight of a vertex or edge $x$ in ${\cBb}$ and ${\hAb}$ is denoted by $\om(x)$ and $\hom(x)$ respectively.\\
(i) The weight of an original vertex $v$ in $\cAb'$ is $\hom(v)=\om(v)+1=d+1$.\\
(ii) The weight of an original inner edge $e$ in $\cAb'$ is $\hom(e)=\om(e)-1=\deg(f)+\deg(f')-d-3$.\\
(iii) If the inner face $f$ of $\Gb$ has degree $d$, then the weight of $v_f$ in $\cAb'$ is $1=d+1-\deg(f)$. If the inner face $f$ of $\Gb$ has degree less than $d$, then it is easy to see that the weight of $v_f$ in $\cAb'$ is equal to $\hom(v_f)$ (upon recalling that all the star vertices and star edges of $\hGb^+$ have weight 1 in  ${\hAb}$), hence is equal to $\om(v_f)+1=d+1-\deg(f)$.\\
(iv) If the inner face $f$ of $\Gb$ has degree $d$, then the weight of the star edge $e$ in $\cAb'$ is $1=d+1-\deg(f)$. If the inner face $f$ of $\Gb$ has degree less than $d$, then it is easy to see that the weight of $e$ in $\cAb'$ is $\hom(e)+2=\om(e)+1=d+1-\deg(f)$.
\end{proof}

By Claim~\ref{claim:final-weights}, the orientation $\cAb'$ satisfies Conditions (A1) and (A2) of $(d+1)$-GS angular orientations of $\Gb$. 

In order to show that $\cAb'$ can be modified to satisfy Condition~(A0), 
  we prove the following lemma. 

\begin{lemma}\label{lem:conditions_A_prime}
In the orientation $\cA_\bullet'$, we call \emph{border star vertices} those in faces of $\Gb$ incident to at least one outer edge. 
Then any arc starting at an outer vertex has weight $0$ if it does not end at a border star vertex,  
and any arc starting at a border star vertex has weight $0$ if it ends at an inner (original) vertex. 
Finally, for every border star vertex $v_f$, the outer neighbours  of $v_f$ are consecutive along the outer face contour.   
\end{lemma}
\begin{proof}
Let $C$ be the outer cycle of $G$.  Lemma~\ref{lem:ingoing-weight-angular} 
gives
 $\inweight(C)=\sum_{i=1}^{d+1}(d+1-\deg(f_i))$, where $f_i$ is the inner face incident to the outer edge $(v_{i-1},v_i)$. 
Let $S$ be the total weight  of arcs from border star vertices to inner vertices. 
Since Conditions (A1) and (A2) hold, the contribution to $\inweight(C)$ of the arcs from outer vertices to border star vertices is at least  $S+\sum_{i=1}^{d+1}(d+1-\deg(f_i))$, with equality if and only if every border star vertex $v_f$ has its outer neighbours consecutive along the outer face contour (which means that the number of outer vertices incident to $f$ is one more than the number of outer edges incident to $f$). 
Hence, the total weight of the arcs that start at an outer vertex and do not end at a border star vertex is at most $-S$. Hence, $S=0$, giving the first two statement, and moreover the above inequality has to be tight, giving the third statement.
\end{proof}
To turn $\cA_\bullet'$ into a $(d+1)$-GS angular orientation $\cA_{\bullet}$, we modify some arc weights in the neighborhood of the outer face, keeping the weights of edges and of inner vertices unchanged, as follows. 
For each edge $(v_i,v_{f_i})$, with $f_i$ the inner face of $G_\bullet$ incident to $(v_{i-1},v_i)$,   
we put all the weight on the arc out of $v_i$; and if $f_{i+1}\neq f_i$, we put all the weight of the edge $(v_i,v_{f_{i+1}})$ on the arc ending at $v_i$.  
With these modifications, and given Lemma~\ref{lem:conditions_A_prime}, one easily checks that all conditions of (d+1)-GS angular orientations 
 are now satisfied.

Finally, we construct from $\cAb$ a $(d+1)$-GS angular orientation $\cA$ of $G$. Recall that $\Gb$ is obtained from $G$ by adding a copy of $X_d$ in each inner face of degree $d+1$. For each copy of $X_d$, the weighted orientation $\cAb$ restricted to the interior of $X_d$ satisfies Conditions (A1) and (A2) of $(d+1)$-GS angular orientations of $X_d$. Thus, by Lemma~\ref{lem:ingoing-weight-angular}, the total weight $W$ of the arcs strictly inside $X_d$ having initial vertex an outer vertex of $X_d$ satisfies $W=d+1$. 
Let us call \emph{border faces} of $X_d$ the $d+1$ inner faces of $X_d$ that share an edge with its outer face. The border star vertices of $X_d$ are the star vertices corresponding to the border faces, and the \emph{border star edges} of $X_d$ are the star edges connecting a border star vertex to an outer vertex (there are two such edges incident to each border star vertex). 
By definition, border star vertices and border star edges have weight 1 in $\cAb$. Hence, for each of the $d+1$ border star vertices, the contribution to $W$ of the two incident border star edges is in $\{1,2\}$. Since $W=d+1$ we conclude that this contribution is exactly $1$ for each border star vertex, and that $W$ has no other contribution. 


Consider the following operations on $\Gb^+$ and $\cAb$, which are represented in Figure~\ref{fig:existence-proof-construct}(b):
\begin{compactenum}
\item Delete all the non-border vertices and edges from every copy of $X_d$. 
\item Contract one incident border star edge for each border star edge.
\item Add a star vertex and star edges of weight 0 in each face of degree $d+1$ previously containing a copy of $X_d$.
\end{compactenum}
By the above, these operations do not affect the weight of the vertices of $G$ (hence they all have weight $d+1$). The remaining border edges are parallel to some original edges of $G$. For each original inner arc $a$ of $G$, we add to the weight of $a$ in $\cAb$ the weight of the parallel border arcs with the same initial and terminal vertices as $a$. Finally we delete all border edges and obtain the map $G^+$ endowed with a weighted orientation $\cA$. It is easy to check that $\cA$ is a $(d+1)$-angular orientation $\cA$ of $G$. This completes the proof of the existence of a $(d+1)$-GS angular orientation by induction.

The above process for getting a $(d+1)$-GS angular orientation of $G$ from a $d$-GS angular orientation of $\bGb$ is constructive and can easily be translated into an algorithm working in linear time in the number of vertices of $\bGb$ (for any fixed $d$). Since the number of vertices of $\bGb$ is linear in the number of vertices of $G$, and since the induction hypothesis ensures that a $d$-GS angular orientation of $\bGb$ can be computed in linear time, we conclude that a $(d+1)$-GS angular orientation of $G$ can be constructed in linear time.

\subsection{Case of maps with multiple edges}
In this subsection we complete the proof of Theorem~\ref{thm:main}. The existence of a $d$-GS angular orientation has been established for $d$-adapted maps without multiple edges. Let us now consider a $d$-adapted map $G$ with some multiple edges.

Since $G$ is $d$-adapted, any cycle made by 2 parallel edges must be facial. By contracting the faces of degree 2 of $G$ into a single edge, one obtains a $d$-adapted map $G'$ without multiple edges. Let us call \emph{digons} the faces of $G$ of degree 2, and \emph{digon edges} the corresponding edges of $G'$. As established above, $G'$ admits a $d$-GS angular orientation $\cA'$. Since $G$ is $d$-adapted, any digon edge of $G'$ is incident to two faces of degree $d$. By definition of angular orientations, the digon edges of $G'$ have weight $d-2$. We construct from $\cA'$ an angular orientation $\cA$ of $G$ as indicated in Figure~\ref{fig:digon-opening}.
It is easy to see that the resulting orientation is a $d$-GS angular orientation. Moreover the above construction of $\cA$ is performed in linear time, which completes the proof.

\fig{width=.9\linewidth}{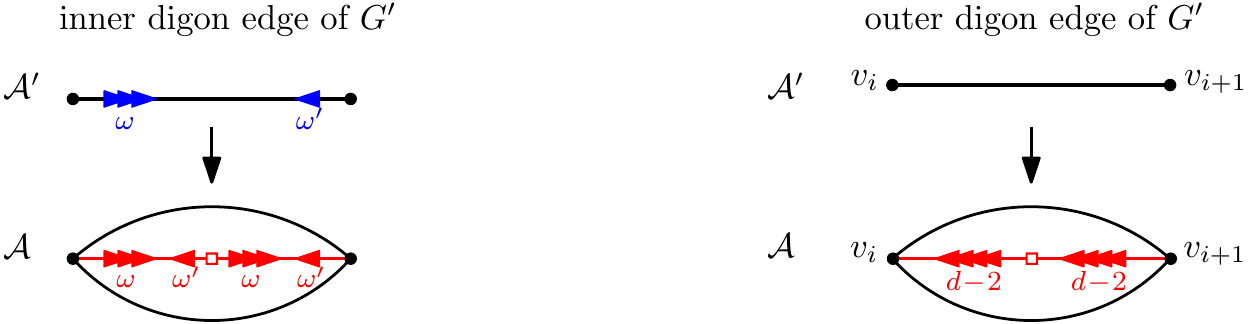}{The construction of the angular orientation $\cA$ starting from $\cA'$. The weights of the arcs in $\cA$ and $\cA'$ are the same except for digon edges of $G'$ which are transformed as indicated in this figure. The transformation for an inner digon edge is indicated on the left (the weights $\om$ and $\om'$ sum to $d-2$), while the transformation of an outer digon edge is indicated on the right. The two edges of $G$ incident to the digon have weight 0.}


%% file: quasi_structures.tex

In this section, we give an extension of grand-Schnyder structures to so-called quasi $d$-adapted maps. This extension of the domain of definition of grand-Schnyder structures is partly motivated by the case of 5-connected triangulations, for which quasi-Schnyder structures have incarnations which are very reminiscent of Schnyder's original definition (except that there are 5 spanning trees instead of 3) and can be used to define a graph drawing algorithm~\cite{OB-EF-SL:5QS-drawing}.

For $d\geq 3$, a $d$-map $G$ is called \emph{quasi $d$-adapted} if every non-facial cycle has length at least $d-1$, and those enclosing at least one vertex have length at least $d$. In other words, the only non-facial cycles of length smaller than $d$ have length $d-1$ and have no vertex in their interior. Note that if a quasi $d$-adapted $d$-map $G$ has no inner face of degree less than $3$, then the cycles of length $d-1$ have at most one edge (and no vertex) in their interior. To avoid overly complicated formulations, we will only give the incarnations as angular orientations and as corner labelings, and will work with $d$-maps having only inner faces of degree at least $3$.


\subsection{Quasi-Schnyder angular orientations}\label{subsec:WGS-angular}

Let $d\geq 3$, and let $G$ be a $d$-map. 
For an inner edge $e$ of $G$, we let $\delta(e):=\mathrm{deg}(f)+\mathrm{deg}(f')-d-2$, where $f,f'$ are the faces incident to $e$.  
Note that if $G$ is $d$-adapted (resp. quasi $d$-adapted), then for each inner edge $e$ of $G$, we have $\delta(e)\geq 0$ (resp. $\delta(e)\geq -1$). 
An inner edge $e$ of $G$ is called \emph{special} if $\delta(e)=-1$. 
We let $\Gcr$ be the map obtained from $G^+$ by placing a vertex $v_e$, called an \emph{edge-vertex}, in the middle of every special edge; and adding one edge between $v_e$ and a star vertex $v_f$ on each side of $e$ (where $f$ is the face of $G$ incident to $e$ on that side of $e$), so that every edge-vertex has degree $4$ in $\Gcr$.
The edges of $\Gcr$ incident to edge-vertices are called \emph{extra edges}. The map $\Gcr$ is represented in Figure~\ref{fig:5_quasi_other}.



\begin{definition}\label{def:quasiAngular}
Let $d\geq 3$, and let $G$ be a $d$-map. 
A \emph{quasi-Schnyder angular orientation}, or \emph{$d$-QS angular orientation}, of $G$ is a weighted orientation of $\Gcr$ satisfying the following conditions (shown in the top-part of Figure~\ref{fig:def-incarnations-QS}).
\begin{itemize}
\item[(A0')] The weight of every outer arc is 0. Any inner arc $a$ of $\Gcr$ whose initial vertex is an outer vertex $v_i$ has weight $0$, unless $a$ is the arc (on a star-edge) following the outer edge $(v_i,v_{i-1})$ around~$v_i$ (for this arc there is no condition), or $a$ is on an extra edge and has 
the inner face of $G$ incident to $(v_i,v_{i-1})$ on its right.
\item[(A1')] The outgoing weight of a star vertex $v_f$ is $d-\deg(f)$. For every star edge $e$ incident to $v_f$ and an original vertex $v$ the weight of $e$ is $\si(\eps)=d-\mathrm{deg}(f)-s$, where $s\in\{0,1,2\}$ is the number of special edges among the edges $e',e''$ preceding and following $e$ around~$v$.
\item[(A2')] The outgoing weight of every inner original vertex is $d$. The weight of every original non-special inner edge $e$ is $\delta(e)$. 
\item[(A3')] The outgoing weight of every edge-vertex is $1$. The weight of every extra edge is $1$. 
\end{itemize}
\end{definition}

A 5-QS angular orientation is shown in Figure~\ref{fig:5_quasi_other} (bottom-left part).

\fig{width=\linewidth}{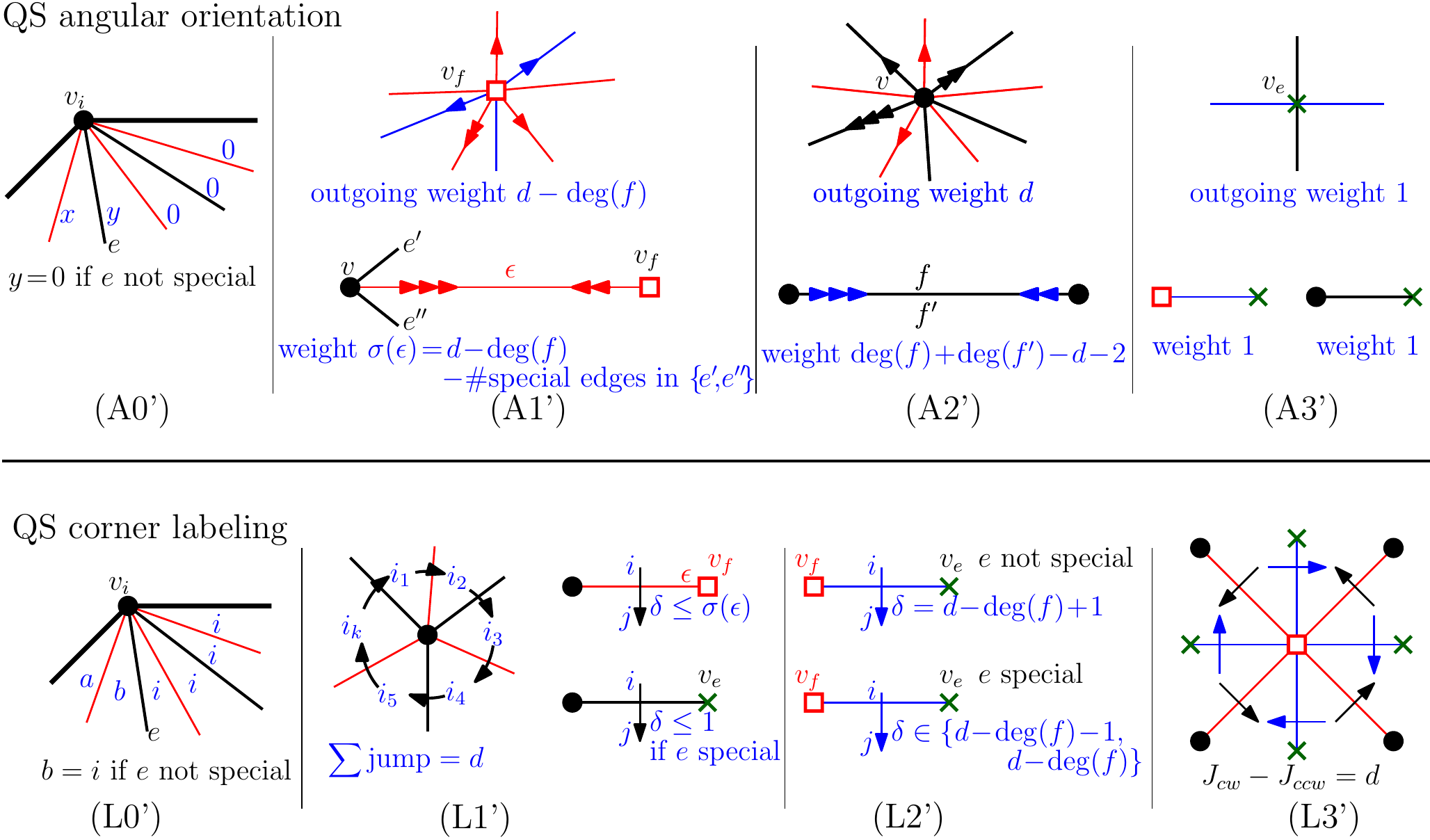}{Definition of quasi-Schnyder structures.}

\begin{remark}
Note that if $G$ has no special edge, the definition of $d$-QS angular orientation exactly matches Definition~\ref{def:angular} of $d$-GS angular orientations.
\end{remark}

We now prove some rigidity properties in the vicinity of the outer face, similar to those in Remark~\ref{rk:frozen}.

\begin{lemma}\label{lem:rigid_qs}
Let $G$ be a $d$-map endowed with a $d$-QS angular orientation $\cA$. For $i\in[d]$ let $f_i$ be the inner face of $G$ incident to the outer edge $(v_{i-1},v_i)$, and 
let $s_i$ be the star vertex in $f_i$. Then the arc from $s_i$ to $v_i$ has weight $0$. Moreover, if there is a special inner edge $e$ incident to $v_i$ and $f_i$, 
  then the unique arc of weight 1 starting at $v_e$ ends at a star vertex of $G$. Furthermore, if $G$ has an outer vertex incident to at least two inner edges in $G$, then the arc of weight 1 starting at $v_e$ ends at $s_i$.
\end{lemma}
\begin{proof}
We define a \emph{tricky} edge of $G$ as a special edge $e$ incident to an outer vertex $v_i$, such that in $\cA$ the outgoing edge of $v_e$ is the next edge after $(v_e,v_i)$ 
in clockwise order around $v_e$, as represented in Figure~\ref{fig:rigid}(a).

We first prove that the lemma holds if there is no tricky edge. Consider any portion $v_{i'},\ldots,v_i$ of consecutive outer vertices that are incident to $f_i$, such that $v_{i'}$ 
and $v_i$ are incident to at least one inner edge in $G$, and let $e$ (resp. $e'$) be the inner edge of $G$ incident to $v_i$ (resp. $v_{i'}$) and $f_i$.  
If $e'$ is non-special, then by Conditions (A0') and (A1') the arc from $s_i$ to $v_{i'}$ has weight $d-\deg(f_i)$. Since $s_i$ has outgoing weight $d-\deg(f_i)$ by Condition~(A1'), all the other arcs from $s_i$ have weight $0$, in particular the arcs from $s_i$ to all vertices in $v_{i'+1},\ldots,v_i$ have weight $0$. Moreover, if $e$ is special then the arc from $s_i$ to $v_e$ has weight $0$, hence the unique
outgoing edge of $v_e$ is the edge ending at $s_i$, and thus the arc from $v_e$ to $v_i$ has weight $0$.
If $e'$ is special, then by Conditions (A0') and (A1') the arc from $s_i$ to $v_{i'}$ has weight $d-\deg(f_i)-1$, and moreover the arc from $s_i$ to $v_{e'}$ has weight $1$,
since $e'$ is not tricky. Hence the whole outgoing weight of $s_i$ is taken by these two arcs, so that all other arcs from $s_i$ have weight $0$. In particular the arcs from $s_i$ to all   vertices in $v_{i'+1},\ldots,v_i$ have weight $0$, and if $e$ is special then the arc from $s_i$ to $v_e$ has weight $0$, so that the unique outgoing edge of $v_e$ leads to $s_i$, and thus
the arc from $v_e$ to $v_i$ has weight $0$.

\fig{width=0.8\linewidth}{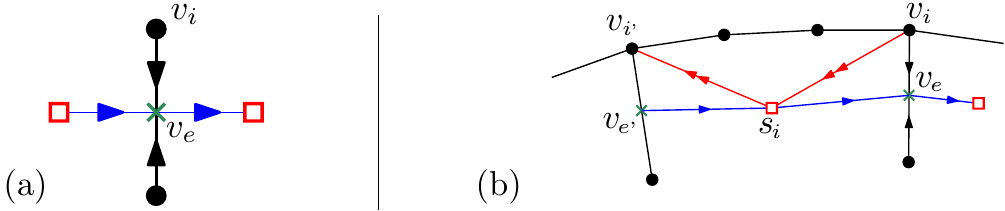}{(a) A tricky edge $e$, as defined in the proof of Lemma~\ref{lem:rigid_qs}.  (b) The situation when there exists a tricky edge $e$.}

Assume now that there is a tricky edge $e$, with $v_i$ the outer vertex adjacent to $v_e$.   
Since $v_e$ has outdegree~$1$, the edge $(v_i,v_e)$ is directed toward $v_e$. Hence, by Condition (A0'), $e$ is incident to $f_i$ in $G$.
Let $s_i$ be the star vertex in $f_i$, let $v_{i'}$ be the next outer vertex after $v_i$ in counterclockwise order around the
outer face contour that is incident to $f_i$ and to an inner edge, and let 
$e'$ be the inner edge of $G$ incident to $v_{i'}$ and to $f_i$. 
Conditions (A0') and (A1') imply that the arc from 
$s_i$ to $v_{i'}$ has weight $d-\deg(f_i)-\delta_{e'\ \mathrm{special}}$, and Condition (A3') implies that the arc from $s_i$ to $v_e$ has weight $1$ (as $v_e$ has its outgoing edge on the other side of $e$), as represented in Figure~\ref{fig:rigid}(b). Since $s_i$
has outgoing weight $d-\deg(f_i)$, we conclude that $e'$ has to be special, and that the arc from $v_{e'}$ to $s_i$ has weight $1$. Hence $e'$ is also tricky. 
Condition~(A3') implies that the edge $(v_{i'},v_{e'})$ is directed toward $v_{e'}$, and then Condition~(A0') implies that $e'$ 
has to be the unique inner edge of $G$ incident to $v_{i'}$. Continuing iteratively in counterclockwise order around the outer face contour, until reaching $v_i$ again, 
we conclude that if there exists a tricky edge, then every outer vertex is either incident to no inner edge of $G$, or to a unique inner edge of $G$ that is tricky. 
In this situation it is then easy to check that the lemma holds (by a similar argument as the one used when there is no tricky edge); note that the very last point (which fails) does
not need to be checked since all outer vertices are incident to $0$ or $1$ inner edge in $G$.
\end{proof}

\subsection{Quasi-Schnyder labelings}\label{subsec:WGS-labeling}
For $G$ a $d$-map, the \emph{derived map} $G'$ is the map obtained from $G^+$ by placing a vertex $v_e$, called an \emph{edge-vertex}, in the middle of every original edge $e$, and for each side of $e$ 
incident to an inner face $f$ of $G$, we add an edge (on that side of $e$) from $v_e$ to the star vertex $v_f$ (compared to $\Gcr$, we place an edge-vertex on every edge of $G$, not just on special edges). 
The derived map $G'$ is represented in Figure~\ref{fig:5_quasi_other}.

\begin{definition}\label{def:WGS-labeling}
Let $d\geq 3$, and let $G$ be a $d$-map with no inner face of degree less than~$3$. 
A \emph{quasi-Schnyder labeling}, or \emph{$d$-QS labeling}, of $G$ is an assignment to each inner face of $G'$ of a label in $[d]$ satisfying the following conditions (shown in the bottom-part of Figure~\ref{fig:def-incarnations-QS}).
\begin{itemize}
\item[(L0')] For $i\in [d]$, all the inner faces of $G'$ incident to $v_i$ have label $i$, except possibly for the last one (incident to $(v_{i-1},v_i)$),
and for the next to last one if that face is incident to (the extra edge on) a special edge of $G$.
\item[(L1')] For every inner vertex $v$ of $G$, the sum of clockwise label jumps around $v$ is $d$. The clockwise label jump across any star edge $\eps$ incident to $v$ is at most $\si(\eps)$, with $\si(\eps)$ given in (A1') of Definition~\ref{def:quasiAngular}; and the clockwise label jump across an edge $(v,v_e)$ (with $v_e$ an edge-vertex) is in $\{0,1\}$ if $e$ is special (there is no constraint if $e$ is not special).
\item[(L2')] For an edge $\eps$ between a star vertex $v_f$ and an edge-vertex $v_e$, the clockwise jump at $v_f$ across $\eps$ 
is $d-\mathrm{deg}(f)+1$ if $e$ is non-special, and is in $\{d-\mathrm{deg}(f)-1,d-\mathrm{deg}(f)\}$ if $e$ is special. 
\item[(L3')] For each star vertex $v_f$, let $\Sigmacw$ be the sum of clockwise jumps across edges from $v_f$ to edge-vertices, and let $\Sigmaccw$ be the sum of 
counterclockwise jumps across edges from $v_f$ to original vertices. Then $\Sigmacw-\Sigmaccw=d$. 
\end{itemize}
\end{definition}


A $5$-QS labeling is shown in Figure~\ref{fig:5_quasi_other} (bottom-right part).

\fig{width=.9\linewidth}{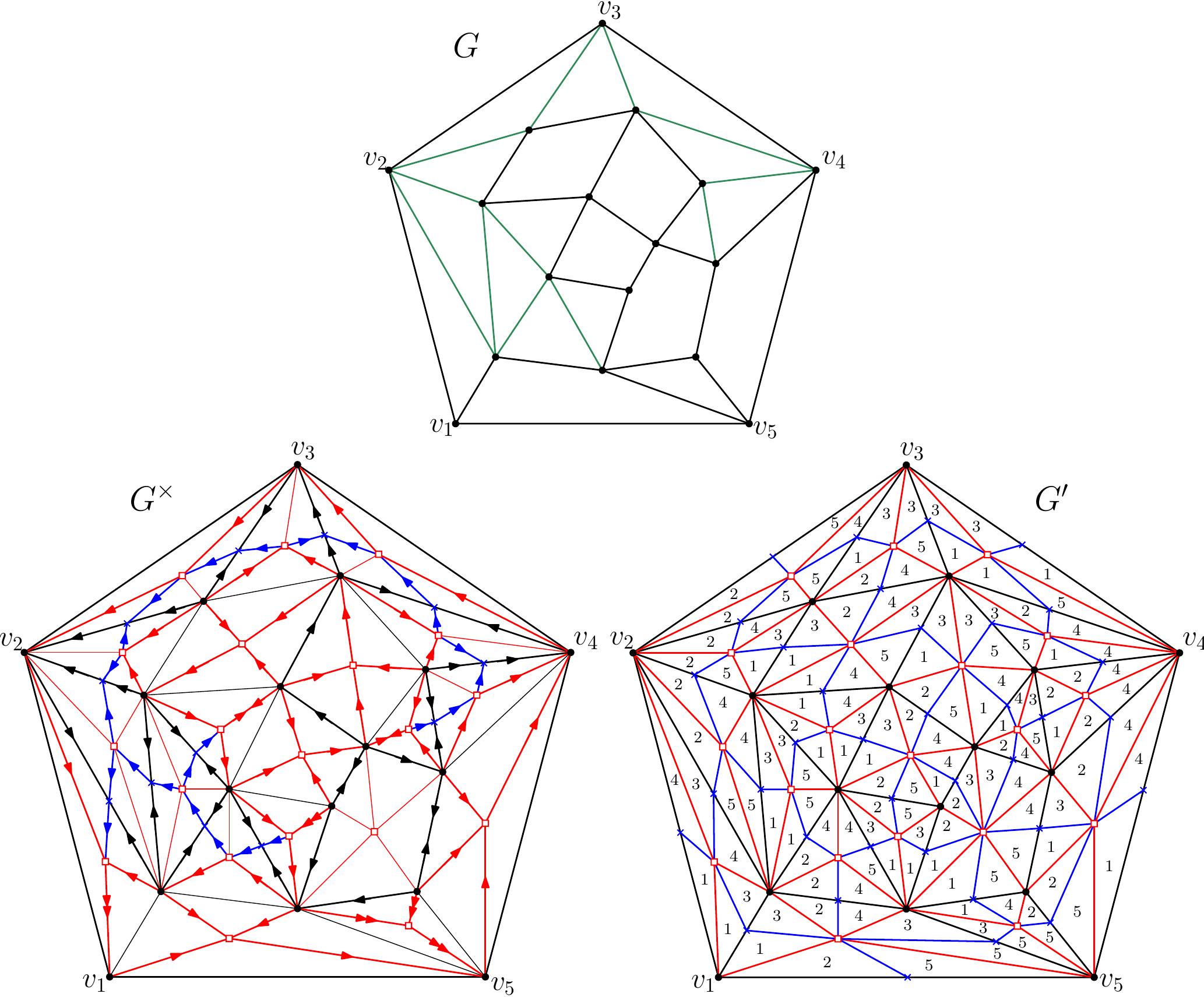}{Top: a quasi $5$-adapted map $G$ (special edges are shown in green). Bottom: on the left a $5$-QS angular orientation on $G^{\times}$; on the right the corresponding $5$-QS labeling (labeling of the inner faces of~$G'$).}

\begin{remark}\label{rk:ccw-quasi}
Similarly as for Lemma~\ref{lem:ccw-jumps-edges}, one can prove (using the Euler relation) that in a $d$-QS labeling, the sum of the 4 label jumps in counterclockwise order around any inner edge-vertex is equal to $d$. 
 \end{remark}
 
In order to explain the relation between $d$-QS labeling and $d$-GS labeling we introduce a bit of vocabulary.
We call \emph{right triangle} (resp. \emph{left triangle}) an inner face of $G'$ on the right (resp. left) of an arc of $G'$ from a star vertex to an original vertex. Note that every inner edge of $G'$ separates a right triangle and a left triangle.

\begin{remark}\label{rk:jumps-differences}
Let $G$ be a $d$-map.
Let $t_1,t_2,t_3$ be three consecutive faces of $G'$ in clockwise order around a star vertex $v_f$, with $t_1$ a right triangle (so that $t_2$ is a left triangle and $t_3$ is a right triangle). Let $i_1,i_2,i_3$ be the labels of $t_1,t_2,t_3$ respectively in a $d$-QS of $G$; see Figure~\ref{fig:labeling_wtos}(a). Then, by (L1') and (L2') the label jumps satisfy:
$$\jp(i_1,i_2)\geq \jp(i_3,i_2),$$
and this inequality is strict if $e$ is not special.
In particular, in the notation of Definition~\ref{def:WGS-labeling}, the difference $\Sigmacw-\Sigmaccw$ in Condition (L3') is the sum of the (non-negative) label jumps between right triangles in clockwise order around $v_f$.
\end{remark}

 \fig{width=\linewidth}{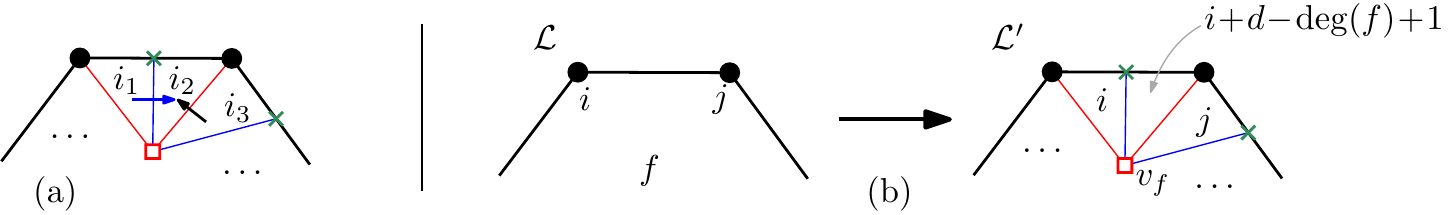}{(a) Situation described in Remark \ref{rk:jumps-differences} (the inequality $\jp(i_1,i_2)\geq \jp(i_3,i_2)$ always holds). (b) Turning a $d$-GS labeling into a $d$-QS labeling.}

\begin{remark} 
Let $G$ be a $d$-adapted map. A $d$-GS corner labeling $\cL$ can be turned into a $d$-QS labeling $\cL'$ as follows (see Figure~\ref{fig:labeling_wtos}(b)). 
Every inner corner $c$ of $G$ correspond to star edge $\eps_c$ of $G'$, which is incident to a right triangle $t_c$; and we set the label of $t_c$ in $\cL'$ to be the label of $c$ in $\cL$. The label of the left triangles in $\cL'$ are then determined by Condition (L2'): for every extra edge $a=(v_e,v_f)$ incident to a star vertex $v_f$ the label of the left triangle incident to $e$ is equal to the label of the right triangle incident to $e$ plus $d-\deg(f)+1$.
One can easily check that, the constraints of Definition~\ref{def:GS-labeling} for $d$-GS labelings imply that $\cL'$ satisfy the constraints of Definition~\ref{def:WGS-labeling} for $d$-QS labelings. (For instance, Condition (L3) for $\cL$ implies that the sum of label jump around an original vertex $v$ are the same in $\cL'$ as in $\cL$ (hence equal to $d$); while Conditions (L1-L2) imply that the label jumps in $\cL$ from a corner to the next around a face $f$ are at most $d-\deg(f)-1$, ensuring that the label jumps across a star edge $\eps$ in $\cL'$ is at most $\si(\eps)$.)

Conversely if $\cL$ is a corner labeling of $G$ such that $\cL'$ is a $d$-QS labeling, then it is easy to check that $\cL$ is a $d$-GS labeling of $G$. (For instance, by Remark~\ref{rk:jumps-differences}, Condition (L3') for $\cL'$ implies that the sum of clockwise label jumps in $\cL$ around any inner face of $G$ is $d$, and Remark~\ref{rk:jumps-differences} also readily imply Condition (L2) and (L3) for $\cL$.) In conclusion, when $G$ is $d$-adapted, the above is a bijection between the $d$-GS labelings of $G$ and the $d$-QS labelings of $G$.
 \end{remark}



\subsection{Main result on quasi-Schnyder structures}

We state here the existence condition for $d$-QS structures, which extends Theorem~\ref{thm:main} (for the two relevant incarnations, and assuming no face has degree smaller than $3$).

\begin{thm}\label{thm:main-quasi}
Let $d\geq 3$ and let $G$ be a $d$-map with all inner faces of degree at least $3$. There exists a $d$-QS angular orientation (resp. labeling) for $G$ if and only if $G$ is quasi $d$-adapted.

Moreover for any fixed $d$, there is an algorithm which takes as input a quasi $d$-adapted map, and computes a $d$-QS angular orientation (resp. labeling) 
in linear time in the number of vertices.

Lastly, the set $\bA_G'$ of $d$-GS angular orientations of $G$ and the set $\bL_G'$ of $d$-QS labelings of $G$ are in bijection.
\end{thm}


We will prove the necessity of being quasi $d$-adapted in Section~\ref{sec:proof-necessity-quasi}, and then the existence result in Section~\ref{sec:proof-existence-quasi}, 
and will also describe there the algorithm for computing a $d$-QS structure.\\

We now describe the bijection between $\bL_G'$ and $\bA_G'$. Let $d\geq 3$, and let $G$ be a $d$-map, with all inner faces of degree at least $3$, endowed with a $d$-QS labeling $\cL'$. 
We produce a weighted orientation $\cA'=\Phi'(\cL')$ of $\Gcr$ as follows. 
\begin{itemize}
\item
For each arc $a$ of $\Gcr$ whose initial vertex is an original vertex $v$ of $G$, the weight assigned to $a$ is the label-jump from the face of $G'$ on the left of $a$ to the face on the right of $a$ around $v$. If $a$ is on a star edge $\eps$, then the weight of the opposite arc $-a$ is assigned so that the sum of the two weights is $\si(\eps)$, with $\si(\eps)$ given in Definition~\ref{def:quasiAngular}.  If $a$ is on an extra edge, then the weight of the opposite arc $-a$ is assigned so that the sum of the two weights is 1.
\item For each arc $a$ of $\Gcr$ on an extra inner edge from a star vertex $v_f$ to an edge-vertex $v_e$, the weight of $a$ is 1 (resp. 0) if the label-jump from the face on the left to the face on the right of $a$ is $d-\mathrm{deg}(f)$ (resp. is $d-\mathrm{deg}(f)-1$). The weight of the opposite arc $-a$ is assigned so that the sum of the two weights is 1.
\end{itemize}


It is easy to check that for any corner labeling $\cL\in\bL_G'$, the weighted orientation $\Phi'(\cL)$ is in $\bA_G'$. Indeed Conditions (L1') and (L2') (for edges) ensure that all the arc weights in $\Phi'(\cL)$ are non-negative, Remark~\ref{rk:ccw-quasi} ensures that $\Phi'(\cL)$ satisfies Conditions (A2') (for original edges) and (A3') (for edge-vertices), while Condition (L3') ensures that $\Phi'(\cL)$ satisfies Conditions (A1') (for star vertices). 
The inverse mapping $\bPhi'$ amounts to specifying the label jumps from the weighted orientation:
\begin{compactitem}
\item
For any inner arc $a$ of $\Gcr$ starting from an original inner vertex $v$, the weight of $a$ specifies the label jump from the inner face of $G'$ on the left of $a$ (considering the first half of $a$ if $a$ is on a non-special edge of $G$) to the inner face of $G'$ on the right of $a$. 
\item
For any arc $a$ from a star vertex $v_f$ to an edge-vertex $v_e$, if $e$ is non-special then the label-jump from the face on the left of $a$ to the face on the right of $a$ is 
$d-\mathrm{deg}(f)+1$; if $e$ is special and $a$ has weight $1$ (resp. $0$) then the label-jump from the face on the left of $a$ to the face on the right of $a$ is 
$d-\mathrm{deg}(f)$ (resp. $d-\mathrm{deg}(f)-1$).
\end{compactitem}
With these specifications, the sum of label-jumps in clockwise (resp. counterclockwise) order around an inner original vertex (resp. edge-vertex) is $d$, and 
for any star vertex, the jumps around it satisfy (L3'). 
Similarly as in the proof of Proposition~\ref{prop:bij-beta}, this ensures that the labels can be uniquely propagated, up to a global shift. Moreover, the weight configuration for inner arcs 
starting from an outer vertex ensure that there is a unique global shift so that the resulting labeling satisfies (L0').
By construction, the labeling also satisfies the other conditions of Definition~\ref{def:WGS-labeling}. Hence, $\bPhi'(\cA)\in\bL_G'$ for all $\cA\in\bA_G'$. 
Moreover, by construction, the two mappings $\Phi',\bPhi'$ are inverse of each other, hence bijections between $\bA_G'$ and $\bL_G'$.

\begin{remark}
In the case where $G$ is a triangulation of the 5-gon, being quasi 5-adapted is equivalent to a notion of strong irreducibility: every 3-cycle must be the boundary of a face (irreducibility), and moreover every 4-cycle must bound an adjacent pair of faces. 
As we now explain, the QS angular orientations and labelings can be simplified in this case, and lead to an orientation of the \emph{primal-dual completion} of $G$ (the map obtained by superimposing $G$ and $G^*$), and a corner labeling of $G$. An example is shown in Figure~\ref{fig:5c_structure_simple}.

Consider first the $5$-QS angular orientations of the triangulation $G$. All original inner edges of~$G$ are special, and all the star edges have weight $0$, except for the ones (of weight $1$) connecting an outer vertex $v_i$ or $v_{i-1}$ to the star vertex in the inner face incident to $(v_{i-1},v_i)$. We now assume that $G$ has at least $2$ inner vertices, so that the last property stated in Lemma~\ref{lem:rigid_qs} holds, hence the situation in the vicinity of outer edges is as shown in the left part of Figure~\ref{fig:5c_outer}(a).
 By performing the operations in Figure~\ref{fig:5c_outer}(a), and deleting the star edges, one gets an orientation of the inner edges of the primal-dual completion of $G$ characterized by the following conditions:
 \begin{itemize}
 \item
the outer vertices (which are primal vertices) have outdegree $0$, 
\item
the inner primal vertices have outdegree $5$, 
\item
the dual vertices have outdegree $2$, 
\item
the edge-vertices have outdegree $1$. 
 \end{itemize}
  
Consider now the $5$-QS labelings of the triangulation $G$. Note that each pair of inner faces of $G'$ corresponding to an inner corner of $G$ must have the same label (this follows from (L1'), the star edge $\eps$ in that corner satisfies $\si(\eps)=0$), except for the pairs corresponding to the first inner corner of $G$ at $v_i$ (the one incident to $(v_{i-1},v_{i})$).  
Hence, up to performing the label transfer shown in Figure~\ref{fig:5c_outer}(b) at the special pairs, and performing the natural label transfer for the other pairs (keeping the labels), the QS labeling of $G'$ yields a labeling of the inner corners of $G$ characterized by the following conditions:
\begin{itemize}
\item
for all $i\in[5]$, the inner corners incident to $v_i$ have label $i$,
\item
around every inner vertex, the incident corners form, in clockwise order, 5 non-empty intervals $I_1,I_2,I_3,I_4,I_5$, with all corners in $I_i$ having label $i$ for $i\in[5]$. 
\item
around every inner face, in clockwise order, there are two label-jumps equal to $2$ and one label-jump equal to $1$. 
\end{itemize}
These $5$-QS structures, the additional incarnation as woods, and a related drawing algorithm are discussed in~\cite{OB-EF-SL:5QS-drawing}.  The wood incarnation is represented in Figure~\ref{fig:5c_wood}.
\end{remark}

\fig{width=\linewidth}{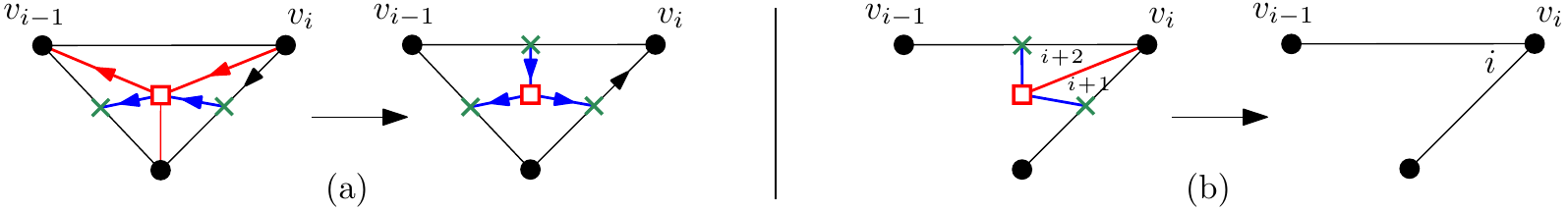}{Operations to be performed in the neighborhood of outer edges for 5-QS structures of a quasi 5-adapted triangulation. (a) Operations for orientations. (b) Operations for labelings. The configuration shown here is frozen (see Lemma~\ref{lem:rigid_qs}).}

\fig{width=\linewidth}{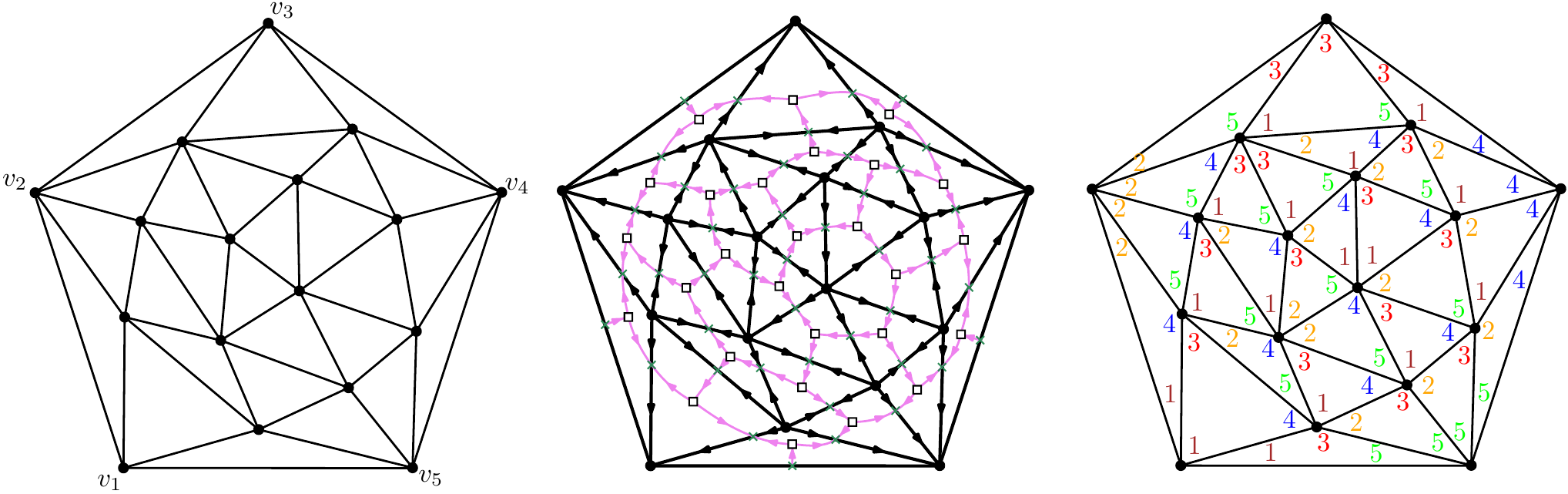}{Left: A quasi 5-adapted triangulated map $G$. Middle: the primal-dual completion of $G$, endowed with a $5$-QS angular orientations (more precisely, the orientation obtained from $5$-QS angular orientations by applying the operations of Figure~\ref{fig:5c_outer}(a), and deleting the star edges). Right: the corresponding $5$-QS labeling of $G$ (more precisely, the corner labeling of $G$ obtained from the $5$-QS labeling by applying the operations of Figure~\ref{fig:5c_outer}(b), and transferring the labels from inner faces of $G'$ to inner corners of $G$).}

\subsection{Proof of necessity in Theorem~\ref{thm:main-quasi}}\label{sec:proof-necessity-quasi}
A \emph{relaxed weighted orientation} is a weighted orientation where the arc-weights are in 
$\mathbb{Z}$. As usual, the \emph{outgoing weight} (shortly, the \emph{weight}) of a vertex $v$ is the sum of the weights of the arcs going out of $v$, and the \emph{weight} of an edge is the sum of weights of its two arcs. 
For a $d$-map $G$, with $G^+$ the angular map, a \emph{relaxed $d$-GS angular orientation} of $G$ is a relaxed weighted orientation of $G^+$ with exactly the same requirements as in Definition~\ref{def:angular}, the only difference being that negative arc-weights are allowed. By the exact same arguments as in the proof of Lemma~\ref{lem:ingoing-weight-angular} (these do not use positivity of the arc-weights), we obtain:

\begin{lemma} \label{lem:ingoing-weight-quasi-angular} 
Let $G$ be a $d$-map. Let $\cA$ be a relaxed weighted orientation of the angular map $G^+$ satisfying Condition (A1) and (A2) of $d$-GS angular orientations. Let $C$ be a simple cycle of $G$. The total weight $\inweight(C)$ of the arcs strictly inside $C$ with origin on $C$ is
$$\inweight(C)=\ell(C)-d+\sum_{a\in C}(d-\deg(f_a)),$$
where $\ell(C)$ is the length of $C$, the sum is over the arcs of $C$, and for an arc $a$ of $C$ the face incident to $a$ inside $C$ is denoted by $f_a$.
\end{lemma}

\fig{width=.9\linewidth}{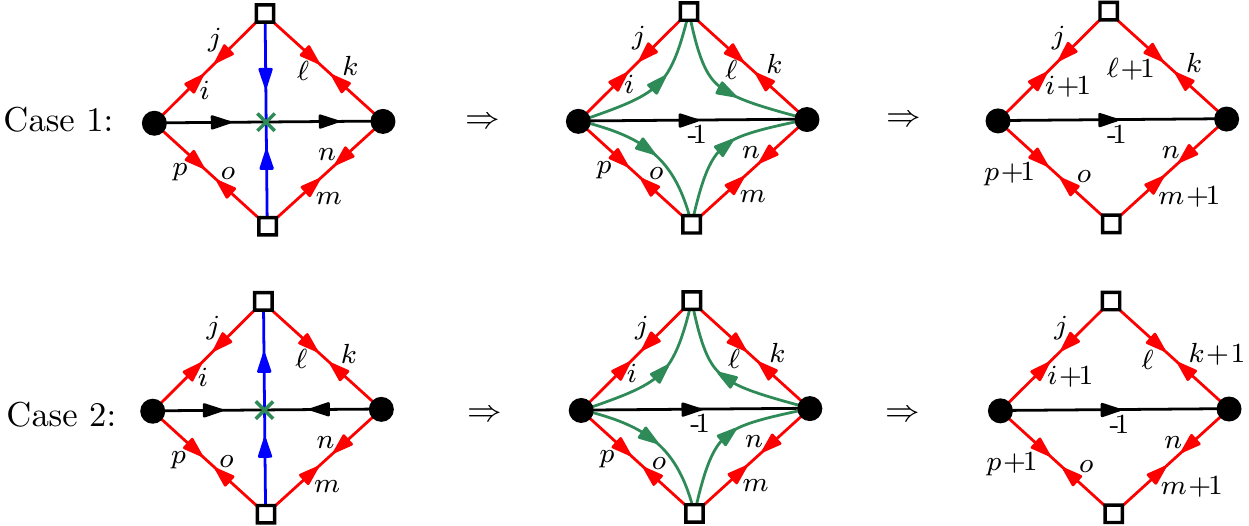}{Starting from a $d$-QS angular orientation, and applying the shown local rules at every edge-vertex, one obtains a relaxed $d$-GS angular orientation (arcs with no weight indication have weight $1$, arcs not shown on their edge have weight~$0$).}

Let $G$ be a $d$-map with no face of degree smaller than $3$, admitting a $d$-QS angular orientation $\cA'$ of $\Gcr$. Let $\cA$ be the weighted orientation of $G^+$ obtained from $\cA'$ by applying the rules of Figure~\ref{fig:local_rules_transition_other} at each (special) edge-vertex of $\Gcr$. Note that $\cA$ is a relaxed $d$-GS angular orientation of $G$, where the only arcs of negative weight are on the special (original) edges: each such edge has an arc of weight $-1$ and an arc of weight $0$. 

We now aim at proving that $G$ is quasi $d$-adapted. For a simple cycle $C$ of $G$, we call \emph{chord} an edge strictly inside $C$ joining two vertices on $C$.
Let us first consider the length of chordless cycles. Let $C$ be a chordless non-facial simple cycle of $G$. 
For a face $f$ inside $C$ incident to at least one arc of $C$, let $k$ be the number of arcs of $C$ incident to $f$, and let $s$ be the number of arcs of $G$ inside $C$ that have weight $-1$, start from $C$, and have $f$ on their right (since $C$ is chordless, such an arc ends at a vertex strictly inside $C$). 
Then, Condition (A1) and the rules of Figure~\ref{fig:local_rules_transition_other} ensure that the total contribution to $\inweight(C)$ of the edges incident to $v_f$ is at least $k(d-\deg(f))+s$ (indeed, letting $W_f$ be this contribution, $m$ be the number of vertices on $C$ adjacent to $v_f$, and $b$ be the sum of weights of arcs from $v_f$ to vertices not on $C$, we have $W_f=(m-1)(d-\deg(f))+b$, with $m\geq k+1$; and from Figure~\ref{fig:local_rules_transition_other} we see that $b\geq s$). 
This ensures that $\inweight(C)\geq\sum_{a\in C}(d-\deg(f_a))$, hence $\ell(C)\geq d$. 

Next, we consider a cycle $C$ of $G$ containing at least one vertex and having at least one chord. 
It admits a decomposition (at chords) into chordless cycles, as shown in Figure~\ref{fig:chord}. Since $C$ contains a vertex, one of these chordless components is non-facial, hence has length at least $d$. Moreover the length of $C$ is at least the length of any of its components, hence the length of $C$ is at least $d$. 

\fig{width=.6\linewidth}{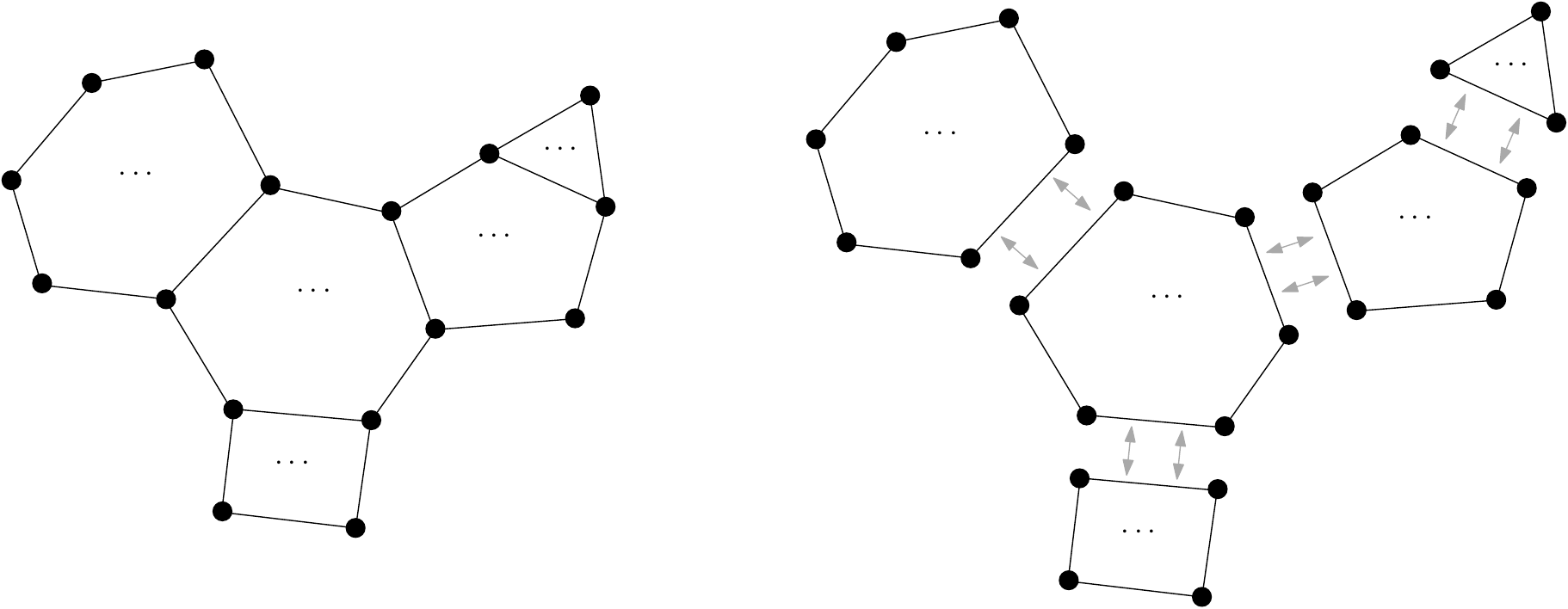}{Decomposition of a cycle into chordless components.}

Lastly, consider a non-facial simple cycle $C$ containing no vertex. The components of the chord decomposition of $C$ are facial cycles. Consider a pair of adjacent components, and let $C''$ be the contour of the union of these two faces. The length of $C''$ is at least $d-1$ since $G$ admits a $d$-QS angular orientation (so that every inner edge $e$ satisfies $\delta(e)\geq -1$). Hence the length of $C$ is at least $d-1$ as well, which concludes the proof that $G$ is $d$-adapted.

\subsection{Proof of existence in Theorem~\ref{thm:main-quasi}}\label{sec:proof-existence-quasi}
The proof follows the same lines as the proof of existence of Theorem~\ref{thm:main}, so we only sketch it and highlight the main differences. 
For convenience in view of later proofs, we assume that $G$ has no inner edge connecting two outer vertices. The only quasi $d$-adapted maps having such edges are easily checked to be those with no inner vertex, and either one inner edge (separating two inner faces whose degrees add up to $d+2$, this case is actually $d$-adapted) 
or two inner edges (separating a chain of three inner faces $f_1,f_2,f_3$, where $\deg(f_1)=\deg(f_3)=3$, and $\deg(f_2)=d-2$). The existence of a $d$-QS angular orientation is 
then readily checked for those cases; see Figure~\ref{fig:dQS_with_chords} for the case $d=5$ with 3 inner faces.

\fig{width=0.6\linewidth}{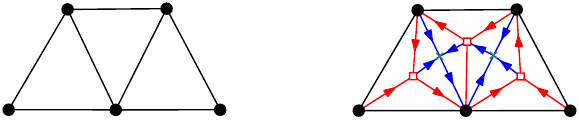}{Quasi 5-adapted map with 3 inner faces and no inner vertex (left drawing), endowed with its unique 5-QS angular orientation (right drawing).}
 
For the base case $d=3$, since we disallow here inner faces of degree smaller than $3$, every inner edge in a quasi $3$-adapted map $G$ 
satisfies $\delta(e)=1$. Hence there is no special edge, so being quasi 3-adapted is the same as being 3-adapted (here, being a simple triangulation), and the existence of 3-GS angular orientations is already established. The same remark actually holds for $d=4$ (assuming no inner face of degree less than $3$, the first value of $d$ for which there are quasi $d$-adapted maps that are not $d$-adapted is $d=5$).

For $d\geq 3$, let $G$ be a quasi $(d+1)$-adapted map with no inner face of degree less than $3$, and no inner edge connecting two outer vertices. We construct $\overline{G}_{\bullet}$ exactly as in the beginning of Section~\ref{sec:induction-step} (insertion of a copy of $X_d$ in each $(d+1)$-face, including the outer one, and re-rooting at a $d$-face in the outer copy of $X_d$). 
There is no non-facial cycle of length less or equal to $d$ in $G$, and the insertions of the copies of $X_d$ do not create non-facial cycles of length less or equal to $d$ either (by the same arguments as in Claim~\ref{claim:still-adapted}). Hence $\overline{G}_{\bullet}$ is $d$-adapted, so it admits a $d$-GS angular orientation $\overline{\cA}_{\bullet}$. Let $G_{\bullet}$ be the $(d+1)$-map obtained from $\overline{G}_{\bullet}$ by deleting the outer copy of $X_d$. In the orientation induced by $\overline{\cA}_{\bullet}$ on $G_{\bullet}^+$, as before the star edges incident to star vertices of degree $d$ have weight $0$, but now we also have inner original edges of weight $0$, which are precisely the special edges. 
Let $\widehat{G}_{\bullet}$ be the map obtained from $G_{\bullet}^+$ by deleting the star vertices of degree $d$ and their incident edges, and deleting 
the special (original) edges. Let $\cB_{\bullet}$ be the weighted orientation induced by $\overline{\cA}_{\bullet}$ on $\widehat{G}_{\bullet}$. 

The same arguments as in Claim~\ref{claim:co-accessible} ensure that $\cB_\bullet$ is co-accessible (indeed, the crucial property in that proof is that there is no $d$-cycle with at least one vertex inside). We can thus consider a spanning tree $T$ of $\widehat{G}_{\bullet}$ 
such that it contains all the outer edges of $G_{\bullet}$ except $(v_1,v_2)$, and for every inner edge of $T$, the arc traversed in the direction away from the outer vertices has positive weight. 
Then we use the transfer process of Section~\ref{sec:induction-step} (described after Claim~\ref{claim:co-accessible}) 
to obtain a weighted orientation $\widehat{\cA}_{\bullet}^+$ on $\widehat{G}_{\bullet}^+$ 
(the process can be formulated for any co-accessible orientation endowed with a compatible spanning tree). 
By the same arguments as in Section~\ref{sec:induction-step}, in $\widehat{\cA}_\bullet$ the weight of every inner vertex of $\widehat{G}_\bullet$ has increased by $1$ compared to $\cB_\bullet$, and the weight of every inner edge of $\widehat{G}_\bullet$ has decreased by $1$ compared to $\cB_\bullet$, while all the vertices and edges of $\widehat{\cA}_{\bullet}^+$ which were not in $\widehat{G}_{\bullet}$ have weight $1$. 
 
\fig{width=\linewidth}{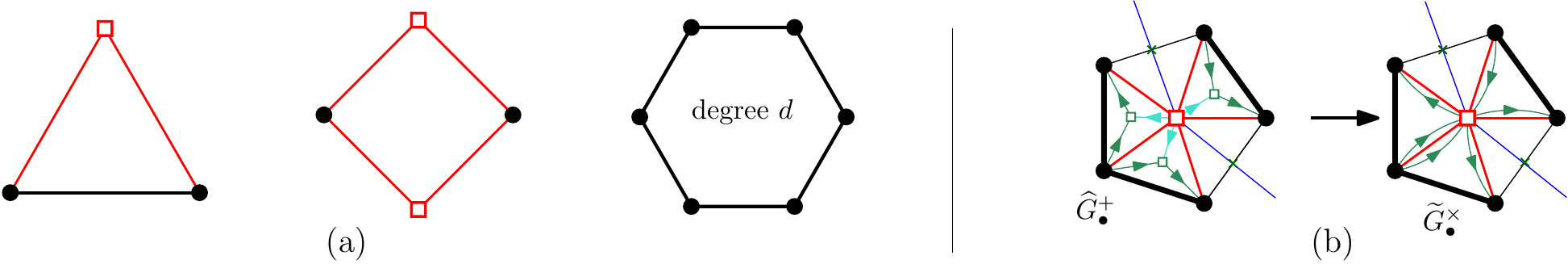}{(a) The 3 types of inner faces in the map $\widehat{G}_\bullet$ (the second type, arising from special edges of $G$, does not appear when $G$ is $(d+1)$-adapted). (b) Contracting the small edges incident to a star vertex of $G_{\bullet}^+$.}

A notable difference with Section~\ref{sec:induction-step} is that there are now $3$ types of inner faces in $\widehat{G}_{\bullet}$, as shown in Figure~\ref{fig:face_three_types}(a). 
We call \emph{small triangles} the faces of the first type, that is, the inner faces of $\widehat{G}_{\bullet}$ made of two vertices of $G_{\bullet}$
and a star vertex of $G_{\bullet}^+$ of degree less than $d$. We call \emph{small edges} and \emph{small vertices} of $\widehat{G}_{\bullet}^+$ the vertices and 
edges that have been added inside the small triangle. We call \emph{edge-vertices} and \emph{extra edges} of $\widehat{G}_{\bullet}^+$ the vertices and edges that have been added inside the (quadrangular) faces of the second type in Figure~\ref{fig:face_three_types}(a). 

We then apply to $\widehat{G}_{\bullet}^+$ the contraction process indicated in Figure~\ref{fig:face_three_types}(b): contracting all the small edges incident to star vertices of $G_{\bullet}^+$. We obtain a map $\widetilde{G}_\bullet^\times$ closely related to $G_\bullet^\times$: it is $G_\bullet^\times$ except that for every star edge $e$ incident to a star vertex of degree less than $d$ in $\Gcr$ one has $t\in\{0,1,2\}$ additional small edges with the same endpoints as $e$. Precisely, we have $t=t_{\ell}+t_r$, where $t_{\ell}\in\{0,1\}$ (resp. $t_{r}\in\{0,1\}$) indicates whether there is an edge $e_\ell$ (resp. $e_r$) on the left (resp. right) side of $e$ resulting from the contraction of a small edge. Note that $t=2-s(e)$, with $s(e)$ the number of special edges among the two original edges that delimit the inner corner of $G_{\bullet}$ associated to $e$. 

We then let $\cA_\bullet'$ be the weighted orientation of $G_\bullet^\times$ obtained from $\widehat{\cA}_{\bullet}$ by this contraction process. That is to say, for every star edge $e$ of $G_\bullet^\times$ incident to a star vertex $v_f$ of degree less than $d$, the weights of the arcs of $e$ are taken as the sum of weights of the corresponding arcs in $e,e_{\ell},e_r$ (whenever $e_\ell,e_r$ exist). 
From the preceding, the weight of such an edge $e$ is $d+1-\mathrm{deg}(f)-s(e)$. This formula also holds for an edge $e$ that is a star edge of $G_\bullet^\times$ 
incident to a star vertex $v_f$ of $G_\bullet^\times$ of degree $d$. Indeed,
the weight in $\cA_{\bullet}'$ of such an edge is $1$, and it has $s(e)=0$ (this is ensured by the fact that there is no inner face of degree less than $3$, so that the $(d+1)$-map $G_\bullet$ has no special edge incident to a $d$-face). Since the weights of inner vertices and inner edges of $G_{\bullet}$ have not changed compared to $\cB_\bullet$, we conclude that the orientation $\cA_\bullet'$ (on $G_\bullet^\times$) satisfies all conditions of Definition~\ref{def:quasiAngular}, except possibly Condition (A0').



A \emph{border} star vertex of $G_\bullet^\times$ (or $G_\bullet^+$) is a star vertex in a face of $G_\bullet$ that is incident to at least one outer edge of $G_\bullet$. 
A \emph{border} edge-vertex of $G_\bullet^\times$ is an edge-vertex that is adjacent to an outer vertex and to a border star vertex. 
To show that $\cA_{\bullet}'$ can be modified to satisfy Condition~(A0'), 
 we prove the following:

\begin{lemma}\label{lem:conditions_A_prime_quasi}
In the orientation $\cA_\bullet'$, any arc starting from an outer vertex has weight $0$, except possibly if it ends at a border star vertex or at a border edge-vertex. 
Moreover, any arc starting from a border star vertex has weight $0$, except possibly if it ends at an outer vertex or at a border edge-vertex. 
Finally, for every border star vertex $v_f$, the outer vertices adjacent to $v_f$ are consecutive along the outer face.  
\end{lemma}
\begin{proof}
We let $\cR_\bullet$ be the weighted orientation of $\Gb^+$ obtained from $\cA_\bullet'$ 
by applying the operations of Figure~\ref{fig:local_rules_transition_other} at every star vertex. As in Section~\ref{sec:proof-necessity-quasi}, 
$\cR_\bullet$ is a relaxed orientation that satisfies Conditions (A1) and (A2) of Definition~\ref{def:angular}. Thus, letting $C$ be the outer cycle of $\Gb$, Lemma~\ref{lem:ingoing-weight-quasi-angular} applies: in $\cR_{\bullet}$ we have $\inweight(C)=\sum_{i=1}^{d+1}(d+1-\deg(f_i))$, where $f_i$ is the inner face incident to the outer edge $(v_{i-1},v_i)$. 
Let $S$ be the total weight (in $\cR_{\bullet}$) of arcs from border star vertices to inner vertices. 
Since the arcs on star edges have non-negative weight in $\cR_\bullet$, and since Conditions (A1) and (A2) hold, the contribution to $\inweight(C)$ of the arcs from outer vertices to border star vertices is at least $S+\sum_{i=1}^{d+1}(d+1-\deg(f_i))$, with equality if and only if for every border star vertex $v_f$, the outer vertices adjacent to $v_f$ are consecutive along the outer face. 
Hence, the total weight (in $\cR_\bullet$) of the arcs that start at an outer vertex and do not end at a border star vertex is at most $-S$. 
Hence, if we let $W$ be the total weight (in $\cR_\bullet$) of arcs that either start at an outer vertex and do not end at a border star vertex, or 
start at a border star vertex and end at an inner original vertex, then we have $W\leq 0$. 

We now look at the contributions to $W$ 
 in the intermediate step (with green edges) in Figure~\ref{fig:local_rules_transition_other}. Recall that we assume $G$ has no inner edge connecting two outer vertices,  and the same  holds for $\Gb$ (as the insertions of copies of $X_d$ can not create such edges), hence every edge-vertex $v_e$ has at most one outer neighbor.  
Then it is not difficult to see that every edge-vertex $v_e$ adjacent to an outer vertex or to a border star vertex yields a non-negative contribution to $W$
(contribution given by the special edge of weight $-1$ and the 4 green edges associated to $v_e$), 
and that the contribution is strictly positive whenever $v_e$ does not end at an outer vertex or at a border star vertex. 
On the other hand, for every arc $a$ in $\cA_\bullet'$ starting at an outer vertex, 
on a non-special original edge or on a star edge, and not ending at a border star vertex, the contribution of $a$ to $W$ is given by its weight in $\cA_{\bullet}'$. 
Similarly, for every arc $a$ in $\cA_\bullet'$ starting at a border star vertex and ending at an original inner vertex, the contribution of $a$ to $W$ is given by its weight in $\cA_{\bullet}'$. 
Since $W\leq 0$, and since all possible contributions are non-negative (grouping those associated to an edge-vertex as explained above), we conclude that the positive ones do not exist. 

Finally, the above inequality (for $S+\sum_{i=1}^{d+1}(d+1-\deg(f_i))$) has to be tight, so that the last statement holds.
\end{proof}

To turn $\cA_\bullet'$ into a $(d+1)$-QS angular orientation, we modify some arc weights in the neighborhood of the outer face, keeping the weights of edges and of inner vertices unchanged, as follows. 
For each edge $(v_i,v_{f_i})$, with $f_i$ the inner face of $G_\bullet$ incident to $(v_{i-1},v_i)$,  
we put all the weight on the arc out of $v_i$; and if $f_{i+1}\neq f_i$, we put all the weight of the edge $(v_i,v_{f_{i+1}})$ on the arc ending at $v_i$. 
For each border edge-vertex $v_e$ we choose the unique outgoing edge of $v_e$ as the one leading to the adjacent outer vertex $v$, unless
the next neighbor after $v$ in counterclockwise order around $v_e$ is a border star vertex $v_f$, in which case the unique outgoing edge of $v_e$ is the one leading to $v_f$.  
With these modifications, and given Lemma~\ref{lem:conditions_A_prime_quasi}, one easily checks that all conditions of Definition~\ref{def:quasiAngular} are now satisfied.

Finally, by the exact same arguments as those at the end of Section~\ref{sec:induction-step}, the $(d+1)$-GS angular orientation $\cA_\bullet'$ on $G_\bullet^\times$ yields a $(d+1)$-GS angular orientation $\cA'$ on $\Gcr$; see Figure~\ref{fig:existence-proof-construct}(b). Note that the special edges do not interfere with this last step,  
since they are not incident to inner faces of $G_\bullet$ of degree $d$, hence there is no special edge of $G_\bullet$ on a boundary of a copy of $X_d$, 
and there is no special edge inside a copy of $X_d$ either.



%% file: conclusion.tex
As shown in Sections \ref{sec:classical} and \ref{sec:bip-classical} many of the combinatorial structures on plane graphs that have been used in the literature to define  graph drawing algorithms are grand-Schnyder structures for special classes of plane graphs. 
Hence, grand Schnyder woods offer a framework for extending known drawing algorithms to new classes of plane graphs, or to look for brand new algorithms in the spirit of the known ones. 
Following this perspective, in~\cite{OB-EF-SL:4-GS-drawing} we present extensions of several graph drawing algorithms (straight-line drawing and orthogonal drawing) related to $4$-GS structures \cite{Barriere-Huemer:4-Labelings-quadrangulation,OB-EF:Schnyder,Fu07b,He93:reg-edge-labeling}, and establish some connections between them. 
And in~\cite{OB-EF-SL:5QS-drawing} we use $5$-QS structures on quasi 5-adapted triangulations of the pentagon (closely related to 5-connected triangulations)  to define a face-counting  straight-line drawing algorithm for 5-connected triangulations. 
In this context, the algorithm established in the present article to compute grand-Schnyder structures in linear time is valuable in order to obtain graph drawing algorithms of linear time complexity.

Several instances of grand-Schnyder structures have also been used to define bijections between classes of planar maps and classes of trees~\cite{albenque2013generic,Bernardi-Fusy:dangulations,Fu07b,FuPoScL,Poulalhon:triang-3connexe+boundary,Schaeffer:these}. Thus the general framework established in the present article offers the promise of extending these bijections (using the ``master-bijection'' approach developed in \cite{Bernardi-Fusy:dangulations,OB-EF:girth}). In particular, this approach seems well suited to tackle the $d$-irreducible maps enumerated in~\cite{bouttier2014irreducible}.

Lastly, the original impetus for the definition of Schnyder woods were results about the dimension of the incidence poset of planar graphs~\cite{Schnyder:wood1}. We wonder if some generalization of these results can be deduced from the existence of grand-Schnyder structures.
 